\documentclass{memo-l}

\usepackage[T1]{fontenc}
\usepackage[utf8]{inputenc}
\usepackage{amsmath,amsfonts,amssymb,amsthm,stackrel,leftidx}
\usepackage{mathtools,bbold}
\usepackage{csquotes}

\usepackage{tikz, tikz-cd}
\usetikzlibrary{decorations.markings}
\usetikzlibrary{shapes.geometric}
\usetikzlibrary{decorations.pathmorphing}
\usetikzlibrary{calc}

\usepackage{comment}

\usepackage{enumitem}
\setlist[enumerate,1]{label=(\arabic*)}

\usepackage[bookmarksdepth=2]{hyperref}

\newtheorem*{Theorem*}{Theorem}
\newtheorem*{Conjecture*}{Conjecture}
\newtheorem*{Corollary*}{Corollary}
\newtheorem{IntroTheorem}{Theorem}
\newtheorem*{IntroTheorem*}{Theorem}
\newtheorem{IntroConjecture}[IntroTheorem]{Conjecture}
\newtheorem*{IntroConjecture*}{Conjecture}

\newtheorem{Theorem}{Theorem}[chapter]

\newtheorem{Proposition}[Theorem]{Proposition}
\newtheorem{Corollary}[Theorem]{Corollary}
\newtheorem{Lemma}[Theorem]{Lemma}
\newtheorem{Conjecture}[Theorem]{Conjecture}

\theoremstyle{definition}
\newtheorem{Definition}[Theorem]{Definition}

\theoremstyle{remark}
\newtheorem{Remark}[Theorem]{Remark}

\newtheorem{Example}[Theorem]{Example}
\newtheorem{Examples}[Theorem]{Examples}

\numberwithin{section}{chapter}
\numberwithin{equation}{chapter}

\usepackage{suffix}  
\usepackage{ifthen}
\usepackage{xspace}

\newif\ifdgcal
\dgcaltrue    

\newcommand\basecat{{\cat V}}       
\newcommand\refined{refined }
\let\simeq\cong
\newcommand\dg{DG\xspace}
\newcommand\Dg{DG\xspace}

\newcommand\isoto{\xrightarrow{\sim}}

\mathchardef\mhyphen="2D    

\newcommand\kk{{\mathbb{k}}}            
\newcommand\NN{{\mathbb{N}}}
\newcommand\ZZ{{\mathbb{Z}}}
\newcommand\CC{{\mathbb{C}}}

\newcommand\cat\mathcal             
\newcommand\bicat\mathbf            

\DeclareMathOperator{\obj}{Ob}
\newcommand\Hom{\operatorname{Hom}}
\newcommand\End{\operatorname{End}}

\newcommand\id{\operatorname{id}}

\let\Id\id

\newcommand\Tr{\operatorname{Tr}}                   

\newcommand\Cat{{\bicat{Cat}}}                      
\newcommand\catfVect{\mathrm{Vect^f_{\kk}}}         
\newcommand\catgrVect{\mathrm{gr\mhyphen Vect_{\kk}}} 

\DeclareMathOperator\Kar{Kar}                       
\newcommand\addsym{\sym}                            

\newcommand\Ind{\operatorname{Ind}}     
\newcommand\Res{\operatorname{Res}}     
\newcommand\dual{\vee}                  
\DeclareMathOperator\Sym{Sym}           
\newcommand\SymGrp[1]{S_{#1}}           
\newcommand\sgn{\operatorname{sgn}}     
\newcommand\triv{{\mathrm{triv}}}       
\newcommand\sign{{\mathrm{sign}}}       
\newcommand\SL{\mathrm{SL}}             
\newcommand\stimes{\vee}                

\newcommand{\lfrp}{\mathrm{lfrp}}                   
\newcommand{\kc}{\mathrm{kc}}                       
\newcommand{\mdg}{\mathrm{dg}}                       
\newcommand{\Dr}{\mathrm{Dr}}                       
\newcommand{\old}{\mathrm{old}}                       

\newcommand\DGCat{{\bicat{dgCat}}}                  
\newcommand\DGCatone{{\bicat{dgCat}^1}}             
\newcommand\DGCatdg{{\bicat{dgCat}^{\mdg}}}           
\newcommand\DGModCat{{\bicat{dgModCat}}}            
\newcommand\DGBiMod{{\bicat{dgMor}}}                
\newcommand\HoDGCat{{\bicat{Ho}(\DGCat)}}           
\newcommand\HoDGCatone{{\bicat{Ho}(\DGCatone)}}     
\newcommand\MoDGCat{{\bicat{Mor}(\DGCat)}}          
\newcommand\MoDGCatone{{\bicat{Mor}(\DGCatone)}}    
\def\EnhCat{{\bicat{EnhCat}}}                       
\def\EnhCatone{{\bicat{EnhCat}^1}}                  
\def\EnhCatKC{{\bicat{EnhCat}_{\kc}}}                
\def\EnhCatKCdg{{\bicat{EnhCat}_{\kc}^{\mdg}}}         
\newcommand{\PDGCat}{\mathcal{P}dg\mathcal{C}at}    

\ifdgcal
  \DeclareMathOperator{\modd}{\mathcal{M}\mkern-1.5mu\mathit{o\mkern-1mud}}                    
  \DeclareMathOperator{\hperf}{\mathcal{H}\mathit{perf}}                 
  \DeclareMathOperator{\perf}{{\mathcal{P}\mkern-1.5mu\mathit{erf}}}
  \DeclareMathOperator{\hproj}{\mathcal{P}}
  \DeclareMathOperator{\pretriag}{{\mathcal{P}\mkern-1.5mu\mathit{re\mhyphen}\mkern-2mu\mathcal{T}\mkern-3mu\mathit{r}}}
  \DeclareMathOperator{\DGFun}{{\mathcal{DGF}\mkern-1.5mu\mathit{un}}}
  \DeclareMathOperator{\acyc}{\mathcal{A}\mathit{c}}
\else
  \DeclareMathOperator{\hperf}{Hperf}                 
  \DeclareMathOperator{\modd}{Mod}                    
  \DeclareMathOperator{\perf}{Perf}
  \DeclareMathOperator{\hproj}{P}
  \DeclareMathOperator{\pretriag}{Pre\mhyphen Tr}
  \DeclareMathOperator{\DGFun}{DGFun}
  \DeclareMathOperator{\acyc}{Ac}
\fi

\DeclareMathOperator{\bihperf}{\bicat{Hperf}}               
\DeclareMathOperator{\bipretriag}{\bicat{Pre\mhyphen Tr}}   

\newcommand\catdgfVect{\mathrm{dg\mhyphen Vect^f_{\kk}}} 

\newcommand\sym{{\cat{S}}}                              

\newcommand\numGgp[1]{\mathrm{K}_0^{\mathrm{num}}\ifthenelse{\equal{#1}{}}{}{(#1)}}      

\newcommand\hh{\mathrm{HH}_{\bullet}}      %

\newcommand\modbar{{\overline{\modd}}}              
\newcommand\perfbar{{\overline{\perf}}}              
\newcommand\rightmod[1]{{\modd\mkern-2mu\mhyphen #1}}   
\newcommand\leftmod[1]{{#1\mhyphen\mkern-2.5mu\modd}}   
\newcommand\bimod[2]{{#1\mhyphen\mkern-2.5mu\modd\mkern-2mu\mhyphen #2}}   
\newcommand\bimodbar[2]{{#1\mhyphen\modbar\mhyphen #2}} 

\newcommand\halg[1]{H_{#1}}              
\newcommand\chalg[1]{\underline{H}_{#1}} 
\newcommand\falg[1]{F_{#1}}              
\newcommand\cfalg[1]{\underline{F}_{#1}} 

\newcommand\hcat[1]{\bicat{H}_{#1}}                     
\WithSuffix\newcommand\hcat*[1]{\bicat{H}'_{#1}}        
\newcommand\hcatadd[1]{\bicat{H}^{\mathrm{add}}_{#1}}                       
\WithSuffix\newcommand\hcatadd*[1]{\bicat{H}^{\mathrm{add}\prime}_{#1}}     
\newcommand\monhcat[1]{\underline{\bicat{H}}_{#1}}                          
\newcommand\monfcat[1]{\underline{\bicat{F}}_{#1}}
\WithSuffix\newcommand\monhcat*[1]{\underline{\bicat{H}}'_{#1}}             
\newcommand\PP{{\mathsf{P}}}            
\newcommand\QQ{{\mathsf{Q}}}            
\newcommand\RR{{\mathsf{R}}}            
\newcommand\hunit{{\mathbb{1}}}         
\newcommand\ho{N}                       
\WithSuffix\newcommand\ho*{N'}          

\newcommand\fcat[1]{\bicat{F}_{#1}}                   
\WithSuffix\newcommand\fcat*[1]{\bicat{F}'_{#1}}      
\newcommand\fcatadd[1]{\bicat{F}^{\mathrm{add}}_{#1}}                         
\WithSuffix\newcommand\fcatadd*[1]{\bicat{F}^{\mathrm{add}\prime}_{#1}}       
\newcommand\starmap[1]{\star_{#1}}                      

\tikzset{dot/.style={circle, fill, inner sep=0, minimum size=4pt}}
\tikzset{serre/.style={star, fill, star points=10, star point ratio = 2, scale=0.4}}
\tikzset{up/.style={color=blue}}
\tikzset{down/.style={color=red}}
\tikzset{cc/.style={color=green!50!black}}
\tikzset{desc/.style={fill=white}}
\tikzset{sym/.style={rectangle, draw, minimum width=1cm}}
\tikzset{many/.style={very thick}}
\newcommand\dcross[5]{\phi^{#1,#2}_{#3,#4,#5}}      
\newcommand\sdcross[3]{\psi^{#1}_{#2,#3}}           

\newcommand\Spec{\operatorname{Spec}}   
\newcommand\as[1]{\mathbb{A}^{#1}}      
\newcommand\ps[1]{\mathbb{P}^{#1}}      
\newcommand\sheaf\mathcal               
\newcommand\sO{{\sheaf O}}              
\newcommand\catCoh[1]{\mathrm{Coh}(#1)} 
\newcommand\catDQCoh[1]{\mathrm{D}_{\mathrm{qc}}(#1)}               

\newcommand\catDbCoh[1]{\mathrm{D}^{\mathrm{b}}_{\mathrm{coh}}(#1)} 
\newcommand\catDGCoh[1]{\cat{I}(#1)}                                
\newcommand\catD{\mathrm{D}}                                 
\newcommand\catDc{\mathrm{D}_{\mathrm{c}}}                   
\newcommand\Hzero{\mathrm{H}^{\mathrm{0}}}                   

\DeclareMathOperator{\cok}{Coker}

\DeclareMathOperator{\homm}{Hom}
\DeclareMathOperator{\cone}{Cone}
\DeclareMathOperator{\opp}{{opp}}

\DeclareMathOperator{\lder}{\bf L}

\DeclareMathOperator{\ldertimes}{\overset{\lder}{\otimes}}

\DeclareMathOperator{\action}{act}
\DeclareMathOperator{\unit}{unit}
\DeclareMathOperator{\counit}{counit}

\DeclareMathOperator{\bimodapx}{{\underline{Apx}}}
\DeclareMathOperator{\tensorfn}{{\underline{\otimes\vphantom{p}}}}  

\def\F{{\mathcal{F}}}

\def\N{{\mathcal{N}}}

\def\T{{\mathcal{T}}}

\def\aM{\leftidx{_{a}}{M}{}}

\def\cN{\leftidx{_{c}}{N}{}}
\def\aMb{\leftidx{_{a}}{M}{_{b}}}

\def\A{{\cat{A}}}
\def\B{{\cat{B}}}
\def\C{{\cat{C}}}
\def\D{{\cat{D}}}
\def\I{{\cat{I}}}
\def\Aopp{{\A^{\opp}}}
\def\Bopp{{\B^{\opp}}}

\def\biA{{\bicat{A}}}
\def\biB{{\bicat{B}}}
\def\biC{{\bicat{C}}}
\def\biD{{\bicat{D}}}
\def\biI{{\bicat{I}}}

\def\hperfA{{\hperf(\A)}}
\def\hperfB{{\hperf(\B)}}

\def\modk{{\rightmod{\kk}}}
\def\modA{{\rightmod \A}}
\def\modB{{\rightmod \B}}
\def\modC{{\rightmod \C}}

\def\Amod{{\leftmod \A}}

\def\AmodA{{\bimod\A\A}}
\def\AmodB{{\bimod\A\B}}

\def\BmodA{{\bimod\B\A}}
\def\AmodC{{\bimod\A\C}}

\def\BmodC{{\bimod\B\C}}

\def\CmodA{{\bimod\C\A}}
\def\CmodB{{\bimod\C\B}}

\def\AbimA{{\A\mhyphen \A}}
\def\AbimB{{\A\mhyphen \B}}

\def\AmodbarA{{\bimodbar\A\A}}
\def\AmodbarB{{\bimodbar\A\B}}
\def\AmodbarC{{\bimodbar\A\C}}

\def\BmodbarC{{\bimodbar\B\C}}

\def\Bperf{{\B\mhyphen\perf}}
\def\hprojA{{\hproj(\A)}}

\def\perfA{{\perf(\A)}}
\def\perfB{{\perf(\B)}}

\def\bartimes{\mathrel{\overline{\otimes}}}
\def\Ainfty{{A_{\infty}}}
\def\degzero{{\mathrm{deg.\,0}}}

\newcommand\symbc[1]{{\sym^{#1}\basecat}}           
\newcommand\addsymbc[1]{\addsym^{#1}\basecat}       
\def\symbcn{{\symbc{\ho}}}
\def\symbcnmone{{\symbc{\ho-1}}}
\def\symbcnmtwo{{\symbc{\ho-2}}}

\def\symbcnpone{{\sym^{\ho+1}\basecat}}

\makeindex
\begin{document}

\frontmatter

\title{The Heisenberg category of a category}

\author{Ádám Gyenge}
\address{Department of Algebra and Geometry, Institute of Mathematics,  
	Budapest University of Technology and Economics, 
	M\H{u}egyetem rakpart 3, 1111, 
	Budapest, Hungary}
\email{Gyenge.Adam@ttk.bme.hu}

\author{Clemens Koppensteiner}
\address{Mathematical Institute, University of Oxford, Andrew Wiles Building, OX2 6GG, Oxford UK}

\email{Clemens.Koppensteiner@maths.ox.ac.uk}

\author{Timothy Logvinenko}
\address{School of Mathematics, Cardiff University, Senghennydd Road, CF24 4AG, Cardiff UK}
\email{LogvinenkoT@cardiff.ac.uk}

\date{}




\begin{abstract}
	Starting with a $\kk$-linear or \dg category 
	admitting a (homotopy) Serre functor, we construct a $\kk$-linear or \dg
	$2$-category categorifying the Heisenberg algebra of the numerical 
	$K$-group of the original category. 
	We also define a $2$-categorical analogue of the Fock space
	representation of the Heisenberg algebra.  Our construction
	generalises and unifies various categorical Heisenberg algebra actions
	appearing in the literature. In particular, we give a full categorical
	enhancement of the action on derived categories of symmetric quotient
	stacks introduced by Krug, which itself categorifies a Heisenberg
	algebra action proposed by Grojnowski.
\end{abstract}

\maketitle

\tableofcontents

\mainmatter

\chapter{Introduction}

The Heisenberg algebra of a lattice is a much investigated object
originating in quantum theory. It appears in many areas of
mathematics, including the representation theory of affine Lie algebras.
For a smooth projective surface,
Grojnowksi and Nakajima \cite{grojnowski1995instantons,
	nakajima1997heisenberg} identified the total cohomology of 
its Hilbert schemes of points with the Fock space representation 
of the Heisenberg algebra associated to the cohomology of the surface. 
As proposed by Grojnowksi \cite[footnote~3]{grojnowski1995instantons} 
and proved by Krug \cite{krug2018symmetric}, this occurs more
generally for the symmetric quotient stacks\index{symmetric quotient
stack} of any smooth projective
variety on the level of K-theory, and, more fundamentally, on the
level of derived categories of coherent sheaves.

On the other hand, Khovanov \cite{khovanov2014heisenberg}
introduced a categorification of the infinite Heisenberg algebra
associated to the free boson or, equivalently, a rank 1 lattice.
It used a graphical construction involving planar diagrams. A related
graphically defined category was constructed by Cautis and Licata
\cite{cautis2012heisenberg} for ADE type root lattices. Both of these
Heisenberg categories admit categorical representations on
categorifications of the corresponding Fock spaces. 
They were much studied since
\cite{gal2016hopf,rosso2017general,Brundan:DefinitionOfHeisenbergCategory,brundan2018degenerate,Savage:FrobeniusHeisenberg}.

In this paper we unify and generalise many of these constructions. We start with a $\kk$-linear and $\Hom$-finite category $\basecat$ equipped with 
a Serre functor $S$. That is, $S$ is a $\kk$-linear autoequivalence 
equipped with natural isomorphisms \begin{equation}\label{eq:intro:serre}
	\eta_{a,b}\colon\Hom_\basecat(b, Sa)^* \cong \Hom_\basecat(a,b)
	\quad \quad \forall\; a,\, b \in \basecat.
\end{equation}
A typical example is the derived category $\catDbCoh{X}$ of a smooth and proper variety $X$ with $S = (-) \otimes \omega_X[\dim X]$.
We further allow $\basecat$ to be graded or a \dg category. In the
latter case, $S$ only needs to be a homotopy Serre functor.
The following summarises our main results:

\begin{Theorem*}[Summary of the main results]
	There exists a \emph{Heisenberg $2$-category} $\hcat\basecat$ of $\basecat$
	defined using a graphical calculus, together with 
	a \emph{Fock space representation} on the categories of 
	$\SymGrp \ho$-equivariant objects in $\basecat^{\otimes \ho}$.
\end{Theorem*}

We now make this statement more precise.

\section{Heisenberg algebras of categories}

The numerical Grothendieck group\index{numerical Grothendieck group} $\numGgp{\basecat}$ 
has a bilinear pairing $\chi$ given by the dimension of 
$\Hom_{\basecat}(a,b)$ or its Euler characteristic in the graded or \dg case, 
cf.~Section~\ref{subsec:prelim_grothendieck_group}. 
If $\chi$ is symmetric, we can define a Heisenberg algebra 
$\halg\basecat$ with generators
$$ \left\{ a_{b}(n) \right\}_{b \in \numGgp{\basecat},\; n \in \ZZ \setminus \{0\}}$$
and relations
\[
[a_{b}(m),\, a_{c}(n)] = m\delta_{m,-n}\langle b,\, c\rangle_\chi.
\]
However, in practice $\chi$ is rarely symmetric,
cf.~Example~\ref{ex:nonsymmetric_pairing}\footnote{In particular, in \cite[Corollary~1.5]{krug2018symmetric} the algebra $\mathsf{H}_{\mathrm{K}(X)}$ is a priori not well-defined for a general smooth and projective variety $X$.  We are thankful to Pieter Belmans for this remark.}.

As observed in \cite{khovanov2014heisenberg, cautis2012heisenberg}, it can be more convenient to choose a different set of generators
\[ \left\{  p_{b}^{(n)}, q_{b}^{(n)} \right\}_{b \in
	\numGgp{\basecat},\; n \in \ZZ_{{ \geq} 0}} \]
and a different set of relations 
\begin{equation}\label{eq:heisrelintro00}
	p_{b}^{(0)} = q_{b}^{(0)} = 1
\end{equation}
\begin{equation}\label{eq:heisrelintro0}
	p_{a+b}^{(n)} = \textstyle\sum_{k=0}^{n} p_{a}^{(k)}p_{b}^{(n-k)}
	\quad\text{and}\quad 
	q_{a+b}^{(n)} = \textstyle\sum_{k=0}^{n}  q_{a}^{(k)}q_{b}^{(n-k)},
\end{equation}
\begin{equation}\label{eq:heisrelintro1}
	p_{a}^{(n)}p_{b}^{(m)} = p_{b}^{(m)}p_{a}^{(n)}
	\quad\text{and}\quad
	q_{a}^{(n)}q_{b}^{(m)} = q_{b}^{(m)}q_{a}^{(n)},
\end{equation}
\begin{equation}\label{eq:heisrelintro2}
	q_{a}^{(n)}p_{b}^{(m)} = 
	\textstyle\sum_{k = 0}^{\min(m,n)} s^k \left( \langle a, b
	\rangle_{\chi}\right)\, p_{b}^{(m-k)}q_{a}^{(n-k)},
\end{equation}
and $s^k(n) = \dim \Sym^k
\kk^n$. 
These relations are consistent even when $\chi$ is non-symmetric. 
Thus the above defines the Heisenberg algebra $\halg\basecat$ of any $\basecat$.
We prove in Corollary~\ref{cor:halg_iso_to_symmetric} that it 
is always isomorphic to one induced by a symmetric pairing.

\section{Categorification}
The goal is to define a monoidal category\index{monoidal category}
$\hcat\basecat$ with objects generated by symbols $\PP_a$ and
$\QQ_a$ for each $a \in \basecat$ and the morphisms set up
so that we can define $\PP^{(n)}_a$
and $\QQ^{(n)}_a$ in terms $\PP_a$'s and $\QQ_a$'s
and so that the relations above become isomorphisms of objects. 
For example, relation~\eqref{eq:heisrelintro2} should become an isomorphism
\begin{equation}\label{eq:intro:heisenberg_iso}
	\QQ_a^{(m)}\PP_{b}^{(n)} \cong\ \bigoplus_{i=0}^{\mathclap{\min(m,n)}}\ \Sym^i \Hom_{\basecat}(a,b) \otimes_\kk \PP_{b}^{(n-i)}\QQ_a^{(m-i)}.
\end{equation}
We construct $\hcat\basecat$ as a
\emph{$2$-category}\index{$2$-category} with objects $\mathbb{Z}$, $1$-morphisms generated by 
$\PP_a\colon \ho \to \ho+1$ and $\QQ_b\colon \ho \to \ho-1$, and 
appropriate $2$-morphisms.  A representation of $\hcat\basecat$ is a $2$-functor into 
the $2$-category of categories, sending each integer to a \enquote{weight space category}. This \emph{idempotent modification}\index{idempotent modification} is done for convenience, and our 
construction can be easily repackaged into a monoidal category, cf.~Section~\ref{subsec:fock_cat_2}.

The crux of the categorification is to define \enquote{useable}
$2$-morphism spaces which imply only the necessary isomorphisms such as
\eqref{eq:intro:heisenberg_iso}. 
We define these by planar string diagrams such as
\begin{equation}
	\label{eq:diagram}
	\begin{tikzpicture}[baseline={(0,1)}, scale=0.9]
		\draw[->] (-1,0) node[below] {$\PP_{a}$} -- (-1,0.2) to[out=90, in=270] (0,2) -- (0,2.5) node[label=right:{$\alpha$}, dot, pos=0.25]{} node[above] {$\PP_{e}$} ;
		
		\draw[->] (0,0) node[below] {$\PP_{b}$} -- (0,0.2) to[out=90, in=180] (0.6,1.6) to[out=0, in=90] 
		(1.3,1.3) to[out=-90, in=0] (1,1) to[out=180, in=180] (1,2) to[out=0,in=90] (1.6,1.6) to[out=270,in=90](1,0.2) -- (1,0) 
		node[below] {$\QQ_{Sb}$} ;
		
		\draw[<-] (2,0) node[below] {$\QQ_{c}$} -- (2,0.6) to[out=90, in=270] (-1,2) -- (-1,2.5) node[label=left:{$\beta$}, dot, pos=0.25]{} node[above] {$\QQ_{d}$} ;
		
		\draw[decoration={markings, mark=at position 0.32 with {\arrow{>}}, mark=at position 0.82 with {\arrow{>}}}, postaction={decorate}] (2.8,0.8)
		--
		node[label=right:{$\gamma$}, dot, pos=0.5] {}
		(2.8,1.3)
		arc[start angle=0, end angle=180, radius=.5]
		--
		(1.8,0.8)
		arc[start angle=180, end angle=360, radius=.5];
		
	\end{tikzpicture}
\end{equation}
read from bottom to top. 
These are built out of a handful of generators, subject to relations.

\section{Main results} 
\label{section-main-results}

Our approach differs depending on whether our input datum $\basecat$ is 
a graded additive category with a genuine Serre functor or a \dg
category with only a homotopy Serre functor. We call these two setups 
the \emph{additive} and \emph{\dg} settings, respectively. We construct
the Heisenberg $2$-category $\hcat\basecat$ in the additive setting
in Chapter~\ref{sec:additive-Heisenberg-2cat} and in the \dg setting
in Chapter~\ref{sec:dg-Heisenberg-2-cat}. We then prove 
Theorems~\ref{thm:main1},~\ref{thm:main2} and ~\ref{thm:main3} stated
below for the \dg setting and Theorem~\ref{thm:main2} in the additive setting. 
Theorems~\ref{thm:main1} and ~\ref{thm:main3} are also expected to
hold in the additive setting if the numerical
Grothendieck group $\numGgp{\basecat}$ is a finitely generated abelian
group. In such case, our \dg proofs can be adapted and even simplified 
for the additive setting. Let us therefore state our main results in 
the language of the \dg setting.  

Let $\basecat$ be a smooth and proper \dg category. We view it as
a Morita enhanced triangulated category,
cf.~Chapter~\ref{section-enhanced-categories}. It is the noncommutative 
analogue of a smooth and proper algebraic variety $X$: the enhanced derived 
category of $X$ is an example of such $\basecat$. 
The graphical calculus described in Chapter~\ref{sec:dg-Heisenberg-2-cat} 
yields a \dg bicategory $\hcat\basecat$ together with maps of $\kk$-algebras
\begin{equation}\label{eq:decat_morphism}
	\pi\colon \halg\basecat \to \numGgp{\hcat{\basecat},\, \kk},\\
\end{equation}
where $\halg\basecat$ is the Heisenberg algebra of
$\numGgp{\basecat}$.  
Here, a \em bicategory \rm is a certain kind of weak $2$-category.
To be precise, we actually mean a bicategory 
enriched over the homotopy $2$-category $\HoDGCat$ of \dg
categories, see Chapter~\ref{section-enriched-bicategories}. 
We treat these subtle differences carefully in the main text of the paper, 
but here refer to these merely as \dg bicategories.  

As in the literature of Heisenberg categorification (numerical) Grothendieck groups appear more frequently, let us first state our results towards this direction. 
Our first main result shows that a $2$-full subcategory of
$\hcat\basecat$ categorifies the Heisenberg algebra $\halg\basecat$:
\begin{IntroTheorem}[{Theorem~\ref{thm:injective}}]
	\label{thm:main1}
	The map $\pi\colon \halg\basecat \to \numGgp{\hcat{\basecat},\, \kk}$ is injective.
\end{IntroTheorem}
Indeed, Theorem \ref{thm:main1} implies that 
the $2$-full subcategory of $\hcat{\basecat}$
comprising the objects whose class in $\numGgp{\hcat{\basecat},\, \kk}$ 
lies in the image of $\pi$ is a categorification of $\halg\basecat$.
{Since this subcategory is $2$-full and contains the objects $\PP_a$
	and $\QQ_a$ for $a \in \basecat$, which generate $\hcat\basecat$
	under taking $1$-compositions and perfect hulls, any 2-representation 
	of this subcategory extends uniquely to one of $\hcat\basecat$.}
Thus we work with $\hcat\basecat$ instead. 

Let $\EnhCatKCdg$ be the \dg bicategory of enhanced triangulated
categories, cf.~Chapter~\ref{section-enhanced-categories}. Here
and throughout the paper the subscript $\kc$ means ``Karoubi-complete''. 
Let $\fcat\basecat$ be its $2$-full subcategory comprising 
the symmetric powers $\symbcn$.
If $\basecat$ is the derived category of a variety $X$, then $\symbcn$ is the derived category of 
the symmetric quotient stack $[X^{\ho}/\SymGrp{\ho}]$.

Our second main result constructs a $2$-action of $\hcat\basecat$ on 
$\fcat\basecat$ which implies that a $2$-full subcategory of
$\fcat\basecat$ categorifies the classical Fock space representation 
$\falg\basecat$ of $\halg\basecat$:

\begin{IntroTheorem}[{Theorem~\ref{thm:cat_Fock_representation}}]\label{thm:main2}
	There is a 2-representation of $\hcat\basecat$ on
	$\fcat\basecat$. More precisely, 
	there is a homotopy strong \dg $2$-functor $\Phi_\basecat\colon\hcat\basecat \to \fcat\basecat$.
\end{IntroTheorem}

Indeed, this $2$-action induces a representation of 
$\numGgp{\hcat{\basecat},\, \kk}$ and hence of $\halg\basecat$ on
$\numGgp{\fcat{\basecat},\, \kk}$. We analyze it in  
Section~\ref{subsec:grothfock} and show that it induces an embedding 
of $\phi\colon \falg{\basecat} \hookrightarrow \numGgp{\fcat{\basecat},\, \kk}$ as the subrepresentation generated by $1 \in 
\numGgp{\sym^0 \basecat,\, \kk} \cong \kk$. Thus
the $2$-full subcategory of $\fcat{\basecat}$
comprising the objects whose class in 
$\numGgp{\fcat{\basecat},\, \kk}$ lies in the image of $\phi$
gives a categorification of $\falg\basecat$.

In many cases, for example if $\numGgp{\basecat}$ satisfies a Künneth-type formula for symmetric
powers, the embedding $\phi$ above is an isomorphism. 
Then the whole of $\fcat\basecat$
is a categorification of $\falg\basecat$. 
In any case, we call $\fcat\basecat$ the \emph{categorical Fock space}\index{categorical Fock space} of $\basecat$.

Our third main result gives another sufficient condition for
$\fcat\basecat$ to exactly categorify $\falg\basecat$, while at the same time 
exhibiting an obstruction for $\pi$ to be an isomorphism.
\begin{IntroTheorem}[{Theorem~\ref{thm:piisothenphi}}]
	\label{thm:main3} 
	If~$\hcat\basecat$ categorifies $\halg\basecat$, 
	that is, if $\pi$ is an isomorphism, then $\fcat\basecat$ categorifies $\falg\basecat$. In particular, in such case for all $N \geq 0$
	{
		\[\numGgp{\symbcn} \cong \bigoplus_{1^{\lambda_1}2^{\lambda_2} \dots \dashv \ho} \Sym^{\lambda_1}\numGgp{\basecat} \otimes \Sym^{\lambda_2}\numGgp{\basecat} \otimes \cdots\]
		where the direct sum is taken over all integer partitions {of} $N$.}
\end{IntroTheorem}

We conjecture that the converse of this statement holds as well.
\begin{IntroConjecture}
	If $\fcat\basecat$ categorifies $\falg\basecat$, then $\pi\colon \halg\basecat \to \numGgp{\hcat{\basecat},\, \kk}$ is an isomorphism.
\end{IntroConjecture}

{ We provide examples in
	Section~\ref{subsec:genuine_categorification} where $\phi$ is an
	isomorphism. We also give an example in
	Section~\ref{section:counterexample} where it fails to be an
	isomorphism. In the latter case $\pi$ also can not be an
	isomorphism by Theorem~\ref{thm:main3}. In fact, the numerical 
	Grothendieck group decategorifications of $\hcat\basecat$
	and $\fcat\basecat$ are generally larger than the classical 
	Heisenberg algebra $\halg\basecat$ and its Fock space $\falg\basecat$. 
	However, our decategorifications always \emph{contain} $\halg\basecat$ 
	and $\falg\basecat$. It becomes an interesting new problem to compute 
	the surplus and find ways to interpret it.
	
	In the sequel paper
	\cite{GyengeLogvinenko-TheHeisenbergAlgebraOfAVectorSpaceAndHochschildHomology},
	we show that our $2$-category $\hcat\basecat$ can also be decategorified
	using the Hochschild homology $\hh$. Specifically, we settle some
	foundational issues to define the Heisenberg algebra 
	$\halg\basecat^H$ of the $\mathbb{Z}_2$-graded vector space
	$\hh(\basecat)$. We then prove the following:
	\begin{IntroTheorem*}
		[\cite{GyengeLogvinenko-TheHeisenbergAlgebraOfAVectorSpaceAndHochschildHomology}]
		For any smooth and proper \dg category $\basecat$:
		\begin{enumerate}
			\item There exists an injective map $\pi^H\colon \halg\basecat^H
			\longrightarrow \hh(\hcat\basecat)$. 
			\item The map $\pi^H$ and the $2$-representation $\Phi_\basecat$ induce
			an action of $\halg\basecat^H$ on $\hh(\fcat{\basecat})$.
			There is an injective map 
			$\phi^H\colon \falg\basecat^H \hookrightarrow \hh(\fcat{\basecat})$
			which embeds the Fock space $\falg\basecat^H$ of $\halg\basecat^H$
			as the subrepresentation generated by $1 \in \hh(\fcat{\basecat})$. 
			\item The map $\phi^H$  is always an isomorphism 
			and therefore $\fcat\basecat$ always categorifies $\falg\basecat^H$. 
		\end{enumerate}
	\end{IntroTheorem*}
	
	This leads us to conjecture the following:
	\begin{IntroConjecture*}
		The map $\pi^H$ is always an isomorphism, so $\hcat\basecat$ always 
		categorifies $\halg\basecat^H$. 
\end{IntroConjecture*}}

\section{Relation to earlier results}
Our results recover as special cases the earlier
Heisenberg categorification and Fock space action results mentioned
above. We bring forward these specialisations
throughout the paper as sequences of examples; here we just preview them briefly.
For $\basecat=\kk$, the field $\kk$ considered as a single object \dg
category concentrated in degree $0$, our category $\hcat\basecat$ is a \dg
enhancement of Khovanov's original category
\cite{khovanov2014heisenberg}; see Examples~\ref{ex:Khovanov-add},~\ref{ex:Khovanov-dg} and ~\ref{ex:khovanovnumgrp}.  When $X$ is a smooth and projective
variety and $\basecat$ its \dg enhanced coherent derived category, 
a subcategory of $\hcat{\basecat}$ categorifies the
Heisenberg algebra modeled on the numerical K-theory of $X$. Its
action on $\fcat\basecat$ constructed in Theorem \ref{thm:main2}
coincides, after taking homotopy categories, with that of Krug
\cite{krug2018symmetric}; see Examples~\ref{ex:Xsmoothpropenh},~\ref{ex:symcomp}, ~\ref{ex:Krug_part1} and ~\ref{ex:Krug_part2}. This answers the questions raised in
\cite[Section~3.5]{krug2018symmetric}. When $X$ is Calabi-Yau, the
direct sum of the Hochschild (co)homologies of $X$ carries the structure
of a Frobenius algebra. In this case our categories essentially
coincide with those of \cite{rosso2017general}, although we do not
consider super-Frobenius algebras. Let $\Gamma \subset \SL(2,\CC)$ be 
a finite subgroup and let $\basecat$ be the \dg enhanced derived category 
of coherent sheaves supported on the exceptional divisor $E$ of the 
minimal resolution $X$ of the quotient singularity
$\mathbb{C}^2/\Gamma$. Then our construction yields the Heisenberg
category constructed by Cautis and Licata \cite{cautis2012heisenberg}, 
see Examples~\ref{ex:CautisLicata_part1},
~\ref{ex:CautisLicata_part3} and ~\ref{ex:CautisLicata_part2}. 

There are several advantages to our approach compared to the earlier
ones. Our definition allows any \dg category $\basecat$ as
the input of the machinery. This fits well into the
framework of noncommutative motives \cite{tabuada2015noncommutative}.
We do not need the form $\chi$ on the Grothendieck
group to be symmetric. In particular, if $\basecat$ comes from a
variety, the latter does not have to be a Calabi-Yau. In fact, our
construction works with $\basecat$ being a \dg enhancement of any smooth
and proper scheme $X$, as opposed to the construction in 
\cite{cautis2012heisenberg} which is specific to the case where $X$
is (a local model of) the minimal resolution of a Kleinian surface 
singularity. Finally, working with \dg categories, we obtain a
natural framework for working with complexes of operators, as is
necessary when categorifying alternating sums which appear, for
example, in the Frenkel--Kac construction \cite[Chapter~7]{frenkel1989vertex}.

\section{The additive construction} We now describe our construction of 
$\hcat\basecat$ in more detail. We begin with the simpler additive
construction. 

In categorification, one often encounters the following diagram of
categories and functors: 
\begin{equation}\label{eq:E_F_adj}
	\begin{tikzcd}
		\cat C \arrow[r, bend left, "\mathsf{E}"] &
		\cat D \arrow[l, bend left, "\mathsf{F}"]
	\end{tikzcd}
\end{equation}
Frequently, these functors are required to be biadjoint.
For example, in Khovanov's Heisenberg category
\cite{khovanov2014heisenberg} the generating objects $Q_{+}$ and
$Q_{-}$ are biadjoint, while in the Cautis--Licata categorification
\cite{cautis2012heisenberg} the $1$-morphisms $\PP_i$ and $\QQ_i$ are
biadjoint up to a shift.

The biadjointness assumption can be a powerful tool, 
but it can also be very restrictive.
For example, in Krug's action of a Heisenberg algebra on derived categories of symmetric quotient stacks \cite{krug2018symmetric} the functors $Q_\beta^{(n)}$ are only \emph{right} adjoint to $P_\beta^{(n)}$.

Inspired by \cite{bondal1989representable}, we use Serre
functors to overcome this. In \eqref{eq:E_F_adj}, if $\mathsf{E}$ is the left adjoint of $\mathsf{F}$ and $S_{\cat C}$ and $S_{\cat D}$ are Serre functors on
$\cat C$ and $\cat D$, then 
$S_{\cat D}^{\vphantom{-1}} \mathsf E S_{\cat C}^{-1}$
is the right adjoint of $\mathsf F$. We use this to relax Khovanov's 
biadjunction condition for our categorification. 

Thus, let $\basecat$ be a $\Hom$-finite graded 
$\kk$-linear category endowed with a Serre functor $S$.
To construct the \emph{additive Heisenberg category} $\hcatadd\basecat$
we first construct a simpler $2$-category $\hcatadd*\basecat$ 
whose objects are the integers $\ho \in \mathbb{Z}$
and whose $1$-morphisms are freely generated by  
$$ \PP_a\colon {\ho} \rightarrow {\ho + 1} \quad \text{ and } 
\QQ_a\colon {\ho} \rightarrow {\ho -  1} $$ 
for each $a \in \basecat$ and $\ho \in \mathbb{Z}$. The identity 
$1$-morphism of each $\ho$ is denoted by $\hunit$.

The $2$-morphisms of $\hcatadd*\basecat$ we define below 
ensure that $\PP_a$ is the left adjoint of $\QQ_a$. Motivated by the above, 
we also ensure that $\PP_{Sa}$ is the right adjoint of $\QQ_a$.
Thus, we have 
\begin{equation}
	\label{eqn-intro-P-Q-adjunctions}
	\PP_a \dashv \QQ_{a} \dashv \PP_{Sa}.
\end{equation}

We define the $2$-morphisms by planar string diagrams 
similar to those of Khovanov \cite{khovanov2014heisenberg}; 
an example is \eqref{eq:diagram} above. Similarly to the work of 
Cautis and Licata \cite{cautis2012heisenberg} our strings 
are decorated by morphisms of $\basecat$. For
every $\alpha \in \Hom_\basecat(a,b)$ we have vertical oriented strings
\[
\begin{tikzpicture}[baseline={(0,0.4)}]
	\draw[->] (0,0) node[below] {$\PP_a$} -- node[label=right:{$\alpha$}, dot, pos=0.5] {} (0,1) node[above] {$\PP_{b}$};
\end{tikzpicture}
\qquad\text{and}\qquad
\begin{tikzpicture}[baseline={(0,0.4)}]
	\draw[<-] (0,0) node[below] {$\QQ_a$} -- node[label=right:{$\alpha$}, dot, pos=0.5] {} (0,1) node[above] {$\QQ_{b}$};
\end{tikzpicture}
\]
As a shorthand, the strings decorated by the identity morphism are 
drawn unadorned. Strings are also allowed to cross and bend. 
Thus, for any $a,b \in \basecat$ we have the crossings
\[
\begin{tikzpicture}[baseline={(0,0.4)}]
	\draw[<-] (0,0) node[below] {$\QQ_{a}$}-- (1,1) node[above] {$\QQ_{a}$};
	\draw[<-] (1,0) node[below] {$\QQ_{b}$}-- (0,1) node[above] {$\QQ_{b}$};
\end{tikzpicture}, 
\begin{tikzpicture}[baseline={(0,0.4)}]
	\draw[->] (0,0) node[below] {$\PP_{a}$} -- (1,1) node[above] {$\PP_{a}$};
	\draw[<-] (0,1) node[above] {$\PP_{b}$} -- (1,0) node[below] {$\PP_{b}$};
\end{tikzpicture}, 
\begin{tikzpicture}[baseline={(0,0.4)}]
	\draw[->] (0,0) node[below] {$\PP_{a}$} -- (1,1) node[above] {$\PP_{a}$};
	\draw[->] (0,1) node[above] {$\QQ_{b}$} -- (1,0) node[below] {$\QQ_{b}$};
\end{tikzpicture}, 
\begin{tikzpicture}[baseline={(0,0.4)}]
	\draw[->] (1,0) node[below] {$\PP_{b}$} -- (0,1) node[above] {$\PP_{b}$};
	\draw[<-] (0,0) node[below] {$\QQ_{a}$} -- (1,1) node[above] {$\QQ_{a}$};
\end{tikzpicture}. 
\]
The cups and caps that appear at the bends need to take into account the Serre
functor. For any $a \in \basecat$ we have the following cups and caps 
\begin{equation}
	\label{eqn-intro-cups-and-caps}
	\begin{tikzpicture}[baseline={(0,0.25)}]
		\draw[->] (0,0) node[below] {$\PP_a$} arc[start angle=180, end angle=0, radius=.5] node[label=above:{$\hunit$},pos=0.5]{} node[below] {$\QQ_a$};
	\end{tikzpicture},\quad
	\begin{tikzpicture}[baseline={(0,0.25)}]
		\draw[->] (0,0) node[below] {$\PP_{Sa}$} arc[start angle=0, end angle=180, radius=.5] node[label=above:{$\hunit$},pos=0.5]{} node[below] {$\QQ_{a}$};
	\end{tikzpicture},\quad
	\begin{tikzpicture}[baseline={(0,-0.25)}]
		\draw[->] (0,0) node[above] {$\QQ_a$} arc[start angle=0, end angle=-180, radius=.5] node[label=below:{$\hunit$},pos=0.5]{} node[above] {$\PP_{Sa}$};
	\end{tikzpicture},\quad
	\begin{tikzpicture}[baseline={(0,-0.25)}]
		\draw[->] (0,0) node[above] {$\QQ_{a}$} arc[start angle=-180, end angle=0, radius=.5] node[label=below:{$\hunit$},pos=0.5]{} node[above] {$\PP_{a}$};
	\end{tikzpicture}.
\end{equation}

As in \cite{khovanov2014heisenberg}, the planar diagrams generated by the above
are subject to a number of relations. The full list is in 
Chapter~\ref{sec:additive-Heisenberg-2cat}. For example, for any $a
\in \basecat$ we have the straightening relations 
\begin{equation*}
	\begin{tikzpicture}[baseline={(0,0.9)}, scale=0.75]
		\draw[->] (0,0)  node[below] {$\PP_a$} -- (0,1)  arc[start angle=180,
		end angle=0, radius=.5] arc[start angle=-180, end angle=0, radius=.5]
		-- (2,2) node[above] {$\PP_a$};
		\draw (2.5,1) node {$=$};
		\draw[->] (3,0) node[below]  {$\PP_a$}  -- (3,2) node[above] {$\PP_a$};
		\draw (3.5,1) node {$=$};
		\draw[->] (6,0) node[below] {$\PP_a$}  -- (6,1)  arc[start angle=0, end angle=180, radius=.5] arc[start angle=0, end angle=-180, radius=.5] -- (4,2) node[above] {$\PP_a$};
	\end{tikzpicture},
	\quad
	\begin{tikzpicture}[baseline={(0,0.9)}, scale=0.75]
		\draw[<-] (0,0)  node[below] {$\QQ_a$} -- (0,1)  arc[start angle=180,
		end angle=0, radius=.5] arc[start angle=-180, end angle=0, radius=.5]
		-- (2,2) node[above] {$\QQ_a$};
		\draw (2.5,1) node {$=$};
		\draw[<-] (3,0) node[below]  {$\QQ_a$}  -- (3,2) node[above] {$\QQ_a$};
		\draw (3.5,1) node {$=$};
		\draw[<-] (6,0) node[below] {$\QQ_a$}  -- (6,1)  arc[start angle=0, end angle=180, radius=.5] arc[start angle=0, end angle=-180, radius=.5] -- (4,2) node[above] {$\QQ_a$};
	\end{tikzpicture},
\end{equation*}
ensuring the $2$-categorical adjunctions
\eqref{eqn-intro-P-Q-adjunctions} with units and counits given by the caps
and cups \eqref{eqn-intro-cups-and-caps}. 

The relations on the planar string diagrams take into account the Serre functor. The details are
in Chapter~\ref{sec:additive-Heisenberg-2cat}, while here we give 
one representative example. In Khovanov's category one has the 
\enquote{biadjunction}\index{biadjunction} or
\enquote{bubble}\index{bubble relation} relation specifying 
that the diagram composition
\[
\hunit \xrightarrow{\unit} \QQ \PP \xrightarrow{\counit} \hunit,
\quad \quad\textrm{pictorially}\quad \quad\
\begin{tikzpicture}[baseline={(0,-0.5ex)}]
	\draw[decoration={markings, mark=at position 0.27 with {\arrow{>}}, mark=at position 0.76 with {\arrow{>}}}, postaction={decorate}] (0.5,0)
	arc[start angle=0, end angle=180, radius=.5] node[label=above:{$\hunit$}, pos=0.5] {}
	--
	(-0.5,0)
	arc[start angle=180, end angle=360, radius=.5] node[label=below:{$\hunit$}, pos=0.5] {}; \end{tikzpicture} 
\]
is the identity. Here we set $\QQ = Q_-$ and $\PP = Q_+$ in the notation of \cite{khovanov2014heisenberg}, and the first map is the unit of $(\PP,\QQ)$-adjunction, while the second map is the counit of $(\QQ,\PP)$-adjunction.

In the absence of biadjunction, the above cannot possibly hold.
Instead, we demand that for any 
$\alpha \in \Hom_{\basecat}(a,\, Sa)$ the composition
\[
\hunit \xrightarrow{\unit} \QQ_a\PP_a \xrightarrow{\left(\id_{\QQ_a}\right)\alpha} \QQ_a\PP_{Sa} \xrightarrow{\counit} \hunit,
\quad \quad\textrm{pictorially}\quad \quad\
\begin{tikzpicture}[baseline={(0,0.0)}]
	\draw[decoration={markings, mark=at position 0.29 with {\arrow{>}}, mark=at position 0.79 with {\arrow{>}}}, postaction={decorate}] (1,0)
	--
	node[label=right:{$\alpha$}, dot, pos=0.5] {}
	(1,0.25)
	arc[start angle=0, end angle=180, radius=.5] node[label=above:{$\hunit$}, pos=0.5] {}
	--
	(0,0)
	arc[start angle=180, end angle=360, radius=.5] node[label=below:{$\hunit$}, pos=0.5] {};
\end{tikzpicture}
\]
is the multiplication by the \emph{Serre trace}\index{Serre trace} $\Tr(\alpha) \in \kk$, 
defined in Section~\ref{subsec:serre}. 

Finally, as in some previous works on the categorification of Heisenberg 
algebras, having constructed the smaller $2$-category $\hcatadd*\basecat$ 
$1$-generated only by $\PP_a = \PP_a^{(1)}$ and $\QQ_a = \QQ_a^{(1)}$ for 
$a \in \basecat$, we define $\hcatadd\basecat$ to be its idempotent
completion. The remaining elements $\PP_a^{(n)}$ and $\QQ_a^{(n)}$ 
are then the direct summands of $1$-compositions
$\PP_a^{n}$ and $\QQ_a^{n}$ defined by the symmetrising idempotents
of the action of the permutation group $\SymGrp n$ by braid
diagrams. Thus, for constructing a $2$-representation of $\hcat\basecat$
one only needs to specify the actions of $\PP_a$ and $\QQ_a$.

In Section~\ref{subset:cat_fock_add} we give such an action on 
the categorical version of the Fock space, consisting of the categories 
of  $\SymGrp\ho$-equivariant objects in $\basecat^{\otimes\ho}$.

\section{The \texorpdfstring{\dg}{DG} construction}

From the viewpoint of algebraic geometry, we want to work   
with a \dg category $\basecat$ which Morita enhances\index{Morita
enhancement} the derived category
of an algebraic variety $X$. This means that the compact derived 
category $\catDc(\basecat)$ of \dg modules over $\basecat$ is equivalent to 
the bounded derived category $\catDbCoh{X}$ of coherent sheaves on $X$. 
{ This is different from the older notion of a (non-Morita) DG
	enhancement, which required $\basecat$ to have special properties
	(being pre-triangulated) and the triangulated category it enhanced
	was $\Hzero(\basecat)$. The two notions are connected: if $\basecat$
	Morita enhances $\catDbCoh{X}$, then the \em perfect hull \rm $\hperf
	\basecat$ enhances it in the usual sense. On triangulated level, 
	the perfect hull corresponds to taking the Karoubi-completed 
	triangulated hull. Thus, with Morita enhancements we can work with 
	smaller \dg categories which explicitly enhance only a small part of 
	the triangulated category from which the rest can be generated by 
	taking cones, shifts, and idempotent completions\index{idempotent completion}. 
	
	A nice example is provided by the symmetric quotient stacks.  
	A naive symmetric power of a triangulated category is not triangulated.  
	In \cite{SymCat} Kapranov and Gantner took a pretriangulated category
	$\A$ and defined its completed $n$-th symmetrical power 
	$\widehat{\sym}^n\A$ which ensured that $\Hzero(\widehat{\sym}^n\A)$
	is the correct symmetric power of $\Hzero(\A)$. In 
	\S\ref{subsec:equivariant_cats} we give for any \dg category $\A$ 
	a simpler construction $\mathcal{S}^n \A$ which ensures that 
	$\catDc(\mathcal{S}^n \A)$ is the correct symmetric power of $\catDc(\A)$. 
	It is a categorification of the skew group algebra construction and
	its perfect hull coincides 
	with the Kapranov-Gantner's $\widehat{\sym}^n\A$ on 
	the \dg level (see Lemma
	\ref{lemma-A-rtimes-G-modules-are-G-equiv-A-modules}). 
	It is, in a sense, the smallest natural \dg category which does this job.
	In particular, when $\basecat$ Morita enhances $\catDbCoh{X}$, 
	$\symbcn$ Morita enhances the symmetric quotient stack $[X^N/S_N]$. 
}

Let $\basecat$ be a smooth and proper \dg category 
(see Chapter~\ref{sec:dg-prelim} for a review on \dg categories).
Then $\hperf \basecat$ always possesses a \emph{homotopy Serre functor}\index{homotopy Serre functor}, 
i.e.~a quasi-autoequivalence $S$ together with quasi-isomorphisms
\[
\eta_{a,b}\colon \homm_\basecat(a,b) \to \homm_\basecat(b,Sa)^*,
\]
natural in $a,b \in \basecat$ (see Section~\ref{subsec:dg-homotopy-serre}).
In other words, $S$ is only a Serre functor \emph{up to homotopy}.

Thus the adjunction relation $\QQ_a^{(n)} \dashv \PP_{Sa}^{(n)}$
in the \dg Heisenberg category $\hcat\basecat$ needs to be
homotopically weakened. One option would be to upgrade $\hcat\basecat$ to 
an $(\infty,2)$-category and have the additional homotopical 
information come from the topology of string diagrams. 
However, at the moment the authors still find it difficult 
to construct $(\infty, 2)$-categories by means of generators and relations. 
In this paper we take a different approach which stays entirely within 
the realm of \dg categories.

Our main idea is to introduce three sets of generating objects $\PP_a$, $\QQ_a$ and $\RR_a$, related by strict adjunctions 
$\PP_a \dashv \QQ_a$ and $\QQ_a \dashv \RR_a$.
To relate the left and right adjoints of $\QQ_a$, we add for each $a
\in \basecat$ the starred string $2$-morphism
\begin{equation*}
	\starmap a\colon  
	\begin{tikzpicture}[baseline={(0,0.4)}]
		\draw[->] (0,0) node[below] {$\PP_{Sa}$} -- node[serre, pos=0.5] {} (0,1) node[above] {$\RR_{a}$};
	\end{tikzpicture}. 
\end{equation*}
By the considerations above, all these $\starmap a$ should be homotopy 
equivalences. To impose this in a consistent way, without having to
specify the higher homotopies by hand, we take 
the Drinfeld quotient by the cone of $\starmap{a}$.
This makes $\starmap{a}$ a homotopy equivalence, and thus makes 
each $\PP_{Sa}$ a homotopy right adjoint of $\QQ_a$. 

Thus, we first define a strict \dg $2$-category $\hcat*\basecat$ 
with objects $\ho \in \mathbb{Z}$, $1$-morphisms freely generated by 
$\PP_a$, $\QQ_a$ and $\RR_a$, and $2$-morphisms given by planar string 
diagrams similar to those in $\hcatadd*\basecat$ with the addition of 
the star-morphisms $\starmap a\colon \PP_{Sa} \to \RR_a$.
We then take the h-perfect hull $\bihperf(\hcat*\basecat)$ 
to obtain a \dg bicategory whose $1$-morphism \dg categories are 
pretriangulated and homotopy Karoubi complete\index{homotopy Karoubi
completion}. Finally, we define 
$\hcat\basecat$ to be the Drinfeld quotient of $\bihperf(\hcat*\basecat)$ 
by the two-sided ideal $I_{\basecat}$ generated by
the cones of $\starmap{a}$ and of another $2$-relation 
we only want to hold up to homotopy. This is one of the subtler
points of our construction: the original Drinfeld quotient
construction \cite{Drinfeld-DGQuotientsOfDGCategories} is very much
incompatible with monoidal structures such as that of
a $1$-composition in a $2$-category. However, this was already
considered by Shoikhet \cite{Shoikhet-DifferentialGradedCategoriesAndDeligneConjecture} who refined Drinfeld's construction to obtain on it the structure of a 
\emph{weak Leinster monoid}. We use this to define the notion of
a \em monoidal Drinfeld quotient \rm\index{monoidal Drinfeld
quotient} of a \dg bicategory by 
a two-sided ideal of $1$-morphisms. It has all the expected universal
properties. The price is that $\hcat\basecat$ becomes a 
$\HoDGCat$-enriched bicategory. In other words, its $1$-composition is
now given by quasi-functors: compositions of 
genuine \dg functors with formal inverses of quasi-equivalences.
However, the homotopy category of $\hcat\basecat$ is a genuine 
$2$-category whose $1$-morphism categories are triangulated and
Karoubi-complete. In particular, it recovers all the combinatorics 
of the additive setting.

In Chapter~\ref{sec:cat_fock} we construct a categorical version of
the Fock space for the \dg setting. As noted in
\cite{bondal2004grothendieck}, the naive tensor product of 
categories does not behave well with respect to
triangulated structures. In the \dg enhanced setting this is solved 
by taking the h-perfect hull of the naive tensor product
(often called the completed tensor product). This was one of
our reasons to develop the machinery of Heisenberg categories on the level 
of \dg categories. 

We thus proceed in two steps again: first, we define a strict
$2$-functor $\Phi'_{\basecat}$ from $\hcat*{\basecat}$ to the strict
\dg $2$-category $\DGModCat$ of \dg categories, \dg functors between
their module categories and natural transformations.  The image of
$\Phi_\basecat'$ is contained in the $1$-full subcategory
$\fcat*\basecat$ whose objects are the symmetric powers $\symbcn$.
This concrete definition is at the heart of our categorical Fock
space representation. 

We next apply some abstract \dg wizardry. We use the bimodule
approximation $2$-functor $\bimodapx$  to approximate the
$1$-morphisms of $\fcat*{\basecat}$ by \dg bimodules. This yields a
homotopy strong $2$-functor from $\hcat*{\basecat}$ into the
bicategory $\EnhCatKCdg$ of enhanced triangulated
categories. We next take perfect hulls and verify that on the homotopy
level the resulting $2$-functor $\bihperf(\hcat*{\basecat})
\rightarrow \EnhCatKCdg$ kills all $1$-morphisms of $I_{\basecat}$ and
thus descends to a homotopy strong $2$-functor $\Phi_{\basecat}\colon
\hcat{\basecat} \rightarrow \EnhCatKCdg$. Its image 
is our categorical Fock space $\fcat{\basecat}$. 

\section{Results on \texorpdfstring{\dg}{DG} categories}
\label{subsec:intro:dg_results}

To construct the \dg Heisenberg algebra and its Fock space representation, 
we needed to develop several new results on \dg categories.
Most of these are $2$-categorical analogues of common 
\dg-categorical constructions. We hope that these results and
techniques may have applications outside of our work. We thus summarise
them here in the order in which we perceive them to be potentially useful
to others. For the technical details, see the indicated sections.

In Section~\ref{subsec:monoidal-drinfeld-quotients}, we use Shoikhet's
construction \cite{Shoikhet-DifferentialGradedCategoriesAndDeligneConjecture}
to define a \emph{monoidal Drinfeld quotient} $\bicat C/\bicat I$ 
of a \dg bicategory $\bicat C$ by a two-sided 
$1$-morphism ideal 
$\bicat I$. We want this to be a $2$-category with the same objects
as $\bicat C$ whose $1$-morphism categories are Drinfeld quotients of 
those of $\bicat C$ by $\bicat I$. The problem is to define 
the $1$-composition, as the interchange law would force relations to 
exist between the contracting homotopies, which were freely
introduced. Following Shoikhet
\cite{Shoikhet-DifferentialGradedCategoriesAndDeligneConjecture}, we
define $1$-composition by resolving tensor products of
Drinfeld quotients of $1$-morphism categories of $\bicat C$
by a refined construction which admits a natural $1$-composition 
functor. The resulting $1$-composition is then a quasi-functor in 
the homotopy category $\HoDGCat$ of \dg categories. In  
Theorem \ref{theorem-the-universal-properties-of-monoidal-drinfeld-quotient}
we prove that the resulting $\HoDGCat$-enriched bicategory $\bicat C/\bicat I$
has the expected universal property with respect to the $2$-functors
out of $\bicat C$ which are null-homotopic on the $1$-morphisms of 
$\bicat I$. 

In Chapter~\ref{section-enhanced-categories}, we define the \dg bicategory 
$\EnhCatKCdg$ of enhanced triangulated categories. It is where the 
main action of this paper takes place. Its homotopy category, the strict 
$1$-triangulated $2$-category $\EnhCatKC$ has been 
understood for a while
\cite{Toen-TheHomotopyTheoryOfDGCategoriesAndDerivedMoritaTheory}{,}\cite{lunts2010uniqueness}. However, there are well-known technical difficulties in 
constructing a \dg bicategory enhancing it. We propose two 
constructions which are both almost a \dg bicategory. One uses
the technology of bar-categories of modules 
\cite{AnnoLogvinenko-BarCategoryOfModulesAndHomotopyAdjunctionForTensorFunctors}. The result is a homotopy unital \dg bicategory. Its unitor morphisms are 
homotopy equivalences with canonical homotopy inverses which are 
genuine inverses on one side. This approach is more elegant and 
its structures are explicitly defined and thus easily computable.  
Alternatively, we use our new notion of the
monoidal Drinfeld quotient\index{monoidal Drinfeld quotient} to construct $\EnhCatKCdg$ as the quotient
of the Morita $2$-category of \dg bimodules by acyclics. The result is
a bicategory, but enriched over $\HoDGCat$ and not
$\DGCat$. This definition is simpler, not requiring familiarity with
\cite{AnnoLogvinenko-BarCategoryOfModulesAndHomotopyAdjunctionForTensorFunctors}, but less explicit and less practical to compute with. Either
construction works well for the purposes of this paper. 

In Section~\ref{subsec:hperf}, we define the
\emph{h-perfect hull} of a \dg bicategory $\bicat C$. It is
a \dg bicategory with the same objects as $\bicat C$ whose 
$1$-morphism categories are h-perfect hulls of those of 
$\bicat C$. 

In Section~\ref{section-bimodule-approximation}, 
we define the \em bimodule approximation \rm 
$2$-functor $\bimodapx$ which approximates 
\dg functors by \dg bimodules. Some of these formalities are
well-known to experts \cite[Section~6.4]{Keller-DerivingDGCategories}, 
but it may be useful to have them written down. 

In Section~\ref{subsec:dg-homotopy-serre} we define
the notion of a \emph{homotopy Serre functor} and show 
that every smooth and proper \dg category $\basecat$ admits one
on $\hperf \basecat$. Again, this is well-known to experts, but 
the point is that the genuine Serre functor constructed on $\Hzero(\hperf
\basecat)$ in 
\cite{Shklyarov-OnSerreDualityForCompactHomologicallySmoothDGAlgebras}
lifts together with all its natural morphisms to $\hperf \basecat$
itself.

\section{Further questions and remarks}
Next, we outline some further questions and related results that we believe
to be interesting for future investigations.  

Gal \cite{gal2016hopf} showed that the structure of a Hopf category on a semisimple symmetric monoidal abelian category implies the existence of a categorical Heisenberg action in the sense of Khovanov.
It would be interesting to see whether this construction can be generalised to obtain a category isomorphic to $\hcat\basecat$ for any $\basecat$.
Several examples of categorifications of algebraic structures seemingly related to ours carry actions of braid groups. It would also be interesting to see if there is a deeper relationship between our categorification, Hopf categories and braid group actions.

Extending the work of Grojnowski and Nakajima, Lehn \cite{lehn1999chern} constructed Virasoro operators on the cohomology of Hilbert schemes of points of smooth projective surfaces. The present article is motivated in part by a desire to generalise this construction to the Heisenberg algebra action on derived categories of symmetric quotient stacks. Such operators should arise as convolutions of certain complexes of $2$-morphisms on $\hcat\basecat$.
The desire to obtain a good framework for working with such complexes is one of the reasons we work with \dg categories in this paper.
We intend to return to this question in future work.

In a different direction, the BGG category $\mathcal O$ of prominence in representation theory has a Serre functor (see Example~\ref{ex:bgg} and \cite{Koppensteiner:Traces}).
It would be enlightening to understand the associated Heisenberg category and its Fock space in detail.

Theorem~\ref{thm:main3} shows that it is interesting to consider 
when the morphism \eqref{eq:decat_morphism} is an isomorphism.
Following \cite{brundan2018degenerate}, one way to understand
surjectivity of this morphism seems to be via a suitable
generalisation of degenerate affine Hecke algebras and their
categorifications. This may also lead to the answers for 
the questions raised in \cite[Section~10.3]{cautis2012heisenberg}.

\section{Structure of the paper}

The structure of the paper is as follows. In
Chapter~\ref{sec:general-prelim} we give preliminaries 
relevant to both the additive and the \dg settings.
We recall the concept of Serre functors and introduce the idempotent modification of Heisenberg
algebras which we categorify. 
In Chapter~\ref{sec:additive-Heisenberg-2cat} we construct the additive Heisenberg $2$-category $\hcatadd\basecat$ and investigate its properties.

In Chapter~\ref{sec:dg-prelim} we give preliminaries required  for the
\dg setting. We encourage the reader uninterested in \dg
technicalities to skip this section and refer back to it when needed.

In Chapter~\ref{sec:dg-Heisenberg-2-cat} we construct the Heisenberg
$2$-category $\hcat\basecat$ in the \dg setting. In
Chapter~\ref{sec:structure}, 
we investigate the structure of $\hcat\basecat$ and, in
particular, deduce the categorical version of the Heisenberg
commutation relations and prove Theorem~\ref{thm:main1}.
In Chapter~\ref{sec:cat_fock} we construct the categorical Fock space
representation $\fcat\basecat$ and the $2$-functor $\hcat\basecat \to
\fcat\basecat$, and prove Theorem~\ref{thm:main2}.
We note that the proof of Theorem~\ref{thm:main1} depends 
on Theorem~\ref{thm:main2}.
Finally, in Chapter~\ref{sec:strfock} we investigate the properties of
$\fcat\basecat$ and prove Theorem~\ref{thm:main3}.

\section{Acknowledgements}  The authors are thankful to Pieter
Belmans, Alexander Efimov, Adam Gal, Elena Gal, Eugene Gorsky, Dmitri Kaledin, Andreas Krug, Alexander Kuznetsov, Boris Shoikhet, Bal\'azs Szendr\H{o}i, Tam\'as Szamuely and Gon{\c{c}}alo Tabuada for helpful comments and discussions. The authors are especially grateful to the anonymous referees for their valuable comments, suggestions and questions that improved the manuscript.
This project received funding from EPSRC grant EP/R045038/1 and from
the European Union's Horizon 2020 research and innovation programme
under the Marie Sk\l odowska-Curie grant agreement No 891437. {
	Á.Gy.~was} also supported by the János Bolyai Research Scholarship of the Hungarian Academy of Sciences and by the National Research, Development and Innovation Fund of Hungary, within the Program of Excellence TKP2021-NVA-02 at the Budapest University of Technology and Economics.

\section{Notation}

Throughout the paper, $\kk$ is an algebraically closed field of 
characteristic $0$. All categories and functors are assumed to be $\kk$-linear.
By a variety we mean an integral, separated scheme of finite type over $\kk$. All of our tensor products are over $\kk$, unless indicated otherwise. The tensor product of two complexes over $\kk$ is understood as the total complex of the double complex containing the tensor products of the terms.

We always denote $2$-categories in bold (such as $\hcat\basecat$ or $\DGCat$) and $1$-categories in calligraphic letters (such as $\basecat$).
Objects in a $1$-category are denoted by lowercase Latin letters, while morphisms are denoted in lowercase Greek letters.

\chapter{Preliminaries}\label{sec:general-prelim}

\section{Serre functors}\label{subsec:serre}

Let $\A$ be a graded $\kk$-linear category with finite-dimensional
$\Hom$-spaces. A \enquote{graded $\kk$-linear} category means a category enriched in graded vector spaces.

A \emph{Serre functor}\index{Serre functor} on $\A$ is a degree zero autoequivalence $S$ of $\A$ equipped with isomorphisms 
\[
\eta_{a,b}\colon \Hom_\A(a,b) \isoto \Hom_\A(b,Sa)^*,
\]
natural in $a,b \in \A$ \cite{bondal1989representable}.
If a Serre functor exists, then it is unique up to an isomorphism \cite[Proposition 1.5]{BondalOrlov:2001:ReconstructionOfAVariety}.

\begin{Example}\label{ex:dbcoh-serre}
	If $X$ is a smooth and proper variety over $\kk$, then $\catDbCoh{X}$ admits a Serre functor $S = (-) \otimes_X \omega_X[\dim X]$, where $\omega_X$ is the canonical line bundle of $X$. 
\end{Example}

\begin{Example}\label{ex:bgg}
	Let $G$ be a reductive algebraic group over $\kk$, with Borel subgroup $B$.
	Then the category of Schubert-constructible sheaves on the flag variety $G/B$ has a Serre functor given by the square of the intertwining operator\index{intertwining operator} associated to the longest element of the Weyl group \cite{TiltingExercises}.
	We note that by Beilinson--Bernstein
localisation\index{Beilinson--Bernstein localisation} and the
Riemann--Hilbert equivalence\index{Riemann--Hilbert equivalence} this category is the same as the principal block of the Beilinson--Gelfand--Gelfand category $\mathcal O$ associated to the Lie algebra of $G$.
	The Serre functors for similar categories of importance to representation theory are further explored in \cite{GaitsgoryYomDin:2018:AnalogOfDelignLusztigDuality}.
\end{Example}

\begin{Remark}
	Serre functors are particularly useful for producing adjoint functors.
	If $F\colon \cat C \to \cat D$ is a functor between $\kk$-linear categories with Serre functors $S_{\cat C}$ and $S_{\cat D}$ respectively, then
	\[
	F^L \cong S_{\cat C}^{-1}F^R S_{\cat D},
	\]
	where $F^R$ and $F^L$ are the right and left adjoint of $F$.
	Indeed, for $x \in \cat C$ and $y \in \cat D$ one has
	\[
	\Hom_{\cat D}(y,\, Fx) \cong
	\Hom_{\cat D}(Fx,\, S_{\cat D}y)^* \cong
	\Hom_{\cat C}(x,\, F^RS_{\cat D}y)^* \cong
	\Hom_{\cat C}(S_{\cat C}^{-1}F^RS_{\cat D}y,\, x).
	\]
	Our usage of the Serre functor in the definition of the Heisenberg category is closely related to this observation.
\end{Remark}

The Serre functor $S$ induces a \emph{Serre trace\index{Serre trace}} map 
\begin{equation}\label{eq:Serre-trace-add}
	\Tr \colon \homm_{\A}(a,\, Sa) \rightarrow \kk, 
	\quad \quad
	\alpha \mapsto \eta_{a,a}(\id_{a})(\alpha).
\end{equation}

\begin{Proposition}
	\label{prop:serre_trace_cyclic_additive}
	Let $\C$ be a Hom-finite $\kk$-linear category which admits a Serre functor $S$.
	For any $a,b \in \C$ and any $\alpha \in \Hom_\C(a,\,b)$, $\beta \in \Hom_\C(b,\,Sa)$ we have 
	\begin{equation*}
		\Tr(\beta \circ \alpha) = (-1)^{\deg\alpha\deg\beta}\Tr(S\alpha \circ \beta).
	\end{equation*}
\end{Proposition}

\begin{proof}
	We note that if $\cat C$ is a graded category, then composition of two morphisms $\alpha$ and $\beta$ in $\cat C^{\opp}$ is twisted by $(-1)^{\deg\alpha\deg\beta}$.
	Thus acting on the first argument of the bifunctor\index{bifunctor} $\Hom_{\C}(-,-)\colon \C^{\opp} \times \C \to \catgrVect$ involves a sign twist.
	Naturality of $\eta$ therefore implies that the diagram
	\begin{equation*}
		\begin{tikzcd}
			\homm_\A(b,b)
			\ar{d}[']{(-1)^{\deg(-)\deg(\alpha)} (-) \circ \alpha}
			\ar{r}{\eta}[']{\sim}
			& 
			\homm_\A(b,Sb)^*
			\ar{d}{f(-) \mapsto 
				(-1)^{\deg(f)\deg(\alpha)}f\left(S\alpha \circ (-)\right)}
			\\
			\homm_\A(a,b)
			\ar{r}{\eta}[']{\sim}
			&
			\homm_\A(b,Sa)^*
			\\
			\homm_\A(a,a)
			\ar{r}{\eta}[']{\sim}
			\ar{u}{\alpha \circ (-)}
			&
			\homm_\A(a,Sa)^*.
			\ar{u}[']{f(-) \mapsto (-1)^{(\deg(f) +
					\deg(-))\deg(\alpha)}f\left((-) \circ \alpha\right)}
		\end{tikzcd}
	\end{equation*}
	commutes.
	
	Chasing $\id_b$ through the upper square and $\id_a$ through the lower square yields
	\[
	\Tr(S\alpha \circ -) = \eta(\alpha)(-) = (-1)^{\deg \alpha\deg(-)} \Tr(- \circ \alpha),
	\]
	whence the desired assertion follows. 
\end{proof}

\section{Heisenberg algebra\index{Heisenberg algebra}s}\label{subsec:heisenberg_algebra}

Recall that a lattice\index{lattice} is a free $\ZZ$-module $M$ of finite rank equipped with a bilinear form
\[ 
\chi\colon M \times M \to \ZZ, \quad v,w \mapsto \langle v,w \rangle_{\chi}.
\]
We do not require the form $\chi$ to be symmetric or antisymmetric; to the knowledge of the authors no treatment of Heisenberg algebras has been this general.
If the bilinear form $\chi$ on $M$ is degenerate, then the Heisenberg algebra defined as below has a non-trivial centre.
Thus it is common to assume that $\chi$ is non-degenerate and we do so from now on.

Let $(M,\chi)$ be a lattice.
As a preliminary definition of the Heisenberg algebra we let $\chalg{M} \coloneqq \chalg{(M,\chi)}$ to be the unital $\kk$-algebra with generators $p_{a}^{(n)}$, $q_{a}^{(n)}$ for $a \in M$ and integers $n\geq 0$ modulo the following relations for all $a,b\in M$ and $n,m\geq 0$:
\begin{equation}\label{eq:heisrel0}
	p_a^{(0)} = 1 = q_a^{(0)},
\end{equation}
\begin{equation}\label{eq:heisrel1}
	p_{a+b}^{(n)} = \sum_{k=0}^{n} p_{a}^{(k)}p_{b}^{(n-k)}
	\quad\text{and}\quad 
	q_{a+b}^{(n)} = \sum_{k=0}^{n}  q_{a}^{(k)}q_{b}^{(n-k)},
\end{equation}
\begin{equation}\label{eq:heisrel2}
	p_{a}^{(n)}p_{b}^{(m)} = p_{b}^{(m)}p_{a}^{(n)}
	\quad\text{and}\quad
	q_{a}^{(n)}q_{b}^{(m)} = q_{b}^{(m)}q_{a}^{(n)},
\end{equation}
\begin{equation}\label{eq:heisrel3}
	q_{a}^{(n)}p_{b}^{(m)} = 
	\sum_{k = 0}^{\mathclap{\min(m,n)}} s^k \langle a, b \rangle_{\chi}\, p_{b}^{(m-k)}q_{a}^{(n-k)}.
\end{equation}
Here for any pair of integers $k\ge 0$ and $r$ we set
\[ 
s^k r \coloneqq \binom{r+k-1}{k} = \frac{1}{k!}(r+k-1)(r+k-2)\dotsm(r+1)r,
\]
which for positive $r$ coincides with the dimension of the $k$-th symmetric power\index{symmetric power} of a vector space of dimension $r$, that is,
\[ s^k r = \dim (S^k(\CC^r)),\]
and for negative $r$ analogously
\[ s^k r = (-1)^k\dim (\Lambda^k(\CC^{-r})).\]
We use the convention that $p^{(n)}_a = q^{(n)}_b = 0$ for $n <0$.

Let $r = \operatorname{rank} M$ and fix an identification $M \cong \ZZ^r$.
Let $S$ and $T$ be integral $r \times r$ matrices which are invertible over $\ZZ$. In particular, both $S$ and $T$ are unimodular.
Moreover, the form 
\begin{equation}
	\label{eqn-heis-change-of-form}
	\langle a , b \rangle_{S\chi T} \coloneqq \langle S^{ t} a ,\, T b\rangle_{\chi}
\end{equation}
gives again a new pairing on $M$. It is non-degenerate if and only if $\chi$ is non-degenerate.
If $X$ denotes the matrix of $\chi$ in the chosen basis of $M$, then the matrix of $S \chi T$ 
is $SXT$. To the knowledge of the authors, the following observation has not yet appeared 
in the literature.

\begin{Lemma}
	Let $S$ and $T$ be as above.
	The algebras $\chalg{(M,\,\chi)}$ and $\chalg{(M,\, S \chi T)}$ are isomorphic.
\end{Lemma}

\begin{proof}
	Define a map $\chalg{(M,\, S \chi T)} \rightarrow \chalg{(M,\,\chi)}$
	on generators by 
	\begin{equation}\label{eq:heisiso}
		\begin{aligned}
			q_{a}^{(n)} & \mapsto q_{S^{ t} a}^{(n)} \quad \quad \text{ and } \quad \quad 
			p_{a}^{(n)} & \mapsto p_{Ta}^{(n)}.
		\end{aligned}
	\end{equation}
	As $S$ and $T$ are invertible, it is a bijection on the sets of
	generators. It remains to show that it respects the relations. This 
	is immediate for relations \eqref{eq:heisrel0}--\eqref{eq:heisrel2},
	while for relation \eqref{eq:heisrel3} it follows from
	\eqref{eqn-heis-change-of-form}. 
\end{proof}

\begin{Corollary}\label{cor:halg_iso_to_symmetric}
	The Heisenberg algebra on every lattice is isomorphic to one which is induced by a symmetric (in fact, a diagonal) form.
\end{Corollary}

\begin{proof}
	The Smith normal form of $\chi$ (or more precisely of its matrix $X$) provides matrices $S$ and $T$, such that $S\chi T$ (in fact, $SXT$) is diagonal.
\end{proof}

\begin{Remark}
	The result above says that every Heisenberg algebra arises as the
	Heisenberg algebra of a lattice with a symmetric pairing. 
	In the geometrical context, our lattice is the numerical Grothendieck 
	group\index{numerical Grothendieck group} 
	of an algebraic variety and our pairing is the Euler pairing\index{Euler pairing}. 
	Drawing loose parallels, it is 
	tempting to interpret the result above as saying that Heisenberg
	algebra is an intrinsically Calabi–Yau construction. It is certainly 
	the case in the original constructions by Khovanov
	\cite{khovanov2014heisenberg} who works on a point, by Cautis and Licata 
	\cite{cautis2012heisenberg} who work on a minimal
resolution\index{minimal resolution} of an $ADE$ singularity\index{ADE
singularity}, and by Grojnowski and Nakajima 
	\cite{grojnowski1995instantons, nakajima1997heisenberg} who make use
	of the Poincar{\'e} duality on cohomology. 
	
	The authors hope to revisit this issue in a future work which would
	extend our categorification\index{categorification} from Heisenberg agebras to the associated 
	vertex algebras.
\end{Remark}

\begin{Remark}
	When $\chi$ is symmetric, the matrices $S$ and $T$ can be chosen to be equal. Hence, they represent a base change on the underlying lattice $M$. Moreover, in this case there is another common set of generators of the Heisenberg algebra.
	It is given by polynomials (possibly with constant term) on the symbols $a_b(n)$ for $n\in \ZZ \setminus \{0\}$, $b \in M$.
	The set of relations between these is given by
	\[
	\bigl[a_{b}(m),a_{c}(n)\bigr] = \delta_{m,-n}  m\langle b,c\rangle_{\chi}.
	\]
	The proof that these (in the symmetric case) define the same algebra is given for example in \cite[Lemma 1.2]{krug2018symmetric}. The advantage of using the presentation \eqref{eq:heisrel0}--\eqref{eq:heisrel3} is that it also makes sense when $\chi$ is not symmetric. Hence, it is more natural in our context.
\end{Remark}

\subsection{Idempotent modification}
\label{subsubsec:idempotent_modification}

In this paper, we do not work with the Heisenberg algebra $\chalg{M}$ itself,  
but with its idempotent modification\index{idempotent modification} $\halg M$. We define it as
follows. 

Recall that a unital $\kk$-algebra $R$ is the same as a $\kk$-linear category 
$\C$ with a single object whose endomorphism space is $R$. Similarly, 
a unital algebra $R$ with a choice of a decomposition 
$1_R = \sum_1^n 1_i$ of its unit into a finite sum of orthogonal 
idempotents can be viewed as a $\kk$-linear category $\C$ whose objects 
are $\{ 1, \dots, n \}$ and whose $\homm$-spaces are given by 
$\homm_{\C}(i,j) = 1_j R 1_i$. Conversely, we can recover $R$ from
$\C$ as a direct sum of its $\homm$-spaces. 

We would like to decompose the unit of $\chalg{M}$ into an infinite sum 
of idempotents $\sum_{i \in \ZZ} 1_i$. This is not possible directly, 
as infinite sums of elements are not well-defined. However, the categorical 
analogy above suggests the following construction. 

Introduce a $\mathbb{Z}$-grading on $\chalg{M}$ by setting 
$\deg p^{(m)}_a = m$ and $\deg q^{(n)}_a = -n$ for all $n,m \in
\mathbb{Z}$ and $a \in M$. 
Let $\C_{M}$ be a category whose object set is $\mathbb{Z}$ and whose
$\homm$-space $\homm_{\C_M}(i,j)$ is the degree $j-i$ part of
$\chalg{M}$. The identity element $1_i$ in each $\homm_{\C_M}(i,i)$ is
the corresponding copy of the unit $1$ of $\chalg{M}$. The composition 
is given by multiplication in $\chalg{M}$. For any
element $x \in \chalg{M}$ of degree $j - i$ we write $1_j x$, $1_j x 1_i$ or $x 1_i$
to differentiate the copy of $x$ in $\homm_{\C_M}(i,j)$ from 
its counterparts in any other $\homm_{\C_M}(l, l+j-i)$.

Now let $\halg M$ be the direct sum of $\homm$-spaces of $\C_{M}$:
\[ \halg M \coloneqq \bigoplus_{i,j \in \ZZ} \homm_{\C_M}(i,j). \]
This is a non-unital algebra as it does not contain the infinite sum 
$\sum_{i \in \ZZ} 1_i$. Instead, it has a collection of orthogonal 
idempotents $\{ 1_i \}_{i \in \mathbb{Z}}$ and each defining relation 
\eqref{eq:heisrel0}--\eqref{eq:heisrel3} of 
the unital algebra $\chalg{M}$ gives rise, for each $i \in \ZZ$, 
to a relation in $\halg M$. Namely, take the original 
relation and add the idempotent $1_k$ at the end of each expression. 
For example,
\[ p_{a}^{(n)}p_{b}^{(m)}1_i = p_{b}^{(m)}p_{a}^{(n)}1_i, \quad a,b \in M, \, n,m \in \NN,\, i \in \ZZ. \]
Note that elements $p^{(m)}_a$ and $q^{(n)}_a$ themselves do not 
exist in $\halg M$ anymore, as they should correspond 
to infinite sums $\sum_{i \in \ZZ}p^{(m)}_a 1_i$ and 
$\sum_{i \in \ZZ}q^{(n)}_a 1_i$. 

We have a canonical projection $\halg M \rightarrow \chalg{M}$ given 
by sending each idempotent $1_i$ to the unit $1_{\chalg{M}}$. 
A representation of the category $\C_M$ into the category of vector
spaces is the same as a graded module over $\halg M$. Moreover, 
any graded module over $\chalg{M}$
induces a representation of $\halg M$ via restriction of scalars.

\subsection{The transposed generators}
\label{subsubsec:transposedgenalg}

Fix $a\in M$ and let $z$ be a formal variable. Let 
\[ \sum_{n \geq 0}p_{a}^{(n)}z^n \quad \textrm{and} \quad \sum_{n \geq 0} q_{a}^{(n)}z^n\]
be the generating series of the $p$, resp. $q$ elements associated with $a$.
Define a new set of elements $p^{(1^n)}_a$ and $q^{(1^n)}_a$, $n \in \ZZ_{> 0}$ so that the generating series
\[ \sum_{n \geq 0}(-1)^np_{a}^{(1^n)}z^n \quad \textrm{and} \quad \sum_{n \geq 0} (-1)^nq_{a}^{(1^n)}z^n\]
are the inverses of those of $p^{(n)}$ and $q^{(n)}$ respectively:
\[  \left(\sum_{n \geq 0}p_{a}^{(n)}z^n\right) \left(\sum_{n \geq 0}(-1)^np_{a}^{(1^n)}z^n\right)=1 \]
and
\[  \left(\sum_{n \geq 0}q_{a}^{(n)}z^n\right) \left(\sum_{n \geq 0}(-1)^nq_{a}^{(1^n)}z^n\right)=1.\]
Compare \cite[Section~2.2.2]{cautis2012heisenberg} and \cite[Section~3.2]{krug2018symmetric}. One can show that the relations among these generators are exactly the same as
those between the $p_a^{(n)}$ and $q_a^{(n)}$, just replace $(n)$ by $(1^n)$ everywhere. In particular, they also give a set of generators of $\chalg{M}$.
Additionally, for all $a, b \in M$ one has the following relations:
\begin{equation*}
	\begin{gathered}
		p_{a}^{(n)}p_{b}^{(1^n)}=p_{b}^{(1^n)}p_{a}^{(n)},\quad q_{a}^{(n)}q_{b}^{(1^n)}=q_{b}^{(1^n)}q_{a}^{(n)}\\
		q_{a}^{(1^{n})}p_{b}^{(m)} = \textstyle\sum_{k=0}^{\min(m,n)} s^k\left(-\langle a,b \rangle_{\chi} \right)  p_{b}^{(n-k)}q_a^{(1^{m-k})}.
	\end{gathered}
\end{equation*}

\subsection{The Fock space\index{Fock space}}

Let $\halg{M}^- \subset \halg M$ denote the subalgebra generated by the set 
\[
\bigl\{ q_a^{(n)}1_k : a \in M,\, k \leq 0,\, n \geq 0 \bigr\}.
\] 
Let $\triv_0$ denote the trivial representation of $\halg{M}^-$, where $1_0$ acts as identity and $1_k$ acts by zero for $k < 0$.
The Fock space representation of the Heisenberg algebra $\halg M$ is defined as the induced representation
\[ 
\falg M = \Ind_{\halg{M}^-}^{\halg M}(\triv_0) \cong \halg M \otimes_{\halg{M}^-} \kk.
\]
We note that in $\falg M$ one has $1_k \otimes 1 = 1_k \otimes (1_0
\cdot 1) = 1_k1_0 \otimes 1 = 0$ for all $k \ne 0$.
It follows that $\falg M$ is generated by elements  
$p_a^{(n)} 1_0$ for $a \in M$ and $n \geq 0$. 
The $\ZZ$-grading on $\chalg{M}$ induces a grading on $\falg M$
where the degree $k$ part is canonically isomorphic to 
\begin{equation}
	\label{eq:fockdegk}
	\falg{M}^k \simeq \bigoplus_{k_1+2k_2+\dots=k}\bigotimes_{i}\Sym^{k_i}  (M \otimes_\ZZ {\kk}).\end{equation}
The idempotent $1_k \in \halg M$ acts by projection onto $\falg{M}^k$.
Alternatively, the Fock space can be described as $\falg M=\halg M/I$
where $I$ is the left ideal generated by the operators 
$1_k$ for $k \ne 0$ and  $q_a^{(n)}1_{k}$ for $k = 0$ and $n > 0$.

For $\chi$ non-degenerate, the Fock space is an irreducible and faithful representation of $\halg{M}$ with highest weight vector 1. If $\chi$ is of rank 1, irreducibility and faithfulness follows from the description of the Fock space representation as differential operators on an infinite polynomial algebra \cite[Section~2]{frenkel1981two}. As the form can be chosen to be diagonal, the higher rank case follows by taking a direct sum; the Fock space of the Heisenberg algebra of a direct sum of lattices is the tensor product of the Fock spaces of the Heisenberg algebras of the summands. Hence the representation can 
be described as differential operators on a polynomial algebra.

The next claim follows from the definition and irreducibility of the Fock space.

\begin{Lemma}\label{lem:fock_embeds}
	Let $\halg M \to \End(V)$ be a representation and let $v \in V$ be an element annihilated by $\halg{M}^- \setminus \{ 1_0\}$  which is invariant under $1_0$.
	Then the map $1 \mapsto v$ induces an embedding $\falg M \to V$ of $\halg M$-representations.
\end{Lemma}

\chapter{The Additive Heisenberg $2$-category}\label{sec:additive-Heisenberg-2cat}

In this section, we fix a $\Hom$-finite graded $\kk$-linear category
$\basecat$ which is  closed under shifts and has a Serre
	functor $S$. We then define a $2$-category
$\hcatadd\basecat$, the \emph{(additive) Heisenberg category}\index{additive Heisenberg category} of $\basecat$.
We present the results in this section for graded categories for comparison with the homotopy category of the dg version in Chapter~\ref{sec:dg-Heisenberg-2-cat}.
Any $\kk$-linear category can be seen as graded $\kk$-linear by
viewing the $\Hom$-spaces as placed in degree $0$. In such case 
all sign rules in this section can be ignored.

The category $\hcatadd\basecat$ is the Karoubi completion\index{Karoubi completion} of a simpler $2$-category $\hcatadd*\basecat$ which we set up in the following first two subsections.
This additive version of the Heisenberg category is less powerful than
the \dg version constructed in Chapter~\ref{sec:dg-Heisenberg-2-cat}.
We include it in the paper as it might be of wider interest and because the similarities and differences to the earlier constructions are more readily apparent in the purely $\kk$-linear setting.

{ In our constructions, we want to work with objects of {the} form $a \otimes V$ 
	where $a \in \basecat$ and $V \in \mathcal{G}rVect^{\mathrm{fin}}$,
	the category of finite-dimensional graded vector spaces. By this we
	mean a direct sum of $\dim V$ shifted copies of $a$ indexed by a choice of
	basis of $V$. The maps between two such objects $a \otimes V$ and $b
	\otimes W$ then correspond to matrices with values in $\homm_\basecat(a,b)$.  
	
	To do this without having to choose a basis, we replace $\basecat$ by 
	the category $\basecat \otimes_k  \mathcal{G}rVect^{\mathrm{fin}}$
	which is (non-canonically) equivalent to $\basecat$. The equivalence 
	is defined by choosing a {homogeneous } basis $\left\{e_1, \dots, e_n\right\}$ for
	every $V \in \mathcal{G}rVect^{\mathrm{fin}}$ and setting 
	$$ (a,V) \mapsto \bigoplus_{e_i \in \left\{e_1, \dots, e_n\right\} } 
	a[\deg(e_i)] \quad \quad a \in \basecat, V \in 
	\mathcal{G}rVect^{\mathrm{fin}} $$
	$$ \alpha \otimes \beta \mapsto \sum \beta_{ij}\left(a[\deg(e_j)]
	\xrightarrow{\alpha} b[\deg(f_i)]\right)
	\quad \quad \alpha \in \homm_{\basecat}(a,b), \beta \in \homm(V,W) $$
	where $(\beta_{ij})$ is the matrix of $\beta$ with respect to the
	chosen bases. 
	
	The inverse equivalence is given by 
	$$ a \mapsto (a,\kk) \quad \quad a \in \basecat $$
	$$ \alpha \mapsto \alpha \otimes \id \quad \quad \alpha \in
	\homm_\basecat(a,b). $$
}

\section{The category \texorpdfstring{$\hcatadd*\basecat$}{H'}: generators}
\label{subsec:heisencatdef-additive}

We now define a (strict) $2$-category $\hcatadd*\basecat$.
The objects of $\hcatadd*\basecat$ are the integers $\ho \in \ZZ$.

The $1$-morphism\index{$1$-morphism} categories are additive graded $\kk$-linear
categories whose $1$-morphisms are freely generated under
	$1$-composition by symbols
\[
\PP_a\colon \ho \to \ho+1
\quad\text{and}\quad
\QQ_a\colon \ho+1 \to \ho
\]
for each $a \in \basecat$ and $\ho \in \ZZ$.
Thus the objects of $\Hom_{\hcatadd*\basecat}(\ho,\,\ho*)$ are direct sums of finite strings generated by the symbols $\PP_a$ and $\QQ_a$ with $a \in \basecat$, such that the difference of the number of $\PP$'s and the number of $\QQ$'s in each summand is $\ho*-\ho$. 
The identity $1$-morphism of any $\ho \in \ZZ$ is denoted by $\hunit$. 

Strictly speaking, one should distinguish between $1$-morphisms with different sources in the notation, i.e.~write $\PP_a\hunit_{\ho}$ and $\hunit_{\ho}\QQ_a$. 
However, we will have $\Hom_{\hcatadd*\basecat}(\ho,\,\ho*) =
\Hom_{\hcatadd*\basecat}(\ho+i,\,\ho*+i)$ for each integer $i$, 
and do not distinguish these in our notation.

The $2$-morphism\index{$2$-morphism}s between a pair of $1$-morphisms form a $\kk$-vector space. 
These vector spaces are freely generated by a number of generators listed below, subject to the axioms of a strict $2$-category as well as certain relations which we detail in the next subsection.
We usually represent these $2$-morphisms as planar diagrams.
This requires certain sign rules, see Remark~\ref{rem:$2$-morphism-interchange-additive} below.
The diagrams are read bottom to top, i.e.~the source of a given $2$-morphism lies on the lower boundary, while the target lies on the upper boundary.

The $2$-morphism spaces are generated by three types of symbols.
Firstly, for every $\alpha \in \Hom_\basecat(a,b)$ there are arrows
\[
\begin{tikzpicture}[baseline={(0,0.4)}]
	\draw[->] (0,0) node[below] {$\PP_a$} -- node[label=right:{$\alpha$}, dot, pos=0.5] {} (0,1) node[above] {$\PP_{b}$};
\end{tikzpicture}
\qquad\text{and}\qquad
\begin{tikzpicture}[baseline={(0,0.4)}]
	\draw[->] (0,1) node[above] {$\QQ_a$} -- node[label=right:{$\alpha$}, dot, pos=0.5] {} (0,0) node[below] {$\QQ_{b}$};
\end{tikzpicture}.
\]
These $2$-morphisms are homogeneous of degree $|\alpha|$.
The remaining generators listed below are all of degree $0$.
By convention a strand without a dot is the same as one marked with the identity morphism.
Any such unmarked strand is an identity $2$-morphism in $\hcatadd*\basecat$.
The identity $2$-morphisms of the identity $1$-morphism $\hunit$ are
denoted by blank space.

Secondly, for any object $a \in \basecat$ there are cup\index{cup}s and cap\index{cap}s
\[
\begin{tikzpicture}[baseline={(0,0.25)}]
	\draw[->] (0,0) node[below] {$\PP_a$} arc[start angle=180, end angle=0, radius=.5] node[label=above:{$\hunit$},pos=0.5]{} node[below] {$\QQ_a$};
\end{tikzpicture},\quad
\begin{tikzpicture}[baseline={(0,0.25)}]
	\draw[->] (0,0) node[below] {$\PP_{Sa}$} arc[start angle=0, end angle=180, radius=.5] node[label=above:{$\hunit$},pos=0.5]{} node[below] {$\QQ_{a}$};
\end{tikzpicture},\quad
\begin{tikzpicture}[baseline={(0,-0.25)}]
	\draw[->] (0,0) node[above] {$\QQ_a$} arc[start angle=0, end angle=-180, radius=.5] node[label=below:{$\hunit$},pos=0.5]{} node[above] {$\PP_{Sa}$};
\end{tikzpicture},\quad
\begin{tikzpicture}[baseline={(0,-0.25)}]
	\draw[->] (0,0) node[above] {$\QQ_{a}$} arc[start angle=-180, end angle=0, radius=.5] node[label=below:{$\hunit$},pos=0.5]{} node[above] {$\PP_{a}$};
\end{tikzpicture}.
\]

Thirdly, for any pair of objects $a,\, b \in \basecat$ there is a crossing of two downward\footnote{We use the downward crossing rather than the upward crossing as a basic generator since in the \dg version of the Heisenberg category described in Chapter~\ref{sec:dg-Heisenberg-2-cat} this will lead to a more symmetric presentation.} strands: 
\begin{equation*}
	\begin{tikzpicture}[baseline={(0,0.4)}]
		\draw[<-] (0,0) node[below] {$\QQ_{a}$}-- (1,1) node[above] {$\QQ_{a}$};
		\draw[<-] (1,0) node[below] {$\QQ_{b}$}-- (0,1) node[above] {$\QQ_{b}$};
	\end{tikzpicture}.
\end{equation*}

For convenience, we define three further types of strand crossings from this basic one by composition with cups and caps:
\begin{equation}\label{eq:mixed-crossings-add}
	\begin{tikzpicture}[baseline={(0,1.45)}]
		\draw[->] (0,1) node[below] {$\PP_{a}$} -- (1,2) node[above] {$\PP_{a}$};
		\draw[->] (0,2) node[above] {$\QQ_{b}$} -- (1,1) node[below] {$\QQ_{b}$};
		\draw (1.5,1.5) node {$\coloneqq$};
		\draw[->, decoration={markings, mark=at position 0.25 with {\arrow{>}}, mark=at position 0.575 with {\arrow{>}}}, postaction={decorate}] 
		(2,0) node[below] {$\PP_{a}$} --
		(2,2) arc[start angle=180, end angle=0, radius=.5] --
		(4,1) arc[start angle=180, end angle=360, radius=.5] --
		(5,3) node[above] {$\PP_{a}$}; 
		\draw[<-, decoration={markings, mark=at position 0.32 with {\arrow{<}}, mark=at position 0.74 with {\arrow{<}}}, postaction={decorate}] 
		(3,0) node[below] {$\QQ_{b}$} --
		(3,1) --
		(4,2) --
		(4,3) node[above] {$\QQ_{b}$}; 
	\end{tikzpicture} 
	,\quad
	\begin{tikzpicture}[baseline={(0,1.45)}]
		\draw[->] (1,1) node[below] {$\PP_{b}$} -- (0,2) node[above] {$\PP_{b}$};
		\draw[<-] (0,1) node[below] {$\QQ_{a}$} -- (1,2) node[above] {$\QQ_{a}$};
		\draw (1.5,1.5) node {$\coloneqq$};
		\draw[->, decoration={markings, mark=at position 0.25 with {\arrow{>}}, mark=at position 0.575 with {\arrow{>}}}, postaction={decorate}] 
		(5,0) node[below] {$\PP_{b}$} --
		(5,2) arc[start angle=0, end angle=180, radius=.5] -- 
		(3,1) arc[start angle=360, end angle=180, radius=.5] --
		(2,3) node[above] {$\PP_{b}$}; 
		\draw[->, decoration={markings, mark=at position 0.29 with {\arrow{>}}, mark=at position 0.70 with {\arrow{>}}}, postaction={decorate}]
		(3,3) node[above] {$\QQ_{a}$} --
		(3,2) -- 
		(4,1) -- 
		(4,0) node[below] {$\QQ_{a}$}; 
	\end{tikzpicture},
\end{equation}
\begin{equation}\label{eq:up-crossing-add}
	\begin{tikzpicture}[baseline={(0,1.45)}]
		\draw[->] (0,1) node[below] {$\PP_{a}$} -- (1,2) node[above] {$\PP_{a}$};
		\draw[<-] (0,2) node[above] {$\PP_{b}$} -- (1,1) node[below] {$\PP_{b}$};
		\draw (1.5,1.5) node {$\coloneqq$};
		\draw[->, decoration={markings, mark=at position 0.25 with {\arrow{>}}, mark=at position 0.575 with {\arrow{>}}}, postaction={decorate}] 
		(2,0) node[below] {$\PP_{a}$} --
		(2,2) arc[start angle=180, end angle=0, radius=.5] --
		(4,1) arc[start angle=180, end angle=360, radius=.5] --
		(5,3) node[above] {$\PP_{a}$}; 
		\draw[->, decoration={markings, mark=at position 0.28 with {\arrow{>}}, mark=at position 0.70 with {\arrow{>}}}, postaction={decorate}] 
		(3,0) node[below] {$\PP_{b}$} --
		(3,1) --
		(4,2) --
		(4,3) node[above] {$\PP_{b}$}; 
	\end{tikzpicture}.
\end{equation}

\begin{Remark}
	\label{rem:$2$-morphism-interchange-additive}
	We draw compositions of basic $2$-morphisms as planar diagrams, as in
	\eqref{eq:mixed-crossings-add}-\eqref{eq:up-crossing-add}. In the ungraded case, the
	interchange law of $2$-categories guarantees that such diagrams can be 
	read without ambiguity.
	
	However, the interchange law for graded $2$-categories includes a sign:
	\begin{equation}
		\label{eq:interchange-law-for-graded-2-categories}
		(\alpha \circ_1 \beta) \circ_2 (\gamma \circ_1 \delta) = (-1)^{|\beta||\gamma|} (\alpha \circ_2 \gamma) \circ_1 (\beta \circ_2 \delta),
	\end{equation}
	where we write $\circ_1$ and $\circ_2$ for the 1- and $2$-composition\index{$2$-composition} operations respectively.
	This can lead to ambiguities.
	For example, the diagram
	
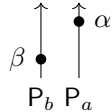
\[
	\begin{tikzpicture}[scale=0.5]
		\draw[->] (0,0) node[below] {$\PP_b$} -- node[label=left:{$\beta$}, dot] {} (0,1) -- (0,2);
		\draw[->] (1,0) node[below] {$\PP_a$} -- (1,1) -- node[label=right:{$\alpha$}, dot] {} (1,2);
	\end{tikzpicture}
	\]
	could be read either as the $1$-composition of
	
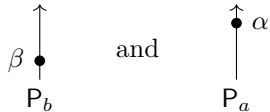
\[
	\begin{tikzpicture}[scale=0.5, baseline={(0,0.3)}]
		\draw[->] (0,0) node[below] {$\PP_b$} -- node[label=left:{$\beta$}, dot] {} (0,1) -- (0,2);
	\end{tikzpicture}
	\qquad
	\text{and}
	\qquad
	\begin{tikzpicture}[scale=0.5, baseline={(0,0.3)}]
		\draw[->] (1,0) node[below] {$\PP_a$} -- (1,1) -- node[label=right:{$\alpha$}, dot] {} (1,2);
	\end{tikzpicture}
	\]
	or the $2$-composition of 
	
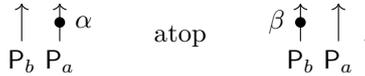
\[
	\begin{tikzpicture}[scale=0.5, baseline={(0,0)}]
		\draw[->] (0,0) node[below] {$\PP_b$} -- (0,1);
		\draw[->] (1,0) node[below] {$\PP_a$} -- node[label=right:{$\alpha$}, dot] {} (1,1);
	\end{tikzpicture}
	\qquad
	\text{atop}
	\qquad
	\begin{tikzpicture}[scale=0.5, baseline={(0,0)}]
		\draw[->] (0,0) node[below] {$\PP_b$} -- node[label=left:{$\beta$}, dot] {} (0,1);
		\draw[->] (1,0) node[below] {$\PP_a$} -- (1,1);
	\end{tikzpicture}.
	\]
	These differ by a factor of $(-1)^{|\alpha||\beta|}$.
	
	We impose the latter convention.
	Thus to read a diagram, one first slices it into lines containing no $2$-composition of basic $2$-morphisms and no dots at different heights.
	Every such line is a $1$-composition of basic $2$-morphisms, and the overall diagram is then the $2$-composition of these $1$-compositions.
	
	With this convention, a diagram with two or more dots at the same
	height represents the same $2$-morphism as the diagram with the
	rightmost of these dots moved a small distance downwards. 
	Graphically, $1$-composition corresponds to placing diagrams side-by-side and $2$-composition corresponds to stacking diagrams on top of each other.
\end{Remark}

\begin{Remark}
	When the domain or target of a diagram is irrelevant or evident from
	the context, we may omit the labels.
	This is the case usually with the empty string occurring as the target of caps and the domain of cups.
	We also usually smooth out the strings in the diagram. 
	For example, we may draw the left definition of \eqref{eq:mixed-crossings-add} more succinctly as
	\[
	\begin{tikzpicture}[baseline={(0,1.45)}]
		\draw[->] (0,1) -- (1,2);
		\draw[->] (0,2) -- (1,1);
		\draw (1.5,1.5) node {$\coloneqq$};
		\draw[->] 
		(2,1) to[out=90, in=180]
		(2.5,2) to[out=0, in=180]
		(3.5,1) to[out=0, in=270]
		(4,2);
		\draw[->] (3.5,2) -- (2.5,1);
	\end{tikzpicture}.
	\]
\end{Remark}

\section{The category \texorpdfstring{$\hcatadd*\basecat$}{H'}: relations between $2$-morphisms}
\label{subsec:heisencatdef2-additive}

In Section~\ref{subsec:heisencatdef-additive} we gave a list of 
generating symbols. The $2$-morphisms in 
$\hcatadd*\basecat$ are $1$- and $2$-compositions of these symbols, 
subject to the following list of relations. 
As a shorthand, a relation specified by an unoriented
diagram holds for all permissible orientations\index{permissible orientations} of this diagram. 

First, we impose the linearity relations\index{linearity relations}
\[
\begin{tikzpicture}[baseline={(0,0.4)}]
	\draw (0,0)
	-- node[label=left:{$\alpha$}, dot, pos=0.5] {}
	(0,1);
\end{tikzpicture}
+
\begin{tikzpicture}[baseline={(0,0.42)}]
	\draw (0,0)
	-- node[label=right:{$\beta$}, dot, pos=0.5] {}
	(0,1);
\end{tikzpicture}
=
\begin{tikzpicture}[baseline={(0,0.42)}]
	\draw (0,0)
	-- node[label=right:{$\alpha+\beta$}, dot, pos=0.5] {}
	(0,1);
\end{tikzpicture}
\qquad\qquad
c\,\;
\begin{tikzpicture}[baseline={(0,0.42)}]
	\draw (0,0)
	-- node[label=right:{$\alpha$}, dot, pos=0.5] {}
	(0,1);
\end{tikzpicture}
=
\begin{tikzpicture}[baseline={(0,0.42)}]
	\draw (0,0)
	-- node[label=right:{$c\alpha$}, dot, pos=0.5] {}
	(0,1);
\end{tikzpicture}
\]
for any $\alpha, \beta \in \Hom(a,b)$ and any scalar $c \in \kk$ for any compatible orientation of the strings.

Neighboring dots along a downward string can merge with a sign twist:
\begin{equation}\label{eq:colliding_dots_down-add}
	\begin{tikzpicture}[baseline={(0,-0.6)}]
		\draw[->] 
		(0,0) -- node[label=right:{$\alpha$}, dot, pos=0.33] {}
		node[label=right:{$\beta$}, dot, pos=0.66] {}
		(0,-1);
		\draw (1.5, -0.5) node {$=(-1)^{|\alpha||\beta|}$};
		\draw[->] (2.5,0) -- node[label=right:{$\beta\circ\alpha$}, dot, pos=0.5] {} (2.5,-1);
	\end{tikzpicture}.
\end{equation}

Dots may \enquote{slide} through caps and downwards crossings as follows:
\begin{equation}\label{eq:caps_and_dots-add}
	\begin{tikzpicture}[baseline=0]
		\draw[->] 
		(0,-0.2) node[below] {$\PP_a$}
		-- node[label=left:{$\alpha$}, dot, pos=1] {}
		(0,0) 
		arc[start angle=180, end angle=0, radius=.5]
		--
		(1,-0.2) node[below] {$\QQ_b$};
		\draw (1.5,0) node {$=$};
		\draw[->]
		(2,-0.2) node[below] {$\PP_a$}
		--
		(2,0)
		arc[start angle=180, end angle=0, radius=.5] 
		-- node[label=right:{$\alpha$}, dot, pos=0] {}
		(3,-0.2) node[below] {$\QQ_b$};
	\end{tikzpicture}
	\qquad\qquad
	\begin{tikzpicture}[baseline=0]
		\draw[<-]
		(0,-0.2) node[below] {$\QQ_b$}
		-- node[label=left:{$\alpha$}, dot, pos=1] {}
		(0,0) 
		arc[start angle=180, end angle=0, radius=.5]
		--
		(1,-0.2) node[below] {$\PP_{Sa}$};
		\draw (1.5,0) node {$=$};
		\draw[<-]
		(2,-0.2) node[below] {$\QQ_b$}
		--
		(2,0) 
		arc[start angle=180, end angle=0, radius=.5]
		-- node[label=right:{$S\alpha$}, dot, pos=0] {}
		(3,-0.2) node[below] {$\PP_{Sa}$};
	\end{tikzpicture}
\end{equation}
\begin{equation}\label{eq:downcross-and-dot-1-add}
	\begin{tikzpicture}[baseline={(0,0.4)}]
		\draw[<-] (0,0) -- node[label=left:{$\alpha$}, dot, pos=0.25] {} (1,1);
		\draw[<-] (1,0) -- (0,1);
	\end{tikzpicture}
	=
	\begin{tikzpicture}[baseline={(0,0.4)}]
		\draw[<-] (0,0) -- node[label=right:{$\alpha$}, dot, pos=0.75] {} (1,1);
		\draw[<-] (1,0) -- (0,1);
	\end{tikzpicture}.
\end{equation}
Note that when drawing diagrams, dots need to keep their relative heights when doing these operations in order to avoid accidentally introducing signs (cf.~Lemma~\ref{lem:dotslide-add} below).

Next, there are two sets of local relations for unmarked strings:
the \emph{adjunction relations}\index{adjunction relations}
\begin{equation}\label{eq:straighten-add}
	\begin{tikzpicture}[baseline={(0,0.9)}]
		\draw (0,0) -- (0,1)  arc[start angle=180, end angle=0, radius=.5] arc[start angle=-180, end angle=0, radius=.5] -- (2,2);
		\draw (2.5,1) node {$=$};
		\draw (3,0) -- (3,2);
		\draw (3.5,1) node {$=$};
		\draw (6,0) -- (6,1)  arc[start angle=0, end angle=180, radius=.5] arc[start angle=0, end angle=-180, radius=.5] -- (4,2);
	\end{tikzpicture}
\end{equation}
and the \emph{symmetric group\index{symmetric group}} relations\index{symmetric group relations} on downward strand\index{downward strand}s
\begin{equation}\label{eq:symmetric_group_relations-add}
	\begin{tikzpicture}[baseline={(0,0.9)}]
		\draw[rounded corners=15pt,<-] (0,0) -- (1,1) -- (0,2);
		\draw[rounded corners=15pt,<-] (1,0) -- (0,1) -- (1,2);
		\draw (1.5,1) node {$=$};
		\draw[<-] (2,0) -- (2,2) ;
		\draw[<-] (3,0) -- (3,2) ;
	\end{tikzpicture},
	\qquad\qquad\qquad
	\begin{tikzpicture}[baseline={(0,0.9)}]
		\draw[<-] (0,0) -- (2,2);
		\draw[rounded corners=15pt,<-] (1,0) -- (0,1) -- (1,2);
		\draw[<-] (2,0) -- (0,2);
		\draw (2.5,1) node {$=$};
		\draw[<-] (3,0) -- (5,2);
		\draw[rounded corners=15pt,<-] (4,0) -- (5,1) -- (4,2);
		\draw[<-] (5,0) -- (3,2);
	\end{tikzpicture}.
\end{equation}

Further, for any $\alpha \in \Hom_\basecat(a,Sa)$ and with $\Tr$ being
the Serre trace~\eqref{eq:Serre-trace-add} we have:
\begin{equation}\label{eq:circle_and_curl-add}
	\begin{tikzpicture}[baseline={(0,-0.1)}, xscale=-1]
		\draw[<-]
		(1,-0.9)  node[below] {$\QQ_{a}$} -- 
		(1,-0.5)  to[out=90, in=0]
		(0.3,0.5) to[out=180,in=90]
		(-0.1,0)  to[out=270,in=180]
		(0.3,-.5) to[out=0,in=270]
		(1,0.5)   -- 
		(1,0.9)   node[above] {$\QQ_{Sa}$};
	\end{tikzpicture}
	= 0,
	\qquad\qquad\qquad
	\begin{tikzpicture}[baseline={(0,0.15)}]
		\draw[decoration={markings, mark=at position 0.30 with {\arrow{>}}, mark=at position 0.805 with {\arrow{>}}}, postaction={decorate}]
		(1,0.1) -- node[label=right:{$\alpha$}, dot, pos=0.5] {}
		(1,0.4) arc[start angle=0, end angle=180, radius=.5] --
		(0,0.1) arc[start angle=180, end angle=360, radius=.5];
	\end{tikzpicture}
	=
	\Tr(\alpha).
\end{equation}

Finally, we have relations for crossings of opposite oriented strands. 
Consider the map 
\[     
\Psi\colon \Hom_\basecat(a, b) \otimes_\kk \Hom_\basecat(a, b)^* \to \Hom(\QQ_a\PP_b,\, \QQ_a\PP_b)
\]
sending $\alpha \otimes \beta \in \Hom(a, b) \otimes_\kk \Hom(a, b)^* \cong \Hom(a, b) \otimes_\kk \Hom(b, Sa)$ to
\[
\Psi(\alpha \otimes \beta) = \mkern-12mu 
\begin{tikzpicture}[baseline={(0,0.92)}]
	\draw[->] (4,1.8) node[above] {$\QQ_a$}
	arc[start angle=-180, end angle=0, radius=.5] node[label=right:{$\alpha$}, dot, pos=0.8] {}
	node[above] {$\PP_b$};
	\node at (4.5,1) {$\hunit$};
	\draw[<-] (4,0.2) node[below] {$\QQ_a$}
	arc[start angle=180, end angle=0, radius=.5] node[label=right:{$\beta$}, dot, pos=0.8] {} 
	node[below] {$\PP_b$};
\end{tikzpicture}.
\]
Consider $\id \in \End_\kk\bigl(\Hom(a, b)\bigr) \cong \Hom(a, b) \otimes_\kk \Hom(a, b)^*$. The final two relations are 
\begin{equation}\label{eq:up_down_braids-add}
	\begin{tikzpicture}[baseline={(0,0.9)}]
		\draw[->] (0,0) node[below] {$\PP_{a}$} -- (0.5,0.5) to[out=45, in=-45] (0.5,1.5) -- (0,2) node[above] {$\PP_{a}$};
		\draw[->] (1,2) node[above] {$\QQ_{b}$} -- (0.5,1.5) to[out=225, in=135] (0.5,0.5) -- (1,0) node[below] {$\QQ_{b}$};
		\draw (1.5,1) node {$=$};
		\draw[->] (2,0) node[below] {$\PP_{a}$} -- (2,2) node[above] {$\PP_{a}$} ;
		\draw[<-] (3,0) node[below] {$\QQ_{b}$} -- (3,2) node[above] {$\QQ_{b}$} ;
	\end{tikzpicture},
	\qquad\qquad\qquad
	\begin{tikzpicture}[baseline={(0,0.9)}]
		\draw[<-] (0,0) node[below] {$\QQ_{a}$} -- (0.5,0.5) to[out=45, in=-45] (0.5,1.5) -- (0,2) node[above] {$\QQ_{a}$};
		\draw[<-] (1,2) node[above] {$\PP_{b}$} -- (0.5,1.5) to[out=225, in=135] (0.5,0.5) -- (1,0) node[below] {$\PP_{b}$};
		\draw (1.5,1) node {$=$}; 
		\draw[<-] (2,0) node[below] {$\QQ_{a}$} -- (2,2) node[above] {$\QQ_{a}$} ;
		\draw[->] (3,0) node[below] {$\PP_{b}$} -- (3,2) node[above] {$\PP_{b}$} ;
		\draw (4,1) node {$-\ \ \Psi(\id)$};
	\end{tikzpicture}
\end{equation}

\section{Remarks on the $2$-morphism relations in \texorpdfstring{$\hcatadd*\basecat$}{HV'}}
\label{subsec:remarks_on_relations-additive}
In order to reduce the number of relations necessary to verify when defining a representation of the Heisenberg category, we have chosen to keep the number of generators and relations on the definition of $\hcatadd*\basecat$ small.
We now note some of their consequences.
One such consequence is that essentially we can homotopy deform string diagrams. This is made precise in the following sequence of lemmas.

\begin{Lemma}\label{lem:dotslide-add}
	Dots may freely \enquote{slide along} strands as well as through cups, caps and all types of crossings, picking up a sign when sliding past each other.
	That is, one has the following additional relations:
	\[
	\begin{tikzpicture}[baseline={(0,0.4)}]
		\draw (0,0) -- node[label=left:{$\alpha$}, dot, pos=0.33] {} (0,1);
		\draw (0.5, 0.5) node {$\cdots$};
		\draw (1,0) -- node[label=right:{$\beta$}, dot, pos=0.66] {} (1,1);
	\end{tikzpicture}
	=
	(-1)^{|\alpha||\beta|}
	\begin{tikzpicture}[baseline={(0,0.4)}]
		\draw (4,0) -- node[label=left:{$\alpha$}, dot, pos=0.66] {} (4,1);
		\draw (4.5, 0.5) node {$\cdots$};
		\draw (5,0) -- node[label=right:{$\beta$}, dot, pos=0.33] {} (5,1);
	\end{tikzpicture}
	\]
	\begin{equation*}
		\begin{tikzpicture}[baseline=0]
			\draw[->] 
			(0,0.2) node[above] {$\QQ_a$}
			-- node[label=left:{$\alpha$}, dot, pos=1] {}
			(0,0)   arc[start angle=180, end angle=360, radius=.5] 
			--
			(1,0.2) node[above] {$\PP_b$};
			\draw (1.5,0) node {$=$};
			\draw[->]
			(2,0.2) node[above] {$\QQ_a$}
			--
			(2,0)   arc[start angle=180, end angle=360, radius=.5] 
			-- node[label=right:{$\alpha$}, dot, pos=0] {}
			(3,0.2) node[above] {$\PP_b$};
		\end{tikzpicture}
		\qquad\qquad
		\begin{tikzpicture}[baseline=0]
			\draw[<-]
			(0,0.2) node[above] {$\PP_{Sb}$}
			-- node[label=left:{$S\alpha$}, dot, pos=1] {}
			(0,0)   arc[start angle=180, end angle=360, radius=.5]
			--
			(1,0.2) node[above] {$\PP_{a}$};
			\draw (1.5,0) node {$=$};
			\draw[<-]
			(2,0.2) node[above] {$\PP_{Sb}$}
			--
			(2,0)   arc[start angle=180, end angle=360, radius=.5]
			-- node[label=right:{$\alpha$}, dot, pos=0] {}
			(3,0.2) node[above] {$\PP_{a}$};
		\end{tikzpicture}
	\end{equation*}
	\[
	\begin{tikzpicture}[baseline=0]
		\draw[-] (0,0) -- node[label=left:{$\alpha$}, dot, pos=0.25] {} (1,1);
		\draw[-] (1,0) -- (0,1);
	\draw (1.5,0.5) node {$=$};
		\draw[-] (2,0) -- node[label=right:{$\alpha$}, dot, pos=0.75] {} (3,1);
		\draw[-] (3,0) -- (2,1);
	\end{tikzpicture}
	\qquad\qquad
	\begin{tikzpicture}[baseline=0]
		\draw[-] (0,0) -- (1,1);
		\draw[-] (1,0) -- node[label=right:{$\alpha$}, dot, pos=0.25] {} (0,1);
	\draw (1.5,0.5) node {$=$};
		\draw[-] (2,0) -- (3,1);
		\draw[-] (3,0) -- node[label=left:{$\alpha$}, dot, pos=0.75] {} (2,1);
	\end{tikzpicture}.
	\]
\end{Lemma}

\begin{proof}
	The first relation is simply a graphical depiction of the interchange law in graded $2$-categories.
	The relations in the second line follow from those in~\eqref{eq:caps_and_dots-add} by applying~\eqref{eq:straighten-add}:
	\[
	\begin{tikzpicture}[baseline={(0,-0.2)}]
		\draw[->] 
		(0,0.2) -- node[label=left:{$\alpha$}, dot, pos=1] {}
		(0,0)   arc[start angle=180, end angle=360, radius=.5]
		--
		(1,0.2);
	\end{tikzpicture}
	\ = \ 
	\begin{tikzpicture}[baseline={(0,-0.2)}]
		\draw[->]
		(0,0.2)    -- 
		(0,0)      to[out=270, in=180]
		(0.5,-0.5) to[out=0, in=180]
		(1.5, 0.2) to[out=0, in=180] node[label=right:{$\alpha$}, dot, pos=0.5] {}
		(2.5,-0.5) to[out=0, in=270]
		(3,0)      --
		(3,0.2);
	\end{tikzpicture}
	\ = \ 
	\begin{tikzpicture}[baseline={(0,-0.2)}]
		\draw[->]
		(2,0.2) --
		(2,0)  arc[start angle=180, end angle=360, radius=.5] 
		-- node[label=right:{$\alpha$}, dot, pos=0] {}
		(3,0.2);
	\end{tikzpicture}.
	\]
	Relations~\eqref{eq:downcross-and-dot-1-add} and~\eqref{eq:symmetric_group_relations-add} imply:
	\[
	\begin{tikzpicture}[baseline={(0,0.4)}, scale=0.8]
		\draw[<-] (0,0) -- (1,1);
		\draw[<-] (1,0) -- node[label=right:{$\mathrlap{\alpha}$}, dot, pos=0.25] {} (0,1);
	\end{tikzpicture}
	=
	\begin{tikzpicture}[baseline={(0,1.2)}, scale=0.8]
		\draw[<-] (0,0) -- (0,2) -- (1,3);
		\draw[<-] (1,0) -- node[label=right:{$\mathrlap{\alpha}$}, dot, pos=1] {}  (1,2) -- (0,3);
	\end{tikzpicture}
	=
	\begin{tikzpicture}[baseline={(0,1.2)}, scale=0.8]
		\draw[<-] (0,0) -- (1,1) -- (0,2) -- (1,3);
		\draw[<-] (1,0) -- (0,1) -- node[label=right:{$\mathrlap{\alpha}$}, dot, pos=1] {}  (1,2) -- (0,3);
	\end{tikzpicture}
	=
	\begin{tikzpicture}[baseline={(0,1.2)}, scale=0.8]
		\draw[<-] (0,0) -- (1,1) -- (0,2) -- (1,3);
		\draw[<-] (1,0) -- (0,1) -- node[label=right:{$\alpha$}, dot, pos=0] {}  (1,2) -- (0,3);
	\end{tikzpicture}
	=
	\begin{tikzpicture}[baseline={(0,0.4)}, scale=0.8]
		\draw[<-] (0,0) -- (1,1);
		\draw[<-] (1,0) -- node[label=left:{$\alpha$}, dot, pos=0.75] {} (0,1);
	\end{tikzpicture}.
	\]
	The remaining interactions of dots and crossings follow from the relations for downward crossings, cups and caps via the definition of the crossings.
\end{proof}

\begin{Lemma}\label{lem:colliding_dots_up-add}
	Dots on upward strands\index{upward strand} merge without a sign twist:
	\[
	\begin{tikzpicture}[baseline={(0,0.4)}]
		\draw[->]
		(0,0) --
		node[label=right:{$\alpha$}, dot, pos=0.33] {}
		node[label=right:{$\beta$}, dot, pos=0.66] {}
		(0,1);
		\draw (0.8, 0.5) node {$=$};
		\draw[->]
		(1.25,0) --
		node[label=right:{$\beta\circ\alpha$}, dot, pos=0.5] {}
		(1.25,1);
	\end{tikzpicture}
	\]
\end{Lemma}

\begin{proof}
	With $\epsilon = (-1)^{|\alpha||\beta|}$ we have
	\begin{multline*}
		\begin{tikzpicture}[baseline={(0,0.4)},scale=0.8]
			\draw[->] (0,0) --
			node[label=right:{$\alpha$}, dot, pos=0.3] {}
			node[label=right:{$\beta$}, dot, pos=0.7] {}
			(0,1);
		\end{tikzpicture}
		\ \overset{\mathclap{\text{\eqref{eq:straighten-add}}}}{=} \ 
		\begin{tikzpicture}[baseline={(0,0.4)},scale=0.8]
			\draw[->] (0,-0.5) -- (0,0) --
			node[label=right:{$\alpha$}, dot, pos=0.3] {}
			node[label=right:{$\beta$}, dot, pos=0.7] {}
			(0,1) 
			arc[start angle=180, end angle=0, radius=0.5] --
			(1,0)
			arc[start angle=180, end angle=360, radius=0.5] --
			(2,1.5);
		\end{tikzpicture}
		\ \overset{\mathclap{\text{\eqref{eq:caps_and_dots-add}}}}{=} \ 
		\begin{tikzpicture}[baseline={(0,0.4)},scale=0.8]
			\draw[->] (0,-0.5) -- (0,0) --
			node[label=right:{$\alpha$}, dot, pos=0.3] {}
			(0,1) 
			arc[start angle=180, end angle=0, radius=0.5] --
			node[label=right:{$\beta$}, dot, pos=0.3] {}
			(1,0)
			arc[start angle=180, end angle=360, radius=0.5] --
			(2,1.5);
		\end{tikzpicture}
		\ =\ 
		\epsilon\ 
		\begin{tikzpicture}[baseline={(0,0.4)},scale=0.8]
			\draw[->] (0,-0.5) -- (0,0) --
			node[label=right:{$\alpha$}, dot, pos=0.7] {}
			(0,1) 
			arc[start angle=180, end angle=0, radius=0.5] --
			node[label=right:{$\beta$}, dot, pos=0.7] {}
			(1,0)
			arc[start angle=180, end angle=360, radius=0.5] --
			(2,1.5);
		\end{tikzpicture}
		\ \overset{\mathclap{\text{\eqref{eq:caps_and_dots-add}}}}{=} \ 
		\epsilon\ 
		\begin{tikzpicture}[baseline={(0,0.4)},scale=0.8]
			\draw[->] (0,-0.5) -- (0,0) --
			(0,1) 
			arc[start angle=180, end angle=0, radius=0.5] --
			node[label=right:{$\alpha$}, dot, pos=0.3] {}
			node[label=right:{$\beta$}, dot, pos=0.7] {}
			(1,0)
			arc[start angle=180, end angle=360, radius=0.5] --
			(2,1.5);
		\end{tikzpicture}
		\\[0.7em]
		\ \overset{\mathclap{\text{\eqref{eq:colliding_dots_down-add}}}}{=} \ 
		\begin{tikzpicture}[baseline={(0,0.4)},scale=0.8]
			\draw[->] (0,-0.5) -- (0,0) --
			(0,1) 
			arc[start angle=180, end angle=0, radius=0.5] --
			node[label=right:{$\beta \circ \alpha$}, dot, pos=0.5] {}
			(1,0)
			arc[start angle=180, end angle=360, radius=0.5] --
			(2,1.5);
		\end{tikzpicture}
		\ \overset{\mathclap{\text{\eqref{eq:caps_and_dots-add}}}}{=} \ 
		\begin{tikzpicture}[baseline={(0,0.4)},scale=0.8]
			\draw[->] (0,-0.5) -- (0,0) --
			node[label=right:{$\beta \circ \alpha$}, dot, pos=0.5] {}
			(0,1) 
			arc[start angle=180, end angle=0, radius=0.5] --
			(1,0)
			arc[start angle=180, end angle=360, radius=0.5] --
			(2,1.5);
		\end{tikzpicture}
		\ \overset{\mathclap{\text{\eqref{eq:straighten-add}}}}{=} \ 
		\begin{tikzpicture}[baseline={(0,0.4)},scale=0.8]
			\draw[->] (1.25,0) --
			node[label=right:{$\beta\circ\alpha$}, dot, pos=0.5] {}
			(1.25,1);
		\end{tikzpicture}.
	\end{multline*}
\end{proof}

\begin{Remark}
	The adjunction relations~\eqref{eq:straighten-add} say that we have adjunctions of $1$-morphisms $(\PP_a,\,\QQ_a)$ and $(\QQ_a,\, \PP_{Sa})$ for any $a \in \basecat$.
\end{Remark}

\begin{Lemma}[Pitchfork relations, part I]\label{lem:pitchfork-I-add}
	The following relations hold in $\hcatadd*\basecat$:
	\[
	\begin{tikzpicture}[baseline={(0,0.4)}]
		\draw[->] (0,0) to[out=55, in=180] (1,0.8) to[out=0, in=100] (1.8,0);
		\draw[->] (0.1,0.9) -- (1,0);
	\end{tikzpicture}
	\ = \ 
	\begin{tikzpicture}[baseline={(0,0.4)}]
		\draw[->] (0,0) to[out=80, in=180] (0.8,0.8) to[out=0, in=125] (1.8,0);
		\draw[->] (1.7,0.9) -- (0.8,0);
	\end{tikzpicture}
	\qquad\qquad
	\begin{tikzpicture}[baseline={(0,0.4)}]
		\draw[->] (0,0) to[out=55, in=180] (1,0.8) to[out=0, in=100] (1.8,0);
		\draw[<-] (0.1,0.9) -- (1,0);
	\end{tikzpicture}
	\ = \ 
	\begin{tikzpicture}[baseline={(0,0.4)}]
		\draw[->] (0,0) to[out=80, in=180] (0.8,0.8) to[out=0, in=125] (1.8,0);
		\draw[<-] (1.7,0.9) -- (0.8,0);
	\end{tikzpicture}
	\]
	\[
	\begin{tikzpicture}[baseline={(0,0.4)}]
		\draw[<-] (0,0) to[out=55, in=180] (1,0.8) to[out=0, in=100] (1.8,0);
		\draw[->] (0.1,0.9) -- (1,0);
	\end{tikzpicture}
	\ = \ 
	\begin{tikzpicture}[baseline={(0,0.4)}]
		\draw[<-] (0,0) to[out=80, in=180] (0.8,0.8) to[out=0, in=125] (1.8,0);
		\draw[->] (1.7,0.9) -- (0.8,0);
	\end{tikzpicture}
	\qquad\qquad
	\begin{tikzpicture}[baseline={(0,-0.4)}, yscale=-1]
		\draw[<-] (0,0) to[out=55, in=180] (1,0.8) to[out=0, in=100] (1.8,0);
		\draw[<-] (0.1,0.9) -- (1,0);
	\end{tikzpicture}
	\ = \ 
	\begin{tikzpicture}[baseline={(0,-0.4)}, yscale=-1]
		\draw[<-] (0,0) to[out=80, in=180] (0.8,0.8) to[out=0, in=125] (1.8,0);
		\draw[<-] (1.7,0.9) -- (0.8,0);
	\end{tikzpicture}
	\]
	\[
	\begin{tikzpicture}[baseline={(0,-0.4)}, yscale=-1]
		\draw[->] (0,0) to[out=55, in=180] (1,0.8) to[out=0, in=100] (1.8,0);
		\draw[<-] (0.1,0.9) -- (1,0);
	\end{tikzpicture}
	\ = \ 
	\begin{tikzpicture}[baseline={(0,-0.4)}, yscale=-1]
		\draw[->] (0,0) to[out=80, in=180] (0.8,0.8) to[out=0, in=125] (1.8,0);
		\draw[<-] (1.7,0.9) -- (0.8,0);
	\end{tikzpicture}
	\qquad\qquad
	\begin{tikzpicture}[baseline={(0,-0.4)}, yscale=-1]
		\draw[->] (0,0) to[out=55, in=180] (1,0.8) to[out=0, in=100] (1.8,0);
		\draw[->] (0.1,0.9) -- (1,0);
	\end{tikzpicture}
	\ = \ 
	\begin{tikzpicture}[baseline={(0,-0.4)}, yscale=-1]
		\draw[->] (0,0) to[out=80, in=180] (0.8,0.8) to[out=0, in=125] (1.8,0);
		\draw[->] (1.7,0.9) -- (0.8,0);
	\end{tikzpicture}
	\]
\end{Lemma}

\begin{proof}
	These relations follow immediately from the definition of the crossings in~\eqref{eq:mixed-crossings-add} and~\eqref{eq:up-crossing-add} together with the adjunction relations~\eqref{eq:straighten-add}.
	For example, for the first relation one has
	\[
	\begin{tikzpicture}[baseline={(0,0.4)}]
		\draw[->] (0,0) to[out=55, in=180] (1,0.8) to[out=0, in=100] (1.8,0);
		\draw[->] (0.1,0.9) -- (1,0);
	\end{tikzpicture}
	\ \overset{\text{\eqref{eq:mixed-crossings-add}}}{=} \ 
	\begin{tikzpicture}[baseline={(0,1.4)}]
		\draw[->] 
		(2,1)     to[out=90, in=180]
		(2.5,1.8) to[out=0, in=180]
		(3.5,1)   to[out=0, in=180]
		(4.5,1.8) to[out=0,in=90]
		(5,1);
		\draw[->] (3.5,2) -- (2.5,1);
	\end{tikzpicture}
	\ \overset{\text{\eqref{eq:straighten-add}}}{=} \ 
	\begin{tikzpicture}[baseline={(0,0.4)}]
		\draw[->] (0,0) to[out=80, in=180] (0.8,0.8) to[out=0, in=125] (1.8,0);
		\draw[->] (1.7,0.9) -- (0.8,0);
	\end{tikzpicture}.
	\]
\end{proof}

\begin{Lemma}[Counter-clockwise loops]\label{lem:curls-add}
	The following relations hold in $\hcatadd*\basecat$:
	\begin{equation*}
		\begin{tikzpicture}[baseline={(0,-1.2)}, yscale=-1]
			\draw[<-] (0,0) to [out=90, in=270] (1,1.5) arc[start angle=0, end angle=180, radius=0.5] to[out=270, in=90] (1,0);
		\end{tikzpicture}
		=0,
		\qquad\qquad
		\begin{tikzpicture}[baseline={(0,-0.2)}]
			\draw[->]
			(1,-1)   -- 
			(1,-0.5) to[out=90, in=0] (0.3, 0.5) to[out=180, in=90]
			(-0.1,0) to[out=270,in=180] (0.3, -0.5) to[out=0, in=270]
			(1,0.5)  --
			(1,1);
		\end{tikzpicture}
		= 0,
		\qquad\qquad
		\begin{tikzpicture}[baseline={(0,0.8)}]
			\draw[->] (0,0) to [out=90, in=270] (1,1.5) arc[start angle=0, end angle=180, radius=0.5] to[out=270, in=90] (1,0);
		\end{tikzpicture}
		=0.
	\end{equation*}
\end{Lemma}

\begin{proof}
	Using the left relation in \eqref{eq:circle_and_curl-add} and
a pitchfork move\index{pitchfork relation} across the bottom cup, we have
	\[
	0 = 
	\begin{tikzpicture}[baseline={(0,0)}, xscale=-1, yscale=-1]
		\draw[->]
		(1,-0.9) --
		(1,-0.5) to[out=90, in=0] (0.3,0.5) to[out=180, in=90]
		(-0.1,0) to[out=270, in=180] (0.3,-0.5) to[out=0, in=270]
		(1,0.5)  arc[start angle=180, end angle=0, radius=0.5]
		--
		(2,-0.9);
	\end{tikzpicture}
	=
	\begin{tikzpicture}[baseline={(0,-0.3)}, xscale=-1, yscale=-1]
		\draw[->]
		(1,-0.9)  to[out=90, in=270]
		(2.5,1)   to[out=90, in=0]
		(1.5,1.7) to[out=180, in=90]
		(0,0.5)   arc[start angle=180, end angle=360, radius=0.5]
		arc[start angle=180, end angle=0, radius=0.5]
		--
		(2,-0.9);
	\end{tikzpicture}
	.
	\]
	Straightening out via \eqref{eq:straighten-add}, one obtains the first relation.
	The other two relations are obtained in a similar manner.
\end{proof}

\begin{Lemma}\label{lem:upward_symmetric_group_relations-add}\index{upward
strand}
	The following relations hold in $\hcatadd*\basecat$:
	\[
	\begin{tikzpicture}[baseline={(0,0.9)}]
		\draw[rounded corners=15pt, ->] (0,0) -- (1,1) -- (0,2);
		\draw[rounded corners=15pt, ->] (1,0) -- (0,1) -- (1,2);
		\draw (1.5,1) node {$=$};
		\draw[->] (2,0) -- (2,2) ;
		\draw[->] (3,0) -- (3,2) ;
	\end{tikzpicture}
	\]
	\par 
	\[
	\begin{tikzpicture}[baseline={(0,0.9)}]
		\draw[->] (0,0) -- (2,2);
		\draw[rounded corners=15pt, ->] (1,0) -- (0,1) -- (1,2);
		\draw[->] (2,0) -- (0,2);
		\draw (2.5,1) node {$=$};
		\draw[->] (3,0) -- (5,2);
		\draw[rounded corners=15pt, ->] (4,0) -- (5,1) -- (4,2);
		\draw[->] (5,0) -- (3,2);
	\end{tikzpicture}
	\qquad\qquad\quad
	\begin{tikzpicture}[baseline={(0,0.9)}]
		\draw[->] (0,0) -- (2,2);
		\draw[rounded corners=15pt, ->] (1,0)-- (0,1) -- (1,2);
		\draw[<-] (2,0) -- (0,2);
		\draw (2.5,1) node {$=$};
		\draw[->] (3,0) -- (5,2);
		\draw[rounded corners=15pt, ->] (4,0) -- (5,1) -- (4,2);
		\draw[<-] (5,0) -- (3,2);
	\end{tikzpicture}
	\]
\end{Lemma}

\begin{proof}
	These relations are obtained by adding appropriate cups and caps to
	\eqref{eq:symmetric_group_relations-add} and using the pitchfork and
	adjunction relations. For example, for the first relation, one has
	
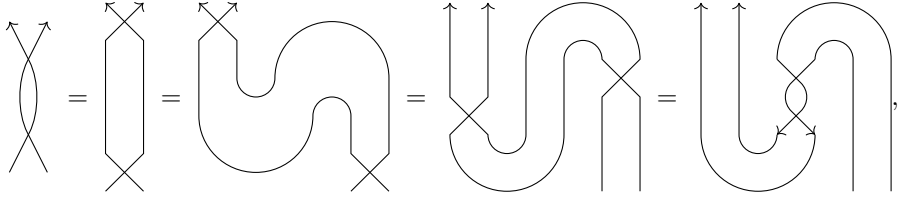
\[
	\begin{tikzpicture}[baseline={(0,0.9)}, xscale=0.5]
		\draw[rounded corners=15pt, ->] (0,0) -- (1,1) -- (0,2);
		\draw[rounded corners=15pt, ->] (1,0) -- (0,1) -- (1,2);
	\end{tikzpicture}
	\ = \ 
	\begin{tikzpicture}[baseline={(0,0.4)},scale=0.5]
		\draw[<-] (-2,3.5) -- (-1,2.5) -- (-1,-0.5) -- (-2,-1.5);
		\draw[<-] (-1,3.5) -- (-2,2.5) -- (-2,-0.5) -- (-1,-1.5);
	\end{tikzpicture}
	\ = \
	\begin{tikzpicture}[baseline={(0,0.4)},scale=0.5]
		\draw[<-] (-2,3.5) -- (-1,2.5) -- (-1,1.5) arc[start angle=180, end angle=360, radius=0.5] arc[start angle=180, end angle=0, radius=1.5] -- (3,-0.5) -- (2,-1.5);
		\draw[<-] (-1,3.5) -- (-2,2.5) -- (-2,0.5) arc[start angle=180, end angle=360, radius=1.5] arc[start angle=180, end angle=0, radius=0.5] -- (2,-0.5) -- (3,-1.5);
	\end{tikzpicture}
	\ = \
	\begin{tikzpicture}[baseline={(0,0.4)},scale=0.5]
		\draw[<-] (-2,3.5) -- (-2,1) -- (-1,0) arc[start angle=180, end angle=360, radius=0.5] -- (0,2) arc[start angle=180, end angle=0, radius=1.5] -- (2, 1) -- (2,-1.5);
		\draw[<-] (-1,3.5) -- (-1,1) -- (-2,0) arc[start angle=180, end angle=360, radius=1.5] -- (1,2) arc[start angle=180, end angle=0, radius=0.5] -- (3, 1) -- (3,-1.5);
	\end{tikzpicture}
	\ = \
	\begin{tikzpicture}[baseline={(0,0.4)},scale=0.5]
		\draw[<-] (-2,3.5) -- (-2,0) arc[start angle=180, end angle=360, radius=1.5];
		\draw[rounded corners=7.5pt, <-] (0,0) -- (1,1) -- (0,2);
		\draw (0,2) arc[start angle=180, end angle=0, radius=1.5] -- (3,-1.5);
		\draw[<-] (-1,3.5) -- (-1,0) arc[start angle=180, end angle=360, radius=0.5];
		\draw[rounded corners=7.5pt, <-] (1,0) -- (0,1) -- (1,2);
		\draw (1,2) arc[start angle=180, end angle=0, radius=0.5] -- (2,-1.5);
	\end{tikzpicture},
	\]
	where the second equality is relation~\eqref{eq:straighten-add}, the
	third equality is the interchange law in the $2$-category
	$\hcatadd*\basecat$ and the fourth is obtained by applying the
	pitchfork relations twice at the top and twice at the bottom. The
	first relation now follows by \eqref{eq:symmetric_group_relations-add}.
\end{proof}

\begin{Remark}\label{rem:symmetric_group_action-add}
	Relations \eqref{eq:symmetric_group_relations-add} imply that
we have an action of the symmetric group $\SymGrp n$ on $\QQ_a^n$ by
twisted unmarked downward strands\index{downward strand}, i.e., we have a morphism $\kk[\SymGrp n] \to \End(\QQ_a^n)$.
	Similarly, Lemma~\ref{lem:upward_symmetric_group_relations-add} shows that there exists an action of the symmetric group on $\PP_a^n$.
\end{Remark}

Fixing a basis $\{\beta_\ell\}$ of $\Hom(a,b)$ one can write the term $\Psi(\id)$ in \eqref{eq:up_down_braids-add} as
\[
\Psi(\id) = \sum_\ell
\begin{tikzpicture}[baseline={(0,0.7)}]
	\draw[->] (4,1.5) arc[start angle=-180, end angle=0, radius=.5] node[label=right:{$\beta_\ell$}, pos=0.8, dot] {};
	\draw[<-] (4,0) arc[start angle=180, end angle=0, radius=.5] node[label=right:{$\beta^\dual_\ell$}, pos=0.8, dot] {};
\end{tikzpicture},
\]
where $\{\beta^\dual_\ell\}$ is the dual basis of $\Hom(b,Sa) \cong \Hom(a,b)^\dual$.
It can also be written as the composition of $2$-morphisms
\begin{equation}\label{eq:half-Psi}
	\psi_1\colon \QQ_a\PP_b \to \Hom(a,b) \otimes_\kk \hunit
	\quad\text{and}\quad
	\psi_2\colon \Hom(a,b) \otimes_\kk \hunit \to \QQ_a\PP_b.
\end{equation}
Here $\psi_1$ is obtained form the map $\Hom(a,b)^\dual \to \Hom(\QQ_a\PP_b, \hunit)$ sending $\beta \in \Hom(b,Sa) \cong \Hom(a,b)^\dual$ to
$\smash[t]{
	\tikz[scale=0.5, baseline=0]
	\draw[<-] (0,0) arc[start angle=180, end angle=0, radius=.5] node[label=right:{$\beta$}, pos=0.8, dot] {};
}$
and $\psi_2$ is similarly obtained from the natural map $\Hom(a,b) \to \Hom(\hunit, \QQ_a\PP_b)$.
We note that the right relation in \eqref{eq:circle_and_curl-add} implies that the composition $\psi_1 \circ \psi_2$ is the identity on $\Hom(a,b)\otimes_\kk \hunit$.

\begin{Lemma}[Pitchfork relations, part II]\label{lem:pitchfork-II-add}
	The two remaining pitchfork relations hold in $\hcatadd*\basecat$, that is, one has
	
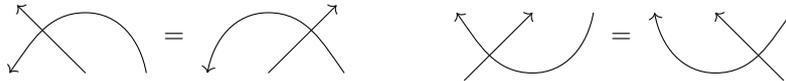
\[
	\begin{tikzpicture}[baseline={(0,0.4)}]
		\draw[<-] (0,0) to[out=55, in=180] (1,0.8) to[out=0, in=100] (1.8,0);
		\draw[<-] (0.1,0.9) -- (1,0);
	\end{tikzpicture}
	\ = \ 
	\begin{tikzpicture}[baseline={(0,0.4)}]
		\draw[<-] (0,0) to[out=80, in=180] (0.8,0.8) to[out=0, in=125] (1.8,0);
		\draw[<-] (1.7,0.9) -- (0.8,0);
	\end{tikzpicture}
	\qquad\qquad
	\begin{tikzpicture}[baseline={(0,-0.4)}, yscale=-1]
		\draw[<-] (0,0) to[out=55, in=180] (1,0.8) to[out=0, in=100] (1.8,0);
		\draw[->] (0.1,0.9) -- (1,0);
	\end{tikzpicture}
	\ = \ 
	\begin{tikzpicture}[baseline={(0,-0.4)}, yscale=-1]
		\draw[<-] (0,0) to[out=80, in=180] (0.8,0.8) to[out=0, in=125] (1.8,0);
		\draw[->] (1.7,0.9) -- (0.8,0);
	\end{tikzpicture}
	\]
\end{Lemma}

In particular, these relations show that we could have defined the upward crossing as a rotation of the left-wards crossing (instead of the right-wards one in \eqref{eq:up-crossing-add}) and obtained the same $2$-morphism.
The proof is inspired by the proof of \cite[Lemma~2.6]{Brundan:DefinitionOfHeisenbergCategory}.

\begin{proof}
	These two pitchfork relations are slightly harder to see than the ones in Lemma~\ref{lem:pitchfork-I-add}.
	First, the relations in Section~\ref{subsec:heisencatdef2-additive}
	imply that for any $a,b \in \basecat$ the morphism
	\begin{equation}\label{eq:lem:pitchfork-II-add:iso1}
		\left[
		\begin{tikzpicture}[baseline={(0,0.4)}]
			\draw[->] (0,0) -- (0,1);
			\draw[->] (0.5,0) -- (1.5,1);
			\draw[->] (0.5,1) -- (1.5,0);
		\end{tikzpicture}
		\,,\,
		\begin{tikzpicture}[baseline={(0,0.4)}]
			\draw[->] (0,0) -- (0,1);
		\end{tikzpicture}
		\,\psi_2
		\right]
		\colon
		\PP_a\PP_{Sb}\QQ_b \oplus \bigl(\Hom(b,Sb) \otimes_\kk \PP_a\bigr) \to \PP_a \QQ_b \PP_{Sb}
	\end{equation}
	is an isomorphism with inverse
	\[
	\begin{bmatrix}
		\begin{tikzpicture}[baseline={(0,0.4)}]
			\draw[->] (0,0) -- (0,1);
			\draw[<-] (0.5,0) -- (1.5,1);
			\draw[<-] (0.5,1) -- (1.5,0);
		\end{tikzpicture}
		\\ \\
		\begin{tikzpicture}[baseline={(0,0.4)}]
			\draw[->] (0,0) -- (0,1);
		\end{tikzpicture}
		\,\psi_1
	\end{bmatrix}
	\colon
	\PP_a \QQ_b \PP_{Sb} \to \PP_a\PP_{Sb}\QQ_b \oplus \bigl(\Hom(b,Sb) \otimes_\kk \PP_a\bigr).
	\]
	Next, we show that
	\begin{equation}\label{eq:lem:pitchfork-II-add:cap-crossing}
		\begin{tikzpicture}[baseline={(0,0.4)}]
			\draw[->] (1,0) to[out=90, in=0] (0.5,1) to[out=180, in=90] (0,0);
			\draw[->] (-0.3,0) -- (1.1,1);
		\end{tikzpicture}
		=
		\begin{tikzpicture}[baseline={(0,0.4)}]
			\draw[->] (0,0) -- (0,1);
			\draw[->] (1.5,0) to[out=90, in=0] (1,1) to[out=180, in=90] (0.5,0);
		\end{tikzpicture}.
	\end{equation}
	Precomposing with isomorphism \eqref{eq:lem:pitchfork-II-add:iso1}, 
	it remains to show that for any $\alpha \in \Hom(b,Sb)$: 
	\[
	\begin{tikzpicture}[baseline={(0,0.8)}]
		\draw[->] (0,0) to [out=60, in=270] (1,1.5) arc[start angle=0, end angle=180, radius=0.5] to[out=270, in=120] (1,0);
		\draw[->] (-0.3,0) -- (1.1,2);
	\end{tikzpicture}
	=
	\begin{tikzpicture}[baseline={(0,0.8)}]
		\draw[->] (-0.5,0) -- (-0.5,2);
		\draw[->] (0,0) to [out=90, in=270] (1,1.5) arc[start angle=0, end angle=180, radius=0.5] to[out=270, in=90] (1,0);
	\end{tikzpicture}
	\quad\text{and}\quad
	\begin{tikzpicture}[baseline={(0,0.15)}]
		\draw[decoration={markings, mark=at position 0.32 with {\arrow{>}}, mark=at position 0.82 with {\arrow{>}}}, postaction={decorate}]
		(1,0)   -- node[label=right:{$\alpha$}, dot, pos=0.5] {}
		(1,0.5) arc[start angle=0, end angle=180, radius=.5]
		--
		(0,0)   arc[start angle=180, end angle=360, radius=.5];
		\draw[->] (-0.1,-0.75) -- (1.1,1.25);
	\end{tikzpicture}
	=
	\begin{tikzpicture}[baseline={(0,0.15)}]
		\draw[->] (-0.5,-0.75) -- (-0.5,1.25);
		\draw[decoration={markings, mark=at position 0.32 with {\arrow{>}}, mark=at position 0.82 with {\arrow{>}}}, postaction={decorate}]
		(1,0)   -- node[label=right:{$\alpha$}, dot, pos=0.5] {}
		(1,0.5) arc[start angle=0, end angle=180, radius=.5]
		--
		(0,0)   arc[start angle=180, end angle=360, radius=.5];
	\end{tikzpicture}.
	\]
	The right diagram of the left equality has a counter-clockwise curl, hence is vanishing. Applying the third equality (read from its right to left) of Lemma~\ref{lem:upward_symmetric_group_relations-add} to the left diagram of the left equality, we can move the upward diagonal arrow to below the counter-clockwise curl.  Hence, this diagram also equals zero.
	Further we have
	\[
	\begin{tikzpicture}[baseline={(0,0.15)}]
		\draw[decoration={markings, mark=at position 0.32 with {\arrow{>}}, mark=at position 0.82 with {\arrow{>}}}, postaction={decorate}]
		(1,0)   -- node[label=right:{$\alpha$}, dot, pos=0.5] {}
		(1,0.5) arc[start angle=0, end angle=180, radius=.5]
		--
		(0,0)   arc[start angle=180, end angle=360, radius=.5];
		\draw[->] (-0.1,-0.75) -- (1.1,1.25);
	\end{tikzpicture}
	=
	\begin{tikzpicture}[baseline={(0,0.15)}]
		\draw[decoration={markings, mark=at position 0.32 with {\arrow{>}}, mark=at position 0.82 with {\arrow{>}}}, postaction={decorate}]
		(1,0)   -- node[label=right:{$\alpha$}, dot, pos=0.5] {}
		(1,0.5) arc[start angle=0, end angle=180, radius=.5]
		--
		(0,0)   arc[start angle=180, end angle=360, radius=.5];
		\draw[->] (1.1,-0.75) to[out=120,in=240] (1.1,1.25);
	\end{tikzpicture}
	=
	\begin{tikzpicture}[baseline={(0,0.15)}]
		\draw[decoration={markings, mark=at position 0.32 with {\arrow{>}}, mark=at position 0.82 with {\arrow{>}}}, postaction={decorate}]
		(1,0)   --
		(1,0.5) arc[start angle=0, end angle=180, radius=.5]
		-- node[label=left:{$\alpha$}, dot, pos=0.5] {}
		(0,0)   arc[start angle=180, end angle=360, radius=.5];
		\draw[->] (1.1,-0.75) to[out=120,in=240] (1.1,1.25);
	\end{tikzpicture}
	=
	\begin{tikzpicture}[baseline={(0,0.15)}]
		\draw[decoration={markings, mark=at position 0.32 with {\arrow{>}}, mark=at position 0.82 with {\arrow{>}}}, postaction={decorate}]
		(1,0)   --
		(1,0.5) arc[start angle=0, end angle=180, radius=.5]
		-- node[label=left:{$\alpha$}, dot, pos=0.5] {}
		(0,0)   arc[start angle=180, end angle=360, radius=.5];
		\draw[->] (1.3,-0.75) -- (1.3,1.25);
	\end{tikzpicture},
	\]
	which is $\Tr(\alpha)$ times the identity $2$-morphism and thus agrees
	with the rightmost $2$-morphism.
	
	Finally, applying \eqref{eq:lem:pitchfork-II-add:cap-crossing} to
	the first pitchfork relation we get
	\[
	\begin{tikzpicture}[baseline={(0,0.4)}]
		\draw[<-] (0,0) to[out=90, in=180] (0.5,1) to[out=0, in=90] (1,0);
		\draw[->] (0.3,0) to[out=90, in=270] (-0.1,1);
	\end{tikzpicture}
	\ = \ 
	\begin{tikzpicture}[baseline={(0,0.65)}]
		\draw[<-] (0,0) to[out=90, in=180] (0.5,1.5) to[out=0, in=90] (1,0);
		\draw[->] (0.3,0) to[out=90, in=270] (-0.1,1) to[out=90, in=210] (1.1,1.5);
	\end{tikzpicture}
	\ = \ 
	\begin{tikzpicture}[baseline={(0,0.65)}]
		\draw[<-] (0,0) to[out=90, in=180] (0.5,1.5) to[out=0, in=90] (1,0);
		\draw[->] (0.3,0) -- (0.3,0.5) to[out=90, in=210] (1.1,1.5);
	\end{tikzpicture}
	+
	\sum_\ell 
	\begin{tikzpicture}[baseline={(0,0.65)}]
		\draw[<-] (0,0) arc[start angle=180, end angle=0, radius=0.25] node[label=right:{\tiny$\tilde\beta_\ell$}, pos=0.8, dot] {};
		\draw[->] (1,0) to[out=90, in=0] (0.5,1.5) to[out=180,in=90] (0,1) arc[start angle=180, end angle=360, radius=0.25] node[label=right:{\tiny$\beta_\ell$}, pos=0.8, dot] {} to[out=90, in=210] (1.1,1.5);
	\end{tikzpicture}
	=
	\begin{tikzpicture}[baseline={(0,0.4)}]
		\draw[<-] (0,0) to[out=90, in=180] (0.5,1) to[out=0, in=90] (1,0);
		\draw[->] (0.7,0) to[out=90, in=270] (1.1,1);
	\end{tikzpicture},
	\]
	where the last equality holds because of the presence of counter-clockwise curls.
	
	The second relation immediately follows from the first one:
	\[
	\begin{gathered}
	\begin{tikzpicture}[baseline={(0,-0.4)}]
		\draw[<-] (0,0) to[out=-55, in=-180] (1,-0.8) to[out=0, in=-100] (1.8,0);
		\draw[->] (0.1,-0.9) -- (1,0);
	\end{tikzpicture}
	\ \overset{\text{\eqref{eq:straighten-add}}}{=} \ 
	\begin{tikzpicture}[baseline={(0,-1.4)}]
		\draw[<-] 
		(2,-1)     to[out=-90, in=-180]
		(2.5,-1.8) to[out=0, in=-180]
		(3.5,-1)   to[out=0, in=-180]
		(4.5,-1.8) to[out=0,in=-90]
		(5,-1);
		\draw[->] (3.5,-2) -- (4.5,-1);
	\end{tikzpicture}
	\\
	\ \overset{\text{first rel.}}{=} \ 
	\begin{tikzpicture}[baseline={(0,-1.4)}]
		\draw[<-] 
		(2,-1)     to[out=-90, in=-180]
		(2.5,-1.8) to[out=0, in=-180]
		(3.5,-1)   to[out=0, in=-180]
		(4.5,-1.8) to[out=0,in=-90]
		(5,-1);
		\draw[->] (3.5,-2) -- (2.5,-1);
	\end{tikzpicture}
	\ \overset{\text{\eqref{eq:straighten-add}}}{=} \ 
	\begin{tikzpicture}[baseline={(0,-0.4)}]
		\draw[<-] (0,0) to[out=-80, in=-180] (0.8,-0.8) to[out=0, in=-125] (1.8,0);
		\draw[->] (1.7,-0.9) -- (0.8,0);
	\end{tikzpicture}.
\end{gathered}
	\]
\end{proof}

Using the pitchfork relations one shows that the remaining triple moves\index{triple move relation} also hold.

\begin{Lemma}[Triple moves]\label{lem:triple_moves-add}
	The following relations holds in $\hcatadd*\basecat$:
	\[
	\begin{tikzpicture}[baseline={(0,0.9)}]
		\draw[-] (0,0) -- (2,2);
		\draw[rounded corners=15pt,-] (1,0) -- (0,1) -- (1,2);
		\draw[-] (2,0) -- (0,2);
		\draw (2.5,1) node {$=$};
		\draw[-] (3,0) -- (5,2);
		\draw[rounded corners=15pt,-] (4,0) -- (5,1) -- (4,2);
		\draw[-] (5,0) -- (3,2);
	\end{tikzpicture}
	\]
\end{Lemma}

\begin{Remark}
	For any object $a \in \basecat$, $\alpha \in \Hom(a,a)$ and $\beta \in \Hom(a, Sa)$ one has
	\[
	\begin{tikzpicture}[baseline={(0,0.15)}]
		\draw[decoration={markings, mark=at position 0.36 with {\arrow{>}}, mark=at position 0.85 with {\arrow{>}}}, postaction={decorate}]
		(1,0) --
		node[label=right:{$\alpha$}, dot, pos=0.15] {}
		node[label=right:{$\beta$}, dot, pos=0.85] {}
		(1,0.5) arc[start angle=0, end angle=180, radius=.5]
		--
		(0,0) arc[start angle=180, end angle=360, radius=.5];
	\end{tikzpicture}
	=
	\begin{tikzpicture}[baseline={(0,0.15)}]
		\draw[decoration={markings, mark=at position 0.36 with {\arrow{>}}, mark=at position 0.85 with {\arrow{>}}}, postaction={decorate}]
		(1,0) --
		node[label=right:{$\beta$}, dot, pos=0.85] {}
		(1,0.5) arc[start angle=0, end angle=180, radius=.5]
		--
		node[label=left:{$\alpha$}, dot, pos=0.85] {}
		(0,0) arc[start angle=180, end angle=360, radius=.5];
	\end{tikzpicture}
	=
	(-1)^{|\alpha||\beta|}\,
	\begin{tikzpicture}[baseline={(0,0.15)}]
		\draw[decoration={markings, mark=at position 0.36 with {\arrow{>}}, mark=at position 0.85 with {\arrow{>}}}, postaction={decorate}]
		(1,0) --
		node[label=right:{$\beta$}, dot, pos=0.15] {}
		(1,0.5) arc[start angle=0, end angle=180, radius=.5]
		--
		node[label=right:{$\alpha$}, dot, pos=0.85] {}
		(0,0) arc[start angle=180, end angle=360, radius=.5];
	\end{tikzpicture}
	=
	(-1)^{|\alpha||\beta|}\,
	\begin{tikzpicture}[baseline={(0,0.15)}]
		\draw[decoration={markings, mark=at position 0.36 with {\arrow{>}}, mark=at position 0.85 with {\arrow{>}}}, postaction={decorate}]
		(1,0) --
		node[label=right:{$\beta$}, dot, pos=0.15] {}
		node[label=right:{$S\alpha$}, dot, pos=0.85] {}
		(1,0.5) arc[start angle=0, end angle=180, radius=.5]
		--
		(0,0) arc[start angle=180, end angle=360, radius=.5];
	\end{tikzpicture}
	\]
	This matches the identity of Proposition~\ref{prop:serre_trace_cyclic_additive}.
\end{Remark}

\section{The category \texorpdfstring{$\hcatadd \basecat$}{H}: Karoubi-completion}\label{subsec:idempotent-hcat-add}

{ A category is \emph{Karoubian}\index{Karoubian category} or \emph{idempotent
		complete} if all its idempotents are split. Given a category $\C$, 
	its \emph{Karoubi envelope\index{Karoubi completion}} or \emph{idempotent completion}\index{idempotent completion}
	is the universal pair $(\text{kar}(\C), \iota)$ where $\text{kar}(\C)$ 
	a Karoubian category\index{Karoubian category} and $\iota$ is a functor $C \rightarrow \text{kar}(\C)$.  
	The functor $\iota$ is necessarily fully faithful, see
	\cite[Exercice 7.5]{GrothedieckVerdier-Topos}.
}

\begin{Definition}\label{def:hcat-add}
	The \emph{(additive) Heisenberg category}\index{additive Heisenberg category} $\hcatadd\basecat$ of $\basecat$ is the Karoubi envelope of $\hcatadd*\basecat$.
\end{Definition}

The objects of $\hcatadd\basecat$ are those of $\hcatadd*\basecat$.
Its $1$-morphisms are pairs $(R,e)$, where $R$ is a $1$-morphism of
$\hcatadd*\basecat$ and $e\colon R \to R$ is a idempotent in
$\End_{\hcatadd*\basecat}(R)$. Its $2$-morphisms $(R_1,e_1) \to (R_2,e_2)$ are $2$-morphisms $f\colon R_1\to R_2$ from $\hcatadd*\basecat$ which satisfy $f = e_2 \circ f \circ e_1$.

\begin{Example}\label{ex:Khovanov-add}
	Let $\basecat = \catfVect$ be
	the category of finite-dimensional vector spaces over $\kk$.  It is the
	additive hull of the field $\kk$ considered as a single-object
	category. Then the Serre functor on $\basecat$ is the identity, and the category $\hcatadd\basecat$ reproduces Khovanov's categorification of the infinite Heisenberg algebra \cite{khovanov2014heisenberg}.
	More precisely, collapsing our category $\hcatadd\basecat$ to a monoidal $1$-category by identifying the objects, the morphism $\PP_\kk$ corresponds to $Q_+$ in \cite{khovanov2014heisenberg}, while $\QQ_\kk$ corresponds to $Q_-$.
	Since $\PP_{\kk \oplus \kk} \cong \PP_\kk \oplus \PP_\kk$, and similarly for $\QQ$, all data is encoded in the relations between these two morphisms.
\end{Example}

By Remark~\ref{rem:symmetric_group_action-add}, for each object $a \in \basecat$ there are canonical morphisms $\kk[\SymGrp n] \to \End(\PP_a^n)$ and $\kk[\SymGrp n] \to \End(\QQ_a^n)$.
Let
\[
e_\triv = \frac{1}{n!} \sum_{\sigma \in \SymGrp n} \sigma \in \kk[\SymGrp n]
\]
be the symmetriser idempotent of $\kk[\SymGrp n]$.
Abusing notation, we denote the image of the symmetriser under either of the above maps again by $e_\triv$.
The $2$-morphisms $e_\triv$ are idempotent endomorphisms of $\PP_a^{n}$ and $\QQ_a^{n}$ respectively, and hence split in $\hcatadd\basecat$.
We write $\PP_a^{(n)}$ and $\QQ_a^{(n)}$ for the corresponding $1$-morphisms $(\PP_a^n,\, e_\triv)$ and $(\QQ_a^n,\, e_\triv)$.

\begin{Theorem}\label{thm:cat_heisenberg_relations-add}
	For any $a, b \in \basecat$ and $n, m \in \NN$ we have the
	following relations in $\hcatadd\basecat$:
	\[
	\PP_a^{(m)}\PP_b^{(n)} \cong \PP_b^{(n)} \PP_a^{(m)}, \quad
	\QQ_a^{(m)}\QQ_b^{(n)} \cong \QQ_b^{(n)} \QQ_a^{(m)},
	\]
	\par
	\[
	\QQ_a^{(m)}\PP_b^{(n)} \cong \bigoplus_{i=0}^{\min(m,n)}\Sym^i \Hom_{\basecat}(a,b) \otimes_\kk \PP_b^{(n-i)}\QQ_a^{(m-i)}.
	\]
\end{Theorem}

The symmetric powers of $\Hom_{\basecat}(a,b)$ in the last
isomorphism of Theorem~\ref{thm:cat_heisenberg_relations-add} categorify
the coefficient $s^k \langle a,b\rangle$ in \eqref{eq:heisrel3}. In
Remark~\ref{rem:why_sym} we explain that from any
$\PP_a^{(i)}$ to any $\PP_b^{(i)}$ there are morphisms which correspond to 
$i$ parallel strands labelled by elements of $\Sym^i \Hom_{\basecat}(a,b)$. 
The last isomorphism of Theorem~\ref{thm:cat_heisenberg_relations-add} is 
then naturally expressed in terms of these morphisms. In particular, 
in the case $m=n=1$, the $1$-precomposition of this $2$-isomorphism 
with $\id_{\PP_a}$ on the left is the isomorphism used in the proof 
of Lemma~\ref{lem:pitchfork-II-add}.

The proof of Theorem~\ref{thm:cat_heisenberg_relations-add} is entirely combinatorial and virtually the same as the one for the \dg version, Theorem~\ref{thm:cat_heisenberg_relations-dg}.
We thus skip it.
Similarly, the constructions and the results of Section~\ref{subsec:transposed_generators} have obvious analogues in the additive setting.

\section{The categorical Fock space\index{categorical Fock space} in the additive case}
\label{subset:cat_fock_add}

In this section we construct a categorical Fock space
$\fcatadd\basecat$ of the base category $\basecat$.
It consists of the categorical symmetric powers of $\basecat$.
We show that $\hcatadd\basecat$ has a representation on the categorical Fock space.

Once this is established, the same decategorification argument as in 
Section~\ref{subsec:grothfock} shows that 
{$\numGgp{\hcatadd\basecat}$} acts on $\numGgp{\fcatadd\basecat}$. 
Theorem~\ref{thm:cat_heisenberg_relations-add}, we have a group 
homomorphism from the classical Heisenberg algebra\index{classical Heisenberg algebra} $\halg\basecat$
to {$\numGgp{\hcatadd\basecat}$}. Thus $\halg\basecat$ acts on 
$\numGgp{\fcatadd\basecat}$ and the same argument as in 
Section~\ref{subsec:genuine_categorification} shows that 
the subrepresentation of $\numGgp{\fcatadd\basecat}$ generated
by $1 \in \numGgp{\sym^0 \basecat} \simeq \kk$ is the Fock space
representation $\falg\basecat$ of $\halg\basecat$. 

If $\numGgp{\basecat}$ is a finitely generated abelian group and 
if we have for all $N \geq 0$
\[\numGgp{\symbcn} \cong \bigoplus_{1^{\lambda_1}2^{\lambda_2} \dots
	\dashv \ho} \Sym^{\lambda_1}\numGgp{\basecat} \otimes
\Sym^{\lambda_2}\numGgp{\basecat} \otimes \cdots\]
then a dimension count shows that $\falg\basecat$ is 
the whole of $\numGgp{\fcatadd\basecat}$. In other words, 
our categorical Fock space categorifies the classical Fock space\index{classical Fock space}.

The $\ho$-fold tensor power $\basecat^{\otimes \ho}$ is the additive hull
({ that is}, the closure under finite direct sums)
of the category of $\ho$-tuples $a_1 \otimes \dots \otimes a_{\ho}$ of
objects of $\basecat$ with morphism spaces
\[
\homm_{\basecat^{\otimes \ho}}\left(a_1 \otimes \dots \otimes a_{\ho},\, b_1 \otimes \dots \otimes b_{\ho}\right) \coloneqq
\homm_{\basecat}(a_1,b_1) \otimes_\kk \dots \otimes_\kk \homm_{\basecat}(a_{\ho}, b_{\ho}).
\]
The category $\basecat^{\otimes \ho}$ can be endowed with an action of $\SymGrp \ho$, given on objects by
\begin{equation}
	\label{eqn-action-of-sn-on-v-otimes-n}
	\sigma(a_1 \otimes \dots \otimes a_{\ho}) \coloneqq a_{\sigma^{-1}(1)} \otimes \dots \otimes a_{\sigma^{-1}(\ho)}.
\end{equation}
{ The category of $\SymGrp \ho$-equivariant objects in $\basecat^{\otimes \ho}$ 
	\[
	\addsymbc{\ho} := (\basecat^{\otimes \ho})^{\SymGrp \ho}
	\]
	has as objects all tuples $\bigl(\underline a,\, (\epsilon_\sigma)_{\sigma
		\in \SymGrp \ho}\bigr)$ with $\underline a \in \basecat^{\otimes \ho}$
	and $\epsilon_\sigma \colon \underline a \isoto \sigma(\underline a)$
	isomorphisms compatible with the $\SymGrp \ho$-action. A morphism  
	$\bigl(\underline a,\, \epsilon_\sigma) \rightarrow 
	\bigl(\underline b,\, \tau_\sigma)$ is a morphism $\alpha\colon \underline{a}
	\rightarrow \underline{b}$ in {$\basecat^{\otimes \ho}$} such that $\sigma(\alpha) \circ \epsilon_\sigma = 
	\tau_\sigma \circ \alpha$ for all $\sigma \in \SymGrp \ho$.} 
We refer to \cite[Section~2]{SymCat} for details.
For ease of notation, we set $\addsymbc0 = \catfVect$ and $\addsymbc{\ho} = 0$ for $\ho < 0$.

\begin{Remark}
	If $\basecat$ is a $\kk$-linear category equipped with additional structure
	and/or conditions, e.g.~an abelian category, then  
	$\basecat^{\otimes \ho}$ will not automatically also have these. 
	In such case, in the definition above one should replace the additive 
	hull with an appropriate completion. For example, Deligne's tensor product 
	of abelian categories takes the abelian hull of $\ho$-tuples.
	We are particularly interested in the case of \dg enhanced triangulated categories, which we discuss in detail in Section~\ref{subsec:equivariant_cats} and Chapter~\ref{sec:cat_fock}. 
\end{Remark}

Let $\fcatadd*\basecat$ be the strict $2$-category with objects $\addsymbc{\ho}$, $1$-morphisms $\kk$-linear functors and $2$-morphisms natural transformations.
We want to define a $2$-functor\index{$2$-functor} 
$\Psi_\basecat'\colon \hcatadd*\basecat \to \fcatadd*\basecat$. For
this, we need the functors of restriction and induction. 
Let  $1 \times \SymGrp{\ho-1} \leq \SymGrp \ho$ be the subgroup
comprising the elements fixing the first letter. The restriction
functor is defined as 
\[
\begin{array}{r c c c} 
	\Res_{\SymGrp \ho}^{1 \times \SymGrp{\ho-1}} \colon & \addsymbc{\ho} &\to &  \bigl(\basecat^{\otimes \ho}\bigr)^{1 \times \SymGrp{\ho-1}} \\
	& \bigl(\underline a,\, (\epsilon_g)_{g \in \SymGrp \ho}\bigr)
	& \mapsto & \bigl(\underline a,\, (\epsilon_g)_{g \in 1 \times \SymGrp{\ho-1}}\bigr)
\end{array}
\]
on objects and by $\Id$ on morphisms. Its left and right adjoint, 
the induction functor, is 
\[
\begin{array}{r c c c}
	\Ind^{\SymGrp \ho}_{1 \times \SymGrp{\ho-1}} \colon & 
	\bigl(\basecat^{\otimes \ho}\bigr)^{1 \times \SymGrp{\ho-1}}
	& \to &  \addsymbc{\ho}
	\\
	& \bigl(\underline a,\, (\epsilon_h)_{h \in 1 \times
		\SymGrp{\ho-1}}\bigr) & \mapsto & \bigl(\bigoplus_{[f] \in \SymGrp \ho/(1
		\times \SymGrp{\ho-1})} f(\underline a),\, (\varepsilon_g)_{g \in \SymGrp \ho}\bigr)
\end{array}
\]
on objects. Here $\SymGrp \ho/(1 \times \SymGrp{\ho-1})$ is the set of
left cosets, the summation happens over a fixed choice of 
their representatives $f$, and the isomorphism
\[ \varepsilon_g\colon \bigoplus_{[f] \in \SymGrp \ho/(1 \times \SymGrp{\ho-1})} 
f(\underline a) \to \bigoplus_{[f'] \in \SymGrp \ho/(1 \times
	\SymGrp{\ho-1})} g f'(\underline a) \]
maps each summand $f(\underline a)$ to the summand $gf'(\underline a)$ 
with $[f]=[gf']$ via the isomorphism $f(\epsilon_h)$ where 
$h \in 1 \times \SymGrp{\ho-1}$ is such that $gf'=fh$. On morphisms, 
$\Ind^{\SymGrp \ho}_{1 \times \SymGrp{\ho-1}}$ is given by
\begin{equation*}
	\alpha \rightarrow 
	\sum_{[f] \in \SymGrp \ho/(1 \times \SymGrp{\ho-1})} f(\alpha). 
\end{equation*}
A more general treatment of these functors is given in Section~\ref{subsec:equivariant_cats} below.

On objects, we define $\Psi_\basecat'$ as
\[
\Psi_\basecat'(\ho) = \addsymbc{\ho}, \quad \quad \forall\; \ho
\in \ZZ. 
\]
On $1$-morphisms $\Psi_\basecat'$ sends $\PP_a\colon (\ho-1) \to \ho$ to the composition
\[
P_a\colon 
\addsymbc{\ho-1} \xrightarrow{a \otimes - }
\bigl(\basecat^{\otimes \ho}\bigr)^{1 \times \SymGrp{\ho-1}}
\xrightarrow{\Ind^{\SymGrp \ho}_{1 \times \SymGrp{\ho-1}}}
\addsymbc{\ho},
\]
and $\QQ_a\colon \ho \to \ho-1$ to the composition
\[
Q_a\colon
\addsymbc{\ho}
\xrightarrow{\Res_{\SymGrp \ho}^{1 \times \SymGrp{\ho-1}}}
\bigl(\basecat^{\otimes\ho}\bigr)^{1 \times \SymGrp{\ho-1}}
\xrightarrow{\Hom_\basecat(a, -) \otimes \id}
\addsymbc{\ho-1}.
\]
Tensor-Hom adjunction\index{Tensor-Hom adjunction} implies that $P_a$ is left adjoint to $Q_a$ and
the definition of a Serre functor further implies that $Q_a$ is left
adjoint to $P_{Sa}$.

\begin{Example}\label{ex:QaPb_add}
	Let  $(a_1 \otimes \dots \otimes a_{\ho},\,
	(\epsilon_\sigma)_{\sigma \in \SymGrp{\ho}})$ be an object in $\addsymbc{\ho}$.
	There are $\ho+1$ cosets of the subgroup $\SymGrp{\ho} <
	\SymGrp{\ho+1}$ fixing the symbol $1$. A set of representatives of
	these cosets is given by the cycles $\{\bigl(i\dots1\bigr)\}_{1
		\leq i \leq \ho+1}$. Denote each $\bigl(i\dots1\bigr)$ by $\xi_i$. 
	
	By definition of the $P_b$, we have
	\begin{align*}
		P_b(a_1\otimes \dots \otimes a_{\ho}) 
		= \Ind^{\SymGrp \ho + 1}_{1 \times \SymGrp{\ho}}
		\left(b \otimes a_1 \otimes \dots \otimes a_{\ho}\right)
		=
		\bigoplus_{i=1}^{\ho+1} \xi_i\left( b \otimes a_1 \otimes \dots \otimes
		a_{\ho} \right). 
	\end{align*}
	By the definition \eqref{eqn-action-of-sn-on-v-otimes-n} of the action of 
	$\SymGrp{\ho + 1}$ on $\basecat^{\otimes \ho + 1}$, $\xi_i$ acts
	by placing the $\xi_i^{-1}(j)$th factor into $j$th place. Thus we have  
	\begin{align*}
		\xi_i\left( b \otimes a_1 \otimes \dots \otimes a_{\ho} \right) 
		= 
		a_1 \otimes \dots \otimes a_{i-1} \otimes b \otimes a_{i} \otimes \dots \otimes a_{\ho} 
	\end{align*}
	and therefore 
	\begin{align}
		\label{eqn-direct-sum-description-of-p_b}
		P_b = 
		\bigoplus_{i=1}^{\ho+1} a_1 \otimes \dots
		\otimes a_{i-1} \otimes b \otimes a_{i} \otimes \dots \otimes a_{\ho}.
	\end{align}
	We describe the $\SymGrp{\ho+1}$-equivariant structure on this
	direct sum. Let $\sigma \in \SymGrp{\ho + 1}$. For each $\xi_i$,
	the element $\sigma \xi_{\sigma^{-1}(i)}$ lies in the same coset as they both 
	send $1$ to $i$. Thus 
	$$\xi^{-1}_{i} \sigma \xi_{\sigma^{-1}(i)} = 
	\bigl(1\cdots i\bigr) \sigma \bigl(\sigma^{-1}(i)\cdots1\bigr) \in 1
	\times \SymGrp{\ho} \subset \SymGrp{\ho+1}.$$ Let $\tau_i$ be
	the corresponding element of $\SymGrp{\ho}$. By definition, 
	the isomorphism 
	$$
	\varepsilon_\sigma\colon \quad 
	\bigoplus_{i=1}^{\ho+1} \xi_i\left( b \otimes a_1 \otimes \dots \otimes a_{\ho} \right) 
	\longrightarrow  
	\bigoplus_{i=1}^{\ho+1} \sigma \xi_i\left( b \otimes a_1 \otimes
	\dots \otimes a_{\ho} \right)
	$$
	is a sum of components
	$$ 
	\xi_i \circ (b \otimes -) (\epsilon_{\tau_i})\colon 
	\xi_i\left( b \otimes a_1 \otimes \dots \otimes a_{\ho} \right) 
	\rightarrow 
	\xi_i\left( 
	b \otimes \left( \tau_i \left(a_1 \otimes \dots
	\otimes a_{\ho} \right)\right)
	\right). 
	$$
	Hence, in terms of \eqref{eqn-direct-sum-description-of-p_b}, 
	$\varepsilon_\sigma$ is the sum of the components 
	$$
	\begin{multlined}
	a_1 \otimes \dots
	\otimes a_{i-1} \otimes b \otimes a_{i} \otimes \dots \otimes a_{\ho} \\
	\xrightarrow{\xi_i \circ (b \otimes -)(\epsilon_{\tau_i})}
	a_{\tau_i^{-1}(1)} \otimes \dots
	\otimes a_{\tau^{-1}_i(i-1)} \otimes b \otimes a_{\tau_i^{-1}(i)} \otimes
	\dots \otimes a_{\tau_i^{-1}(\ho)}. 
	\end{multlined} 
	$$
	
	It follows that 
	\begin{align*}
		Q_aP_b(a_1\otimes \dots \otimes a_{\ho}) & =
		Q_a \left(\bigoplus_{i=1}^{\ho+1} 
		\xi_i(b \otimes a_1 \otimes \dots \otimes \dots \otimes a_{\ho}{)} \right) 
		\\
		& =
		Q_a \left(\bigoplus_{i=1}^{\ho+1} a_1 \otimes \dots
		\otimes a_{i-1} \otimes b \otimes a_{i} \otimes \dots \otimes a_{\ho} \right) 
		\\ 
		&=
		\Hom(a,b) \otimes_\kk a_1 \otimes \dots \otimes a_{\ho} \oplus \\ &
		\phantom{{}={}}\bigoplus_{i=1}^{\ho} \Hom(a,a_1) \otimes_\kk a_2
		\otimes  \dots \otimes a_{i} \otimes b \otimes a_{i+1} \otimes \dots \otimes a_{\ho}.
	\end{align*}
	We describe the $\SymGrp{\ho}$-equivariant structure on this direct sum. 
	Let $\sigma \in \SymGrp{\ho}$. Let $1 \times \sigma$ be the
	corresponding element of $1 \times \SymGrp{\ho} \subset
	\SymGrp{\ho+1}$ and note that $1 \times \sigma (i) = 1$ if $i = 1$ and 
	$1 + \sigma (i - 1)$ if $i > 1$. As before, we have 
	$\xi^{-1}_{i} (1 \times \sigma) \xi_{(1 \times \sigma^{-1})(i)} \in 1
	\times \SymGrp{\ho}$, so let 
	$\tau_i$ be the corresponding element of $\SymGrp{\ho}$.
	
	Restricting the $\SymGrp{\ho+1}$-equivariant structure on
	$P_b(a_1\otimes \dots \otimes a_{\ho})$ described above, we see that
	the isomorphism 
	$$ \varepsilon'_\sigma \colon 
	Q_aP_b(a_1\otimes \dots \otimes a_{\ho})
	\longrightarrow 
	\sigma\left( Q_aP_b(a_1\otimes \dots \otimes a_{\ho}) \right)
	$$
	is the sum 
	$$ \sum_{i = 1}^{\ho+1} \left(\homm(a,-) \otimes \id\right) \circ \xi_i \circ (b \otimes -)
	(\epsilon_{\tau_i}). $$
	When $i = 1$ we have $\xi_1 = \xi_{(1 \times \sigma^{-1})(1)} = \id$, 
	so $\tau = \sigma$ and the corresponding summand of 
	$\varepsilon'_\sigma$ is 
	$$ \Hom(a,b) \otimes_\kk a_1 \otimes \dots \otimes a_{\ho}
	\xrightarrow{\id \otimes \epsilon_\sigma}
	\Hom(a,b) \otimes_\kk a_{\sigma^{-1}(1)} \otimes \dots \otimes
	a_{\sigma^{-1}(\ho)}. $$
	When $i > 1$, observe that $\tau_i(1) = 1$. This is because
	$$ \tau_i(1) = \xi^{-1}_{i} (1 \times \sigma) \xi_{(1 \times
		\sigma^{-1})(i)}(2) - 1 = \xi^{-1}_{i} (1 \times \sigma)(1) - 1 = 
	\xi^{-1}_{i}(1) - 1 = 2 - 1 = 1. $$
	The corresponding summand $\varepsilon'_\sigma$ is therefore 
	\begin{equation*}
		\begin{tikzcd}
			\Hom(a,a_1) \otimes_\kk a_2 \otimes  \dots \otimes a_{i-1} \otimes b
			\otimes a_{i} \otimes \dots \otimes a_{\ho}
			\ar{d}{(\Hom(a,-) \otimes \id) \circ \xi_i \circ(b \otimes
				-)\;\bigl(\epsilon_{\tau_i}\bigr)}
			\\
			\Hom(a,a_1) \otimes_\kk a_{\tau^{-1}(2)}
			\otimes  \dots \otimes a_{\tau^{-1}(i-1)} \otimes b \otimes
			a_{\tau^{-1}(i)} \otimes \dots \otimes a_{\tau^{-1}(\ho)}.
		\end{tikzcd}
	\end{equation*}
	
	If $a = b$, then the adjunction unit
	\[ 
	a_1 \otimes \dots \otimes a_{\ho} \to Q_aP_a(a_1\otimes \dots \otimes a_{\ho})
	\]
	embeds $a_1 \otimes \dots \otimes a_{\ho}$ 
	into the first summand as 
	$\{\id_a\} \otimes_\kk a_1 \otimes \dots \otimes a_{\ho}$.
\end{Example}

\begin{Example}\label{ex:PbQa_add}
	In the same way, we obtain
	\begin{align*}
		P_bQ_a(a_1\otimes \dots \otimes a_{\ho}) & = 
		P_b \bigl(\Hom(a,a_1) \otimes_\kk a_2 \otimes \dots \otimes a_{\ho}\bigr) \\ & =
		\bigoplus_{i=1}^{\ho} \Hom(a, a_1) \otimes_\kk a_2 \otimes \dots
		\otimes a_{i} \otimes b \otimes a_{i+1} \otimes \dots \otimes a_{\ho}.
	\end{align*}
	The equivariant structure is the same as in the preceding example,
	keeping in mind that 
	$$ (\Hom(a,-) \otimes \id) \circ ((i+1)\cdots1)\; \circ (b \otimes -)
	= (i...1) \circ (b \otimes -) \circ  (\Hom(a,-) \otimes \id). $$
	
	The adjunction counit
	\[
	P_{a}Q_{a}(a_1 \otimes \dots \otimes a_{\ho}) \to a_1 \otimes \dots \otimes a_{\ho}
	\] 
	first applies the adjunction map $\Hom(a,a_1) \otimes a \to a_1$ on each summand yielding
	\[
	\bigoplus_{i=1}^{\ho} a_2 \otimes \dots \otimes a_{i} \otimes a_1 \otimes a_{i+1} \dots \otimes a_{\ho}.
	\]
	Then the equivariant structure of $a_1 \otimes \dots \otimes a_{\ho}$ 
	provides a morphism 
	\[
	\bigoplus_{i=1}^{\ho} a_2 \otimes \dots \otimes a_{i} \otimes a_1 \otimes a_{i+1} \dots \otimes a_{l} 
	\xrightarrow{\sum \epsilon_{(12{\cdots}i)}} 
	a_1 \otimes \dots \otimes a_{\ho}.
	\]
\end{Example}

\begin{Example}
	\label{ex:unitcounit_add}
	In the same way one sees that the unit of the adjunction $Q_a \dashv P_{Sa}$ is given by the canonical map $a_1 \to \Hom(a, a_1) \otimes_\kk Sa$ coming from the $\Hom(a,-) \dashv - \otimes_{\kk} Sa$ adjunction followed by the diagonal map into the product.
	The counit is projection onto the factor $\Hom(a,Sa) \otimes_{\kk} a_1 \otimes \dots \otimes a_{\ho}$ followed by the Serre trace applied to $\Hom(a, Sa)$.
\end{Example}

It follows from the explicit computations above that there is an isomorphism
\begin{equation}\label{eq:QPPQ_add}
	Q_aP_b \cong \bigl(\Hom_\basecat(a,b) \otimes_\kk \hunit\bigr) \oplus P_b Q_a
\end{equation}
natural in $a,b \in \basecat$.

We can now define the action of $\Psi_\basecat'$ on $2$-morphisms.
Firstly, the dotted strings \tikz[baseline=0.2em] \draw[->] (0,0) -- node[label=right:{$\alpha$}, dot, pos=0.4] {} (0,1em); and \tikz[baseline=0.2em] \draw[<-] (0,0) -- node[label=right:{$\alpha$}, dot, pos=0.6] {} (0,1em); for $\alpha \in \Hom_\basecat(a,b)$ are sent to the natural transformations $P_a \Rightarrow P_b$ and $Q_b \Rightarrow Q_a$ induced by the natural transformations
\[
a \otimes - \xRightarrow{\alpha \otimes \id} b \otimes -
\qquad \text{and} \qquad
\Hom_\basecat(b,-) \xRightarrow{\alpha \circ -} \Hom_\basecat(a,-)
\]
respectively.

Next, the caps and cups
\[
\begin{tikzpicture}[baseline=0ex, scale=0.67]
	\draw[->] (0,0) node[below] {$\PP_{a}$} arc[start angle=180, end angle=0, radius=.5] node[below] {$\QQ_{a}$};
\end{tikzpicture}, \quad
\begin{tikzpicture}[baseline=0ex, scale=0.67]
	\draw[->] (0,0) node[below] {$\QQ_{a}$} arc[start angle=0, end angle=180, radius=.5] node[below] {$\PP_{Sa}$};
\end{tikzpicture}, \quad
\begin{tikzpicture}[baseline=0ex, scale=0.67]
	\draw[->] (0,0.5) node[above] {$\QQ_{a}$}arc[start angle=-180, end angle=0, radius=.5] node[above] {$\PP_{a}$};
\end{tikzpicture} \quad\text{and}\quad
\begin{tikzpicture}[baseline=0ex, scale=0.67]
	\draw[->] (0,0.5) node[above] {$\PP_{Sa}$} arc[start angle=0, end angle=-180, radius=.5] node[above] {$\QQ_{a}$};
\end{tikzpicture}
\]
are sent to the adjunction maps
\[
P_aQ_a \to \id, \quad
Q_{a}P_{Sa} \to \id, \quad
\id \to Q_aP_a, \quad
\text{and}\quad
\id \to P_{Sa}Q_a.
\]

Finally, the downward crossing
\[
\begin{tikzpicture}[baseline=-1ex, scale=0.67]
	\draw[<-] (0,-0.5) node[below] {$\QQ_{a}$} -- (1,0.5) node[above] {$\QQ_{a}$};
	\draw[<-] (1,-0.5) node[below] {$\QQ_{b}$} -- (0,0.5) node[above] {$\QQ_{b}$};
\end{tikzpicture}
\]
is sent to
{ the following functorial isomorphism. As functors 
	$\addsymbc{\ho} \rightarrow \addsymbc{\ho-2}$ we have
	$$ \QQ_{a} \QQ_{b} \simeq \bigl(\homm_\basecat(a,-) \otimes \homm_\basecat(b,-)
	\otimes \id_{\addsymbc{\ho-2}} \bigr) \circ 
	\Res_{\SymGrp \ho}^{1 \times 1 \times \SymGrp{\ho-2}},$$
	$$ \QQ_{b} \QQ_{a} \simeq \bigl(\homm_\basecat(b,-) \otimes \homm_\basecat(a,-)
	\otimes \id_{\addsymbc{\ho-2}} \bigr) \circ 
	\Res_{\SymGrp \ho}^{1 \times 1 \times \SymGrp{\ho-2}}.$$
	The latter can be further rewritten as
	$$ \QQ_{b} \QQ_{a} \simeq \bigl(\homm_\basecat(a,-) \otimes \homm_\basecat(b,-)
	\otimes \id_{\addsymbc{\ho-2}} \bigr) \circ (12) \circ 
	\Res_{\SymGrp \ho}^{1 \times 1 \times \SymGrp{\ho-2}}. 
	$$
	With these identifications in mind, we send the downward crossing
	to the functorial isomorphism $\QQ_{a} \QQ_{b} \xrightarrow{\sim}
	\QQ_{b} \QQ_{a}$ induced by the natural isomorphism 
	$$
	\Res_{\SymGrp \ho}^{1 \times 1 \times \SymGrp{\ho-2}}
	\xrightarrow{\sim} 
	(12) \circ \Res_{\SymGrp \ho}^{1 \times 1 \times \SymGrp{\ho-2}}$$
	given on any object $(\underline{a}, \epsilon_{\sigma})$ by
	$\epsilon_{(12)}$. 
}

Explicit computations (making particular use of the decomposition~\eqref{eq:QPPQ_add}) show that this definition of $\Psi_\basecat'$ is compatible with all 2-relations on $\hcatadd*\basecat$.
Thus we have the following result:

\begin{Proposition}
	The above definition gives a $2$-functor
	\[
	\Psi_\basecat'\colon \hcatadd*\basecat \to \fcatadd*\basecat.
	\]
\end{Proposition}

Let $\fcatadd\basecat$ be the $2$-category with objects the Karoubi completions $\Kar(\addsymbc{\ho})$, $1$-morphisms $\kk$-linear functors, and $2$-morphisms natural transformations.
We call $\fcatadd\basecat$ the \emph{Fock category}\index{Fock category} or,
	equivalently, the  \emph{categorical Fock space} of
$\hcatadd\basecat$. By the universal property of the Karoubi
	envelope, we have:

\begin{Corollary}
	The functor $\Psi_\basecat'$ induces a $2$-functor
	\[
	\Psi_\basecat\colon \hcatadd\basecat \to \fcatadd\basecat.
	\]
\end{Corollary}

\begin{Remark}
	The functors $\PP_a$ and $\QQ_a$ in the above definition have both a right and left adjoint.
	Hence, if $\basecat$ is abelian they are exact.
	Thus they extend to the Deligne tensor product, i.e.~there exists an action of $\hcatadd\basecat$ on  the $2$-category with objects $\smash[t]{\widehat{\addsym}}^{\ho}\basecat = (\basecat^{\widehat\otimes \ho})^{\SymGrp \ho}$, where $\widehat\otimes$ is the Deligne tensor product of abelian categories \cite[Proposition 1.46.2]{etingof2009topics}.
\end{Remark}

\chapter{Preliminaries on \dg Categories}\label{sec:dg-prelim}

In this section, we review the existing formalism of \dg categories and enriched bicategories and introduce several new results we need for our construction of a \dg Heisenberg $2$-category.  Below we summarise the key items of notation we employ. 

Given a \dg category $\A$,  we denote by $\modA$ its \dg category of right $\A$-modules. We denote by $\hprojA$ and $\perfA$ the full subcategories of $\modA$ comprising h-projective\index{h-projective} modules and perfect module\index{perfect module}s, respectively. We write $\hperfA$ for their intersection. We denote by  $\catD(\A)$ the derived category\index{derived category} of right $\A$-modules, and by $\catDc(\A)$ its full subcategory of compact objects. Note that $\catD(\A) \simeq \Hzero(\hprojA)$ and $\catDc(\A) \simeq \Hzero(\hperfA)$. 

Given a scheme $X$, we write $\catDQCoh{X}$ for the derived category
of complexes of sheaves on $X$ with quasi-coherent cohomology and
$\catDbCoh{X}$ for its full subcategory of complexes with bounded, coherent cohomology. Let $\catDGCoh{X}$ be the standard \dg enhancement of $\catDbCoh{X}$.

In this paper we arrange \dg categories into a m{\'e}nagerie of $1$-categories, strict $2$-categories and \dg bicategories. Figure~\ref{fig:dg-cat-cats} gives an overview of these and their relation to each other:

\begin{figure}[htb]
	\centering
	\begin{tikzpicture}[commutative diagrams/every diagram,
		ref/.style={font={\tiny}},
		x={(4,0)}, y={(0,3)},
		]
		\begin{scope}[every node/.style={anchor=base}]
			\node (HoDGCatone) at (-1,1.25) {$\HoDGCatone$};
			\node (EnhCatone) at (-0.66,0.75) {$\EnhCatone$};
			\node (DGCatone) at (-1,0.25) {$\DGCatone$};
			\node (MoDGCatone) at (-1,-1) {$\MoDGCatone$};
			
			\node (HoDGCat) at (0,1.25) {$\HoDGCat$};
			\node (EnhCat) at (0,0.75) {$\EnhCat^{\vphantom{1}}$};
			\node (DGCat) at ($(DGCatone.base)+(1,0)$) {$\DGCat^{\vphantom{1}}$};
			\node (MoDGCat) at (0,-1) {$\MoDGCat$};
			\node (EnhCatKC) at ($(0,-1)-(0,2.5em)$) {$\EnhCatKC$};
			
			\node (DGCatdg) at ($(DGCat.base)+(1,0)$) {$\DGCatdg$};
			\node (EnhCatKCdg) at ($(EnhCatKC.base)+(1,0)$) {$\EnhCatKCdg$};
			\node (dgModCat) at ($(EnhCatKCdg.base)!0.75!(DGCatdg.base)$) {$\DGModCat$};
			\node (dgBiMod) at  ($(EnhCatKCdg.base)!0.45!(DGCatdg.base)$) {$\DGBiMod$};
			\node (dgBiModlfrp) at  ($(EnhCatKCdg.base)!0.25!(DGCatdg.base)$) {$\smash[b]{\DGBiMod_{\lfrp}}$};
			
			\node[ref] at (HoDGCatone.north east) {Sec.~\ref{section-enhanced-categories}};
			\node[ref] at (DGCatone.north east) {Def.~\ref{defn-enrichment-2categories}};
			\node[ref] at (MoDGCatone.north east) {Sec.~\ref{section-enhanced-categories}};
			
			\node[ref] at (EnhCatone.north east) {Sec.~\ref{section-enhanced-categories}};
			
			\node[ref] at (HoDGCat.north east) {Def.~\ref{defn-enrichment-2categories}};
			\node[ref] at (EnhCat.north east) {\cite[Sec.~1]{lunts2010uniqueness}};
			\node[ref] at (DGCat.north east) {Def.~\ref{defn-enrichment-2categories}};
			\node[ref] at (EnhCatKC.north east) {Def.~\ref{defn-EnhCatKC}};
			
			\node[ref] at (dgModCat.north east) {Def.~\ref{defn-dgModCat}};
			\node[ref] at (dgBiMod.north east) {Def.~\ref{defn-dgBiMod}};
			\node[ref] at (dgBiModlfrp.north east) {Sec.~\ref{sec:magic-wand}};
			\node[ref] at (EnhCatKCdg.north east) {Def.~\ref{defn-EnhCatKCdg}};
		\end{scope}
		
		\path[commutative diagrams/.cd, every arrow, every label]
		(DGCatone) edge node[anchor=south, rotate=90] {\tiny loc.~by} node[anchor=north, rotate=90] {\tiny quasi-equiv.} (HoDGCatone)
		(DGCatone) edge node[anchor=south, rotate=90] {\tiny loc.~by} node[anchor=north, rotate=90] {\tiny Morita equiv.} (MoDGCatone)
		(MoDGCatone) edge[out=120, in=240, commutative diagrams/hook] node[anchor=south, rotate=90] {$\A \mapsto \hperfA$} (HoDGCatone)
		(EnhCatone) edge[commutative diagrams/hook] (HoDGCatone)
		
		(HoDGCat) edge[dashed] (HoDGCatone)
		(EnhCat) edge[dashed] (EnhCatone)
		(EnhCat) edge[commutative diagrams/hook] (HoDGCat)
		(DGCat) edge[dashed] (DGCatone)
		(MoDGCat) edge[dashed] (MoDGCatone)
		(MoDGCat) edge[commutative diagrams/equal] (EnhCatKC)
		
		
		(DGCatdg) edge[decoration={snake,amplitude=0.33ex,segment length=1.5ex,pre length=1mm, post length=1mm}, decorate] (DGCat)
		(EnhCatKCdg) edge[decoration={snake,amplitude=0.33ex,segment length=1.5ex,pre length=1mm, post length=1mm}, decorate] (EnhCatKC)
		(dgModCat) edge[commutative diagrams/hook] (DGCatdg)
		(dgModCat) edge[transform canvas={xshift=-0.7ex}] node[left, pos=0.40] {$\tensorfn$} (dgBiMod)
		(dgBiMod) edge[transform canvas={xshift=0.5ex}] node[right, pos=0.60] {$\bimodapx$} (dgModCat)
		(dgBiModlfrp) edge[commutative diagrams/hook] (dgBiMod)
		(dgBiModlfrp) edge node[left, pos=0.35] {$L$} (EnhCatKCdg) 
		;
		
		\path (MoDGCat.north -| EnhCat.south west) edge[bend left=15,commutative diagrams/.cd, every arrow, every label, hook] node[anchor=south, rotate=90] {$\A \mapsto \hperfA$} (EnhCat.south west);
	\end{tikzpicture}
	\caption{Summary of various categories of \dg categories. Dashed arrows represent $1$-categorical truncation and the squiggly arrow represents taking homotopy categories of the $1$-morphism categories.}
	\label{fig:dg-cat-cats}
\end{figure}

\begin{itemize}
	\item $\DGCatone$ is the $1$-category of \dg categories and
\dg functors\index{DG functor}
	between them, see Section~\ref{section-enriched-bicategories}, Definition 
	\ref{defn-enrichment-2categories}, 
	\item $\HoDGCatone$ is the localisation of $\DGCatone$ by
	quasi-equivalences, see Section~\ref{section-enhanced-categories} and
	\cite{Toen-TheHomotopyTheoryOfDGCategoriesAndDerivedMoritaTheory}, 
	\item $\EnhCatone$ is the full subcategory of $\HoDGCatone$
	comprising pretriangulated \dg categories. We view it as
	the $1$-category of enhanced triangulated categories, see
	Section~\ref{section-enhanced-categories}. 
	\item $\MoDGCatone$ is the localisation of $\DGCatone$ by
	Morita equivalence\index{Morita equivalence}s. We view it as
the $1$-category of Morita enhanced triangulated categories\index{Morita
enhancement}, see Section~\ref{section-enhanced-categories} and
	\cite{Tabuada-InvariantsAdditifsDeDGCategories}, 
	\item $\DGCat$ is the strict $2$-category of the isomorphism classes of
	\dg categories, \dg functors, and closed degree zero
	\dg natural transformations,
	see Section~\ref{section-enriched-bicategories}, 
	Definition \ref{defn-enrichment-2categories},
	\item $\HoDGCat$ is a strict $2$-categorical version of $\HoDGCatone$
	constructed using the main results of 
	\cite{Toen-TheHomotopyTheoryOfDGCategoriesAndDerivedMoritaTheory}, 
	see 
	Section~\ref{section-enriched-bicategories}, Definition~\ref{defn-enrichment-2categories} and \cite{Toen-TheHomotopyTheoryOfDGCategoriesAndDerivedMoritaTheory}, 
	\item $\EnhCat$ is the strict $2$-category of enhanced triangulated
	categories, see \cite[Sec.~1]{lunts2010uniqueness}. 
	It is the $1$-full subcategory\index{$1$-full subcategory} of $\HoDGCat$ comprising 
	pretriangulated \dg categories.  
	\item $\MoDGCat$ is the strict $2$-category 
	of Morita enhanced triangulated categories, see
	Section~\ref{section-enhanced-categories}, Definition~\ref{defn-EnhCatKC}. 
	It is also known as $\EnhCatKC$, because it can be realised as 
	the $1$-full subcategory of $\EnhCat$ comprising homotopy
Karoubi complete\index{homotopy Karoubi completion}
	\dg categories. Here and throughout the paper the subscript $\kc$ means 
	`Karoubi complete'. 
	\item $\DGCatdg$ is the strict \dg $2$-category of  
	\dg categories, \dg functors, and (all) \dg natural transformations,
	\item $\DGModCat$ is the strict \dg $2$-category of \dg module
	categories. It is the $1$-full subcategory of $\DGCatdg$ consisting 
	of all \dg categories of form $\modA$ for some small \dg category
	$\A$, see Section~\ref{section-bimodule-approximation}, Definition 
	\ref{defn-dgModCat}, 
	\item $\DGBiMod$ is the \dg bicategory\index{DG bicategory} whose objects are small \dg
	categories and whose $1$-morphism categories are \dg categories 
	of \dg bimodules, see 
	Section~\ref{section-bimodule-approximation}, Definition \ref{defn-dgBiMod}, 
	\item $\DGBiMod_{\lfrp}$ is the $2$-full subcategory of $\DGBiMod$
	comprising the same objects and the $1$-morphisms given by
	left-h-flat and right-perfect bimodules, see 
	Sec.~\ref{sec:magic-wand}, 
	\item $\EnhCatKCdg$ is the lax-unital \dg bicategory of Morita enhanced
	triangulated categories. It is a \dg enhancement\index{DG enhancement} of $\EnhCatKC$
	and is a new object introduced in this paper, see Definition
	\ref{defn-EnhCatKCdg}. Alternatively, it can be constructed as
	a strictly unital $\HoDGCat$-enriched bicategory, see
	Section~\ref{section-enhanced-categories}, Definition
	\ref{defn-alternative-enhcatkcdg}. 
\end{itemize}

\section{Enriched bicategories}
\label{section-enriched-bicategories} 

The \dg version of the Heisenberg category, which we define in Chapter~\ref{sec:dg-Heisenberg-2-cat}, is a certain weak $2$-category and its representations are given by weak $2$-functors.  
The notion of a weak $2$-category we use is a \emph{bicategory}\index{bicategory}.  We refer to~\cite{Benabou-IntroductionToBicategories} for 
the original definition and a comprehensive technical treatment of bicategories. 

We need to work with \emph{enriched bicategories}\index{enriched bicategory}. 
The natural structure to enrich bicategories over is a monoidal 
bicategory or, more generally, a tricategory. The formal definitions can be found in 
\cite{GordonPowerStreet-CoherenceForTricategories}, and they are 
rather involved. However, a reader comfortable with the properties of cartesian 
products of categories and tensor products of \dg categories need not consider 
the formal definition of a tricategory for the purposes of reading this paper. 
We only work with enrichments over one of the following three strictly monoidal 
strict $2$-categories:
\begin{Definition}
	\label{defn-enrichment-2categories}\leavevmode
	\begin{enumerate}
		\item $\Cat$: The $2$-category of isomorphism classes of
		small categories, of functors, and of natural transformations. 
		The monoidal structure is the cartesian product $\times$.
		\item $\DGCat$: The $2$-category of isomorphism classes of
		small $\kk$-linear \dg categories, of \dg functors, and of
		(closed degree zero) \dg natural transformations.
		The monoidal structure is the tensor product $\otimes_\kk$ 
		over $\kk$. We further write $\DGCatone$ for the underlying 
		$1$-category of $\DGCat$, where we only consider \dg categories and \dg functors
		between them. 
		\item $\HoDGCat$: The $2$-categorical version considered in
		\cite{Toen-TheHomotopyTheoryOfDGCategoriesAndDerivedMoritaTheory} of the localisation of
		$\DGCatone$ by quasi-equivalences. 
		Its objects are the isomorphism classes of small $\kk$-linear \dg categories, 
		its $1$-morphisms are the isomorphism classes of right quasi-representable
		bimodules in $\catD(\AbimB)$, and its $2$-morphisms are the morphisms 
		between these in $\catD(\AbimB)$.
		The monoidal structure is given by the tensor product $\otimes_\kk$. 
	\end{enumerate}
\end{Definition}

For the general definition of an enriched bicategory we
refer the reader to \cite[Section~3]{GarnerShulman-EnrichedCategoriesAsAFreeCocompletion}. 
Considering only enrichments over strictly monoidal strict
$2$-categories allows us to give a simpler definition 
which is nearly identical to the original definition of 
a bicategory in \cite{Benabou-IntroductionToBicategories}.

\begin{Definition}\label{def:enrichedbicategory}
	Let $(\bicat M, \otimes, 1_{\bicat M})$ be a strictly monoidal strict $2$-category. 
	A \emph{bicategory $\bicat C$ enriched over $\bicat M$}\index{enriched bicategory} comprises the following
	data: 
	\begin{enumerate}
		\item a collection of \emph{objects} $\obj \bicat C$; 
		\item $\forall\; a,b \in \obj \bicat C$ a \emph{$1$-morphism object}\index{$1$-morphism} $\bicat C(a,b)$, which is an object in $\bicat M$; 
		\item $\forall\; a \in \obj \bicat C$ an \emph{identity element} $1_a: 1_{\bicat M} \rightarrow \bicat C(a,a)$, which is a $1$-morphism in $\bicat M$; 
		\item $\forall\; a,b,c \in \obj \bicat C$ the \emph{$1$-morphism composition}\index{$1$-composition}, which is a $1$-morphism in $\bicat M$:
		\begin{equation*}
			\mu\colon \bicat C(b,c) \otimes \bicat C(a,b)  \rightarrow \bicat C(a,c);
		\end{equation*}
		\item $\forall\; a,b,c,d \in \obj \bicat C$  the \emph{associator}\index{associator} $\alpha$ which is a $2$-isomorphism in $\bicat M$:
		\begin{equation*}
			\begin{tikzcd}[row sep = 0.25cm]
				\bicat C(c,d) \otimes \bicat C(b,c) \otimes \bicat C(a,b) 
				\ar{d}{\;\mu \otimes \id}
				\\
				\bicat C(b,d) \otimes \bicat C(a,b) 
				\ar{d}{\;\mu}
				\\
				\bicat C(a,d)
			\end{tikzcd}
			\xrightarrow[\ \ \textstyle\sim\ \ ]{\textstyle\alpha}
			\begin{tikzcd}[row sep = 0.25cm]
				\bicat C(c,d) \otimes \bicat C(b,c) \otimes \bicat C(a,b) 
				\ar{d}{\;\id \otimes \mu}
				\\
				\bicat C(c,d) \otimes \bicat C(a,c) 
				\ar{d}{\;\mu}
				\\
				\bicat C(a,d);
			\end{tikzcd}
		\end{equation*}
		\item $\forall\; a,b \in \obj \bicat C$ the \emph{unitors}\index{unitor} $\rho$ and $\lambda$ which are $2$-isomorphisms of $1$-morphisms in $\bicat M$:
		\begin{equation*}
			\begin{tikzcd}[row sep = 0.25cm]
				\bicat C(a,b) \ar[equals]{d} 
				\\
				\bicat C(a,b) \otimes 1_{\bicat M} \ar{d}{\;\id \otimes 1_a}
				\\
				\bicat C(a,b) \otimes \bicat C(a,a) \ar{d}{\;\mu}
				\\
				\bicat C(a,b)
			\end{tikzcd}
			\xrightarrow[\ \ \textstyle\sim\ \ ]{\textstyle\rho}
			\begin{tikzcd}
				\bicat C(a,b) \ar{ddd}{\id}
				\\
				\\
				\\
				\bicat C(a,b)
			\end{tikzcd}
		\end{equation*}
		and
		\begin{equation*}
			\begin{tikzcd}[row sep = 0.25cm]
				\bicat C(a,b) \ar[equals]{d} 
				\\
				1_{\bicat M} \otimes \bicat C(a,b) \ar{d}{\;1_a \otimes \id}
				\\
				\bicat C(a,a) \otimes \bicat C(a,b) \ar{d}{\;\mu}
				\\
				\bicat C(a,b)
			\end{tikzcd}
			\xrightarrow[\ \ \textstyle\sim\ \ ]{\textstyle\lambda}
			\begin{tikzcd}
				\bicat C(a,b) \ar{ddd}{\id}
				\\
				\\
				\\
				\bicat C(a,b)
			\end{tikzcd}
		\end{equation*}
	\end{enumerate}
	which must satisfy the following conditions:
	\begin{enumerate}[resume*]
		\item\label{item-associativity-coherence-diagram} 
		$\forall\; a,b,c,d,e \in \obj \bicat C$ the following diagram of $2$-morphisms between $1$-morphisms 
		$
		\bicat C(d,e) \otimes \bicat C(c,d) \otimes \bicat C(b,c) \otimes \bicat C(a,b) \rightarrow \bicat C(a,e)
		$
		must commute in $\bicat M$:
		\begin{equation*}
			\begin{tikzcd}[sep=2.1em, font=\footnotesize]
				\mu \circ (\mu \otimes \id) \circ  (\mu \otimes \id \otimes \id)
				\ar{r}{\mu \circ (\alpha \otimes \id) }
				\ar{d}{ \alpha \circ \id }
				&
				\mu \circ (\mu \otimes \id) \circ  (\id \otimes \mu \otimes \id)
				\ar{r}{ \alpha \circ \id } 
				&
				\mu \circ (\id \otimes \mu) \circ  (\id \otimes \mu \otimes \id)
				\ar{d}{\mu \circ (\id \otimes \alpha)}
				\\
				\mu \circ (\id \otimes \mu) \circ  (\mu \otimes \id \otimes \id)
				\ar{rr}{\alpha \circ \id}
				&
				&
				\mu \circ (\id \otimes \mu) \circ  (\id \otimes \id \otimes  \mu);
			\end{tikzcd}
		\end{equation*}
		\item\label{eqn-unitality-coherence-diagram}
		$\forall\; a,b,c \in \obj \bicat C$ the following diagram of $2$-morphisms between $1$-morphisms 
		$
		\bicat C(b,c) \otimes \bicat C(a,b) \rightarrow \bicat C(a,c)
		$
		must commute in $\bicat M$:
		\begin{equation*}
			\begin{tikzcd}
				\mu \circ (\mu \otimes \id) \circ (\id \otimes 1_b \otimes \id) 
				\ar{rr}{\alpha \circ (\id \otimes 1_b \id)}
				\ar{dr}[']{\mu \circ (\rho \otimes \id)}
				&&
				\mu \circ (\id \otimes \mu) \circ (\id \otimes 1_b \otimes \id) 
				\ar{dl}{\mu \circ (\id \otimes \lambda)}
				\\
				&
				\mu
				& 
			\end{tikzcd}.
		\end{equation*}
	\end{enumerate} 
\end{Definition}

\begin{Remark}  The objects of the three $2$-categories we define in 
	Definition \ref{defn-enrichment-2categories} are isomorphism 
	classes of categories. This is to make the strictly associative monoidal 
	structures provided by $\times$ and $\otimes_\kk$ also strictly unital. 
	To work with individual categories one only needs to adjust the definition 
	above to allow the monoidal structure on $\bicat M$ to be lax-unital. 
\end{Remark}

\begin{Examples}\leavevmode
	\begin{enumerate}
		\item A bicategory enriched over $\Cat$ is an ordinary bicategory 
		in the sense of \cite{Benabou-IntroductionToBicategories}. We refer
		to these simply as \emph{bicategories}. Special cases are: 
		\begin{enumerate}
			\item A bicategory with a single object is a \emph{monoidal category}\index{monoidal category}. 
			\item A bicategory whose associator and unitor isomorphisms are identity maps is a \emph{strict $2$-category}. 
		\end{enumerate}
		\item A bicategory enriched over $\DGCat$ is a \emph{\dg bicategory}\index{DG bicategory}. 
	\end{enumerate}
\end{Examples}

\begin{Remark}
	\label{rem:dg-interchange-law}
	Consider a \dg bicategory $\bicat C$.
	Then the data of the $1$-composition functor $\mu = \circ_1\colon  \bicat C(b,c) \otimes \bicat C(a,b) \rightarrow \bicat C(a,c)$ gives rise to the graded interchange law
	\[
	(\alpha \circ_1 \beta) \circ_2 (\gamma \circ_1 \delta) = (-1)^{|\beta||\gamma|} (\alpha \circ_2 \gamma) \circ_1 (\beta \circ_2 \delta),
	\]
	where we write $\circ_2$ for the $2$-composition, i.e., the composition in the $1$-morphism categories.
\end{Remark}

\begin{Definition}
	Let $(\bicat M, \otimes, 1_{\bicat M})$ be a strictly monoidal strict $2$-category.
	Let $\bicat C$ and $\bicat D$ be two bicategories enriched over $\bicat M$. 
	An \emph{enriched $2$-functor} $F\colon \bicat C \rightarrow \bicat D$ comprises 
	\begin{enumerate}
		\item a map $F\colon \obj \bicat C \rightarrow \obj \bicat D$;
		\item $\forall\; a,b \in \obj \bicat C$ a $1$-morphism $F_{a,b}$ in $\bicat M$,
		\[
		F_{a,b}\colon \bicat C(a,b) \rightarrow \bicat D(Fa,Fb);
		\]
		\item $\forall\; a \in \obj \bicat C$ a \emph{unit coherence}\index{unit coherence} $2$-morphism $\iota$ in $\bicat M$ between 
		the following $1$-morphisms $1_{\bicat M} \rightarrow \bicat D(Fa,Fa)$:
		\[
		\iota\colon 1_{Fa} \rightarrow F_{a,a} \circ 1_a;
		\]
		\item $\forall\; a,b,c \in \obj \bicat C$ a \emph{composition coherence}\index{composition coherence}
		$2$-morphism  $\phi$ in $\bicat M$ between the following $1$-morphisms $\bicat C(b,c) \otimes \bicat C(a,b)
		\rightarrow \bicat D(Fa,Fc)$:
		\[
		\phi\colon \mu_{\bicat D} \circ (F_{b,c} \otimes F_{a,b}) \rightarrow F_{a,c} \circ \mu_{\bicat C};
		\]
	\end{enumerate}
	which must satisfy the following conditions:
	\begin{enumerate}[resume*]
		\item \emph{associativity coherence:}  $\forall\; a,b,c,d \in \obj \bicat C$ 
		the following diagram of $2$-morphisms between $1$-morphisms 
		$
		\bicat C(c,d) \otimes \bicat C(b,c) \otimes \bicat C(a,b) \rightarrow \bicat D(Fa,Fd)
		$
		must commute in $\bicat M$:
		\begin{equation*}
			\begin{tikzcd}[column sep = 2.4cm, font=\footnotesize]
				\mu_{\bicat D} \circ (\mu_{\bicat D} \otimes \id) \circ (F_{c,d} \otimes F_{b,c} \otimes F_{a,b})
				\ar{r}{\alpha_{\bicat D} \circ (F_{c,d} \otimes F_{b,c} \otimes F_{a,b})}
				\ar{d}{\mu_{\bicat D}\circ (\phi \otimes  F_{a,b})}
				&
				\mu_{\bicat D} \circ (\id \otimes \mu_{\bicat D}) \circ (F_{c,d} \otimes F_{b,c} \otimes F_{a,b})
				\ar{d}{\mu_{\bicat D} \circ (F_{c,d} \otimes \phi)}
				\\
				\mu_{\bicat D} \circ (F_{b,d} \otimes F_{a,b}) \circ (\mu_{\bicat C} \otimes \id)
				\ar{d}{\phi \circ (\mu_{\bicat C} \otimes \id)}
				&
				\mu_{\bicat D} \circ (F_{c,d} \otimes F_{a,c}) \circ 
				(\id \otimes \mu_{\bicat C})
				\ar{d}{\phi \circ (\id \otimes \mu_{\bicat C})}
				\\
				F_{a,d} \circ \mu_{\bicat C} \circ (\mu_{\bicat C} \otimes \id)
				\ar{r}{F_{a,d} \circ \alpha_{\bicat C}}
				&
				F_{a,d} \circ \mu_{\bicat C} \circ (\id \otimes \mu_{\bicat C});
			\end{tikzcd}
		\end{equation*}
		\item \emph{unitality coherence:} $\forall\; a,b \in \obj \bicat C$
		the following diagrams of $2$-morphisms between $1$-morphisms $\bicat C(a,b) \rightarrow \bicat D(Fa,Fb)$ must commute in $\bicat M$:
		\begin{equation*}
			\begin{tikzcd}[column sep=2.5cm]
				\mu_{\bicat D}\circ (\id \otimes 1_{Fa}) \circ F_{a,b}
				\ar{r}{\mu_{\bicat D} \circ (\id \otimes \iota) \circ F_{a,b}}
				\ar{d}{\rho_{\bicat D} \circ F_{a,b}}
				&
				\mu_{\bicat D}\circ (F_{a,b} \otimes F_{a,a}) \circ (\id \otimes 1_a) 
				\ar{d}{\phi \circ (\id \otimes 1_a)  }
				\\
				F_{a,b}
				&
				F_{a,b} \circ \mu_{\bicat C} \circ (\id \otimes 1_a)
				\ar{l}[']{F_{a,b} \circ \rho_{\bicat C}}
			\end{tikzcd}
		\end{equation*}
		\medskip
		\begin{equation*}
			\begin{tikzcd}[column sep=2.5cm]
				\mu_{\bicat D}\circ (1_{Fb} \otimes \id) \circ F_{a,b}
				\ar{r}{\mu_{\bicat D} \circ (\iota \otimes \id) \circ F_{a,b}}
				\ar{d}{\lambda_{\bicat D} \circ F_{a,b}}
				&
				\mu_{\bicat D}\circ (F_{b,b} \otimes F_{a,b}) \circ (1_b \otimes \id) 
				\ar{d}{\phi \circ (\id \otimes 1_a) }
				\\
				F_{a,b}
				&
				F_{a,b} \circ \mu_{\bicat C} \circ (1_b \otimes \id)
				\ar{l}[']{F_{a,b} \circ \lambda_{\bicat C}}
			\end{tikzcd}
		\end{equation*}
	\end{enumerate}
\end{Definition}

\begin{Definition}
	A $2$-functor is said to be:
	\begin{itemize}
		\item \emph{strict}\index{strict $2$-functor} if its unit and composition coherence maps are the identity maps; 
		\item \emph{strong}\index{strong $2$-functor} if its unit and composition coherence maps are isomorphisms; 
		\item \emph{homotopy strong}\index{homotopy strong $2$-functor} if its unit and composition coherence maps are homotopy equivalences. 
		\item \emph{lax}\index{lax $2$-functor} if its unit and composition maps are not necessarily isomorphisms;
	\end{itemize}
\end{Definition}

\section{\Dg-categories}
\label{sec:dgcats}

For an introduction to \dg categories, \dg modules, and 
the related  technical notions, we refer the reader to
\cite[Section~2]{AnnoLogvinenko-SphericalDGFunctors}. 
For an in-depth treatment in the language of model categories
see \cite{Toen-LecturesOnDGCategories}. 
Below we review the main notions we use. 

\subsection{\dg categories and \dg modules}
\label{section-dg-categories-and-dg-modules}

A \emph{\dg \ (differential graded)}\index{DG category} category $\A$ is a category enriched over the monoidal 
category $\modk$ of complexes of $\kk$-modules. A (right) module $E$ over $
\A$ is a functor $E \colon \Aopp \rightarrow \modk$. For any $a \in \A$ we write $E_a$
for the complex $E(a) \in \modk$, the \emph{fibre of $E$ over $a$}\index{fibre}. 
We write $\modA$ for the \dg category of (right) $\A$-modules. 
Similarly, a left $\A$-module $F$ is a functor $F\colon \A \rightarrow
\modk$. We write $\leftidx{_a}F$ for the fibre $F(a) \in \modk$
of $F$ over $a \in \A$ and $\Amod$ for the \dg category of left 
$\A$-modules. For any $a \in \A$ define the right and left 
\emph{representable}\index{representable module} modules corresponding to $a$ to be 
$h^r(a) = \homm_\A(-,a) \in \modA$ and $h^l(a) = \homm_\A(a,-) \in
\Amod$. We further have Yoneda embeddings 
$\A \hookrightarrow \modA$ and $\Aopp \hookrightarrow \Amod$ 
whose images are the representable modules. 

Given another \dg category $\B$, an
$\AbimB$-bimodule $M$ is an $\Aopp \otimes_\kk \B$-module, that is, 
a functor $M\colon \A \otimes_{\kk} \Bopp \rightarrow \modk$. 
For any $a \in \A$ and $b \in \B$ we write $\aM \in \modB$
for the fibre $M(a,-)$ of $M$ over $a$, 
$M_b \in \Amod$ for the fibre $M(-,b)$ of $M$ over $b$, 
and $\aMb \in \modk$ for the fibre of $M$ over $(a,b)$. 
We write $\AmodB$ for the \dg category of $\A$-$\B$-bimodules. 
The categories $\modA$ and $\Amod$ of right and left $\A$-modules
can therefore be considered as the categories of $\kk$-$\A$- and 
$\A$-$\kk$-bimodules. For any \dg category $\A$, we write $\A$ for 
the diagonal $\AbimA$-bimodule defined by 
$\leftidx{_a}\A_b = \homm_\A(b,a)$ for all $a,b \in \A$ and morphisms
of $\A$ acting on the right and on the left by pre- and post-composition, respectively:
\begin{equation}
	\label{eqn-A-A-action-on-diagonal-bimodule}
	\A(\alpha \otimes \beta) = 
	(-1)^{\deg(\beta) \deg(-)} \alpha \circ (-) \circ \beta,
	\quad \quad 
	\forall\; \alpha \in \homm_\A(a,a'), \beta \in \homm_\A(b',b).
\end{equation}

\Dg bimodules over \dg categories admit 
a closed symmetric monoidal structure\index{closed symmetric monoidal structure}. Given three \dg categories 
$\A$, $\B$ and $\C$, we define functors 
\begin{align*}
	(-)\otimes_\B(-) & \colon \AmodB \otimes \BmodC \rightarrow \AmodC,\\
	\homm_\B(-,-)    & \colon \AmodB \otimes \CmodB \rightarrow \CmodA,\\
	\homm_\B(-,-)    & \colon \BmodA \otimes \BmodC \rightarrow \AmodC,
\end{align*}
by
\[
M \otimes_\B N = 
\cok (M \otimes_\kk \B \otimes_\kk N \xrightarrow{\action \otimes \id - \id \otimes \action} M \otimes_\kk N),
\]
\[\leftidx{_c}{\left(\homm_\B(M,N)\right)}_a = \homm_\B(\aM,\cN)\]
for $M,N$ with right $\B$-action, and 
\[\leftidx{_a}{\left(\homm_\B(M,N)\right)}_c = \homm_\B(M_c,N_a)\]
for $M,N$ with left $\B$-action, 
cf.~\cite[Section~2.1.5]{AnnoLogvinenko-SphericalDGFunctors}. 

\subsection{The derived category of a \dg category}

Let $\A$ be a \dg category. Its \emph{homotopy category}\index{homotopy category} $\Hzero(\A)$ is the $\kk$-linear category whose objects are the same as those of $\A$ and whose morphism spaces are $\Hzero(-)$ 
of the morphism complexes of $\A$. 

The category $\Hzero(\modA)$ has a natural structure of a triangulated category\index{triangulated category} defined fibrewise in $\modk$, that is: the homotopy category $\Hzero(\modk)$ of complexes of $\kk$-modules has a natural triangulated structure, and we apply it in each fibre over each $a \in \A$ to define the triangulated structure on $\Hzero(\modA)$. A \dg category $\A$ is \emph{pretriangulated}\index{pretriangulated category} 
if $\Hzero(\A)$ is a triangulated subcategory of $\Hzero(\modA)$ under the Yoneda embedding.

An $\A$-module $E$ is \emph{acyclic}\index{acyclic module} if it is acyclic fibrewise in
$\modk$. We denote by $\acyc \A$ the full subcategory of  $\modA$
consisting of acyclic modules. A morphism of $\A$-modules is a
\emph{quasi-isomorphism}\index{quasi-isomorphism} if it is one levelwise  in $\modk$. The
derived category $\catD(\A)$ is the localisation of $\Hzero(\modA)$ by
quasi-isomorphisms, constructed as the Verdier quotient $\Hzero(\modA)/\acyc \A$. 

The derived category can also be constructed on the \dg level. 
An $\A$-module $P$ is \emph{h-projective}\index{h-projective} (resp.~\emph{h-flat}\index{h-flat}) if $\homm_\A(P,C)$ (resp. $P \otimes_\A C$) is an acyclic complex of $\kk$-modules for any acyclic $C \in \modA$ (resp. $C \in \Amod$). We denote by $\hprojA$ the full subcategory of $\modA$ consisting of
h-projective modules. It follows from the definition that
in $\hprojA$ every quasi-isomorphism is a homotopy equivalence, 
and therefore we have $\catD(\A) \simeq \Hzero(\hprojA)$. Alternatively, 
one uses Drinfeld quotient\index{Drinfeld quotient}s \cite{Drinfeld-DGQuotientsOfDGCategories}: 
Given a \dg category $\A$ with a full subcategory $\C \subset \A$,
we can form the Drinfeld quotient $\A / \C$. When $\A$ and $\C$ are pretriangulated, 
this recovers the 
Verdier quotient as $\Hzero(\A/\C) \simeq \Hzero(\A)/\C$. 
Thus $\catD(\A) = \Hzero(\modA / \acyc { \A})$. 

An object $a$ of a triangulated category $\T$ is \emph{compact}\index{compact object} if
$\homm_\T(a,-)$ commutes with infinite direct sums. We write $\catDc(\A)$
for the full subcategory of $\catD(\A)$ comprising compact objects. 
An $\A$-module $E$ is \emph{perfect}\index{perfect module} if $E \in \catDc(\A)$.
We write $\perfA$ and $\hperfA$ for the full subcategories of $\modA$
comprising perfect modules and h-projective, perfect modules.

Let $\A$ be a \dg category. We denote by $\pretriag \A$ 
the \dg category of one-sided twisted complex\index{twisted complex}es over $\A$, see
\cite[Section~3.1]{AnnoLogvinenko-SphericalDGFunctors}. It is a \dg 
version of the notion of triangulated hull\index{triangulated hull}. There is a natural fully faithful inclusion 
$\pretriag \A \hookrightarrow \modA$ and $\Hzero(\pretriag \A)$ is the triangulated hull of $\Hzero(\A)$
in $\Hzero(\modA)$. Moreover, we have $\pretriag \A \subset \hperfA$
and $\catDc(\A) \simeq \Hzero(\hperfA)$ is the Karoubi
completion\index{Karoubi completion} of 
$\Hzero(\pretriag \A)$ in $\Hzero(\modA)$. 
We say that $\A$ is \emph{strongly pretriangulated}\index{strongly pretriangulated category}
if $\A \hookrightarrow \pretriag \A$ is an equivalence.  

Let $\A$ and $\B$ be \dg categories and let $M$ be an
$\A$-$\B$-bimodule. 
We say that $M$ is $\A$\emph{-perfect} (resp. $\B$\emph{-perfect}) 
if it is perfect levelwise in $\A$ (resp. $\B$). That is, 
$\aM$ (resp. $M_b$) is a perfect module for all $a \in \A$
(resp. $b \in \B$). Similarly, for other properties of modules
such as h-projective, h-flat, or representable. 

A \dg category $\A$ is \em smooth\index{smooth \dg category} \rm if the diagonal 
bimodule $\A$ is perfect as an $\A$-$\A$-bimodule. It is
\em proper\index{proper \dg category} \rm if $\A$ is Morita equivalent (see Section~\ref{section-enhanced-categories}) 
to a \dg algebra which is perfect over $\kk$. Equivalently, $\A$ is proper if and only if the 
total cohomology of each $\homm$-complex is
finitely-generated and $\catD(\A)$ is compactly generated. See  
\cite[Section~2.2]{ToenVaquie-ModuliOfObjectsInDGCategories} for further
details on these two notions. 

\subsection{Restriction and extension of scalars}
\label{section-restriction-and-extension-of-scalars}

Let $\A$ and $\B$ be two \dg categories and let $M$ be an
$\A$-$\B$-bimodule. Moreover, let $\A'$ and $\B'$ be another two
\dg categories and let $f\colon \A' \rightarrow \A$ and $g\colon
\B' \rightarrow \B$ be \dg functors. Define the \em restriction of
scalars of $M$ along $f$ and $g$ \rm to be the $\A'$-$\B'$-bimodule
${}_{f}{M}_{g}$ defined as $M \circ (f \otimes_\kk g)$. 
In particular, for any $a \in \A$ and $b \in \B$ we have
${}_{a}{({}_{f}{M}_{g})}{}_{b} = 
{}_{f(a)}{M}_{g(b)}$. We write ${}_{f}{M}$ and 
$M_g$ for ${}_{f}{M}_{\id}$ and ${}_{\id}{M}_{g}$,
respectively. 

Let $\A$ and $\B$ be two \dg categories and let $f: \A \rightarrow \B$ 
be a \dg functor\index{DG functor}. We have:
\begin{enumerate}
	\item The \em restriction of scalars \rm functor
	\begin{equation*}
		f_*\colon \modB \rightarrow \modA, 
	\end{equation*}
	is defined to be $(-) \otimes_\B \B_f$. It sends each $E \in \modB$ to its
	restriction $E_f \in \modA$, and therefore sends acyclic modules to
	acyclic modules. 
	\item The \em extension of scalars \rm functor
	\begin{equation*}
		f^*\colon \modA \rightarrow \modB, 
	\end{equation*}
	is defined to be $(-) \otimes_\A \leftidx{_f}{\B}$. For each $a \in \A$  
	it sends the representable module $h^r(a) \in \modA$ 
	to the representable module $h^r(f(a)) \in \modB$. It follows that 
	$f^*$ restricts to a functor $\hperfA \rightarrow \hperfB$. 
	\item The \em twisted extension of scalars \rm functor
	\begin{equation*}
		f^!\colon \modA \rightarrow \modB, 
	\end{equation*}
	is defined to be $\homm_{\A}(\B_f,-)$. 
\end{enumerate}
By Tensor-$\homm$ adjunction, $(f^*, f_*)$ and $(f_*, f^!)$ 
are adjoint pairs. As $f_*$ preserves acyclic modules, 
$f^*$ preserves h-projective modules and $f^!$ preserves h-injective\index{h-injective} modules. 

\section{Bimodule approximation}
\label{section-bimodule-approximation}

In this section, we define and describe the basic properties of a lax $2$-functor
$\bimodapx$ which approximates \dg functors between \dg module categories 
by (the tensor functor\index{tensor functor}s defined by) \dg bimodules. On per-functor basis, 
this was already examined by Keller in \cite[Section~6.4]{Keller-DerivingDGCategories}. 
We will apply the bimodule approximation\index{bimodule approximation} to the first step in 
our construction of a categorical Fock space for our Heisenberg \dg bicategory 
$\hcat\basecat$ (see Section~\ref{sec:magic-wand}). At this first step, a representation of 
a simpler strict \dg $2$-category $\hcat*\basecat$ is constructed 
with (non-derived) \dg functors. The bimodule approximation turns these 
into \dg bimodules which are then considered as enhanced exact functors, 
see Section~\ref{section-enhanced-categories}. 

We first look at bimodule approximation on the $1$-categorical level. 
\begin{Definition}
	Let $\A$ and $\B$ be \dg categories. The \emph{bimodule 
		approximation}\index{bimodule approximation} functor is  
	\begin{align*}
		\bimodapx\colon \DGFun(\modA,\, \modB) & \rightarrow \AmodB, \\
		F & \mapsto F(\A), 
	\end{align*}
	where $F(\A) \in \AmodB$ is the evaluation of $F$ at the diagonal
	bimodule $\A$. In other words,  $\leftidx{_a}{F(\A)}{_b} =
	F(\leftidx{_a}{\A})_b$ for all $a \in \A, b \in \B$. 
\end{Definition}

The bimodule approximation functor $\bimodapx$ is right adjoint to the \emph{`tensor functor'} functor:
\begin{align*}
	\tensorfn\colon \AmodB & \rightarrow \DGFun(\modA,\, \modB),\\
	M & \mapsto (-) \otimes_\A M. 
\end{align*}
The adjunction unit $\eta\colon \Id \rightarrow \bimodapx \circ \tensorfn$
is given by the natural isomorphism 
\[
M \xrightarrow{\sim} \A \otimes_\A M \quad \quad \quad \text{for } M \in \AmodB,
\]
and thus $\tensorfn$ is a fully faithful embedding. 

The adjunction counit $\epsilon\colon \tensorfn \circ \bimodapx  \rightarrow
\Id$ is given by the natural transformation 
\begin{equation}
	\label{eqn-tensor-approx-adjunction-counit}
	(-) \otimes_\A F(\A) \rightarrow F \quad \quad \quad \text{for } F \in \DGFun(\modA,\, \modB),
\end{equation}
defined by the map 
\begin{equation}
	\label{eqn-tensor-approx-adjunction-counit-objectwise}
	E \otimes_\A F(\A) \rightarrow F(E), \quad \quad \quad \text{for } E \in \modA 
\end{equation}
which is adjoint to the composition 
\[
E \xrightarrow{\sim} \homm_\A(\A,E) \xrightarrow{F} \homm_\B(F(\A),\, F(E)).
\]
The map \eqref{eqn-tensor-approx-adjunction-counit-objectwise} is an
isomorphism for representable $E$, and thus a homotopy
equivalence for $E \in \hperfA$. This implies, as noted in 
\cite[Section~6.4]{Keller-DerivingDGCategories}, that 
\eqref{eqn-tensor-approx-adjunction-counit} yields an isomorphism of 
derived functor\index{derived functor}s $\catDc(\A) \rightarrow \catD(\B)$ and hence, for $F$
continuous, of functors $\catD(\A) \rightarrow \catD(\B)$. 

We now consider two \dg bicategories whose $1$-morphisms are \dg functors and \dg bimodules, respectively. The objects of these bicategories are the same:  morally, they are the categories of \dg modules over  small \dg categories. For brevity, however, we define these objects to be the small \dg categories themselves:

\begin{Definition}\label{defn-dgModCat}
	Define $\DGModCat$ to be the strict \dg $2$-category whose objects are
	small \dg categories and whose $1$-morphism categories $\DGModCat(\A,\,\B)$ are 
	the \dg categories $\DGFun(\modA,\, \modB)$ of \dg functors 
	between their \dg module categories. 
\end{Definition}

\begin{Definition}\label{defn-dgBiMod}
	Define $\DGBiMod$ to be the following \dg bicategory:
	\begin{enumerate}
		\item Its \emph{objects} are small \dg categories. 
		\item $\forall\; \A,\, \B \in \obj$, the \dg category of \emph{$1$-morphisms} from $\A$ to $\B$ is $\AmodB$. 
		\item $\forall\; \A,\, \B,\, \C \in \obj$ the \emph{$1$-composition functor} is the tensor product of bimodules:
		\begin{align*}
			\BmodC \otimes \AmodB  & \rightarrow \AmodC \\
			(N,M) & \mapsto M \otimes_\B N. 
		\end{align*}
		\item $\forall\; \A \in \obj$ the \emph{identity $1$-morphism} of $\AmodA$ is the diagonal bimodule\index{diagonal bimodule} $\A$. 
		\item The \emph{associator}\index{associator} isomorphisms are the canonical isomorphisms
		$$ (M \otimes_\B N) \otimes_\C L \xrightarrow{\sim} M \otimes_\B (N \otimes_\C L).  $$
		\item $\forall\; M \in \AmodB$ the \emph{left and right unitor}\index{unitor} isomorphisms are the natural maps
		\begin{equation*}
			\A \otimes_\A M \xrightarrow{\sim} M
			\quad \quad \text{ and } \quad \quad 
			M \otimes_\B \B \xrightarrow{\sim} M.
		\end{equation*}
	\end{enumerate} 
\end{Definition}

The $1$-categorical functors $\tensorfn$ package up into an obvious strong $2$-functor.
\begin{Definition}
	Define the strong $2$-functor 
	\[ \tensorfn\colon \DGBiMod \rightarrow \DGModCat, \]
	\begin{enumerate}
		\item On \em objects\rm, $\tensorfn$ is the identity map, 
		\item On \em $1$-morphism categories\rm, $\tensorfn$ is the $1$-categorical functor $\tensorfn$ defined above. 
		\item For any small \dg category $\A$, the \em unit coherence \rm $2$-morphism 
		\[ 1_{\tensorfn \A} \rightarrow \tensorfn(1_\A), \]
		is the natural isomorphism:
		\[ \id_\modA \xrightarrow{\sim} (-) \otimes_\A \A. \]
		\item For any small \dg categories $\A$, $\B$, $\C$, and any 
		$M \in \AmodB$ and $N \in \BmodC$, the \em composition coherence \rm $2$-morphism  
		\[ \tensorfn(N) \circ_1 \tensorfn(M) \rightarrow \tensorfn(N \circ_1 M), \]
		is the natural transformation defined by canonical isomorphisms 
		$$ (- \otimes_\A M) \otimes_\B N \xrightarrow{\sim} (-) \otimes_\A (M
		\otimes_\B N). $$
	\end{enumerate}
\end{Definition}

Since the $2$-functor $\tensorfn$ is strong, it induces a natural
structure of a (lax) $2$-functor on the right adjoints of 
its $1$-categorical components.
\begin{Definition}
	Define the lax $2$-functor 
	\[ \bimodapx\colon \DGModCat \rightarrow \DGBiMod, \]
	as follows:
	\begin{enumerate}
		\item On \em objects\rm, $\bimodapx$ is the identity map, 
		\item On \em $1$-morphism categories\rm, $\bimodapx$ is the $1$-categorical functor $\bimodapx$ defined above. 
		\item For any small \dg category $\A$, the \em unit coherence \rm $2$-morphism 
		\[ 1_{\bimodapx \A} \rightarrow \bimodapx(1_\A), \]
		is the identity morphism of the diagonal bimodule $\A$. 
		\item For any small \dg categories $\A$, $\B$, $\C$, and any 
		\[F \in \DGFun(\modA,\, \modB)\] and \[G \in \DGFun(\modB,\modC),\] 
		the \em composition coherence \rm $2$-morphism  
		\[ \bimodapx(G) \circ_1 \bimodapx(F) \rightarrow \bimodapx(G \circ_1 F), \]
		is given by the adjunction counit of $(\tensorfn,\, \bimodapx)$ for $G$:
		\begin{equation}
			\label{eqn-bimodule-approximation-composition-coherence} 
			F(\A) \otimes_B G(\B) \xrightarrow{\epsilon_G} GF(\A). 
		\end{equation}
	\end{enumerate}
\end{Definition}

In general, the $2$-functor $\bimodapx$ is not even homotopy strong.
However, its unit coherence morphism is the identity map, while
below we show that under certain assumptions on the \dg functors 
$F$ and $G$ their composition coherence morphism  
\eqref{eqn-bimodule-approximation-composition-coherence}
is a fibrewise homotopy equivalence or quasi-isomorphism. This is
important for us, because then the composition of $\bimodapx$
with one of the homotopy quotients of $\DGBiMod$ by acyclics  
discussed in Section~\ref{section-enhanced-categories} below becomes 
homotopy strong when restricted to such \dg functors.

\begin{Proposition}
	\label{prop-bimod-approximation-is-quasi-iso-lax-on-tensor-and-hom}
	Let $\A$, $\B$, $\C$ be small \dg categories, and let 
	\[F \in \DGFun(\modA,\, \modB)\quad \textrm{and}\quad G \in \DGFun(\modB,\, \modC).\] Then 
	\begin{enumerate}
		\item \label{item-bimod-approx-cmps-coh-if-G-tensor-iso}
		If $G \simeq (-) \otimes_\B M$ for some $M \in \BmodC$, then 
		\eqref{eqn-bimodule-approximation-composition-coherence}
		is an isomorphism. 
		\item \label{item-bimod-approx-cmps-coh-if-G-hom-right-hperf-hmtpy-equiv}
		If $G \simeq \homm_\B(N, -)$ for some $N \in \CmodB$ which is 
		$B$-perfect and $\B$-h-projective, then 
		\eqref{eqn-bimodule-approximation-composition-coherence} is 
		fibrewise a homotopy equivalence in $\modA$. 
		\item \label{item-bimod-approx-cmps-coh-if-F-hperf-to-hperf-hmtpy-equiv}
		If $F$ restricts to a functor $\hperfA \rightarrow \hperfB$,
		then \eqref{eqn-bimodule-approximation-composition-coherence} is 
		fibrewise a homotopy equivalence in $\modC$. 
	\end{enumerate}
\end{Proposition}

\begin{proof}
	Assertion \ref{item-bimod-approx-cmps-coh-if-G-tensor-iso} is clear. 
	
	\ref{item-bimod-approx-cmps-coh-if-G-hom-right-hperf-hmtpy-equiv}: 
	If $G \simeq \homm_B(N,-)$, then the morphism
	\eqref{eqn-bimodule-approximation-composition-coherence} 
	is the evaluation map 
	\[ F(\A) \otimes_B \homm_B(N,\,B) \xrightarrow{\text{eval}} \homm_B(N,\,F(\A)). \]
	Since the fibres of $\N$ over $\C$ are perfect and h-projective
	$\B$-modules, the fibres of this map over $\C$ are homotopy
	equivalences in $\modA$. 
	
	\ref{item-bimod-approx-cmps-coh-if-F-hperf-to-hperf-hmtpy-equiv}:
	Morphism \eqref{eqn-bimodule-approximation-composition-coherence} 
	is the $(\tensorfn,\bimodapx)$-adjunction counit 
	for $G$ applied to $F(\A)$. 
	Thus the fibres of \eqref{eqn-bimodule-approximation-composition-coherence} 
	in $\modC$ are given by applying the natural transformation 
	\[
	\id \otimes_\B G(\B) 
	\xrightarrow{\eqref{eqn-tensor-approx-adjunction-counit-objectwise}}
	G,
	\]
	to the fibres of $F(\A)$ in $\modB$. By assumption these fibres lie in
	$\hperfB$. Hence \eqref{eqn-tensor-approx-adjunction-counit-objectwise}
	is a homotopy equivalence. 
	%
	%
\end{proof}

\section{\Dg enhanced triangulated categories}
\label{section-enhanced-categories}

\Dg enhancements were introduced by Bondal and Kapranov 
in \cite{BondalKapranov-EnhancedTriangulatedCategories}. A \emph{\dg enhancement}\index{DG enhancement} 
of a triangulated category $\cat T$ is a pretriangulated \dg category $\A$
together with an exact equivalence $\Hzero(\A)\simeq \cat T$. 
These are considered up to quasi-equivalences and are naturally objects in $\HoDGCatone$, 
the localisation of $\DGCatone$ by quasi-equivalences
\cite{Toen-TheHomotopyTheoryOfDGCategoriesAndDerivedMoritaTheory}. We write $\EnhCatone$ for the full subcategory of $\HoDGCatone$ comprising pretriangulated \dg categories
and consider this to be the $1$-category of enhanced triangulated
categories\index{enhanced triangulated category}. 

A \emph{Morita \dg enhancement}\index{Morita enhancement} of a triangulated category $\T$
is a small \dg category $\A$ whose compact derived category $\catDc(\A)$ 
is equivalent to $\T$. These are considered up to \emph{Morita equivalences:}\index{Morita equivalence}
functors $\phi\colon \A \rightarrow \B$ such that  
$\phi^*: \catD(\A) \rightarrow \catD(\B)$ restricts to an equivalence $\catDc(\A)
\rightarrow \catDc(\B)$.  They are thus naturally the objects of $\MoDGCatone$, 
the localisation of $\DGCatone$ by Morita equivalences \cite{Tabuada-InvariantsAdditifsDeDGCategories}.  

Let $\A$ be a \dg category.  The Yoneda embedding $\A \hookrightarrow
\hperf(\A)$ is a Morita equivalence. Moreover, it identifies
$\MoDGCatone$ with the full subcategory of $\HoDGCatone$ consisting of
pretriangulated categories whose homotopy categories are
Karoubi-complete. Thus working in the Morita setting means working
with small Karoubi-complete triangulated categories, 
such as bounded derived categories of abelian
categories. If $\A$ is an enhancement of a Karoubi-complete
triangulated category $\T$, then it is also its Morita enhancement.
Conversely, if $\A$ is a Morita enhancement of a
triangulated category $\T$, then $\T$ is Karoubi-complete 
and $\hperfA$ is an ordinary enhancement of $\T$. 

The advantage of Morita enhancements is that
morphisms in $\MoDGCatone$ admit a nice description. 
The morphisms from $\A$ to $\B$ in $\MoDGCatone$ 
are in bijection with the isomorphism classes 
in $\catD(\AbimB)$ of $\B$-perfect $\AbimB$ bimodules 
\cite[Theorems 4.2, 7.2]{Toen-TheHomotopyTheoryOfDGCategoriesAndDerivedMoritaTheory}. 
We define an \em enhanced exact functor\index{enhanced exact functor} \rm $\A \rightarrow \B$ to 
be a $B$-perfect bimodule $M \in \catD(\AbimB)$. The underlying exact
functor between the underlying triangulated categories is 
$(-) \otimes^{\lder} M\colon \catDc(\A) \rightarrow \catDc(\B)$. 
An \em enhanced natural transformation \rm is a morphism in 
$\catD(\AbimB)$ between $\B$-perfect bimodules. 

The $1$-category $\MoDGCatone$ is thus refined to the following
strict $2$-category of Morita enhanced triangulated categories.
\begin{Definition}\label{defn-EnhCatKC}
	Define the strict $2$-category $\EnhCatKC$, with $\kc$ referring to Karoubi-complete, also denoted by $\MoDGCat$, to consist of the following data:
	\begin{enumerate}
		\item Its set of \emph{objects} is the set of all small \dg categories. 
		\item For any two $\A, \B \in \obj \EnhCatKC$, the category
		$\EnhCatKC(\A,\B)$ of \emph{$1$-morphisms}\index{$1$-morphism} from $\A$ to $\B$ is 
		the skeleton of $D_{\Bperf}(\AbimB)$. 
		\item For any triple $\A,\B,\C \in \obj \EnhCatKC$ the \emph{$1$-composition}\index{$1$-composition}
		functor is given by the derived tensor product of bimodules:
		\begin{align*}
			\EnhCatKC(\B,\C) \times \EnhCatKC(\A,\B)  & \rightarrow \EnhCatKC(\A,\C) \\
			(M,N) & \mapsto M \ldertimes_\B N. 
		\end{align*}
		\item For any $\A \in \obj \EnhCatKC$ the \emph{identity $1$-morphism} of
		$\EnhCatKC(\A,\A)$ is the diagonal bimodule $\A$. 
	\end{enumerate}
\end{Definition}

We have the $2$-functor $\EnhCatKC \rightarrow \Cat$ which
sends each Morita enhancement $\A$ to its underlying triangulated
category $\catDc(\A)$, each enhanced functor $M \in D_{\Bperf}(\AbimB)$ 
to its underlying exact functor $(-) \otimes^{\lder} M$, and each 
morphism in $D_{\Bperf}(\AbimB)$ to the induced natural 
transformation of these underlying exact functors. 

The $2$-category $\EnhCatKC$ can be identified, via the assignment $\A \mapsto \hperfA$
with the $2$-full subcategory of Karoubi-complete categories in  
the strict $2$-category $\EnhCat$ of enhanced triangulated categories 
defined in \cite[Section~1]{lunts2010uniqueness}. 
The strict $2$-category $\EnhCat$ 
is a $2$-categorical refinement of the $1$-category $\EnhCatone$. It coincides with 
the homotopy category of the $(\infty,2)$-category of 
\dg categories in \cite{AStudyInDAG1}. 

We next introduce a \dg enhancement $\EnhCatKCdg$ of $\EnhCatKC$: 
\begin{Definition}
	A \dg enhancement of a strict $2$-category $\bicat A$ is
	a \dg bicategory $\bicat C$ such that $\bicat A$ is $2$-equivalent to
	the strictification $\tilde{H}^0(\bicat C)$ of the bicategory
	$\Hzero(\bicat C)$ obtained by taking skeletons of its $1$-morphism 
	categories.  
\end{Definition}

We define the \dg bicategory $\EnhCatKCdg$ in terms of
the bar categories of modules and bimodules introduced in 
\cite{AnnoLogvinenko-BarCategoryOfModulesAndHomotopyAdjunctionForTensorFunctors}. 
Given small \dg categories $\A$ and $\B$, the bar-category\index{bar category} $\AmodbarB$
of $\A$-$\B$-bimodules is isomorphic to the \dg category of \dg
$\A$-$\B$-bimodules with $\Ainfty$-morphisms between them
\cite[Prop.~3.5]{AnnoLogvinenko-BarCategoryOfModulesAndHomotopyAdjunctionForTensorFunctors}. 
However, the bar-category has a simpler definition which avoids the
complexities of the fully general $\Ainfty$-machinery. 

We have $\Hzero(\AmodbarB) \simeq \catD(\AbimB)$, since 
all quasi-isomorphisms in $\AmodbarB$ are homotopy equivalences. 
The bar-category $\AmodbarB$ can be viewed as a more natural 
way to factor out the acyclic modules than taking the Drinfeld quotient:
one does not introduce formal contracting homotopies which do not
interact with the old morphisms, and thus retains
the natural monoidal structure in the form of the bar-tensor product
$\bartimes$ of bimodules. It corresponds to the $\Ainfty$-tensor
product of $\Ainfty$-bimodules under the identification of the
bar-category with the category of \dg bimodules with
$\Ainfty$-morphisms, see
\cite[Section~3.2]{AnnoLogvinenko-BarCategoryOfModulesAndHomotopyAdjunctionForTensorFunctors}.

\begin{Definition}\label{defn-EnhCatKCdg}
	Define the homotopy unital \dg bicategory $\EnhCatKCdg$ as follows: 
	\begin{enumerate}
		\item Its set of \em objects \rm is the set of small \dg categories. 
		\item For any two $\A, \B \in \obj \EnhCatKCdg$, the category of \em
		$1$-morphisms \rm from $\A$ to $\B$ is the full subcategory of $\AmodbarB$
		comprising $\B$-perfect bimodules. 
		\item For any triple $\A,\B,\C \in \obj \EnhCatKCdg$ 
		the $1$-composition functor is given by the bar tensor product of bimodules:
		\begin{align*}
			\BmodbarC \otimes \AmodbarB  & \rightarrow \AmodbarC \\
			(N, M) & \mapsto M \bartimes_\B N. 
		\end{align*}
		\item For any $\A \in \obj \EnhCatKCdg$ the \em identity $1$-morphism \rm of
		$\AmodbarA$ is the diagonal bimodule $\A$. 
		\item The \em associator \rm isomorphisms are the natural isomorphisms 
		$$ (M \bartimes_\B N) \bartimes_\C L \xrightarrow{\sim} 
		M \bartimes_\B (N \bartimes_\C L). $$
		\item The \em left and right unitor \rm morphisms are given 
		for any $1$-morphism $M \in \AmodbarB$ by the natural homotopy equivalences defined in
		\cite[Section~3.3]{AnnoLogvinenko-BarCategoryOfModulesAndHomotopyAdjunctionForTensorFunctors}:
		\begin{equation*}
			\A \bartimes_\A M \xrightarrow{\alpha_\A} M
			\quad \quad \text{ and } \quad \quad 
			M \bartimes_\B \B \xrightarrow{\alpha_\B} M. 
		\end{equation*}
	\end{enumerate}
\end{Definition}

Note that the \dg bicategory $\EnhCatKCdg$
is \em homotopy unital\rm: its unitor morphisms are not
isomorphisms, but only homotopy equivalences. On the homotopy level, such bicategories become genuine bicategories.
Indeed, the strictified homotopy bicategory $\tilde{H}^0(\EnhCatKCdg)$ is
$2$-isomorphic to $\EnhCatKC$. This is because
$\Hzero(\AmodbarB) \simeq \catD(\AbimB)$ and $\Hzero(\bartimes) \simeq \otimes^{\lder}$, see 
\cite[Section~3.2]{AnnoLogvinenko-BarCategoryOfModulesAndHomotopyAdjunctionForTensorFunctors}.

The homotopy unitality of $\EnhCatKCdg$ does not interfere with our
constructions. Its unitor morphisms have 
homotopy inverses which are genuine right inverses, see 
\cite[Section~3.3]{AnnoLogvinenko-BarCategoryOfModulesAndHomotopyAdjunctionForTensorFunctors}. 

We offer the following alternative construction of $\EnhCatKCdg$ where we use the monoidal Drinfeld quotient\index{monoidal Drinfeld quotient} instead of bar-categories to kill the acyclic bimodules in $\DGBiMod$. The original Drinfeld quotient \cite{Drinfeld-DGQuotientsOfDGCategories} is not compatible with monoidal structures such as $1$-composition in a bicategory. A construction by Shoikhet \cite{Shoikhet-DifferentialGradedCategoriesAndDeligneConjecture}
fixes this, and in Section~\ref{subsec:monoidal-drinfeld-quotients} we define the monoidal Drinfeld quotient of a \dg bicategory.  The price is the $1$-composition no longer being a  \dg functor but a quasifunctor\index{quasifunctor}, that is,  
a $1$-morphism in $\HoDGCat$. Thus, in 
this alternative construction $\EnhCatKCdg$ is only a $\HoDGCat$-enriched bicategory. 

\begin{Definition}[Alternative construction of $\EnhCatKCdg$]
	\label{defn-alternative-enhcatkcdg}
	Let $\DGBiMod_{\lfrp}$ denote the $2$-full subcategory of $\DGBiMod$ comprising all objects and the $1$-morphisms given by left-h-flat and right-perfect bimodules. The $\HoDGCat$-enriched bicategory $\EnhCatKCdg$ is the Drinfeld quotient 
	of $\DGBiMod_{\lfrp}$ by its two-sided ideal of $1$-morphisms given by acyclic bimodules. 
\end{Definition}

For this paper, it does not matter which of the two constructions one
uses. We use $\EnhCatKCdg$ as the target for a $2$-represention of our
Heisenberg $\HoDGCat$-enriched bicategory $\hcat\basecat$.  First, we
construct a $2$-functor from a simpler strict \dg $2$-category
$\hcat*\basecat$ to $\DGBiMod_{\lfrp}$, which is naturally a
subcategory of both above versions of $\EnhCatKCdg$. The two
constructions should be viewed merely as two different ways to kill
the acyclics in $\DGBiMod_{\lfrp}$.  Thus we obtain the (same)
$2$-functor $\hcat*\basecat \rightarrow \EnhCatKCdg$ whichever version
of the latter we use.  This $2$-functor is turned into the
desired $2$-representation\index{$2$-representation} of $\hcat\basecat$ via a formal
construction for which it is only important that acyclics are
null-homotopic in $\EnhCatKCdg$. 

\section{The perfect hull\index{perfect hull} of a \dg bicategory}\label{subsec:hperf}

In this section, describe the formalism of taking the perfect hull of a $\DGCat$- or $\HoDGCat$-enriched 
bicategory. On the homotopy level, this corresponds to taking a Karoubi-completed 
triangulated hull of each $1$-morphism category. 

Let $\A$ and $\B$ be \dg categories. We have a natural functor
\begin{equation}
	\label{eqn-modA-otimes-modB-to-modAotimesB}
	\modA \otimes \modB \rightarrow \rightmod{(\A \otimes \B)}
\end{equation}
which is defined as the composition 
\begin{equation*}
	\begin{tikzcd}
		\DGFun(\Aopp,\, \modk) \otimes \DGFun(\Bopp,\, \modk) 
		\ar{d}
		\\
		\DGFun(\Aopp \otimes \Bopp,\, \modk \otimes \modk)
		\ar{d}
		\\
		\DGFun(\Aopp \otimes \Bopp,\, \modk)
	\end{tikzcd}
\end{equation*}
whose first map is due to functoriality of the tensor product of
\dg categories, and whose second map is due to the natural monoidal
structure on $\modk$ given by the tensor product over $\kk$. Explicitly, 
given $E \in \modA$ and $F \in \modB$ the functor  
\eqref{eqn-modA-otimes-modB-to-modAotimesB} maps $E \otimes F$
to an $\A \otimes \B$-module whose fibre over $(a,b) \in \A \otimes
\B$ is the tensor product $E_a \otimes F_b$. 

Let $\C$ be a \dg category and let $\mu\colon \A \otimes \B \rightarrow
\C$ be a \dg functor. It extends naturally to 
\begin{equation*}
	\mu\colon \modA \otimes \modB \rightarrow \modC
\end{equation*}
which is defined as the composition
\[
\modA \otimes \modB
\xrightarrow{\eqref{eqn-modA-otimes-modB-to-modAotimesB}}
\rightmod{(\A \otimes \B)}
\xrightarrow{\mu^*}
\modC. 
\]
Explicitly, given $E \in \modA$ and $F \in \modB$ we have for all $c \in \C$
\[
\mu(E \otimes F)_{c} = \bigoplus_{a \in \A, \, b \in \B} (E_a \otimes
F_b) \otimes \homm_\C\bigl(c,\, \mu(a \otimes b)\bigr) / \text{relations},
\] 
where the relations identify 
the actions of $\A \otimes \B$ on $E_a \otimes \F_b$ and on 
$\mu(a \otimes b)$. 

The above generalises to the following.

\begin{Definition}
	\label{defn-extending-natural-transformations-to-module-categories} 
	Let $\A_1, \dots, \A_n$, $\C$ be \dg categories. 
	\begin{enumerate}
		\item Define the functor
		\begin{equation*}
			\varpi\colon \rightmod{\A_1} \otimes \dots \otimes \rightmod{\A_n}
			\rightarrow 
			\rightmod{(\A_1 \otimes \dots \otimes \A_n)}
		\end{equation*}
		to be the composition 
		\begin{align*}
			\bigotimes_{i=1}^n \DGFun\bigl(\A_i^{\opp},\, \modk\bigr) & \rightarrow 
			\DGFun\Bigl(\bigotimes_{i=1}^n \A_i^{\opp},\, \bigotimes_{i=1}^n \modk\Bigl) \\
			& \rightarrow 
			\DGFun\Bigl(\bigotimes_{i=1}^n \A_i^{\opp},\, \modk\Bigr),
		\end{align*}
		whose first map is due to the functoriality of tensor product 
		of \dg categories and whose second map is due to the natural monoidal structure  
		on $\modk$. 
		\item Define the functor 
		\begin{equation*}
			\Upsilon\colon
			\DGFun\bigl(\A_1 \otimes \dots \otimes \A_n,\, \C\bigr) 
			\rightarrow 
			\DGFun\bigl(\rightmod{\A_1} \otimes \dots \otimes \rightmod{\A_n},\, \rightmod{\C}\bigr)
		\end{equation*}
		to be the composition of the extension 
		of scalars functor 
		\[
		(-)^*\colon\DGFun\bigl(\A_1 \otimes \dots \otimes \A_n,\, \C\bigr)
		\rightarrow 
		\DGFun\bigl(\rightmod{(\A_1 \otimes \dots \otimes \A_n)},\, \modC\bigr)
		\]
		with the functor of precomposition with $\varpi$.
	\end{enumerate}
\end{Definition}

\begin{Lemma}
	\label{lem:properties-of-upsilon}
	For any \dg categories $\A_1, \dots, \A_n$, $\C$ we have:
	\begin{enumerate}
		\item \label{item-upsilon-commutes-with-yoneda}
		The functor $\Upsilon$ commutes with Yoneda embeddings, i.e.~the following diagram commutes for any 
		$\mu \in \DGFun(\A_1 \otimes \dots \otimes \A_n,\, \C)$:
		\begin{equation*}
			\begin{tikzcd}
				\A_1 \otimes \dots \otimes \A_n 
				\ar{r}{\mu} \ar[hook]{d}
				&
				\C
				\ar[hook]{d}
				\\
				\rightmod{\A_1} \otimes \dots \otimes \rightmod{\A_n}
				\ar{r}{\Upsilon(\mu)}
				&
				\rightmod{\C}.
			\end{tikzcd}
		\end{equation*}
		\item When $n = 1$, for any $\mu_1\colon \A_1 \rightarrow C$ we have $\Upsilon(\mu_1) = \mu_1^*$.  
		\item $\Upsilon(\id) = \varpi$.  
		\item Let $\C_1$, \dots, $\C_n$ be \dg categories
		and $\mu_1, \dots, \mu_n$ be \dg functors \[\mu_i\colon
		\A_i \rightarrow \C_i.\] Then 
		\begin{equation*}
			(\mu_1 \otimes \cdots \otimes \mu_n)^* \circ \varpi 
			\;\simeq\;
			\varpi \circ (\mu_1^* \otimes \cdots \otimes \mu_n^*). 
		\end{equation*}
		\item \label{item-upsilon-commutes-with-composition}
		Let $\mu \in \DGFun(\A_1 \otimes \cdots \otimes \A_n,\, \C)$. Let
		$m_1, \dots, m_n \in \mathbb{Z}$, let \[\left\{ \B_{ij}
		\right\}_{1 \leq i \leq n, 1 \leq j \leq m_i}\] be \dg categories, and $\left\{ \lambda_i \right\}$ be \dg functors
		\[\lambda_i \colon \B_{i1} \otimes \cdots \otimes \B_{i m_i} \rightarrow \A_i.\] Then 
		\begin{equation*}
			\Upsilon\bigl( \mu \circ (\lambda_1 \otimes \cdots \otimes \lambda_n) \bigr) 
			\simeq \Upsilon(\mu) \circ \bigl( \Upsilon(\lambda_1) \otimes \cdots
			\otimes \Upsilon(\lambda_n)\bigr).
		\end{equation*} 
	\end{enumerate}
\end{Lemma}

\begin{proof}
	This is a straightforward verification. 
	
	For example, to establish \ref{item-upsilon-commutes-with-yoneda},
	let $\mu$ be a functor $\A_1 \otimes \dots \otimes \A_n \rightarrow
	\C$. Then for any $a_1 \in \A_1, \dots, a_n \in \A_n$ we have
	\[
	\Upsilon(\mu)(h^r(a_1) \otimes \dots \otimes  h^r(a_n)) = 
	\mu^* (h^r(a_1 \otimes \cdots \otimes a_n)) = 
	h^r\left(\mu(a_1 \otimes \cdots \otimes a_n)\right).
	\qedhere
	\]
\end{proof}

\begin{Lemma}
	\label{eqn-perfect-hull-of-a-DG-bifunctor}
	For any \dg categories $\A_1, \dots,\, \A_n$, $\C$ the functor
	\begin{equation*}
		\Upsilon\colon
		\DGFun(\A_1 \otimes \dots \otimes \A_n,\, \C) 
		\rightarrow 
		\DGFun(\rightmod{\A_1} \otimes \dots \otimes \rightmod{\A_n},\, \rightmod{\C}). 
	\end{equation*}
	restricts to a functor 
	\begin{equation*}
		\Upsilon\colon
		\DGFun(\A_1 \otimes \dots \otimes \A_n,\, \C) 
		\rightarrow 
		\DGFun(\hperf \A_1 \otimes \dots \otimes \hperf \A_n,\, \hperf \C). 
	\end{equation*}
\end{Lemma}

\begin{proof}
	For any $\mu\colon \A_1 \otimes \dots \otimes \A_n \rightarrow C$
	the functor $\Upsilon(\mu)\colon \rightmod{\A_1} \otimes \dots \otimes \rightmod{\A_n} \rightarrow \rightmod{\C}$ 
	takes tensor products of representables to representables, and therefore 
	tensor products of h-projective, perfect modules to h-projective perfect
	modules.  
\end{proof}

We have the following key result.
\begin{Proposition}
	\label{prps-the-perfect-hull-of-a-dg-bicategory}
	Let $\bicat C$ be a \dg (resp.~$\HoDGCat$-enriched) bicategory. The following set of data defines a \dg (resp.~$\HoDGCat$-enriched) bicategory $\tilde{\bicat C}$:
	\begin{itemize}
		\item $\obj \tilde{\bicat C} \coloneqq \obj \bicat C$. 
		\item For each $a,b \in \obj \tilde{\bicat C}$, 
		\[ \tilde{\bicat C}(a,b) \coloneqq \hperf\bicat C(a,b).\] 
		\item For each $a \in \tilde{\bicat C}$,
		\[ \tilde{1}_a \coloneqq h^r(1_a). \]
		That is, it is the representable module defined by the
		identity $1$-morphism of $a$ in $\bicat C$. 
		\item For each $a,b,c \in \obj \tilde{\bicat C}$ the $1$-composition functor
		\[ \tilde{\mu}\colon \hperf \bicat C(b,c) \otimes \hperf\bicat C(a,b) \rightarrow 
		\hperf\bicat C(a,c) \]
		is the extension $\Upsilon(\mu)$ given in 
		{ Lemma}~\ref{eqn-perfect-hull-of-a-DG-bifunctor} 
		of the $1$-composition functor of $\bicat C$ 
		\[ \mu\colon C(b,c) \otimes \bicat C(a,b) \rightarrow \bicat C(a,c). \]
		\item For each $a,b,c,d \in \obj \tilde{\bicat C}$ the natural 
		associator isomorphism 
		\begin{equation*}
			\tilde{\alpha}\colon \tilde{\mu}(\tilde{\mu} \otimes \id) \simeq
			\tilde{\mu}(\id \otimes \tilde{\mu})
		\end{equation*}
		of functors 
		\[ \hperf \bicat C(c,d) \otimes \hperf \bicat C(b,c) \otimes \hperf \bicat C(a,b) \rightarrow \hperf \bicat C(a,d) \] 
		is the conjugate of the extension $\Upsilon(\alpha)$ of the associator isomorphism $\alpha$ of $\bicat C$ by the isomorphism of Lemma~\ref{lem:properties-of-upsilon}~\ref{item-upsilon-commutes-with-composition}:
		\begin{equation*}
			\begin{tikzcd}
				\Upsilon(\mu(\mu \otimes \id))
				\ar{r}{\Upsilon(\alpha)}
				&
				\Upsilon(\mu(\id \otimes \mu))
				\ar{d}{\simeq}
				\\
				\Upsilon(\mu)\left(\Upsilon(\mu) \otimes \id\right)
				\ar{u}{\simeq}
				\arrow[dashed]{r}{\tilde{\alpha}}
				&
				\Upsilon(\mu)\left(\id\otimes \Upsilon(\mu) \right).
			\end{tikzcd}
		\end{equation*}
		\item Similarly, for each $a,b \in \obj \tilde{\bicat C}$ the unitor
		isomorphism $\tilde{\lambda}$ (resp.~$\tilde{\rho}$) is the conjugate of the extension
		$\Upsilon(\lambda)$ (reps., $\Upsilon(\rho)$) of the corresponding 
		unitor isomorphism $\lambda$ (resp., $\rho$) of $\bicat C$ by the isomorphism of Lemma~\ref{lem:properties-of-upsilon}~\ref{item-upsilon-commutes-with-composition}. 
	\end{itemize}
\end{Proposition}

\begin{proof}
	This is a straightforward verification. For example, to show that the diagram
	of Definition~\ref{def:enrichedbicategory}~\ref{item-associativity-coherence-diagram} commutes for $\tilde{\bicat C}$
	we write, according to the definition, each instance of $\tilde{\alpha}$ 
	in this diagram as a conjugate of $\Upsilon(\alpha)$ by the isomorphisms
	from Lemma~\ref{lem:properties-of-upsilon}~\ref{item-upsilon-commutes-with-composition}. The resulting
	diagram can then be simplified to the image under $\Upsilon$ of the 
	the same associativity coherence diagram 
	for $\bicat C$. 
	The claim then follows since the image of a commutative diagram
	under a functor is itself a commutative diagram. 
\end{proof}

\begin{Definition}
	\label{defn-the-perfect-hull-of-a-bicategory}
	Let $\bicat C$ be a \dg or $\HoDGCat$-enriched bicategory. The \emph{perfect hull}\index{perfect hull} of $\bicat C$, denoted 
	$\bihperf(\bicat C)$, is the bicategory defined in Proposition 
	\ref{prps-the-perfect-hull-of-a-dg-bicategory}. 
\end{Definition}

\begin{Remark}\label{rem:perfect-hull-is-bicategory}
	Even when  $\bicat C$ is a strict $2$-category, its perfect hull $\bihperf(\bicat C)$ is in general only a bicategory.
	Indeed, since $\Upsilon(\id)$ is only isomorphic to the identity (being given by an extension of scalars), the unitor and associator isomorphisms of $\bihperf(\bicat C)$ will not be equal to the identity.
\end{Remark}

\begin{Proposition}
	\label{prps-the-perfect-hull-of-a-dg-functor-of-dg-bicategories}
	Let $\bicat C$ and $\bicat D$ be \dg or $\HoDGCat$-enriched bicategories and $F\colon \bicat C \rightarrow \bicat D$ a $2$-functor. Then the following set of data
	defines a $2$-functor
	\begin{equation*}
		\bihperf(F): \bihperf(\bicat C) \rightarrow \bihperf(\bicat D). 
	\end{equation*}
	\begin{itemize}
		\item The map \[F\colon \obj \bihperf(\bicat C) \rightarrow \obj\bihperf(\bicat D)\]
		which equals the map $F \colon \obj \bicat C \rightarrow \obj \bicat D$ as 
		taking the perfect hull of a bicategory does not change the objects. 
		\item For every $a,b \in \obj \bihperf(\bicat C)$ the functor
		\[ \bihperf(F)_{a,b}\colon \bihperf \bicat C(a,b) \rightarrow \bihperf \bicat D(Fa,Fb) \]
		is defined to be the extension of scalars functor $F_{a,b}^*$. 
		\item For every $a \in \obj \bihperf(\bicat C)$ the $2$-morphism 
		\[ \iota\colon 1_{Fa} \rightarrow \hperf F(1_a) \]
		is the image under the Yoneda embedding of the corresponding
		$2$-morphism $\iota_F$ for $F$. 
		\item For each $a,b,c, \in \obj \bihperf(\bicat C)$ a natural transformation 
		\[
		\phi\colon \mu_{\bihperf(\bicat D)} \circ (\bihperf(F)_{b,c} \otimes \bihperf(F)_{a,b})
		\rightarrow
		\bihperf(F)_{a,c} \circ \mu_{\bihperf(\bicat C)},
		\]
		which is the conjugate by the isomorphisms 
		from Lemma~\ref{lem:properties-of-upsilon}~\ref{item-upsilon-commutes-with-composition}
		of the extension $\Upsilon(\phi_{F})$ of the corresponding 
		natural transformation for $F$.  
	\end{itemize}
\end{Proposition}
\begin{proof}
	This is a straightforward verification.
\end{proof}

\begin{Remark} 
	\label{rem:pretrdg}
	By replacing the perfect hull $\hperf \bicat C(a,b)$ with the pretriangulated
	hull $\pretriag \bicat C(a,b)$ in
	Proposition~\ref{prps-the-perfect-hull-of-a-dg-bicategory} and
	Definition~\ref{defn-the-perfect-hull-of-a-bicategory}, one obtains
	the pretriangulated hull\index{pretriangulated hull} $\bipretriag(\bicat C)$ of a $\HoDGCat$-enriched 
	bicategory $\bicat C$.
\end{Remark}

\section{Monoidal Drinfeld quotient}\label{subsec:monoidal-drinfeld-quotients}

In this section we give a generalisation 
of the notion of the \emph{Drinfeld quotient} of a \dg category \cite{Drinfeld-DGQuotientsOfDGCategories}. The original notion is not
compatible with monoidal structures, which led Shoikhet to introduce
in \cite{Shoikhet-DifferentialGradedCategoriesAndDeligneConjecture}
the notion of a \emph{refined Drinfeld quotient} and use it to
construct the structure of a \emph{weak Leinster monoid}\index{weak Leinster monoid} on 
the Drinfeld quotient of a monoidal \dg category by a two-sided 
ideal of objects. 

Here, we use Shoikhet's construction to define the Drinfeld quotient
of a \dg bicategory by a two-sided ideal of $1$-morphisms. 
The result is a $\HoDGCat$-enriched bicategory. That is, the $1$-composition 
is no longer given by \dg functors, but by quasi-functors: compositions 
of genuine \dg functors with formal inverses of quasi-equivalences. 

We actually get a richer structure: $1$-morphism spaces in 
the quotient bicategory are not abstract objects of 
$\HoDGCat$, but concrete \dg categories. These admit a multi-object
analogue of a weak Leinster monoid structure. Localising by
quasi-equivalences simplifies it to an ordinary,
associative $1$-composition, whence we obtain a $\HoDGCat$-enriched
bicategory. 

Finally, our quotient construction works just as well with 
a bicategory that is already $\HoDGCat$-enriched and produces again a
$\HoDGCat$-enriched bicategory. 

Recall the original construction by Drinfeld:
\begin{Definition}[{\cite[Section~3.1]{Drinfeld-DGQuotientsOfDGCategories}}]
	\label{defn-original-drinfeld-quotient}
	Let $\C$ be a \dg category and $\A \subseteq \C$ a full
	\dg subcategory. The \emph{Drinfeld quotient}\index{Drinfeld quotient} $\C / \A$ is
	the \dg category freely generated over $\C$ by adding for each 
	$a \in \A$ a degree $-1$ contracting homotopy $\xi_a: a \rightarrow a$ 
	with $d \xi_a = \id_a$. 
	
	Explicitly: 
	\begin{enumerate}
		\item The objects of $\C / \A$ are those of $\C$. 
		\item $\forall\; c,d \in \C$ the morphism complex
		$\homm_{\C/\A}(c,d)$ comprises all composable words 
		\begin{equation*}
			f_{n} \xi_{a_n} f_{n-1} \cdots f_1 \xi_{a_1} f_0 
		\end{equation*}
		with $a_1, \cdots, a_n \in \A$
		and $f_0 \in \homm_\C(c, a_1)$, $f_{n} \in \homm_\C(a_{n}, d)$. 
		Composable here means that $f_i \in \homm_\C(a_{i}, a_{i+1})$ for $0 < i < n$. 
		The degree of 
		such a word
		is $(\sum \deg f_i) - n$. The 
		differential is given by the Leibniz rule and, when differentiating
		one of the $\xi_i$, the subsequent composition of $f_{i-1}$ and 
		$f_{i}$ in $\C$. 
		\item The composition in $\C / \A$ is given by the concatenation of words 
		and the subsequent composition in $\C$ of the two letters at which 
		the concatenation happens. 
		\item The identity morphisms in $\C/\A$ are the identity morphisms 
		of $\C$. 
	\end{enumerate}
\end{Definition}

We have the natural embedding $\C \rightarrow \C / \A$ which is the
identity on objects. On morphisms, it considers morphisms of $\C$ 
as length $1$ composable words; that is, $n = 0$ in the notation of
Definition~\ref{defn-original-drinfeld-quotient}~(2). We thus have 
an embedding $\catDc(\C) \rightarrow \catDc(\C/\A)$. It sends 
the objects of $\catDc(\A) \subset
\catDc(\C)$ to zero, and therefore by the universal property of the
Verdier quotient it filters through a unique functor 
$\catDc(\C)/\catDc(\A) \rightarrow \catDc(\C/\A)$.

The main properties of the Drinfeld quotient are summarised 
as follows: 

\begin{Theorem}[{\cite[Theorem 1.6.2]{Drinfeld-DGQuotientsOfDGCategories}, \cite[Theorem 4.0.3]{Tabuada-OnDrinfeldsDGQuotient}}]
	\label{theorem-main-properties-of-drinfeld-quotients}
	Let $\C$ be a \dg category and let $\A \subseteq \C$ be a full
	subcategory. Then: 
	\begin{enumerate}
		\item \label{item-drinfeld-quotient-descends-to-verdier-quotient} 
		The natural functor $\catDc(\C)/\catDc(\A) \rightarrow \catDc(\C/\A)$ is 
		an exact equivalence.  
		\item \label{item-universal-property-with-respect-to-morphisms-killing-A}
		Let $\B$ be a \dg category. 
		The natural functor $\C \rightarrow \C / \A$ gives a fully 
		faithful functor
		\begin{equation*}
			\homm_{\HoDGCat}(\C/\A,\, \B) \rightarrow \homm_{\HoDGCat}(\C,\, \B),
		\end{equation*}
		whose image comprises the quasi-functors whose underlying
		functors \[\Hzero(\C) \rightarrow \Hzero(\B)\] send the objects of $\A$ to
		zero. 
	\end{enumerate}
\end{Theorem}

Let $\biC$ be a \dg bicategory. For any collections $\biA, \biB$ of $1$-morphisms
of $\biC$, write $\biA \circ_1 \biB$ for the collection of $1$-morphisms 
of $\biC$ $2$-isomorphic 
to $a \circ_1 b$ with $a \in \biA$, $b \in \biB$. 
The \em two-sided ideal $\biI_\biA$ generated by $\biA$ \rm 
is the $2$-full subcategory of $\biC$ supported on objects
and $1$-morphisms of $\biC \circ_1 \biA \circ_1 \biC$. Here, by abuse of
notation, $\biC$ denotes the collection of all its $1$-morphisms. 

Let $\biC$ be a \dg bicategory and $\biA$ a collection of $1$-morphisms
of $\biC$. For any $a,b \in \biC$ write $\biC(a,b)/\biA$ for the Drinfeld
quotient $\biC(a,b)/\biA(a,b)$. These do not apriori form a bicategory. 
First of all, any $1$-composition involving a contractible $1$-morphism 
would have to be contractible. Were a bicategory structure to exist, 
for any $f \in \biA(a,b)$ and any $g \in \biC(b,c)$ 
the $1$-composition $\id_g \circ_1 \xi_f$ would have 
to be a contracting homotopy for $g \circ_1 f$. Unless
$g \circ_1 f$ lies in $\biA(a,c)$, there is no reason 
for it to be contractible in $\biC(a,c)/\biA$. 

This could be rectified by replacing $\biA$ with two-sided ideal
$\biI_\biA$. We could then attempt 
to define the $1$-composition $\id_g \circ_1 \xi_f$
to be contracting homotopy $\xi_{g \circ_1 f}$. However, 
the interchange law  \eqref{eq:interchange-law-for-graded-2-categories} for $1$-composition demands that for any $2$-morphism 
$\alpha\colon g \rightarrow h$ in $\biC(b,c)$ we have: 
\[
(\id_g \circ_1 \xi_f) \circ_2 (\alpha \circ_1 \id_f)
= \alpha \circ_1 \xi_f = (-1)^{\deg(\alpha)}
(\alpha \circ_1 \id_f) \circ_1 (\id_g \circ_1 \xi_f).
\]
If we define $\id_g \circ_1 \xi_f = \xi_{g \circ_1 f}$, 
this would then ask for $\xi_{g \circ_1 f}$ to supercommute with 
$(\alpha \circ_1 \id_f)$. But, by definition, 
there are no relations between $\xi_{g \circ_1 f}$ and 
any $2$-morphisms in $\biC(a,c)$!

This is why the original Drinfeld
quotient works poorly with monoidal structures: it is freely
generated by the contracting homotopies $\xi_f$ over the original 
category. Thus $\xi_f$ cannot satisfy the relations in 
the definition of $1$-composition. The $1$-composition functor
\begin{equation*}
	\circ_1\colon \biC(b,c)/\biI_\biA \otimes \biC(a,b)/\biI_\biA
	\rightarrow \biC(a,c)/\biI_\biA
\end{equation*}
could not exist because its target 
is a free category generated by $\xi_f$, while its source is not. 

In \cite{Shoikhet-DifferentialGradedCategoriesAndDeligneConjecture},
Shoikhet solves this by constructing a free resolution of \[\biC(b,c)/\biI_\biA \otimes \biC(a,b)/\biI_\biA,\] 
which admits a natural $1$-composition functor into $\biC(a,b)/\biI_\biA$.
He defines:

\begin{Definition}[{\cite[Section~4.3]{Shoikhet-DifferentialGradedCategoriesAndDeligneConjecture}}]
	Let $\C$ be a \dg category and let $\A_1, \dots, \A_n$ be full
	subcategories. The \em refined Drinfeld quotient \rm $\C/(\A_1, \dots,
	\A_n)$ is the \dg category whose underlying graded category is freely
	generated over that of $\C$ by introducing for any
	\[
	i_1 < i_2 < \dots < i_k \text{ and } a \in \A_{i_1} \cap \dots \cap \A_{i_k} 
	\]
	a new degree $k$ element
	\begin{equation*}
		\xi^{i_1 \dots i_k}_a.
	\end{equation*}
	The differential on these new elements is defined by setting
	\begin{equation*}
		d \xi^{i_1}_a = \id_a
	\end{equation*}
	and for $k > 1$
	\begin{equation*}
		d \xi^{i_1 \dots i_k}_a =
		\sum_{j = 1}^k (-1)^{j-1} \xi^{i_1 \dots \hat{i}_j \dots  i_k}_a. 
	\end{equation*}
\end{Definition}

\begin{Remark}
	When $n = 1$, the refined Drinfeld quotient $\C / \A_1$ coincides with the ordinary Drinfeld quotient.
	In this case we therefore omit the superscript in the notation above and write $\xi_a$ for $\xi^1_a$.
\end{Remark} 

The reason for considering the above as a refinement of the original 
Drinfeld quotient is the following theorem by Shoikhet:

\begin{Theorem}[{\cite[Lemma 4.3]{Shoikhet-DifferentialGradedCategoriesAndDeligneConjecture}}]
	Let $\C$ be a \dg category and let $\A_1, \dots, \A_n$ be full
	subcategories. The functor
	\[ \Psi\colon \C/(\A_1, \dots, \A_n) \rightarrow \C/\bigcup_{i = 1}^n \A_i \]
	defined as the identity on objects and morphisms of $\C$ and as
	\begin{align*}
		\Psi(\xi^{i_1}_a) &= \xi_a, \\
		\Psi(\xi^{i_1 \dots i_k}_a) & = 0 \quad \quad \text{for } k > 1,
	\end{align*}
	is a quasi-equivalence of \dg categories. 
\end{Theorem}

Observe that $\C/(\A_1, \dots, \A_n)$ and $\C/\bigcup_{i =
	1}^n \A_i$ are therefore isomorphic in $\HoDGCat$. It follows 
from Theorem~\ref{theorem-main-properties-of-drinfeld-quotients} that
the former enjoys the same unique lifting property as the latter
with respect to quasifunctors out of $\C$ which kill $\bigcup_{i =
	1}^n \A_i$ on the homotopy level.

At the same time, the next example shows that the \refined Drinfeld quotient serves as a free
resolution of the tensor product of ordinary Drinfeld quotients.
\begin{Example}
	\label{exmpl-refined-quotient-as-a-resolution-of-the-tensor-product-of-n-quotients-case-n=2}
	Let $\C_1$ and $\C_2$ be \dg categories and $\A_i \subset \C_i$
	be full subcategories. Let 
	\begin{equation*}
		\beta_{\Dr}\colon \C_1 \otimes \C_2 / (\A_1 \otimes \C_2, \C_1 \otimes \A_2)
		\rightarrow \left(\C_1 / \A_1\right) \otimes \left(\C_2 / \A_2\right)
	\end{equation*}
	be the functor defined as the identity on objects and 
	the morphisms of $\C_1 \otimes \C_2$ and as follows
	on the contracting homotopies:
	\begin{align*}
		\beta_{\Dr}(\xi^1_{a_1 \otimes c_2} )  &= \xi_{a_1} \otimes \id_{c_2}, \\
		\beta_{\Dr}(\xi^2_{c_1 \otimes a_2} )  &= \id_{c_1} \otimes \xi_{a_2}, \\ 
		\beta_{\Dr}(\xi^{12}_{a_1 \otimes a_2} )  &= \xi_{a_1} \otimes \xi_{a_2}. 
	\end{align*}
	It can be readily verified that $\beta_{\Dr}$ is a quasi-equivalence of
	\dg categories. 
\end{Example}

The above example can be formalised as follows: 

\begin{Definition}[{\cite[Section~4.4]{Shoikhet-DifferentialGradedCategoriesAndDeligneConjecture}}]
	Let $\PDGCat$ be the following category:
	\begin{itemize}
		\item Its objects are pairs $(\C;\, \A_1, \dots, \A_n)$
		where $\C$ is a \dg category and $\A_1, \dots, \A_n$ is an ordered
		$n$-tuple of full subcategories of $\C$.
		\item A morphism
		\[ (\C;\, \A_1, \dots, \A_n) \rightarrow (\D;\,\B_1, \dots, \B_m) \]
		is a pair $(F,f)$ where $f\colon \{1,\dots,n\}
		\rightarrow \{1,\dots, m\}$ is a map of sets and $F\colon\C \rightarrow \D$
		is a \dg functor such that $F(\A_i) \subset \B_{f(j)}$. 
	\end{itemize}
	We define a monoidal structure on $\PDGCat$
	by setting 
	\[ (\C;\, \A_1, \dots, \A_n) \otimes (\D;\, \B_1, \dots, \B_m) \]
	to be
	\[
	(\C \otimes \D;\, \A_1 \otimes \D, \dots, \A_n \otimes\D, \C \otimes
	\B_1, \dots, \C \otimes \B_m)
	\]
	and the unit to be $(\kk;\emptyset)$. 
\end{Definition}

\begin{Theorem}[{\cite[Section~4.4]{Shoikhet-DifferentialGradedCategoriesAndDeligneConjecture}}]
	The \refined Drinfeld quotient defines a functor:
	\begin{equation*}
		Dr\colon \PDGCat \rightarrow \DGCatone,
	\end{equation*}
	which has a natural homotopy monoidal structure given by 
	quasi-equivalences
	\begin{equation*}
		\begin{multlined}
		\beta\colon 
		Dr\bigl( (\C;\,\A_1,\dots, \A_n) \otimes (\D;\,\B_1,\dots,\B_m) \bigr)
	\\	\rightarrow 
		Dr\left(\C;\,\A_1,\dots, \A_n\right) \otimes Dr\left(\D;\,\B_1,\dots,\B_m\right). 
		\end{multlined}
	\end{equation*}
\end{Theorem}

The case considered in Example
\ref{exmpl-refined-quotient-as-a-resolution-of-the-tensor-product-of-n-quotients-case-n=2}
follows by observing that in $\PDGCat$ we have 
\[
(\C;\,\A) \otimes (\D;\,\B) =
(\C \otimes \D;\, \A \otimes \D, \C \otimes \B).
\]

We now return to the problem of constructing the Drinfeld quotient 
of a \dg bicategory. Let $\biC$ be a \dg bicategory and $\biA$
a collection of its $1$-morphisms. The $1$-composition functor 
\begin{equation*}
	\circ_1\colon \biC(b,c)/\biI_\biA \otimes \biC(a,b)/\biI_\biA
	\rightarrow \biC(a,c)/\biI_\biA,
\end{equation*} 
which does not exist in $\DGCat$, can now be defined in $\HoDGCat$ as
follows. The homotopy monoidal structure of the \refined
Drinfeld quotient functor gives us a quasi-equivalence
\begin{equation*}
	\beta_{\Dr}\colon 
	\biC(b,c) \otimes \biC(a,b) / \left(\biI_\biA \otimes \biC, \biC \otimes
	\biI_\biA\right)
	\longrightarrow
	\biC(b,c)/\biI_\biA \otimes \biC(a,b)/\biI_\biA.
\end{equation*} 
On the other hand, since $\biI_\biA$ is a two-sided ideal, the
original $1$-composition functor 
\[ \circ^{\old}_1\colon \biC(b,c) \otimes \biC(a,b) \rightarrow \biC(a,c), \]
takes $\biI_\biA(b,c) \otimes \biC(a,b)$ and $\biC(b,c) \otimes
\biI_\biA(a,b)$ to $\biI_\biA(a,c)$, and thus uniquely extends in 
$\HoDGCat$ to a quasi-functor 
\[
\circ^{\old}_1\colon 
\biC(b,c) \otimes \biC(a,b) / \left(\biI_\biA \otimes \biC, \biC \otimes \biI_\biA\right)
\rightarrow \biC(a,c)/\biI_\biA.
\]
We can therefore define $\circ_1$ 
in $\HoDGCat$ as the composition 
\[
\biC(b,c)/\biI_\biA \otimes \biC(a,b)/\biI_\biA
\xrightarrow{\beta_{\Dr}^{-1}} 
\biC(b,c) \otimes \biC(a,b) / \left(\biI_\biA \otimes \biC, \biC \otimes \biI_\biA\right)
\xrightarrow{\circ^{\old}_1}
\biC(a,c)/\biI. 
\]

\begin{Theorem}
	\label{theorem-construction-of-the-drinfeld-quotient-of-a-bicategory}
	Let $\biC$ be a \dg bicategory, or more generally a $\HoDGCat$-enriched
	bicategory. Let $\biA$ be a collection of $1$-morphisms
	in $\biC$, and let $\biI_\biA$ be the two-sided ideal generated by $\biA$. 
	Then the following data defines a $\HoDGCat$-enriched bicategory: 
	\begin{itemize}
		\item The same set of objects as $\biC$. 
		\item For any $a,b \in \biC$, the \dg category of $1$-morphisms from 
		$a$ to $b$ is $\biC(a,b)/\biI_\biA$. 
		\item For any $a,b,c \in \biC$, the $1$-composition functor
		\[
		\begin{multlined}
		\circ_1 \colon \biC(b,c)/\biI_\biA \otimes \biC(a,b)/\biI_\biA
		\xrightarrow{\beta_{\Dr}^{-1}}
		\biC(b,c) \otimes \biC(a,b) / \left(\biI_\biA \otimes \biC, \biC \otimes \biI_\biA\right)
		\\ \xrightarrow{\circ^{\old}_1}
		\biC(a,c)/ \biI_\biA,
		\end{multlined}
		\]
		\item The associator and unitor $2$-isomorphisms in 
		$\HoDGCat$ which are similarly obtained from the  
		associator and unitor $2$-isomorphisms of $\biC$ 
		via the precomposition with $\beta^{-1}_{\Dr}$. 
		
		For example, for any $a,b,c,d \in \biC$, 
		the quasi-functors 
		$\circ_1(\circ_1 \otimes \id)$ and  
		$\circ_1(\id \otimes \circ_1):$
		\[ 
		\biC(c,d)/\biI_{\biA} \otimes \biC(b,c)/\biI_{\biA} \otimes \biC(a,b)/\biI_{\biA}
		\rightarrow \biC(a,d)/\biI_{\biA}
		\]
		are the composition of the quasi-functor
		$\beta_{\Dr}^{-1}:$
		\[
		\begin{gathered}
			\biC(c,d)/\biI_{\biA} \otimes \biC(b,c)/\biI_{\biA} \otimes \biC(a,b)/\biI_{\biA} 
			\rightarrow \\
			\biC(c,d)\otimes \biC(b,c) \otimes \biC(a,b)/
			\left( 
			\biI_\biA\otimes \biC \otimes \biC,
			\biC\otimes \biI_\biA \otimes \biC,
			\biC\otimes \biC \otimes \biI_\biA
			\right)
		\end{gathered}
		\]
		with the quasi-functors 
		$\circ^{\old}_1(\circ^{\old}_1 \otimes \id)$ and 
		$\circ^{\old}_1(\id \otimes \circ^{\old}_1)$:
		\[
		\begin{multlined}
		\biC(c,d)\otimes \biC(b,c) \otimes \biC(a,b)/
		\left( 
		\biI_\biA\otimes \biC \otimes \biC,
		\biC\otimes \biI_\biA \otimes \biC,
		\biC\otimes \biC \otimes \biI_\biA
		\right)
		\\ \rightarrow 
		\biC(a,d)/\biI_{\biA}.
		\end{multlined}
		\]
		We thus define the new associator by precomposing the old
		associator with $\beta_{\Dr}^{-1}$. 
	\end{itemize}
\end{Theorem}

\begin{proof}
	
	Shoikhet began his proof of \cite[Theorem
	5.4]{Shoikhet-DifferentialGradedCategoriesAndDeligneConjecture} by
	constructing a Leinster monoid $F_\A$ in $\PDGCat$ out of a certain monoidal 
	DG category $\A_0$ and the two-sided ideal $\I_0$ of acyclic objects in 
	it. His construction works exactly the same for an arbitrary 
	monoidal \dg category $\A$ and an arbitrary two-sided ideal $\I$ in $\A$. 
	
	In general, a Leinster monoid\index{Leistner monoid} $L_\bullet$ 
	is a simplicial structure, generalising the 
	notion of an algebra in a monoidal category\index{monoidal
category}, cf.~\cite[Defn.~2.1]{Shoikhet-DifferentialGradedCategoriesAndDeligneConjecture}.
	It has colax maps $\beta_{m,n}\colon 
	L_{m+n} \rightarrow L_m \otimes L_n$ which 
	are weak equivalences, and thus each $L_n$ is weakly equivalent to 
	$(L_1)^{\otimes n}$. The non-extremal face maps 
	$L_n \rightarrow L_{n-1}$ should be thought of as analogues of 
	applying the algebra operation to subsequent 
	pairs of $L_1$'s in $(L_1)^{\otimes n}$, and the degeneracy maps\index{degeneracy maps}
	$L_{n} \rightarrow L_{n+1}$ as applying the algebra unit in between 
	two subsequent $L_1$'s. It follows that if the colax maps are not just
	weak equivalences, but isomorphisms, we have a unital algebra
	structure on $L_1$ whose algebra operation is 
	$$ L_1 \otimes L_1 \xrightarrow{\beta^{-1}_{1,1}} L_2
	\xrightarrow{\text{the unique non-extremal face}} L_1 $$
	and whose unit is the degeneracy map $\mathbb{1} \simeq L_0
	\rightarrow L_1$.
	
	The colax maps of the Leinster monoid $F_\A$ in $\PDGCat$ constructed
	by Shoikhet are the identity maps and $(F_\A)_1 = (\A;\I)$. 
	Applying the refined Drinfeld quotient functor, we obtain Leinster
	monoid $\Dr(F_\A)$ in $\DGCatone$ whose colax maps are
	the quasi-equivalences $\beta_{\Dr}$ and $(\Dr(F_\A))_1 = \A / \I$. 
	We then view it as a Leinster monoid in $\HoDGCatone$. There its
	colax maps become invertible, and we obtain the induced structure 
	of unital algebra on $\A / \I$ in $\HoDGCatone$. This structure 
	is the one claimed in the assertion of this Theorem. Thus we have
	proved the Theorem for an arbitrary monoidal \dg category, i.e. a DG
	bicategory with a single object. The proof for a general DG
	bicategory works identically, but with a more cumbersome notation. 
\end{proof}

\begin{Definition}\label{def:monoidal-Drinfeld-quotient}
	Let $\biC$ be a $\HoDGCat$-enriched bicategory. 
	Let $\biA$ be a collection of $1$-morphisms
	in $\biC$ and $\biI_\biA$ be the two-sided ideal generated by $\biA$. 
	The \emph{monoidal Drinfeld quotient $\biC/\biI_\biA$}\index{monoidal Drinfeld quotient} is the
	$\HoDGCat$-enriched bicategory constructed in Theorem
	\ref{theorem-construction-of-the-drinfeld-quotient-of-a-bicategory}. 
\end{Definition}

We have a natural functor $\biC \rightarrow \biC/\biI_\biA$ which is
a strict $2$-functorial embedding:

\begin{Definition}
	Let $\biC$ be a $\HoDGCat$-enriched bicategory.
	Let $\biA$ be a collection of $1$-morphisms in $\biC$, and let $\biI_\biA$ be 
	the two-sided ideal generated by $\biA$. Define a strict $2$-functor
	\[ \iota\colon \biC \hookrightarrow \biC/\biI_\biA, \]
	to be the identity on the objects. On $1$-morphisms, for any 
	$a,b \in \biC$ define 
	\[ \iota\colon \biC(a,b) \rightarrow \biC/\biI_\biA(a,b), \]
	to be the natural inclusion of the category into its Drinfeld quotient
	\[ \biC(a,b) \hookrightarrow \biC(a,b)/\biI_\biA(a,b). \]
\end{Definition}

We can now formulate an analogue of 
Theorem \ref{theorem-main-properties-of-drinfeld-quotients}
summarising the main properties of our monoidal Drinfeld quotient:

\begin{Theorem}
	\label{theorem-the-universal-properties-of-monoidal-drinfeld-quotient}
	Let $\biC$ be a $\HoDGCat$-enriched bicategory.
	Let $\biA$ be a collection of $1$-morphisms in $\biC$, and let $\biI_\biA$ be 
	the two-sided ideal generated by $\biA$. Then:
	\begin{enumerate}
		\item \label{item-as-enhanced-triag-cats-$1$-morphism-categories-are-verdier-quotients}
		For any $a,b \in \biC$, the following natural functor is an exact
		equivalence
		\[ 
		D_c\bigl(\biC/\biI_\biA(a,b)\bigr) \rightarrow
		D_c\bigl(\biC(a,b)\bigr)\Big/D_c\bigl(\biI_\biA(a,b)\bigr).
		\]
		\item \label{item-bicat-unique-lifting-property-for-functors-killing-the-ideal}
		Let $\biD$ be another $\HoDGCat$-enriched bicategory and let 
		$F\colon \biC \rightarrow \biD$ be a $2$-functor. If $F(\biI_\biA)$ is
		null-homotopic in the $1$-morphism categories of $\biD$, then there exists
		a unique lift of $F$ to a $2$-functor $F'\colon \biC/\biI_\biA \rightarrow \biD$:
		\begin{equation*}
			\begin{tikzcd}[column sep = 2cm]
				\biC 
				\ar{r}{F}
				\ar{dr}[']{\iota} 
				& 
				\biD
				\\
				&
				\biC/\biI_\biA.
				\ar[dashed]{u}[']{\exists!\;\; F'}
			\end{tikzcd}
		\end{equation*}
	\end{enumerate}
\end{Theorem}

\begin{proof}
	\ref{item-as-enhanced-triag-cats-$1$-morphism-categories-are-verdier-quotients}:
	This is immediate from the corresponding result for ordinary
	Drinfeld quotients. 
	
	\ref{item-bicat-unique-lifting-property-for-functors-killing-the-ideal}:
	This is due to the $2$-categorical unique lifting property of ordinary
	Drinfeld quotients (Theorem \ref{theorem-main-properties-of-drinfeld-quotients}), as follows:
	
	The data of a $2$-functor consists of a map of object sets, a collection 
	of functors between $1$-morphisms categories and composition/unit
	coherence morphisms\index{coherence morphisms}. Since the embedding $\iota\colon \biC \rightarrow
	\biC/\biI_\biA$ is the identity on object sets, the condition 
	$F = F' \circ \iota$ completely determines the action of $F'$ on
	objects. Next, let $a,b \in \biC$ be any pair of objects. Since
	\[ \iota_{a,b}\colon \biC(a,b) \hookrightarrow \biC(a,b)/\biI_\biA(a,b) \]
	is the canonical embedding of a category into its Drinfeld quotient, 
	and since, by assumption, $\Hzero(F_{a,b})$ kills $\biI_\biA(a,b)$,
	the quasifunctor
	\[ F_{a,b}\colon \biC(a,b) \rightarrow \biD(Fa,Fb), \]
	lifts to a unique quasifunctor 
	\[ F'_{a,b}\colon \biC(a,b)/\biI_\biA(a,b) \rightarrow \biD(Fa,Fb), \]
	such that $F'_{a,b} \circ \iota_{a,b} = F_{a,b}$. 
	
	It remains to show that composition and unit coherence morphisms exist
	and are unique. This is due to the lifting property in 
	Theorem \ref{theorem-main-properties-of-drinfeld-quotients} being
	$2$-categorical in $\HoDGCat$. We treat the composition coherence
	morphism below, the proof for unit coherence is similar. 
	
	Let $a,b,c \in \biC$ be objects. Consider the diagram
	\begin{equation*}
		\begin{tikzcd}[column sep = 1.75cm, row sep=1cm]
			\biC(b,c) \otimes \biC(a,b) 
			\ar{d}{\iota_{b,c}\otimes\iota_{a,b}}
			\ar{r}{\mu_\biC}
			&
			\biC(a,c)
			\ar{d}{\iota_{a,c}}
			\\
			\biC(b,c)/\biI_\biA(b,c) \otimes \biC(a,b)/\biI_\biA(a,b) 
			\ar{d}{F'_{b,c}\otimes F'_{a,b}}
			\ar{r}{\mu_{\biC/\biI_\biA}}
			&
			\biC(a,c)/\biI_\biA(a,c) 
			\ar{d}{F'_{a,c}}
			\\
			\biD(Fb,Fc) \otimes \biD(Fa,Fb) 
			\ar{r}{\mu_{\biD}}
			&
			\biD(Fa,Fc).
		\end{tikzcd}
	\end{equation*}
	It follows from our definition of $\mu_{\biC/\biI_\biA}$ that the top
	square commutes on the nose. Indeed, this can be takes as an
	alternative definition of $\mu_{\biC/\biI_\biA}$ since 
	$\biC(b,c)/\biI_\biA(b,c) \otimes \biC(a,b)/\biI_\biA(a,b)$ is quasi-equivalent
	to $\biC(b,c) \otimes \biC(a,b) / (\biI_\biA(b,c) \otimes \biC(b,c),
	\biC(a,b) \otimes \biI_\biA(a,b))$ and thus enjoys its unique lifting
	property with respect to the quasi-functors out of $\biC(b,c) \otimes
	\biC(a,b)$. 
	
	On the other hand, by our definition of $F'_{a,b}$ and $F'_{b,c}$
	it follows that the outer perimeter of the above diagram composes to 
	\begin{equation*}
		\begin{tikzcd}
			\biC(b,c) \otimes \biC(a,b) 
			\ar{d}{F_{b,c}\otimes\F_{a,b}}
			\ar{r}{\mu_\biC}
			&
			\biC(a,c)
			\ar{d}{F_{a,c}}
			\\
			\biD(Fb,Fc) \otimes \biD(Fa,Fb) 
			\ar{r}{\mu_{\biD}}
			&
			\biD(Fa,Fc).
		\end{tikzcd}
	\end{equation*}
	The composition coherence morphism $\phi_F$ is a choice of a $2$-morphism 
	in $\HoDGCat$ which makes this diagram commute. Since the lifting
	property of Drinfeld quotients is $2$-categorical, there
	exists a unique $2$-morphism $\phi_F'$ which makes the bottom square 
	in the first diagram commute, and composes with $\iota_{b,c} \otimes
	\iota_{a,b}$ to give $\phi_F$. 
\end{proof}

\section{Homotopy Serre functors}\label{subsec:dg-homotopy-serre}

Let $\A$ be a pretriangulated \dg category. A \emph{homotopy Serre functor\index{homotopy Serre functor}} on $\A$ is 
a quasi-autoequivalence $S$ of $\A$ equipped with a closed degree zero $\AbimA$-bimodule 
quasi-isomorphism 
\[
\eta\colon \A \rightarrow \left(\leftidx{_S}{\A}\right)^*,
\]
such that $S$ and $\eta$ induce a Serre functor on $\Hzero(\A)$ in the sense of 
Section~\ref{subsec:serre}.  Here $(-)^*$ denotes the dualisation over $\kk$ and 
$_S$ denotes the twist of the left $\A$-action by $S$. Explicitly, the
data of $\eta$ can be thought of as a collection of quasi-isomorphisms
natural in $a,b \in \A$:
\[
\eta_{a,b}\colon \homm_\A(a,b) \simeq \homm_\A(b,Sa)^*,
\]

Since the $\kk$-dualisation $(-)^*$ commutes with taking cohomologies, 
the dual of a quasi-isomorphism is a quasi-isomorphism. 
It also follows that the natural map 
\[
\leftidx{_S}{\A} \rightarrow \left(\leftidx{_S}{\A}\right)^{**}. 
\]
is a quasi-isomorphism if $\A$ is proper. The composition 
\[
\leftidx{_S}{\A} \rightarrow \leftidx{_S}{\A}^{**} \xrightarrow{\eta^*}
\A^*, 
\]
is then also a quasi-isomorphism. By abuse of notation, we also 
denote it by $\eta^*$. 

\begin{Lemma}
	Let $\A$ be a smooth and proper \dg category. Then $\hperfA$ 
	admits a homotopy Serre functor given by $S = (-) \otimes_\A \A^*$.  
\end{Lemma}

\begin{proof}
	It was shown in \cite{Shklyarov-OnSerreDualityForCompactHomologicallySmoothDGAlgebras}
	that $S$ descends to a Serre functor on $\Hzero(\hperf \A) \simeq \catDc(\A)$. 
	It remains to demonstrate that there is a quasi-isomorphism 
	$\eta\colon \hperfA \rightarrow \left(\leftidx{_S}{\hperfA}\right)^*$.
	Since Serre functors are unique, such $\eta$ would then necessarily be
	a \dg lift of the bifunctorial isomorphisms $\eta$ of the Serre functor
	on $\catDc(\A)$.
	
	We prove a more general statement. Let $P \in \hperfA$ and $Q$ be any \dg $\A$-module. 
	Consider the natural morphism functorial in $P$ 
	\[ P \otimes_\A \homm_\kk(\A,\kk) \longrightarrow \homm_\kk\bigl(\homm_\A(P,\A),\,\kk\bigr). \]
	It is an isomorphism on representable $P$ and hence a homotopy 
	equivalence on $P \in \hperfA$. We thus obtain a bifunctorial homotopy equivalence
	\[
	\beta\colon \homm_\A \bigl(Q,\, P \otimes_\A \homm_\kk(\A,\kk)\bigr)
	\longrightarrow \homm_\A\bigl(Q,\, \homm_\kk(\homm_\A(P,\A),\,\kk)\bigr).
	\]
	By Tensor-Hom adjunction, the RHS is canonically isomorphic to 
	\[ \homm_\kk \bigl(Q \otimes_\A \homm_\A(P,\A),\, \kk \bigr), \]
	and since $P \in \hperfA$, the natural morphism 
	$Q \otimes_\A \homm_\A(P,\A) \rightarrow \homm_\A(P,Q)$ 
	is a homotopy equivalence. Thus $\beta$ can be rewritten as a homotopy equivalence
	\[
	\eta\colon \homm_\A\bigl(Q,\, P \otimes_\A \homm_\kk(\A,\kk)\bigr) 
	\longrightarrow \homm_\kk\bigl(\homm_\A(P,Q),\, \kk\bigr),
	\]
	or, in other words, as
	\[ \eta\colon \homm_\A(Q,\, SP) \longrightarrow \homm_\A(P,\,Q)^*.\qedhere\]
\end{proof}

\begin{Example}
	\label{ex:Xsmoothpropenh}
	For $X$ a smooth and proper scheme over $\kk$, the enhanced derived category $\catDGCoh{X}$ is smooth and proper, and hence admits a homotopy Serre functor lifting the Serre functor on $\catDbCoh{X}$ from Example~\ref{ex:dbcoh-serre}.
\end{Example}

As before, a homotopy Serre functor $S$ induces a \emph{Serre trace map}\index{Serre trace} 
\begin{equation*}
	\Tr \colon \homm_{\A}(a, Sa) \rightarrow \kk, 
	\quad \quad
	\alpha \mapsto \eta_{a,a}(\id_{a})(\alpha)=\eta_{a,Sa}(\alpha)(\id_{a}).
\end{equation*}
As in Proposition~\ref{prop:serre_trace_cyclic_additive}, we have
\begin{equation*}
	\Tr(\beta \circ \alpha) =
	(-1)^{\deg(\alpha)\deg(\beta)}\Tr(S\alpha \circ \beta).
\end{equation*}
for any $a,b \in \A$ and any $\alpha \in \homm_\A(a,b)$, $\beta \in \homm_\A(b,Sa)$.

\section{\texorpdfstring{$G$}{G}-equivariant \dg categories for strong group actions}
\label{subsec:equivariant_cats}

Let $\A$ be a small \dg category with a \emph{strong}\index{strong action} action of 
a finite group $G$. That is, with an embedding of $G$ into the group of
\dg automorphisms of $\A$.

\begin{Definition}\label{def:semidirect-product}
	The \emph{semi-direct product\index{semi-direct product}} $\A \rtimes G$ is the following \dg
	category:
	\begin{itemize}
		\item $\obj \A \rtimes G = \obj \A$,
		\item For any $a,b \in \obj(\A \rtimes G)$ their morphism complex is
		\begin{equation*}
			\homm^i_{\A \rtimes G}(a,b) :=
			\left\{ (\alpha, g) \; \middle| \;
			\alpha \in \homm^i_{\A}(g.a,b), g \in G
			\right\}
		\end{equation*}
		with $\deg_{\A \rtimes G} (\alpha,g) = \deg_{\A} \alpha$ and 
		$d_{\A \rtimes G}(\alpha,g) = (d_\A \alpha,g)$,
		\item The composition in $\A \rtimes G$ is given by 
		\begin{equation*}
			(\alpha_1, g_1) \circ (\alpha_2, g_2) =
			(\alpha_1 \circ g_1.\alpha_2,\, g_1 g_2). 
		\end{equation*}
		\item For any $a \in \obj (\A \rtimes G)$ the identity morphism 
		of $a$ is $(\id_a, 1_G)$. 
	\end{itemize}
\end{Definition}

One can think of this as taking $\A$ and
formally adding for every object $a \in \A$ a closed degree $0$ isomorphism 
$a \rightarrow g.a$ for every $g \in G$. We then impose relations:
these isomorphisms compose via the multiplication in $G$, and 
their composition with the native morphisms of $\A$ is subject 
to the relations
\begin{equation}
	\label{eqn-twisted-action-of-G-relations}
	g \circ \alpha = g.\alpha \circ g \quad \quad \forall\;  g \in G, \alpha \in \A. 
\end{equation}
Therefore an action of $\A \rtimes G$ is equivalent to  
the action of $\A$ and an action of $G$ subject to 
\eqref{eqn-twisted-action-of-G-relations}. Here by action of $G$
we mean the action of the above tautological isomorphisms corresponding
to the elements of $G$. 

We have a natural embedding 
\begin{equation*}
	\eta \colon \A \hookrightarrow \A \rtimes G 
\end{equation*}
given by the identity on objects and $\alpha \mapsto (\alpha,\, \id_G)$ 
on morphisms. On the other hand, the projections 
$(\alpha, g) = \alpha \circ g \mapsto \alpha$ and 
$(\alpha, g) = g \circ g^{-1}.\alpha \mapsto g^{-1}\alpha$    
give rise to the decompositions 
\begin{equation}
	\label{eqn-decomposition-of-the-diagonal-A-rtimes-G-bimodule}
	\homm_{\A \rtimes G}(a,\, b) \simeq 
	\bigoplus_{g \in G} \homm_{\A}(g.a,\,b)
	\simeq 
	\bigoplus_{g \in G} \homm_{\A}(a,\,g^{-1}.b).
\end{equation}
We can think of these as decompositions of the diagonal bimodule:
\begin{equation}
	\label{equation-decomposition-of-the-equivariant-diagonal-bimodule}
	\A \rtimes G \simeq 
	\bigoplus_{g \in G} \A_{g}
	\simeq
	\bigoplus_{g \in G} \leftidx{_{g}}\A, 
\end{equation}
where $g$ denotes the autoequivalence $g\colon \A \rightarrow \A$. 
{ Both decompositions respect {the} 
	$\AbimA$-action and so the direct summands
	are themselves $\A$-$\A$-bimodules.} The induced right
and left actions of any $h \in G$ on the first decomposition are given by 
\[ \A_{g}  \xrightarrow{\id} \A_{gh}, \]
\[ \A_{g}  \xrightarrow{h.(-)} \A_{hg}, \]
and similarly for the second decomposition. 

The action of $G$ on $\A$ induces the action of $G$ on $\modA$ 
where each $g \in G$ acts via the extension of scalars functor 
$g^*$ with respect to the action functor $g\colon \A \rightarrow \A$. 
A \emph{$G$-equivariant $\A$-module}\index{equivariant module} 
is a pair $(E, \epsilon)$ where $E \in \modA$ and $\epsilon=(\epsilon_g)_{g \in G}$ is
a collection of isomorphisms 
\[ \epsilon_g\colon E \xrightarrow{\sim} g^* E \quad\quad g \in G \]
such that 
\[
\epsilon_{hg} =  E \xrightarrow{\epsilon_g} g^* E
\xrightarrow{g^* \epsilon_h} g^* h^* E = (hg)^* E \quad \quad g,h \in G.
\]
The \dg category $\modd^G\text{-}\A$ has as its objects $G$-equivariant
$\A$-modules and as its morphisms the morphisms between the
underlying $\A$-modules which commute with the isomorphisms $\psi_g$. 
See \cite[Section~2.1]{SymCat} for further details.  
The following generalises the classical correspondence 
between representations of a group and modules over 
the associated skew group algebra \cite[Chapter~5,
Remark~5.56]{leuschke2012cohen}:

\begin{Lemma}
	\label{lemma-A-rtimes-G-modules-are-G-equiv-A-modules}
	There are mutually inverse isomorphisms of categories
	\[ \rightmod{(\A \rtimes G)} \leftrightarrows \modd^G\mkern-4mu\mhyphen\A. \]
\end{Lemma}

\begin{proof}
	Given a $G$-equivariant $\A$-module $(E,\epsilon)$ we can extend  
	the action of $\A$ on $E$ to the action of $\A \rtimes G$ by 
	having $g$ act by $\epsilon_g\colon E_{a} \rightarrow E_{g^{-1}.a} = (g^* E)_a$. 
	Conversely, given a $\A \rtimes G$-module $E$ we can define 
	a $G$-equivariant structure on the $\A$-module $\eta_* E$ by 
	defining $\epsilon_g$ to be the action of $g$. 
	These operations are functorial and mutually inverse.
\end{proof}

Generalizing the setting from Section~\ref{subset:cat_fock_add}, for any subgroup $H\subset G$ there is a functor
\[ \iota\colon \A \rtimes H  \to  \A \rtimes G\]
given by the identity on objects, and by the identity times the inclusion on morphisms. This functor induces restriction and induction functors
\begin{align*}
	\Res_G^H \coloneqq \iota_{\ast}\colon & \rightmod{(\A \rtimes G)} \to \rightmod{(\A \rtimes H)},\\
	\Ind_H^G \coloneqq \iota^{\ast}\colon & \rightmod{(\A \rtimes H)} \to \rightmod{(\A \rtimes G)}.
\end{align*}

From the viewpoint of equivariant modules, the restriction functor can be written as
\[
\begin{array}{r c c c} 
	\Res_G^H\colon &  \modd^G\mkern-4mu\mhyphen\A & \to & \modd^H\mkern-4mu\mhyphen\A \\
	& \bigl(E,\, (\epsilon_g)_{g \in G}\bigr) & \mapsto & \bigl(E,\, (\epsilon_g)_{g \in H}\bigr)
\end{array}
\]
on objects and as the identity on morphisms.
Similarly, the induction functor is
\begin{align*}
	\Ind_H^G\colon \quad \quad \modd^H\mkern-4mu\mhyphen\A \quad & \to
	\quad \quad \quad \modd^G\mkern-4mu\mhyphen\A \\
	\bigl(E,\, (\epsilon_h)_{h \in H}\bigr) & \mapsto \bigl({ \bigoplus}_{[f] \in G/H} f^\ast E,\, (\epsilon_g)_{g \in G}\bigr)
\end{align*}
on objects where for every $g \in G$ 
\[ \epsilon_g\colon { \bigoplus}_{[f] \in G/H} f^\ast E \to { \bigoplus}_{[f'] \in G/H} g^\ast f'^\ast E \]
maps the $f$-permuted component in the domain to the $gf'$-permuted component in the target via $f^\ast \epsilon_h$ when $[f]=[gf'] \in G/H$ and $h \in H$ is such that $gf'=fh$. A similar formula applies to morphisms.

For us actions by symmetric groups\index{symmetric group}, and in particular symmetric powers\index{symmetrix power} of categories, will be of interest.
The $n$-fold tensor power $\A^{\otimes n}$ of a \dg category $\A$ has as objects $n$-tuples $a_1 \otimes \dots \otimes a_n$ of objects of $\A$ and has morphism complexes
\[
\homm_{\A^{\otimes n}}\left(a_1 \otimes \dots \otimes a_n,\, b_1 \otimes \dots \otimes b_n\right) \coloneqq
\homm_{\A}(a_1,b_1) \otimes_\kk \dots \otimes_\kk \homm_{\A}(a_n, b_n).
\]
We therefore have a natural strong action of the symmetric group $\SymGrp n$ on $\A^{\otimes n}$ by permuting the factors of objects and the factors of morphisms. 

\begin{Definition}\label{def:symmetric-power}
	Let $\A$ be an enhanced triangulated category. 
	We define $\sym^n\A$, the \emph{$n$-th symmetric power}\index{categorical symmetric power} of $\A$, 
	to be the semidirect product $\A^{\otimes n} \rtimes \SymGrp n$. 
\end{Definition}

The corresponding triangulated category is $\catDc(\sym^n\A) \simeq \Hzero\bigl(\hperf(\sym^n\A)\bigr)$. 
We have 
\[
\modd(\sym^n\A) \simeq \modd^{\SymGrp n}(\A^{\otimes n})
\]
by Lemma~\ref{lemma-A-rtimes-G-modules-are-G-equiv-A-modules}. 
It follows from the decomposition~\eqref{eqn-decomposition-of-the-diagonal-A-rtimes-G-bimodule} that this further restricts to 
\begin{align}
	\label{eqn-equivalence-modules-over-twalg-with-equiv-modules}
	\hperf(\sym^n\A) \simeq \hperf^{\SymGrp n}(\A^{\otimes n}), 
\end{align}
{where} $\hperf^{\SymGrp n}(\A^{\otimes n})$ is the full subcategory of
$\modd^{\SymGrp n}(\A^{\otimes n})$ consisting of the equivariant 
modules which are perfect in $\modd(\A^{\otimes n})$ after forgetting
the equivariant structure. 
Since $\hperf^{\SymGrp n}(\A^{\otimes n})$ was the definition of the
completed $n$-th symmetrical power $\widehat{\sym}^n \A$ of $\A$ in
\cite[Section~2.2.7]{SymCat}, that 
category is equivalent to the $\hperf$ hull of our $\sym^n \A$. This discrepancy
is due to us working in the Morita enhancement\index{Morita
enhancement} setting, where 
to pass to the underlying triangulated category one first takes the $\hperf$ 
hull, and then its homotopy category.

\section[The numerical Grothendieck group and the Heisenberg algebra]{The numerical Grothendieck group and the Heisenberg algebra of a \dg category}\label{subsec:prelim_grothendieck_group}\label{subsec:heisenberg_algebra_cat}

Consider a smooth and proper \dg category $\basecat$.
The Grothendieck group of $\basecat$,
\[ \mathrm{K}_0({\basecat}) = \mathrm{K}_0(\catDc(\basecat)), \]
comes equipped with the \emph{Euler}\index{Euler pairing} (or \emph{Mukai}) \emph{pairing}
\[
\bigl\langle [a],[b] \bigr\rangle_{\chi} \coloneqq
\chi\bigl(\Hom_{\hperf \basecat}(a,b)\bigr ) =
\sum_{n\in \ZZ} (-1)^n  \dim \Hom^n_{\catDc(\basecat)}(a,b).
\]

\begin{Example}\label{ex:nonsymmetric_pairing}
	The Euler pairing is in general neither symmetric nor antisymmetric.
	A simple example is given by the Grothendieck group of $\mathrm{K}_0(\ps 1) = \mathrm{K}_0({\catDbCoh{\ps 1}})$.
	It has a semiorthogonal basis given by the classes $\{ [\sO],[\sO(1)] \}$ for which the matrix of $\chi$ is 
	\[
	\begin{pmatrix}
		1 & 2 \\ 0 & 1
	\end{pmatrix}
	\]
	This matrix is clearly not diagonalisable over the integers.
\end{Example}

\begin{Proposition}\label{prop:tabuada_Ggp}
	Let $\basecat$ be a smooth and proper \dg category.
	\begin{enumerate}
		\item For every pair of objects $a$, $b$ of $\catDc(\basecat)$,
		\[
		\bigl\langle [a],[b] \bigr\rangle_{\chi} =
		\bigl\langle [b],[Sa] \bigr\rangle_{\chi} =
		\bigl\langle [S^{-1}b], [a] \bigr\rangle_{\chi},
		\]
		where $S$ is the Serre functor on $\catDc(\basecat)$.
		\item The left and right kernels of $\chi$ agree.
	\end{enumerate}
\end{Proposition}

\begin{proof}
	This is Lemma~4.25 and Proposition~4.24 of \cite{tabuada2015noncommutative}.
\end{proof}

The \emph{numerical Grothendieck group}\index{numerical Grothendieck group} $\numGgp{\basecat}$ of a smooth and proper \dg category $\basecat$ is $\mathrm{K}_0(\basecat)/\ker(\chi)$.
We further set $\numGgp{\basecat,\,\kk} \coloneqq \numGgp{\basecat} \otimes_{\ZZ} \kk$.

\begin{Proposition}[{\cite[Theorem 1.2]{tabuada2016noncommutative}, \cite[Theorem 1.2]{tabuada2017finite}}]
	The numerical Grothendieck group $\numGgp{\basecat}$ of a smooth and proper \dg category $\basecat$ is a finitely generated free abelian group.
\end{Proposition}

As $\chi$ is non-degenerate and integral on $\numGgp{\basecat}$, we
call the pair $(\numGgp{\basecat},\, \chi)$ the \emph{Mukai
	lattice}\index{Mukai lattice} of $\basecat$.

\begin{Example}\label{ex:mukai}
	For ${\basecat}={\catDGCoh X}$, where $X$ is smooth and projective, the Euler form\index{Euler form} can be computed by
	Hirzebruch--Riemann--Roch theorem\index{Hirzebruch--Riemann--Roch theorem} (see, for example, \cite[Section~6.3]{caldararu2010mukai}):
	\[
	\chi\bigl(\Hom(a,b)\bigr) =
	\chi(a^{\dual} \otimes b) =
	\int_X \operatorname{ch}(a^{\dual} \otimes b) \cdot \operatorname{td}(T_X).
	\]
	This implies that the kernel of $\chi$ equals the kernel of the Chern character\index{Chern character} map to Chow groups tensored with $\mathbb{Q}$.
\end{Example}

\begin{Definition}
	\label{def:halgbasecat}
	Let $\basecat$ be a smooth and proper \dg category.
	We write $\halg{\basecat}$ for the idempotent modified Heisenberg algebra $\halg{(\numGgp\basecat,\, \chi)}$.
	The corresponding Fock space representation is denoted by $\falg{\basecat}$.
\end{Definition}

\begin{Example} 
	\label{Ex:ps1pt2}	
	For $\basecat = \catDGCoh{\ps 1}$ as in Example~\ref{ex:nonsymmetric_pairing}, $\chi$  is nondegenerate and its Smith normal form is the unit $2 \times 2$ matrix  $ \mathrm{Id}_2$.
	Therefore, 
	\[\halg{\catDGCoh{\ps 1}} \simeq  H_{\ZZ^2, \mathrm{Id}_2} = \halg{\catDGCoh {\mathrm{pt} \sqcup \mathrm{pt}}}\] by Corollary~\ref{cor:halg_iso_to_symmetric}. 
\end{Example}

\begin{Lemma}
	\label{lem:indmapnum}
	Let $\A$, $\B$ be smooth and proper \dg categories, and let $F\colon \A \to \B$ be a \dg functor.
	Then $F^*\colon \catD(\A) \rightarrow \catD(\B)$ and $F_*\colon \catD(\B) \rightarrow \catD(\A)$ 
	preserve compactness and induce 
	\[
	F^\ast\colon \numGgp{\A} \to \numGgp{\B},
	\]
	\[
	F_\ast\colon \numGgp{\B} \to \numGgp{\A}.
	\]
\end{Lemma}

\begin{proof}
	As explained in \ref{section-restriction-and-extension-of-scalars},
	for any $\A$ and $\B$, not necessarily smooth or proper, 
	the extension of scalars functor $F^*\colon \modA \rightarrow \modB$
	always restricts to a functor $\hperfA \rightarrow \hperfB$. Thus 
	its derived functor preserves compactness. 
	
	We now show that $F_*\colon \modB \rightarrow \modA$ restricts to $\perfA \rightarrow
	\perfB$, whence its derived functor preserves compactness. 
	As $F_*$ is tensoring with the $\B$-$\A$-bimodule
	$\B_F$, it suffices to show $\B_F$ to be $\A$-perfect
	\cite[Prop.~2.14]{AnnoLogvinenko-SphericalDGFunctors}. Let $b \in \B$.
	For an $\A$-module $\leftidx{_b}{\B}{_F}$ to be perfect it
	suffices, since $\A$ is smooth, for it to be $\kk$-perfect 
	\cite[Cor.~2.15]{AnnoLogvinenko-SphericalDGFunctors}. In other words,  
	for any $a \in \A$ the total cohomology of the $\kk$-module
	$\leftidx{_b}{\B}{_{Fa}}$ has to be finite. This holds since $\B$ is proper. 
	
	The remaining assertions now follow by adjunction of $F^*$ and $F_*$.
	Indeed, for any $a \in \catDc(\A)$ and for any $b \in \catDc(\B)$ we have
	\begin{align*}
		\chi(F^\ast(a),\, b) & =
		\sum (-1)^i \dim \Hom^i_{\catDc(\B)}(F^\ast(a),\, b) \\ & = 
		\sum (-1)^i \dim \Hom^i_{\catDc(\A)}(a,\, F_\ast(b)) = 
		\chi(a,\, F_\ast(b)). 
	\end{align*}
	Thus $F^\ast$ and $F_\ast$ take $\ker \chi$ to $\ker \chi$ and 
	so induce maps of numerical Grothendieck groups\index{numerical Grothendieck group}.   
\end{proof}

\chapter{The \dg Heisenberg $2$-category}\label{sec:dg-Heisenberg-2-cat}

{ Let $\basecat$ be any smooth and proper \dg category. 
	We fix this choice throughout the rest of the paper. Now, recall 
	from Section~\ref{section-enhanced-categories} that we work with DG
	categories up to Morita equivalence, viewing them as enhanced
	triangulated categories. Replace therefore $\basecat$ by its 
	perfect hull $\hperf\basecat$. This doesn't change the Morita
	equivalence class of $\basecat$. However, it ensures that $\basecat$
	is homotopy direct summand complete and admits a homotopy Serre functor. 
	Note that, as explained in Section~\ref{subsec:dg-homotopy-serre}, 
	any homotopy Serre functor $S$ induces a Serre trace map 
	$\Tr\colon\Hom_{\basecat}(a, Sa) \to \kk$ for any $a \in \basecat$. }

In this section we define a $\HoDGCat$-enriched bicategory $\hcat\basecat$, 
the \emph{Heisenberg category}\index{Heisenberg category} of $\basecat$. This category is 
a monoidal Drinfeld quotient\index{monoidal Drinfeld
quotient} of the perfect hull of a simpler strict
\dg $2$-category $\hcat*\basecat$ which we set up in the following
paragraphs. We take the Drinfeld quotient to impose certain relations
in $\hcat\basecat$ which we only expect to hold on the level of
homotopy categories, unlike the relations we impose on
$\hcat*\basecat$ which must hold on the \dg level. 

\section{The category \texorpdfstring{$\hcat*\basecat$}{H'}: generators}
\label{subsec:heisencatdef-dg}

The objects of $\hcat*\basecat$ are the integers $\ho \in \ZZ$.

As in the additive setting of Chapter~\ref{sec:additive-Heisenberg-2cat}, we have $1$-morphisms labeled $\QQ_a$ for $a \in \basecat$.
However, as we have only a homotopy Serre functor, we need to more carefully distinguish between the left and right duals of $\QQ_a$.
The $1$-morphisms are therefore freely generated by 
\begin{itemize}
	\item $\PP_a\colon \ho \to \ho+1$, 
	\item $\QQ_a\colon \ho+1 \to \ho$, 
	\item $\RR_a\colon \ho \to \ho+1$, 
\end{itemize}
for each $a \in \basecat$ and $\ho \in \ZZ$.
Thus the objects of $\Hom_{\hcat*\basecat}(\ho,\ho*)$ are finite words in the symbols $\PP_a$, $\QQ_a$, and $\RR_a$ with $a \in \basecat$, 
such that the difference of the number of $\PP$s and $\RR$s and the 
number of $\QQ$s is $\ho*-\ho$. 
The identity $1$-morphism of any $\ho \in \mathbb{Z}$ is denoted as
$\hunit$. 

The $2$-morphisms between two $1$-morphisms form a complex of vector spaces. 
These vector spaces are freely generated by the generators listed
below, subject to the axioms of a (strict) \dg $2$-category as well as
the relations we detail in the next section.
As before, we represent these $2$-morphisms as planar diagrams, using the same sign rules as in Remark~\ref{rem:$2$-morphism-interchange-additive}.
We recall that diagrams are read bottom to top, i.e., the source of a given $2$-morphism lies on the lower boundary, while the target lies on the upper boundary.

We now list the generating $2$-morphisms. 
For every $\alpha \in \Hom_\basecat(a,b)$ there are arrows
\[
\begin{tikzpicture}[baseline=0]
	\draw[->] (0,0) node[below] {$\PP_a$} -- node[label=right:{$\alpha$}, dot, pos=0.5] {} (0,1) node[above] {$\PP_{b}$};
\end{tikzpicture}, \quad
\begin{tikzpicture}[baseline=0]
	\draw[->] (0,1) node[above] {$\QQ_a$} -- node[label=right:{$\alpha$}, dot, pos=0.5] {} (0,0) node[below] {$\QQ_{b}$};
\end{tikzpicture}, \quad 
\begin{tikzpicture}[baseline=0]
	\draw[->] (0,0) node[below] {$\RR_a$} -- node[label=right:{$\alpha$}, dot, pos=0.5] {} (0,1) node[above] {$\RR_{b}$};
\end{tikzpicture}.
\]
These $2$-morphisms are homogeneous of degree $|\alpha|$. The 
remaining generators listed below are all of degree $0$. By convention 
a strand without a dot is the same as one marked with the identity morphism.
Any such unmarked string is an identity $2$-morphism in $\hcat*\basecat$.
The identity $2$-morphisms of the $1$-morphisms $\hunit$ are usually
pictured by a blank space. 

For every $a \in \basecat$, there is a special arrow marked with a star:
\begin{equation}\label{eq:H'V-genereator-serre}
	\begin{tikzpicture}[baseline={(0,0.4)}]
		\draw[->] (0,0) node[below] {$\PP_{Sa}$} -- node[serre, pos=0.5] {} (0,1) node[above] {$\RR_{a}$};
	\end{tikzpicture}. 
\end{equation}

Furthermore, for any objects $a, b \in \basecat$ there are cups and caps
\[
\begin{tikzpicture}[baseline={(0,0.25)}]
	\draw[->] (0,0) node[below] {$\PP_a$} arc[start angle=180, end angle=0, radius=.5] node[label=above:{$\hunit$}, pos=0.5]{} node[below] {$\QQ_a$};
\end{tikzpicture}, \quad
\begin{tikzpicture}[baseline={(0,0.25)}]
	\draw[->] (0,0) node[below] {$\RR_a$} arc[start angle=0, end angle=180, radius=.5] node[label=above:{$\hunit$}, pos=0.5]{} node[below] {$\QQ_{a}$};
\end{tikzpicture}, \quad
\begin{tikzpicture}[baseline={(0,-0.25)}]
	\draw[->] (0,0) node[above] {$\QQ_a$} arc[start angle=0, end angle=-180, radius=.5] node[label=below:{$\hunit$}, pos=0.5]{} node[above] {$\RR_{a}$};
\end{tikzpicture}, \quad
\begin{tikzpicture}[baseline={(0,-0.25)}]
	\draw[->] (0,0) node[above] {$\QQ_{a}$} arc[start angle=-180, end angle=0, radius=.5] node[label=below:{$\hunit$}, pos=0.5]{} node[above] {$\PP_{a}$};
\end{tikzpicture},
\]
as well as crossings of two downward strands\index{downward strand}: 
\begin{equation}
	\label{eqn-H'V-generators-downcrossing}
	\begin{tikzpicture}[baseline={(0,0.4)}]
		\draw[<-] (0,0) node[below] {$\QQ_{a}$} -- (1,1) node[above] {$\QQ_{a}$};
		\draw[<-] (1,0) node[below] {$\QQ_{b}$} -- (0,1) node[above] {$\QQ_{b}$};
	\end{tikzpicture}
\end{equation}

We recall again the sign convention for reading planar diagrams from Remark~\ref{rem:$2$-morphism-interchange-additive}.
As before, we often \enquote{prettify} diagrams by smoothing them out.

We give each $2$-morphism space a \dg structure. With the grading
defined above, it remains to define the differential.  If $f$ is a single
strand with one dot labelled $\alpha$, then $d(f)$ is the same diagram
with the label replaced by $d(\alpha)$. In particular, the
differential of a strand labelled with the identity is
$d(\mathrm{id})=0$. The differentials of the remaining generating
$2$-morphisms --- the caps, the cups, the crossings, and the star ---
are zero. The differential of a general $2$-morphism is then 
determined by  the following graded Leibniz rules for $1$- and
$2$-compositions. These follow from the definition of a \dg bicategory:
\begin{itemize}
	\item $d(h \circ_1 g) = d(h) \circ_1 g + (-1)^{|h|} h \circ_1 d(g)$,
	\item $d(h \circ_2 g) = d(h) \circ_2 g + (-1)^{|h|} h \circ_2 d(g)$.
\end{itemize}

For convenience, we define four further types of strand crossings from the basic one in \eqref{eqn-H'V-generators-downcrossing} by composition with cups and caps:
\begin{equation}\label{eq:other-crossings-dg}
	\begin{split}
		\begin{tikzpicture}[baseline={(0,1.45)}]
			\draw[->] (0,1) node[below] {$\PP_{a}$} -- (1,2) node[above] {$\PP_{a}$};
			\draw[->] (0,2) node[above] {$\QQ_{b}$} -- (1,1) node[below] {$\QQ_{b}$};
			\draw (1.5,1.5) node {$\coloneqq$};
			\draw[->] 
			(2,1) node[below] {$\PP_{a}$} to[out=90, in=180]
			(2.5,2) to[out=0, in=180]
			(3.5,1) to[out=0, in=270]
			(4,2) node[above] {$\PP_{a}$};
			\draw[->] (3.3,2) node[above] {$\QQ_{b}$} -- (2.7,1) node[below] {$\QQ_{b}$};
		\end{tikzpicture},
		\qquad\qquad
		\begin{tikzpicture}[baseline={(0,1.45)}]
			\draw[<-] (0,1) node[below] {$\QQ_{a}$} -- (1,2) node[above] {$\QQ_{a}$};
			\draw[<-] (0,2) node[above] {$\RR_{b}$} -- (1,1) node[below] {$\RR_{b}$};
			\draw (1.5,1.5) node {$\coloneqq$};
			\draw[<-] 
			(2,1) node[below] {$\QQ_{a}$} to[out=90, in=180]
			(2.5,2) to[out=0, in=180]
			(3.5,1) to[out=0, in=270]
			(4,2) node[above] {$\QQ_{a}$};
			\draw[<-] (3.3,2) node[above] {$\RR_{b}$} -- (2.7,1) node[below] {$\RR_{b}$};
		\end{tikzpicture},\\
		\begin{tikzpicture}[baseline={(0,1.45)}]
			\draw[->] (0,1) node[below] {$\PP_{a}$} -- (1,2) node[above] {$\PP_{a}$};
			\draw[<-] (0,2) node[above] {$\PP_{b}$} -- (1,1) node[below] {$\PP_{b}$};
			\draw (1.5,1.5) node {$\coloneqq$};
			\draw[->] 
			(2,1) node[below] {$\PP_{a}$} to[out=90, in=180]
			(2.5,2) to[out=0, in=180]
			(3.5,1) to[out=0, in=270]
			(4,2) node[above] {$\PP_{a}$};
			\draw[<-] (3.3,2) node[above] {$\PP_{b}$} -- (2.7,1) node[below] {$\PP_{b}$};
		\end{tikzpicture},
		\qquad\qquad
		\begin{tikzpicture}[baseline={(0,1.45)}]
			\draw[->] (0,1) node[below] {$\RR_{a}$} -- (1,2) node[above] {$\RR_{a}$};
			\draw[<-] (0,2) node[above] {$\RR_{b}$} -- (1,1) node[below] {$\RR_{b}$};
			\draw (1.5,1.5) node {$\coloneqq$};
			\draw[->] 
			(2,1) node[below] {$\RR_{a}$} to[out=90, in=180]
			(2.5,2) to[out=0, in=180]
			(3.5,1) to[out=0, in=270]
			(4,2) node[above] {$\RR_{a}$};
			\draw[<-] (3.3,2) node[above] {$\RR_{b}$} -- (2.7,1) node[below] {$\RR_{b}$};
		\end{tikzpicture}.
	\end{split}
\end{equation}

\section{The category \texorpdfstring{$\hcat*\basecat$}{H'}: relations between $2$-morphisms}
\label{subsec:heisencatdef2-dg}

In the preceding subsection we gave the list of the generating symbols for $2$-morphisms.
We obtain all $2$-morphisms in $\hcat*\basecat$ $1$- and $2$-compositions of these symbols, subject to the axioms of a strict \dg $2$-category and a list of relations we impose in this section.

First, we impose the linearity relations:
\begin{equation}\label{eq:linearity-dg}
	\begin{tikzpicture}[baseline={(0,0.4)}]
		\draw (0,0) -- node[label=left:{$\alpha$}, dot, pos=0.5] {} (0,1);
	\end{tikzpicture}
	+
	\begin{tikzpicture}[baseline={(0,0.42)}]
		\draw (0,0) -- node[label=right:{$\beta$}, dot, pos=0.5] {} (0,1);
	\end{tikzpicture}
	=
	\begin{tikzpicture}[baseline={(0,0.42)}]
		\draw (0,0) -- node[label=right:{$\alpha+\beta$}, dot, pos=0.5] {} (0,1);
	\end{tikzpicture}
	\qquad\qquad
	c\,\;
	\begin{tikzpicture}[baseline={(0,0.42)}]
		\draw (0,0) -- node[label=right:{$\alpha$}, dot, pos=0.5] {} (0,1);
	\end{tikzpicture}
	=
	\begin{tikzpicture}[baseline={(0,0.42)}]
		\draw (0,0) -- node[label=right:{$c\alpha$}, dot, pos=0.5] {} (0,1);
	\end{tikzpicture}
\end{equation}
for any scalar $c \in \kk$ and any compatible orientation of the strings.

Neighboring dots along a downward string can merge with a sign twist:
\begin{equation}\label{eq:colliding_dots_down-dg}
	\begin{tikzpicture}[baseline={(0,-0.6)}]
		\draw[->] 
		(0,0) -- node[label=right:{$\alpha$}, dot, pos=0.33] {}
		node[label=right:{$\beta$}, dot, pos=0.66] {}
		(0,-1);
		\draw (1.5, -0.5) node {$=(-1)^{|\alpha||\beta|}$};
		\draw[->] (2.5,0) -- node[label=right:{$\beta\circ\alpha$}, dot, pos=0.5] {} (2.5,-1);
	\end{tikzpicture}.
\end{equation}
A dot can swap with a star according to the following rule:
\begin{equation}\label{eq:dot_sliding_past_star}
	\begin{tikzpicture}[baseline={(0,0.4)}]
		\draw[->] (0,0) -- node[serre, pos=0.33] {}
		node[label=right:{$\alpha$}, dot, pos=0.66] {}
		(0,1);
		\draw (0.8, 0.5) node {$=$};
		\draw[->] (1.25,0) -- node[label=right:{$S\alpha$}, dot, pos=0.33] {}
		node[serre, pos=0.66] {}
		(1.25,1);
	\end{tikzpicture}.
\end{equation}
Dots may \enquote{slide} through the generating cups and crossing as follows:
\begin{equation}\label{eq:cupsslide-dg}
	\begin{tikzpicture}[baseline=0]
		\draw[->] 
		(0,-0.2) node[below] {$\PP_a$}
		-- node[label=left:{$\alpha$}, dot, pos=1] {}
		(0,0)    arc[start angle=180, end angle=0, radius=.5]
		--
		(1,-0.2) node[below] {$\QQ_b$};
		\draw (1.5,0) node {$=$};
		\draw[->]
		(2,-0.2) node[below] {$\PP_a$}
		--
		(2,0)    arc[start angle=180, end angle=0, radius=.5] 
		-- node[label=right:{$\alpha$}, dot, pos=0] {}
		(3,-0.2) node[below] {$\QQ_b$};
	\end{tikzpicture}
	\qquad\qquad
	\begin{tikzpicture}[baseline=0]
		\draw[<-]
		(0,-0.2) node[below] {$\QQ_b$}
		-- node[label=left:{$\alpha$}, dot, pos=1] {}
		(0,0)    arc[start angle=180, end angle=0, radius=.5]
		--
		(1,-0.2) node[below] {$\RR_{a}$};
		\draw (1.5,0) node {$=$};
		\draw[<-]
		(2,-0.2) node[below] {$\QQ_b$}
		--
		(2,0)    arc[start angle=180, end angle=0, radius=.5]
		-- node[label=right:{$\alpha$}, dot, pos=0] {}
		(3,-0.2) node[below] {$\RR_{a}$};
	\end{tikzpicture}
\end{equation}
\begin{equation}\label{eq:crossingslide-dg}
	\begin{tikzpicture}[baseline={(0,0.4)}]
		\draw[<-] (0,0) -- node[label=left:{$\alpha$}, dot, pos=0.25] {} (1,1);
		\draw[<-] (1,0) -- (0,1);
	\end{tikzpicture}
	=
	\begin{tikzpicture}[baseline={(0,0.4)}]
		\draw[<-] (0,0) -- node[label=right:{$\alpha$}, dot, pos=0.75] {} (1,1);
		\draw[<-] (1,0) -- (0,1);
	\end{tikzpicture}.
\end{equation}
Note that when drawing diagrams, dots need to keep their relative heights when doing these operations in order to avoid accidentally introducing signs (cf.~Remark~\ref{rem:dg_slide_interchange} below).

There are two sets of local relations for unmarked strings:
the \emph{adjunction relations}
\begin{equation}\label{eq:straighten-dg}
	\begin{tikzpicture}[baseline={(0,0.9)}]
		\draw (0,0) -- (0,1) arc[start angle=180, end angle=0, radius=.5] arc[start angle=-180, end angle=0, radius=.5] -- (2,2);
		\draw ( 2.5,1) node{$=$};
		\draw (3,0) -- (3,2);
		\draw ( 3.5,1) node{$=$};
		\draw (6,0) -- (6,1) arc[start angle=0, end angle=180, radius=.5] arc[start angle=0, end angle=-180, radius=.5] -- (4,2);
	\end{tikzpicture}
\end{equation}
and the \emph{symmetric group} relations\index{triple move relation}
on downward strands\index{downward strand}
\begin{equation}\label{eq:symmetric_group_relations-dg}
	\begin{tikzpicture}[baseline={(0,0.9)}]
		\draw[rounded corners=15pt, <-] (0,0) -- (1,1) -- (0,2);
		\draw[rounded corners=15pt, <-] (1,0) -- (0,1) -- (1,2);
		\draw (1.5,1) node {$=$};
		\draw[<-] (2,0) -- (2,2) ;
		\draw[<-] (3,0) -- (3,2) ;
	\end{tikzpicture}
	\qquad\qquad\qquad
	\begin{tikzpicture}[baseline={(0,0.9)}]
		\draw[<-] (0,0) -- (2,2);
		\draw[rounded corners=15pt, <-] (1,0) -- (0,1) -- (1,2);
		\draw[<-] (2,0) -- (0,2);
		\draw (2.5,1) node {$=$};
		\draw[<-] (3,0) -- (5,2);
		\draw[rounded corners=15pt, <-] (4,0) -- (5,1) -- (4,2);
		\draw[<-] (5,0) -- (3,2);
	\end{tikzpicture}
\end{equation}

Finally there are three relations\index{curl relation}\index{bubble
relation} involving a star-marked string
\begin{equation}\label{eq:circle_and_curl-dg}
	\begin{tikzpicture}[baseline={(0,-0.1)}, xscale=-1]
		\draw[<-] 
		(1,-0.9)  node[below] {$\QQ_{a}$} -- 
		(1,-0.5)  to[out=90, in=0]
		(0.3,0.5) to[out=180,in=90]
		(-0.1,0)  node[serre] {} to[out=270,in=180] (0.3,-.5) to[out=0,in=270]
		(1,0.5)   --
		(1,0.9) node[above] {$\QQ_{Sa}$};
	\end{tikzpicture}
	= 0,
	\qquad\qquad\qquad
	\begin{tikzpicture}[baseline={(0,0.4)}]
		\draw[decoration={markings, mark=at position 0.33 with {\arrow{>}}, mark=at position 0.83 with {\arrow{>}}}, postaction={decorate}]
		(1,0.2) -- node[label=right:{$\alpha$}, dot, pos=0.25] {}
		node[serre, pos=0.75] {}
		(1,0.8) arc[start angle=0, end angle=180, radius=.5]
		--
		(0,0.2) arc[start angle=180, end angle=360, radius=.5];
	\end{tikzpicture}
	=
	\Tr(\alpha),
\end{equation}
where $\alpha \in \Hom_\basecat(a,Sa)$, and
\begin{equation}\label{eq:up_down_braid-dg}
	\begin{tikzpicture}[baseline={(0,0.9)}]
		\draw[->] 
		(0,0)     node[below] {$\PP_{Sa}$} 
		--
		(0.5,0.5) to[out=45,in=-45] node[serre, pos=0.5] {} 
		(0.5,1.5) -- 
		(0,2) node[above] {$\RR_{a}$};
		\draw[->] 
		(1,2)     node[above] {$\QQ_{b}$}
		--
		(0.5,1.5) to[out=225, in=135]
		(0.5,0.5) --
		(1,0) node[below] {$\QQ_{b}$};
		\draw (1.5,1) node {$=$};
		\draw[->]
		(2,0) node[below] {$\PP_{Sa}$}
		-- node[serre, pos=0.5] {}
		(2,2) node[above] {$\RR_{a}$};
		\draw[<-]
		(3,0) node[below] {$\QQ_{b}$}
		--
		(3,2) node[above] {$\QQ_{b}$};
	\end{tikzpicture}.
\end{equation}

The relations \eqref{eq:circle_and_curl-dg} are the analogues of the relations \eqref{eq:circle_and_curl-add}. As leftward caps involve an $\RR$ but rightward cups involve a $\PP$, a star needs to be added between the two. Similarly, to get a consistent diagram a star must appear in both sides of \eqref{eq:up_down_braid-dg}, the analogue of the left relation from \eqref{eq:up_down_braids-add}.

We do not have an equivalent of the right relation in
\eqref{eq:up_down_braids-add} because to define the map $\Psi$ we need
the natural isomorphism 
$\homm(b,Sa) \simeq \homm(a,b)^*$ afforded to us by the genuine Serre
functor. In the present \dg setup we only have a homotopy Serre functor
which only gives us a natural homotopy equivalence $\homm(b,Sa)
\rightarrow \homm(a,b)^*$, but not its natural inverse. We can't
therefore define the map $\Psi$. More spefically, of the two
composants $\psi_1$ and $\psi_2$ of the term $\Psi(\id)$ described 
after Remark \ref{rem:symmetric_group_action-add} in
Section~\ref{subsec:remarks_on_relations-additive} we have $\psi_2$, but 
not $\psi_1$. However, the two relations in
$\eqref{eq:circle_and_curl-add}$ and the left relation in 
\eqref{eq:up_down_braids-add} together are equivalent to the map 
$$
\QQ_a \PP_{b}
\xrightarrow{
	\left[
	\begin{tikzpicture}[baseline={(0,0.15)}, scale=0.5]
		\draw[->] (1.5,0) -- (0.5,1);
		\draw[->] (1.5,1) -- (0.5,0);
	\end{tikzpicture}
	\,,\,
	\,\psi_1
	\right]
}
\PP_{b}\QQ_a \oplus \bigl(\Hom(a,b) \otimes_\kk \hunit \bigr) 
$$
being the left inverse of the map 
$$  \PP_{b}\QQ_a \oplus \bigl(\Hom(a,b) \otimes_\kk \hunit \bigr) 
\xrightarrow{
	\left[
	\begin{tikzpicture}[baseline={(0,0.15)}, scale=0.5]
		\draw[->] (0.5,0) -- (1.5,1);
		\draw[->] (0.5,1) -- (1.5,0);
	\end{tikzpicture}
	\,,\,
	\,\psi_2
	\right]
}
\QQ_a \PP_{b},
$$
while the right relation in \eqref{eq:up_down_braids-add} is
equivalent
to it being the right inverse. 

Thus, having imposed the equivalents of the two relations in
$\eqref{eq:circle_and_curl-add}$ and the left relation in 
\eqref{eq:up_down_braids-add}, to have the equivalent of
the right relation in $\eqref{eq:up_down_braids-add}$ 
we only need the map $\left[
\begin{tikzpicture}[baseline={(0,0.15)}, scale=0.5]
	\draw[->] (0.5,0) -- (1.5,1);
	\draw[->] (0.5,1) -- (1.5,0);
\end{tikzpicture}
\,,\,
\,\psi_2
\right]$
be a homotopy equivalence. We impose it in Section~\ref{subsec:idempotent-hcat}
by taking the Drinfeld quotient by its cone.

\section{Remarks on the $2$-morphism relations in \texorpdfstring{$\hcat*\basecat$}{H'}}
\label{subsec:remarks_on_relations-dg}

Let us remark on some of the above relations for $2$-morphisms and their consequences.

\begin{Remark}
	The reader familiar with the categorifications of Khovanov and Cautis--Licata \cite{khovanov2014heisenberg,cautis2012heisenberg} or the classical Heisenberg algebra might find the appearance of the third type of $1$-morphisms, i.e.~$\RR_a$, confusing.
	In the Fock space representation constructed in Chapter~\ref{sec:cat_fock}, the $1$-morphism $\QQ_a$ is sent to a pushforward functor $\phi_{a,*}$, while $\PP_a$ and $\RR_a$ are sent to the left adjoint $\phi_a^*$ and right adjoint $\phi_a^!$ respectively.
	In $\hcat*\basecat$ this is expressed by the relations~\eqref{eq:straighten-dg} which state that there are adjunctions of $1$-morphisms $(\PP_a,\,\QQ_a)$ and $(\QQ_a,\, \RR_a)$ for any $a \in \basecat$.
	
	Up to homotopy, the Serre functor lets us switch between left and right adjoints: $\phi_{Sa}^*$ and $\phi_a^!$ are identified in the homotopy category (note that in Khovanov's case the Serre functor is trivial, while in the Cautis--Licata setting it is a shift by $2$, see Examples~\ref{ex:Khovanov-dg} and~\ref{ex:CautisLicata_part1}).
	However, on the \dg level, there is only a canonical natural transformation $\phi_{Sa}^* \to \phi_a^!$.
	This natural transformation is represented by the starred arrow~\eqref{eq:H'V-genereator-serre}.
	In Section~\ref{subsec:idempotent-hcat}, we take the Drinfeld
	quotient by this arrow, forcing it to be an isomorphism on the homotopy level.
\end{Remark}

\begin{Remark}\label{rem:dg_slide_interchange}
	Since composing with the identity on either side doesn't change $2$-morphisms, dots may freely \enquote{slide along} straight strands as long as the relative height of all dots is kept the same.
	The interchange law introduces a sign when two dots slide past each other:
	\begin{equation}\label{eq:commuting_dots-dg}
		\begin{tikzpicture}[baseline={(0,0.4)}]
			\draw (0,0) -- node[label=left:{$\alpha$}, dot, pos=0.33] {} (0,1);
			\draw (0.5, 0.5) node {$\cdots$};
			\draw (1,0) -- node[label=right:{$\beta$}, dot, pos=0.66] {} (1,1);
			\draw (2.5, 0.5) node {$=(-1)^{|\alpha||\beta|}$};
			\draw (4,0) -- node[label=left:{$\alpha$}, dot, pos=0.66] {} (4,1);
			\draw (4.5, 0.5) node {$\cdots$};
			\draw (5,0) -- node[label=right:{$\beta$}, dot, pos=0.33] {} (5,1);
		\end{tikzpicture}.
	\end{equation}
	The axioms governing the differential in a \dg $2$-category are compatible 
	with this super-commutativity:
	\begin{equation}\label{eq:d_and_commutativity}
		\begin{tikzpicture}[baseline={(0,0.4)}]
			\draw (-1, 0.5) node {$d\Bigg($};
			\draw (-0.25,0) -- node[label=left:{$\alpha$}, dot, pos=0.33] {} (-0.25,1);
			\draw (0.25,0) -- node[label=right:{$\beta$}, dot, pos=0.66] {} (0.25,1);
			\draw (1, 0.5) node {$\Bigg)=$};
			\draw (2,0) -- node[label=left:{$\alpha$}, dot, pos=0.33] {} (2,1);
			\draw (2.5,0) -- node[label=right:{$d\beta$}, dot, pos=0.66] {} (2.5,1);
			\draw (4, 0.5) node {$+(-1)^{|\beta|}$};
			\draw (5.5,0) -- node[label=left:{$d\alpha$}, dot, pos=0.33] {} (5.5,1);
			\draw (6,0) -- node[label=right:{$\beta$}, dot, pos=0.66] {} (6,1);
			\draw (7.75, 0.5) node {$=(-1)^{|\alpha||\beta|}d\Bigg($};
			\draw (9.5,0) -- node[label=left:{$\alpha$}, dot, pos=0.66] {} (9.5,1);
			\draw (10,0) -- node[label=right:{$\beta$}, dot, pos=0.33] {} (10,1);
			\draw (10.75, 0.5) node {$\Bigg)$};
		\end{tikzpicture}
	\end{equation}
	In particular, it does not matter which dot one ``moves''  to the bottom of the diagram.
\end{Remark}

\begin{Lemma}\label{lem:colliding_dots_up-dg}
	Dots on upward strands\index{upward strand} merge without a sign change:
	\[
	\begin{tikzpicture}[baseline={(0,0.4)}]
		\draw[->] 
		(0,0) -- node[label=right:{$\alpha$}, dot, pos=0.33] {}
		node[label=right:{$\beta$}, dot, pos=0.66] {}
		(0,1);
		\draw (0.8, 0.5) node {$=$};
		\draw[->] (1.25,0) -- node[label=right:{$\beta\circ\alpha$}, dot, pos=0.5] {} (1.25,1);
	\end{tikzpicture}
	\]
\end{Lemma}

\begin{proof}
	The same proof as in Lemma~\ref{lem:colliding_dots_up-add} applies.
\end{proof}

The sign rules also imply that merging of dots is compatible with the
graded Leibniz rules for $\basecat$ and
$\Hom_{\hcat*\basecat}(\ho,\ho*)$. See \eqref{eq:d_and_commutativity}
and note that in $\basecat$ we have: 
\begin{equation*} 
	d(\beta \circ \alpha) = d(\beta) \circ \alpha + (-1)^{|\beta|} \beta
	\circ d(\alpha).
\end{equation*}

\begin{Lemma}\label{lem:dotslide-dg}
	Dots may freely slide through cups, caps and all types of crossings:
	\[
	\begin{tikzpicture}[baseline=0]
		\draw[->] 
		(0,0.2) node[above] {$\QQ_a$}
		-- node[label=left:{$\alpha$}, dot, pos=1] {}
		(0,0)   arc[start angle=180, end angle=360, radius=.5]
		--
		(1,0.2) node[above] {$\PP_b$};
		\draw (1.5,0) node {$=$};
		\draw[->]
		(2,0.2) node[above] {$\QQ_a$}
		--
		(2,0)   arc[start angle=180, end angle=360, radius=.5] 
		-- node[label=right:{$\alpha$}, dot, pos=0] {}
		(3,0.2) node[above] {$\PP_b$};
	\end{tikzpicture}
	\qquad\qquad
	\begin{tikzpicture}[baseline=0]
		\draw[<-]
		(0,0.2) node[above] {$\RR_{b}$}
		-- node[label=left:{$\alpha$}, dot, pos=1] {}
		(0,0)   arc[start angle=180, end angle=360, radius=.5]
		--
		(1,0.2) node[above] {$\QQ_{a}$};
		\draw (1.5,0) node {$=$};
		\draw[<-]
		(2,0.2) node[above] {$\RR_{b}$}
		--
		(2,0)   arc[start angle=180, end angle=360, radius=.5]
		-- node[label=right:{$\alpha$}, dot, pos=0] {}
		(3,0.2) node[above] {$\QQ_{a}$};
	\end{tikzpicture}
	\]
	\[
	\begin{tikzpicture}[baseline={(0,0.4)}]
		\draw[-] (0,0) -- node[label=left:{$\alpha$}, dot, pos=0.25] {} (1,1);
		\draw[-] (1,0) -- (0,1);
	\end{tikzpicture}
	=
	\begin{tikzpicture}[baseline={(0,0.4)}]
		\draw[-] (0,0) -- node[label=right:{$\alpha$}, dot, pos=0.75] {} (1,1);
		\draw[-] (1,0) -- (0,1);
	\end{tikzpicture}
	\qquad\qquad
	\begin{tikzpicture}[baseline={(0,0.4)}]
		\draw[-] (0,0) -- (1,1);
		\draw[-] (1,0) -- node[label=right:{$\alpha$}, dot, pos=0.25] {} (0,1);
	\end{tikzpicture}
	=
	\begin{tikzpicture}[baseline={(0,0.4)}]
		\draw[-] (0,0) -- (1,1);
		\draw[-] (1,0) -- node[label=left:{$\alpha$}, dot, pos=0.75] {} (0,1);
	\end{tikzpicture}.
	\]
\end{Lemma}

\begin{proof}
	The same proof as in Lemma~\ref{lem:dotslide-add} applies.
\end{proof}

\begin{Lemma}\label{lem:basic-relations-dg}
	The following relations hold in $\hcat*\basecat$ for all objects $a,b,c \in \basecat$.
	\begin{enumerate}
		\item All allowed pitchfork relations\index{pitchfork
relation}:
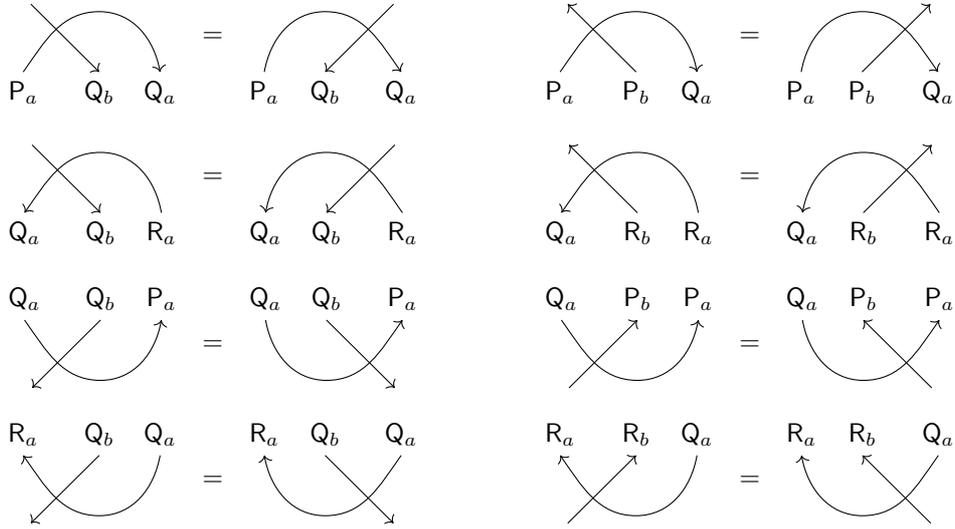
		\[
		\begin{tikzpicture}[baseline={(0,0.4)}]
			\draw[->] (0,0) node[below] {$\PP_a$} to[out=55, in=180] (1,0.8) to[out=0, in=100] (1.8,0) node[below] {$\QQ_a$};
			\draw[->] (0.1,0.9) -- (1,0) node[below] {$\QQ_b$};
		\end{tikzpicture}
		\ = \ 
		\begin{tikzpicture}[baseline={(0,0.4)}]
			\draw[->] (0,0) node[below] {$\PP_a$} to[out=80, in=180] (0.8,0.8) to[out=0, in=125] (1.8,0) node[below] {$\QQ_a$};
			\draw[->] (1.7,0.9) -- (0.8,0) node[below] {$\QQ_b$};
		\end{tikzpicture}
		\qquad\quad
		\begin{tikzpicture}[baseline={(0,0.4)}]
			\draw[->] (0,0) node[below] {$\PP_a$} to[out=55, in=180] (1,0.8) to[out=0, in=100] (1.8,0) node[below] {$\QQ_a$};
			\draw[<-] (0.1,0.9) -- (1,0) node[below] {$\PP_b$};
		\end{tikzpicture}
		\ = \ 
		\begin{tikzpicture}[baseline={(0,0.4)}]
			\draw[->] (0,0) node[below] {$\PP_a$} to[out=80, in=180] (0.8,0.8) to[out=0, in=125] (1.8,0) node[below] {$\QQ_a$};
			\draw[<-] (1.7,0.9) -- (0.8,0) node[below] {$\PP_b$};
		\end{tikzpicture}
		\]
		\[
		\begin{tikzpicture}[baseline={(0,0.4)}]
			\draw[<-] (0,0) node[below] {$\QQ_a$} to[out=55, in=180] (1,0.8) to[out=0, in=100] (1.8,0) node[below] {$\RR_a$};
			\draw[->] (0.1,0.9) -- (1,0) node[below] {$\QQ_b$};
		\end{tikzpicture}
		\ = \ 
		\begin{tikzpicture}[baseline={(0,0.4)}]
			\draw[<-] (0,0) node[below] {$\QQ_a$} to[out=80, in=180] (0.8,0.8) to[out=0, in=125] (1.8,0) node[below] {$\RR_a$};
			\draw[->] (1.7,0.9) -- (0.8,0) node[below] {$\QQ_b$};
		\end{tikzpicture}
		\qquad\quad
		\begin{tikzpicture}[baseline={(0,0.4)}]
			\draw[<-] (0,0) node[below] {$\QQ_a$} to[out=55, in=180] (1,0.8) to[out=0, in=100] (1.8,0) node[below] {$\RR_a$};
			\draw[<-] (0.1,0.9) -- (1,0) node[below] {$\RR_b$};
		\end{tikzpicture}
		\ = \ 
		\begin{tikzpicture}[baseline={(0,0.4)}]
			\draw[<-] (0,0) node[below] {$\QQ_a$} to[out=80, in=180] (0.8,0.8) to[out=0, in=125] (1.8,0) node[below] {$\RR_a$};
			\draw[<-] (1.7,0.9) -- (0.8,0) node[below] {$\RR_b$};
		\end{tikzpicture}
		\]
		\[
		\begin{tikzpicture}[baseline={(0,-0.4)}, yscale=-1]
			\draw[->] node[above] {$\QQ_a$} (0,0) to[out=55, in=180] (1,0.8) to[out=0, in=100] (1.8,0) node[above] {$\PP_a$};
			\draw[<-] (0.1,0.9) -- (1,0) node[above] {$\QQ_b$};
		\end{tikzpicture}
		\ = \ 
		\begin{tikzpicture}[baseline={(0,-0.4)}, yscale=-1]
			\draw[->] (0,0) node[above] {$\QQ_a$} to[out=80, in=180] (0.8,0.8) to[out=0, in=125] (1.8,0) node[above] {$\PP_a$};
			\draw[<-] (1.7,0.9) -- (0.8,0) node[above] {$\QQ_b$};
		\end{tikzpicture}
		\qquad\quad
		\begin{tikzpicture}[baseline={(0,-0.4)}, yscale=-1]
			\draw[->] (0,0) node[above] {$\QQ_a$} to[out=55, in=180] (1,0.8) to[out=0, in=100] (1.8,0) node[above] {$\PP_a$};
			\draw[->] (0.1,0.9) -- (1,0) node[above] {$\PP_b$};
		\end{tikzpicture}
		\ = \ 
		\begin{tikzpicture}[baseline={(0,-0.4)}, yscale=-1]
			\draw[->] (0,0) node[above] {$\QQ_a$} to[out=80, in=180] (0.8,0.8) to[out=0, in=125] (1.8,0) node[above] {$\PP_a$};
			\draw[->] (1.7,0.9) -- (0.8,0) node[above] {$\PP_b$};
		\end{tikzpicture}
		\]
		\[
		\begin{tikzpicture}[baseline={(0,-0.4)}, yscale=-1]
			\draw[<-] (0,0) node[above] {$\RR_a$} to[out=55, in=180] (1,0.8) to[out=0, in=100] (1.8,0) node[above] {$\QQ_a$};
			\draw[<-] (0.1,0.9) -- (1,0) node[above] {$\QQ_b$};
		\end{tikzpicture}
		\ = \ 
		\begin{tikzpicture}[baseline={(0,-0.4)}, yscale=-1]
			\draw[<-] (0,0) node[above] {$\RR_a$} to[out=80, in=180] (0.8,0.8) to[out=0, in=125] (1.8,0) node[above] {$\QQ_a$};
			\draw[<-] (1.7,0.9) -- (0.8,0) node[above] {$\QQ_b$};
		\end{tikzpicture}
		\qquad\quad
		\begin{tikzpicture}[baseline={(0,-0.4)}, yscale=-1]
			\draw[<-] (0,0) node[above] {$\RR_a$} to[out=55, in=180] (1,0.8) to[out=0, in=100] (1.8,0) node[above] {$\QQ_a$};
			\draw[->] (0.1,0.9) -- (1,0) node[above] {$\RR_b$};
		\end{tikzpicture}
		\ = \ 
		\begin{tikzpicture}[baseline={(0,-0.4)}, yscale=-1]
			\draw[<-] (0,0) node[above] {$\RR_a$} to[out=80, in=180] (0.8,0.8) to[out=0, in=125] (1.8,0) node[above] {$\QQ_a$};
			\draw[->] (1.7,0.9) -- (0.8,0) node[above] {$\RR_b$};
		\end{tikzpicture}
		\]
		\item All counterclockwise curls\index{curl relation} vanish:
		\begin{equation*}
			\begin{tikzpicture}[baseline={(0,-1.1)}, yscale=-1]
				\draw[<-]
				(0,0)   node[above] {$\RR_a$}
				to[out=90, in=270]
				(1,1.5) node[serre] {}
				arc[start angle=0, end angle=180, radius=0.5]
				to[out=270, in=90]
				(1,0)   node[above] {$\QQ_{Sa}$};
			\end{tikzpicture}
			=0,
			\qquad
			\begin{tikzpicture}[baseline={(0,-0.1)}]
				\draw[->]
				(1,-1)    node[below] {$\RR_{Sa}$}
				-- 
				(1,-0.5)  to[out=90, in=0]
				(0.3,0.5) to[out=180,in=90]
				(-0.1,0)  to[out=270,in=180]
				(0.3,-.5) to[out=0,in=270]
				node[serre, pos=0.4] {}
				(1,0.5)   --
				(1,1)     node[above] {$\RR_{a}$};
			\end{tikzpicture}
			= 0,
			\qquad
			\begin{tikzpicture}[baseline={(0,-0.1)}]
				\draw[->]
				(1,-1)    node[below] {$\PP_{Sa}$} -- 
				(1,-0.5)  to[out=90, in=0]  
				node[serre, pos=0.6] {}
				(0.3,0.5) to[out=180,in=90]
				(-0.1,0)  to[out=270,in=180]
				(0.3,-.5) to[out=0,in=270]
				(1,0.5)   --
				(1,1)     node[above] {$\PP_{a}$};
			\end{tikzpicture}
			= 0,
			\qquad
			\begin{tikzpicture}[baseline={(0,0.9)}]
				\draw[->]
				(0,0)   node[below] {$\PP_{Sa}$}
				to [out=90, in=270]
				(1,1.5) node[serre] {}
				arc[start angle=0, end angle=180, radius=0.5]
				to[out=270, in=90]
				(1,0)   node[below] {$\QQ_a$};
			\end{tikzpicture}
			=0.
		\end{equation*}
		\item 
		\label{it:basic-relations-dg3}
		The symmetric group relations on upward
strands\index{upward strand} of the same type:
		
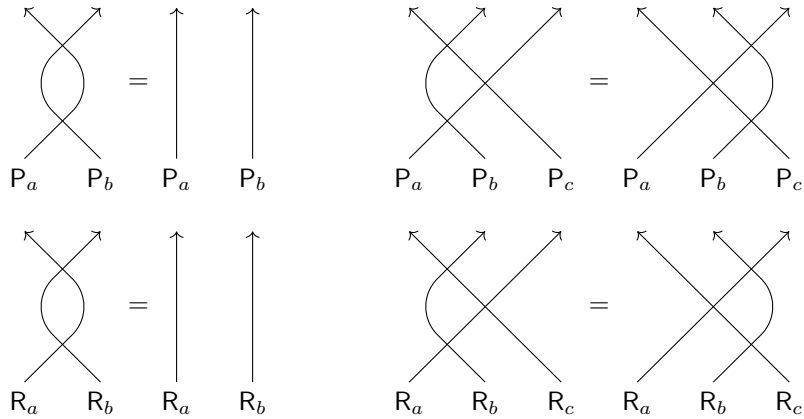
\[
		\begin{tikzpicture}[baseline={(0,0.9)}]
			\draw[rounded corners=15pt, ->] (0,0) node[below] {$\PP_a$} -- (1,1) -- (0,2);
			\draw[rounded corners=15pt, ->] (1,0) node[below] {$\PP_b$} -- (0,1) -- (1,2);
			\draw (1.5,1) node {$=$};
			\draw[->] (2,0) node[below] {$\PP_a$} -- (2,2) ;
			\draw[->] (3,0) node[below] {$\PP_b$} -- (3,2) ;
		\end{tikzpicture}
		\qquad\qquad
		\begin{tikzpicture}[baseline={(0,0.9)}]
			\draw[->] (0,0) node[below] {$\PP_a$} -- (2,2);
			\draw[rounded corners=15pt, ->] (1,0) node[below] {$\PP_b$} -- (0,1) -- (1,2);
			\draw[->] (2,0) node[below] {$\PP_c$} -- (0,2);
			\draw (2.5,1) node {$=$};
			\draw[->] (3,0) node[below] {$\PP_a$} -- (5,2);
			\draw[rounded corners=15pt, ->] (4,0) node[below] {$\PP_b$} -- (5,1) -- (4,2);
			\draw[->] (5,0) node[below] {$\PP_c$} -- (3,2);
		\end{tikzpicture}
		\]
		\[
		\begin{tikzpicture}[baseline={(0,0.9)}]
			\draw[rounded corners=15pt, ->] (0,0) node[below] {$\RR_a$} -- (1,1) -- (0,2);
			\draw[rounded corners=15pt, ->] (1,0) node[below] {$\RR_b$} -- (0,1) -- (1,2);
			\draw (1.5,1) node {$=$};
			\draw[->] (2,0) node[below] {$\RR_a$} -- (2,2) ;
			\draw[->] (3,0) node[below] {$\RR_{b}$} -- (3,2) ;
		\end{tikzpicture}
		\qquad\qquad
		\begin{tikzpicture}[baseline={(0,0.9)}]
			\draw[->] (0,0) node[below] {$\RR_a$} -- (2,2);
			\draw[rounded corners=15pt, ->] (1,0) node[below] {$\RR_{b}$} -- (0,1) -- (1,2);
			\draw[->] (2,0) node[below] {$\RR_{c}$} -- (0,2);
			\draw (2.5,1) node {$=$};
			\draw[->] (3,0) node[below] {$\RR_a$} -- (5,2);
			\draw[rounded corners=15pt, ->] (4,0) node[below] {$\RR_{b}$} -- (5,1) -- (4,2);
			\draw[->] (5,0) node[below] {$\RR_{c}$} -- (3,2);
		\end{tikzpicture}
		\]
		\item The remaining allowed triple moves\index{triple
move relation}:
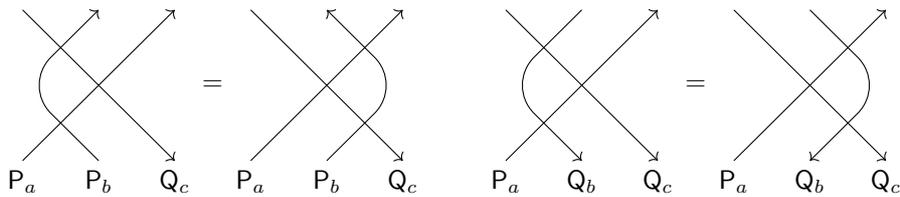
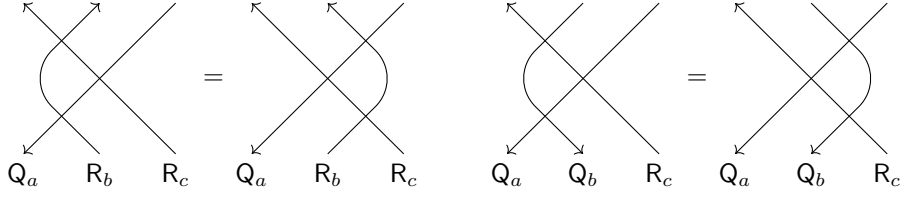
		\[
		\begin{tikzpicture}[baseline={(0,0.9)}]
			\draw[->] (0,0) node[below] {$\PP_a$} -- (2,2);
			\draw[rounded corners=15pt, ->] (1,0) node[below] {$\PP_{b}$} -- (0,1) -- (1,2);
			\draw[<-] (2,0) node[below] {$\QQ_{c}$} -- (0,2);
			\draw (2.5,1) node {$=$};
			\draw[->] (3,0) node[below] {$\PP_a$} -- (5,2);
			\draw[rounded corners=15pt, ->] (4,0) node[below] {$\PP_{b}$} -- (5,1) -- (4,2);
			\draw[<-] (5,0) node[below] {$\QQ_{c}$} -- (3,2);
		\end{tikzpicture}
		\qquad
		\begin{tikzpicture}[baseline={(0,0.9)}]
			\draw[->] (0,0) node[below] {$\PP_a$} -- (2,2);
			\draw[rounded corners=15pt, <-] (1,0) node[below] {$\QQ_{b}$} -- (0,1) -- (1,2);
			\draw[<-] (2,0) node[below] {$\QQ_{c}$} -- (0,2);
			\draw (2.5,1) node {$=$};
			\draw[->] (3,0) node[below] {$\PP_a$} -- (5,2);
			\draw[rounded corners=15pt, <-] (4,0) node[below] {$\QQ_{b}$} -- (5,1) -- (4,2);
			\draw[<-] (5,0) node[below] {$\QQ_{c}$} -- (3,2);
		\end{tikzpicture}
		\]
		\[
		\begin{tikzpicture}[baseline={(0,0.9)}]
			\draw[<-] (0,0) node[below] {$\QQ_a$} -- (2,2);
			\draw[rounded corners=15pt, ->] (1,0) node[below] {$\RR_{b}$} -- (0,1) -- (1,2);
			\draw[->] (2,0) node[below] {$\RR_{c}$} -- (0,2);
			\draw (2.5,1) node {$=$};
			\draw[<-] (3,0) node[below] {$\QQ_a$} -- (5,2);
			\draw[rounded corners=15pt, ->] (4,0) node[below] {$\RR_{b}$} -- (5,1) -- (4,2);
			\draw[->] (5,0) node[below] {$\RR_{c}$} -- (3,2);
		\end{tikzpicture}
		\qquad
		\begin{tikzpicture}[baseline={(0,0.9)}]
			\draw[<-] (0,0) node[below] {$\QQ_a$} -- (2,2);
			\draw[rounded corners=15pt, <-] (1,0) node[below] {$\QQ_{b}$} -- (0,1) -- (1,2);
			\draw[->] (2,0) node[below] {$\RR_{c}$} -- (0,2);
			\draw (2.5,1) node {$=$};
			\draw[<-] (3,0) node[below] {$\QQ_a$} -- (5,2);
			\draw[rounded corners=15pt, <-] (4,0) node[below] {$\QQ_{b}$} -- (5,1) -- (4,2);
			\draw[->] (5,0) node[below] {$\RR_{c}$} -- (3,2);
		\end{tikzpicture}
		\]
	\end{enumerate}
\end{Lemma}

\begin{proof}
	These are proved similarly to Lemmas~\ref{lem:pitchfork-I-add}, \ref{lem:curls-add},
	\ref{lem:upward_symmetric_group_relations-add}
	and~\ref{lem:triple_moves-add}. 
	One notes that the more complicated proof of Lemma~\ref{lem:pitchfork-II-add} is not needed, as we do not require the left and right mates of the downward crossing to coincide.
\end{proof}

\section{The category \texorpdfstring{$\hcat \basecat$}{H}: the perfect hull and homotopy relations}
\label{subsec:idempotent-hcat}

We construct the Heisenberg category $\hcat\basecat$ out of 
category $\hcat*\basecat$ in two steps. First, 
we apply Definition~\ref{defn-the-perfect-hull-of-a-bicategory} to form the
perfect hull $\bihperf(\hcat*\basecat)$. This is no longer a strict
$2$-category, but a bicategory. It has the objects of
$\hcat*\basecat$, but the $1$-morphism categories are replaced by
their perfect hulls. In particular, they are strongly pre-triangulated and
homotopy Karoubi-complete. 

The pre-triangulated structure we obtain on $1$-morphism categories of 
$\bihperf(\hcat*\basecat)$ allows us to formulate the final relations
we need to impose. Roughly, these postulate that certain
$2$-morphisms are isomorphisms in the homotopy category.

Let $a,b \in \basecat$. Since $\basecat$ is proper,
$\homm_{\basecat}(a,b)$ has finite dimensional cohomology and thus is
a perfect \dg $\kk$-module. Hence for any $1$-morphism $E \in
\hcat*\basecat$ the tensor product $\homm_{\basecat}(a,b) \otimes_{\kk} E$ 
lies in $\bihperf(\hcat*\basecat)$. Indeed, since any complex of
vector spaces is homotopy equivalent to the direct sum of its
cohomologies $\homm_{\basecat}(a,b) \otimes_{\kk} E$ is
homotopy equivalent to $\bigoplus_i H^i(\homm_{\basecat}(a,b))
\otimes_{\kk} E$ which is a direct sum of a finite number of copies of $E$.

Similar to \eqref{eq:half-Psi}, we have the natural $2$-morphism
\[
\psi_2\colon \Hom_\basecat(a,b) \otimes_\kk \hunit \to \QQ_a \PP_b
\]
in $\bihperf(\hcat*\basecat)$ obtained 
by adjunction from the map of complexes of vector spaces 
\begin{align*}
	\psi_2^{\text{adj}}\colon \Hom_{\basecat}(a,b) &\to 
	\Hom_{\bihperf(\hcat*\basecat)}(\hunit,\; \QQ_a\PP_b) \\
	\beta \quad &\mapsto \quad
	\smash[b]{\tikz[scale=0.5,baseline={(0,-0.25)}] \draw[->] (0,0) arc[start angle=-180, end angle=0, radius=.5] node[label=right:{$\beta$}, pos=0.7, dot] {};}. 
\end{align*}
We no longer have its counterpart $\psi_1$ as we do not have 
a map \[\Hom_{\basecat}(a,b)^\dual \to \Hom_{\basecat}(b,Sa).\] 

The map $\psi_2^{\text{adj}}$ is closed of degree $0$ since for any
$\beta \in \homm_{\basecat}(a,b)$ we have
\[ 
\begin{multlined}
d\psi_2^{\text{adj}}(\beta) =
d_{\Hom_{\bihperf(\hcat*\basecat)}(\hunit,\;
	\QQ_a\PP_b)}\left(\psi_2^{\text{adj}}(\beta)\right) - 
\psi_2^{\text{adj}}\left(d_{\homm_{\basecat}(a,b)} \beta\right) \\ = 
d\left(\smash[b]{\tikz[scale=0.5,baseline={(0,-0.25)}] \draw[->] (0,0) arc[start angle=-180, end angle=0, radius=.5] node[label=right:{$\beta$}, pos=0.7, dot] {};}\right)
-
\smash[b]{\tikz[scale=0.5,baseline={(0,-0.25)}] \draw[->] (0,0) arc[start angle=-180, end angle=0, radius=.5] node[label=right:{$d\beta$}, pos=0.7, dot] {};}
= 0. 
\end{multlined}
\]
Therefore the map $\psi_2$ is also closed of degree $0$.

Together with a crossing, $\psi_2$ induces a natural degree zero closed $2$-morphism
\begin{equation}\label{eq:dg-baby-Heisenberg-morphism}
	\PP_{b}\QQ_a \oplus \bigl(\Hom(a,b) \otimes_\kk \hunit \bigr) 
	\xrightarrow{
		\left[
		\begin{tikzpicture}[baseline={(0,0.15)}, scale=0.5]
			\draw[->] (0.5,0) -- (1.5,1);
			\draw[->] (0.5,1) -- (1.5,0);
		\end{tikzpicture}
		\,,\,
		\,\psi_2
		\right]
	}
	\QQ_a \PP_{b},
\end{equation}
and on the homotopy level, where $\psi_1$ does exist, we would like
\eqref{eq:dg-baby-Heisenberg-morphism} to be an isomophism. 

Secondly, in the homotopy category of $\basecat$, the functor $S$
becomes an actual Serre functor.  In terms of the graphical calculus, 
this means that on the homotopy level we would like
\begin{equation} 
	\label{eq:dg-star-morphism}
	\PP_{Sa} \xrightarrow{\tikz[scale=0.5] \draw[->](0,0) -- node[serre, pos=0.5] {} (0,1);} \RR_a 
\end{equation}
to be isomorphisms for all $a \in \basecat$.

We therefore take the monoidal Drinfeld quotient\index{monoidal Drinfeld
quotient} (see
Definition~\ref{def:monoidal-Drinfeld-quotient}) of
$\bihperf(\hcat*\basecat)$ by the cones of
\eqref{eq:dg-baby-Heisenberg-morphism} and \eqref{eq:dg-star-morphism}. 
This produces a $\HoDGCat$-enriched bicategory where 
\eqref{eq:dg-baby-Heisenberg-morphism} and \eqref{eq:dg-star-morphism}
are homotopy equivalences:

\begin{Definition}\label{def:hcat}
	The \emph{Heisenberg category} $\hcat\basecat$ of $\basecat$ 
	is the Drinfeld quotient of the h-perfect hull\index{perfect hull} of $\hcat*\basecat$ by 
	the two-sided ideal generated by the $1$-morphisms
	\[
	\cone\bigg(\PP_{Sa} \xrightarrow{\tikz[scale=0.5] \draw[->](0,0) -- node[serre, pos=0.5] {} (0,1);} \RR_a \biggr)
	\]
	\[
	\cone\biggl(
	\PP_{b}\QQ_a \oplus \bigl( \Hom(a,b) \otimes_\kk \hunit \bigr) 
	\xrightarrow{
		\left[
		\begin{tikzpicture}[baseline={(0,0.15)},scale=0.5]
			\draw[->](0.5,0) -- (1.5,1);
			\draw[->](0.5,1) -- (1.5,0);
		\end{tikzpicture}
		\,,\,
		\,\psi_2
		\right]
	}
	\QQ_a \PP_{b}
	\biggr)
	\]
	for all $a, b \in \basecat$.
\end{Definition}

The graded homotopy category $H^*(\A)$ of a \dg category $\A$ is defined to have the same objects as $\A$ and morphism spaces $\Hom_{H^*(\A)}(a,b) = \bigoplus_{i \in \ZZ} H^i(\Hom_\A(a,b))$.
The graded homotopy category $H^*(\hcat\basecat)$ of $\hcat\basecat$ is similarly defined by replacing the $1$-morphism categories with their graded homotopy categories. In particular, each $\Hom_{H^*(\hcat\basecat)}(\ho,\ho*)$ is a Karoubian category.
In $H^*(\hcat\basecat)$ 
one no longer has to distinguish between the $1$-morphisms $\PP$ and $\RR$ and
thus one recovers the formalism of Chapter~\ref{sec:additive-Heisenberg-2cat}, including the labels on cups and caps.

\begin{Lemma}\label{lem:up_down_braids-homotopy}
	Relations~\eqref{eq:up_down_braids-add} hold in $H^*(\hcat\basecat)$.
\end{Lemma}

\begin{proof}
	The left-hand relation is just \eqref{eq:up_down_braid-dg} after identifying $\RR_a$ with $\PP_{Sa}$ and relabeling.
	
	Relations~\ref{eq:circle_and_curl-dg} and \eqref{eq:up_down_braid-dg} together with the curl relation\index{curl relation}s of Lemma~\ref{lem:basic-relations-dg} show that 
	\[
	\begin{bmatrix}
		\begin{tikzpicture}[scale=0.5]
			\draw[<-] (0.5,0) -- (1.5,1);
			\draw[<-] (0.5,1) -- (1.5,0);
		\end{tikzpicture}
		\\
		\psi_1
	\end{bmatrix}
	\circ
	\left[
	\begin{tikzpicture}[baseline={(0,0.15)},scale=0.5]
		\draw[->] (0.5,0) -- (1.5,1);
		\draw[->] (0.5,1) -- (1.5,0);
	\end{tikzpicture}
	\,,\,
	\,\psi_2
	\right]
	\colon
	\PP_{b}\QQ_a \oplus \bigl( \Hom(a,b) \otimes_\kk \hunit \bigr) \to
	\PP_{b}\QQ_a \oplus \bigl( \Hom(a,b) \otimes_\kk \hunit \bigr)
	\]
	is the identity.
	Since in $H^*(\hcat\basecat)$ the $2$-morphism \eqref{eq:dg-baby-Heisenberg-morphism} is an isomorphism, the other composition
	\[
	\left[
	\begin{tikzpicture}[baseline={(0,0.15)},scale=0.5]
		\draw[->] (0.5,0) -- (1.5,1);
		\draw[->] (0.5,1) -- (1.5,0);
	\end{tikzpicture}
	\,,\,
	\,\psi_2
	\right]
	\circ
	\begin{bmatrix}
		\begin{tikzpicture}[scale=0.5]
			\draw[<-] (0.5,0) -- (1.5,1);
			\draw[<-] (0.5,1) -- (1.5,0);
		\end{tikzpicture}
		\\
		\psi_1
	\end{bmatrix}
	=
	\begin{tikzpicture}[baseline={(0,0.4)}, scale=0.5]
		\draw[<-] (0,0) -- (0.5,0.5) to[out=45, in=-45] (0.5,1.5) -- (0,2);
		\draw[<-] (1,2) -- (0.5,1.5) to[out=225, in=135] (0.5,0.5) -- (1,0);
	\end{tikzpicture}  
	+ \Psi(\id)
	\colon \QQ_a\PP_{b} \to \QQ_a\PP_{b}
	\]
	is also the identity, as required.
\end{proof}

\begin{Corollary}\label{cor:graded-to-graded}
	There exists a canonical $2$-functor
	\[
	\hcatadd{H^*(\basecat)} \to H^*(\hcat{\basecat}).
	\]
\end{Corollary}

\begin{proof}
	As all relations in $\hcatadd*{H^*(\basecat)}$ are satisfied in $H^*(\hcat{\basecat})$, there exists a canonical functor $\hcatadd*{H^*(\basecat)} \to H^*(\hcat{\basecat})$.
	Taking Karoubi completion\index{Karoubi completion} gives the desired functor.
\end{proof}

\begin{Example}\label{ex:Khovanov-dg}
	Let $\basecat = \kk$, the field $\kk$ considered as a single-object
	\dg category concentrated in degree $0$.  The Serre functor $S$
	on $\basecat$ is the identity. We have $\basecat =
	H^*(\basecat)$, the additive construction $\hcatadd{\basecat}$
	is Khovanov's categorification of the infinite Heisenberg algebra
	\cite{khovanov2014heisenberg}, and the $2$-functor from
	Corollary~\ref{cor:graded-to-graded} is a fully faithful embedding of
	graded $2$-categories. In the \dg construction we take the perfect 
	hulls of the categories of $1$-morphisms, so 
	the $1$-morphisms in $H^*(\hcat{\basecat})$ are not only words in 
	$\PP$ and $\QQ$ and their idempotents, but also finite complexes
	thereof. The category $\hcat\basecat$ is hence a \dg enhanced
	triangulated hull of Khovanov's categorification.
	The isomorphism \eqref{eq:dg-baby-Heisenberg-morphism} in $\Hzero(\hcat\basecat)$ recovers the defining relation with central charge $k=-1$ from \cite[(1.5)]{brundan2018degenerate}, which was shown to be an alternative of Khovanov's presentation. We expect that our construction has analogues for central charges $k \neq -1$.
\end{Example}

\begin{Example}\label{ex:CautisLicata_part1}
	Consider a finite subgroup $\Gamma$ of $\SL(2,\CC)$ with corresponding simple surface singularity $Y = \as 2/\Gamma$ and minimal resolution $X$.
	Let $\catDGCoh{X}$ be the \dg enhanced bounded derived category of coherent sheaves on $X$.
	For $\cat V$ the full subcategory of $\catDGCoh{X}$ consisting of
	sheaves supported on the exceptional divisor $E$,  the category
	$\hcat\basecat$ is a \dg enhancement of the category
	$\mathcal{H}^\Gamma$ introduced by Cautis--Licata
	\cite[Section~6]{cautis2012heisenberg}.
	
	Indeed, the exceptional divisor $E$ decomposes into $(-2)$-curves $E_i$ labeled by the non-trivial irreducible representations of $\Gamma$.
	Let $I_\Gamma$ be the vertices of the McKay quiver of $\Gamma$.
	Denote by $0$ the vertex corresponding to the trivial representation.
	For $i \in I_\Gamma$ define 
	\[
	\sheaf E_i = \begin{cases}
		\sO_E[-1]     & \text{if $i = 0$} \\
		\sO_{E_i}(-1) & \text{otherwise.}
	\end{cases}
	\]
	The generators $P_i$ and $Q_i$ of $\mathcal{H}^\Gamma$ for $i \in
	I_\Gamma$ correspond in $H^*(\hcat\basecat)$ to $1$-morphisms $\PP_{\sheaf E_i}$ and
	$\QQ_{\sheaf E_i}[ 1 ]$, respectively.
	As $X$ is Calabi--Yau, its Serre functor is $[2]$.
	Thus the shifts chosen above reproduce the grading on turns defined in \cite[Section~6.1]{cautis2012heisenberg}:
	\[
	\begin{tikzpicture}[baseline=0]
		\draw[->] (0,0) node[below] {$P_i$} arc[start angle=180, end angle=0, radius=.5] node[midway, label=above:{$\id[ 1]$}] {} node[below] {$Q_i$};
	\end{tikzpicture},\quad
	\begin{tikzpicture}[baseline=0]
		\draw[->] (0,0) node[below] {$P_i$} arc[start angle=0, end angle=180, radius=.5] node[midway, label=above:{$\id[ -1]$}] {} node[below] {$Q_i$};
	\end{tikzpicture},\quad
	\begin{tikzpicture}[baseline=0]
		\draw[->] (0,0) node[above] {$Q_i$} arc[start angle=0, end angle=-180, radius=.5] node[midway, label=below:{$\id[-1]$}] {} node[above] {$P_i$};
	\end{tikzpicture},\quad
	\begin{tikzpicture}[baseline=0]
		\draw[->] (0,0) node[above] {$Q_i$} arc[start angle=-180, end angle=0, radius=.5] node[midway, label=below:{$\id[ 1]$}] {} node[above] {$P_i$};
	\end{tikzpicture},\quad
	\]
	One has
	\begin{equation}\label{eq:cl_homs}
		\Hom^*(\sheaf E_i,\, \sheaf E_j) = 
		\begin{cases} 
			\CC \oplus \CC[-2], & i = j \\
			\CC[-1],            & |i - j | = 1 \\
			0,                  & \text{otherwise.}
		\end{cases}
	\end{equation}
	Thus a dot on a $2$-morphism in $\mathcal{H}^\Gamma$ corresponds to a
	basis vector of either $\CC[-2]$ or $\CC[-1]$. Picking such a basis, one obtains a $2$-functor
	\[
	\mathcal{H}^\Gamma \to H^*(\hcat\basecat)
	\]
	factoring through the $2$-functor $\hcatadd{H^*(\basecat)} \to H^*(\hcat\basecat)$ of Corollary~\ref{cor:graded-to-graded}.
	
	Equivalently by \cite[Theorem~2.3]{KapranovVasserot:2000:KleinianSingularities}, instead of the sheaves $\sheaf E_i$, one could use the irreducible representations $V_i$ of $\Gamma$ considered as skyscraper sheaves at the origin on the quotient stack $[\as 2/\Gamma]$.
	In this setting one works in the ambient category $\catDGCoh{[\as 2/\Gamma]}$, see also Example~\ref{ex:CautisLicata_part2}.
\end{Example}

\chapter{Structure of the Heisenberg Category}\label{sec:structure}

In this section we deduce a number of properties of the Heisenberg category and we investigate its relationship with the classical Heisenberg algebra.

\section{The Heisenberg commutation relations: \dg level}\label{subsec:cat_heisenberg_relations}

As observed in Remark~\ref{rem:symmetric_group_action-add}, the
symmetric group relations \eqref{eq:symmetric_group_relations-dg}
give us a canonical morphism $\kk[\SymGrp n] \to \End(\QQ_a^n)$.
Similarly, by Lemma~\ref{lem:basic-relations-dg} there are morphisms to $\End(\PP_a^n)$ and $\End(\RR_a^n)$.
Endomorphisms in the image of these maps are made up of unlabelled
strands, thus they are closed and of degree $0$. 

\begin{Remark}\label{rem:Xiprime}
	The homomorphisms $\kk[\SymGrp n] \to \End(\PP_a^n)$ vary in a
	family over $\basecat^{\otimes n}$.
	That is, there exist natural functors
	\[
	\Xi^{\prime\PP}_{\ho,\ho+n} \colon \sym^n\basecat \to \Hom_{\hcat\basecat}(\ho,\, \ho+n),
	\]
	where $\sym^n\basecat := \basecat^{\otimes n} \rtimes \SymGrp n$ is
	the semi-direct product of
	Definition~\ref{def:semidirect-product}.
	On objects $\Xi^{\prime\PP}_{\ho,\ho+n}$ is given by
	\[
	\Xi^{\prime\PP}_{\ho,\ho+n}(a_1 \otimes \dots \otimes a_n) = \PP_{a_1} \dotsm \PP_{a_n}
	\]
	and on morphisms by sending 
	\[
	(\alpha_1 \otimes \dots \otimes \alpha_n,\, \sigma)
	\quad \text{for} \quad \alpha_i \in \Hom_{\basecat}(a_{\sigma^{-1}(i)},\, b_i),\, \sigma \in \SymGrp n
	\]
	to the braid corresponding to $\sigma$ followed by parallel vertical
	strands dotted with $\alpha_1, \dots, \alpha_n$:
	\[
	\begin{tikzpicture}[baseline={(0,0.6)}, scale=0.67]
		\draw[->] (-1,0) node[below] {$\PP_{a_1}$} -- (-1,0.2) to[out=90, in=270] (1,2) -- (1,2.5) node[label=left:{$\alpha_{n-1}$}, dot, pos=0.25]{} node[above] {$\PP_{b_{n-1}}$} ;
		\draw[->] (0,0) node[below] {$\PP_{a_2}$} -- (0,0.2) to[out=90, in=270] (2,2) -- (2,2.5) node[label=right:{$\alpha_n$}, dot, pos=0.25]{} node[above] {$\PP_{b_n}$} ;
		\draw[->] (2,0) node[below] {$\PP_{a_n}$} -- (2,0.2) to[out=90, in=270] (-1,2) -- (-1,2.5) node[label=left:{$\alpha_1$}, dot, pos=0.25]{} node[above] {$\PP_{b_1}$} ;         
		\draw (1,-0.15) node[below] {$\cdots$};
		\draw (-0.15,2.65) node[above] {$\cdots$};
	\end{tikzpicture}
	\]
	Similarly, we have canonical functors $\Xi^{\prime\RR}_{\ho,\ho+n}$ sending $a_1 \otimes \dots \otimes a_n$ to $\RR_{a_1} \dotsm \RR_{a_n}$ and contravariant functors $\Xi^{\prime\QQ}_{\ho,\ho+n}$ sending $a_1 \otimes \dots \otimes a_n$ to $\QQ_{a_1} \dotsm \QQ_{a_n}$.
	
	We can further let $\ho$ and $n$ vary by defining a $2$-category
	$\bicat{Sym}_\basecat$ with objects $\ho \in \ZZ$, $1$-morphism
	categories $\Hom_{\bicat{Sym}_\basecat}(\ho,\, \ho+n) = \sym^n\basecat$ and 
	$1$-composition given by the functors $\sym^{n_1}\basecat \otimes
	\sym^{n_2}\basecat \to \sym^{n_1+n_2}\basecat$ induced by 
	$\SymGrp{n_1} \times \SymGrp{n_2} \hookrightarrow \SymGrp{n_1+n_2}$.
	We then have a natural functor $\Xi^{\PP}\colon \bihperf(\bicat{Sym}_\basecat) \to \hcat\basecat$ and similarly for $\QQ$ and $\RR$.
\end{Remark}

Let
\[
e := e_\triv := \frac{1}{n!} \sum_{\sigma \in \SymGrp n} \sigma \in \kk[\SymGrp n]
\]
be, as in Section~\ref{subsec:idempotent-hcat-add}, the symmetriser
idempotent in $\kk[\SymGrp n]$. Where we work with no other
idempotents of $\kk[\SymGrp n]$ and no confusion is possible, 
we use the shorter notation $e$ for $e_\triv$. 
Denote its image under any of the above maps again by $e$.
The maps $e$ are (strict) idempotent endomorphisms of $\PP_a^n$,
$\QQ_a^n$ and $\RR_a^n$ respectively, and hence split in
$H^*(\hcat\basecat)$. A standard construction gives natural
representatives of the corresponding homotopy direct summands.

\begin{Definition}
	\label{defn-symmetric-powers-of-PP-QQ-and-RR}
	Let $\PP_a^{(n)}$, $\QQ_a^{(n)}$ and
	$\RR_a^{(n)}$ be the convolutions of the twisted
	complexes
	\begin{align*}
		\PP_a^{(n)} :=  \left\{
		\dots 
		\xrightarrow{\; e \;}
		\PP_a^{n}
		\xrightarrow{1 - e}
		\PP_a^{n}
		\xrightarrow{\; e \;}
		\PP_a^{n}
		\xrightarrow{1 - e}
		\underset{\degzero}{\PP_a^{n}}
		\right\},
		\\
		\QQ_a^{(n)} :=  \left\{
		\dots 
		\xrightarrow{\; e \;}
		\QQ_a^{n}
		\xrightarrow{1 - e}
		\QQ_a^{n}
		\xrightarrow{\; e \;}
		\QQ_a^{n}
		\xrightarrow{1 - e}
		\underset{\degzero}{\QQ_a^{n}}
		\right\},
		\\
		\RR_a^{(n)} :=  \left\{
		\dots 
		\xrightarrow{\; e \;}
		\RR_a^{n}
		\xrightarrow{1 - e}
		\RR_a^{n}
		\xrightarrow{\; e \;}
		\RR_a^{n}
		\xrightarrow{1 - e}
		\underset{\degzero}{\RR_a^{n}}
		\right\}.
	\end{align*}
	These are h-projective and perfect modules over 
	{ the}
	$1$-morphism categories of $\hcat*\basecat$. They are h-projective since 
	bounded above complexes of representable modules are semifree.
	They are perfect since in the homotopy categories they 
	are the direct summands of $\PP_a^{n}$, $\QQ_a^{n}$, and $\RR_a^{n}$ 
	defined by the idempotents $e$. Thus, being h-projective
	and perfect, these modules define $1$-morphisms of $\hcat\basecat$
	which we also denote by
	$\PP_a^{(n)}$, $\QQ_a^{(n)}$ and $\RR_a^{(n)}$. 
\end{Definition}

We can now state the main result of this section:

\begin{Theorem}\label{thm:cat_heisenberg_relations-dg}\leavevmode
	\begin{enumerate}
		\item\label{it:cat_heisenberg_relations-dg1}
		\label{item-heisenberg-relations-dg}
		For any $a, b \in \basecat$ and $n, m \in \NN$ the following holds in $\hcat\basecat$:
		\[
		\PP_a^{(m)}\PP_{b}^{(n)} \cong \PP_{b}^{(n)} \PP_a^{(m)}, \quad
		\QQ_a^{(m)}\QQ_{b}^{(n)} \cong \QQ_{b}^{(n)} \QQ_a^{(m)}.
		\]
		\item\label{it:cat_heisenberg_relations-dg2} 
		\label{item-heisenberg-relations-homotopy}
		For any $a, b \in \basecat$ and $n, m \in \NN$  there exists a
		homotopy equivalence in $\hcat\basecat$
		\begin{equation}\label{eq:Heisenberg_map}
			\bigoplus_{i=0}^{\mathclap{\min(m,n)}}\ \Sym^i \Hom_{\basecat}(a, b) \otimes_\kk \PP_{b}^{(n-i)} \QQ_a^{(m-i)} 
			\to
			\QQ_a^{(m)}\PP_{b}^{(n)}, 
		\end{equation}
		and thus the following holds in $H^*(\hcat\basecat)$:
		\[
		\QQ_a^{(m)}\PP_{b}^{(n)} \cong\ \bigoplus_{i=0}^{\mathclap{\min(m,n)}}\ \Sym^i\Hom_{H^*(\basecat)}(a,b) \otimes_\kk \PP_{b}^{(n-i)}\QQ_a^{(m-i)}.
		\]
	\end{enumerate}
\end{Theorem}

\begin{Remark}
	Dually, one can formulate a version of Theorem~\ref{thm:cat_heisenberg_relations-dg}, using the $1$-morphisms $\RR$ instead of $\PP$.
	That is, one has isomorphisms
	\[
	\RR_a^{(m)}\RR_{b}^{(n)} \cong \RR_{b}^{(n)} \RR_a^{(m)}
	\]
	and a homotopy equivalence $2$-morphism
	\begin{equation} 
		\label{eq:Heisenberg_map_R_version}
		\QQ_a^{(m)}\RR_{b}^{(n)}
		\to
		\ \bigoplus_{i=0}^{\mathclap{\min(m,n)}}\ \Sym^i \Hom_{\basecat}(b, a)^* \otimes_\kk \RR_{b}^{(n-i)} \QQ_a^{(m-i)}.
	\end{equation} 
	In the  graded homotopy category, identifying 
	$\Hom_{H^*(\basecat)}(b,a)^*$ with $\Hom_{H^*(\basecat)}(a, Sb)$ and
	$\RR_{b}$ with $\PP_{Sa}$ identifies 
	\eqref{eq:Heisenberg_map_R_version} with 
	\eqref{eq:Heisenberg_map}.
\end{Remark}

\begin{Remark}\label{rem:why_sym}
	The appearance of the symmetric powers of $\Hom_{\basecat}(a,b)$ 
	is related to the following observation.
	Since $\sigma e = e = e \sigma$ for any $\sigma \in \SymGrp{n}$, one sees that any crossings of parallel strands can be absorbed into the symmetrisers.
	In particular, for $\alpha, \beta \in \Hom(a,b)$ one has
	\[
	\begin{tikzpicture}[baseline=-0.5ex]
		\node[sym] (top) at (0,1.25) {$\PP_b^{(2)}$};
		\node[sym] (bottom) at (0,-1.25) {$\PP_a^{(2)}$};
		\draw[->] (bottom.120) -- node[label=left:{$\alpha$}, near start, dot] {} (top.240);
		\draw[->] (bottom.60) -- node[label=right:{$\beta$}, near end, dot] {} (top.300);
	\end{tikzpicture}
	=
	\begin{tikzpicture}[baseline=-0.5ex]
		\node[sym] (top) at (0,1.25)  {$\PP_b^{(2)}$};
		\node[sym] (bottom) at (0,-1.25) {$\PP_a^{(2)}$};
		\draw[->] (bottom.120) -- node[label=left:{$\alpha$}, pos=0.1, dot] {} (top.300);
		\draw[->] (bottom.60) -- node[label=right:{$\beta$}, pos=0.3, dot] {} (top.240);
	\end{tikzpicture}
	=
	\begin{tikzpicture}[baseline=-0.5ex]
		\node[sym] (top) at (0,1.25)  {$\PP_b^{(2)}$};
		\node[sym] (bottom) at (0,-1.25) {$\PP_a^{(2)}$};
		\draw[->] (bottom.120) -- node[label=right:{$\alpha$}, pos=0.7, dot] {} (top.300);
		\draw[->] (bottom.60) -- node[label=left:{$\beta$}, pos=0.85, dot] {} (top.240);
	\end{tikzpicture}
	=
	\begin{tikzpicture}[baseline=-0.5ex]
		\node[sym] (top) at (0,1.25)  {$\PP_b^{(2)}$};
		\node[sym] (bottom) at (0,-1.25) {$\PP_a^{(2)}$};
		\draw[->] (bottom.60) -- node[label=right:{$\alpha$}, pos=0.7, dot] {} (top.300);
		\draw[->] (bottom.120) -- node[label=left:{$\beta$}, pos=0.85, dot] {} (top.240);
	\end{tikzpicture}
	=
	(-1)^{|\alpha|\cdot|\beta|}
	\begin{tikzpicture}[baseline=-0.5ex]
		\node[sym] (top) at (0,1.25) {$\PP_b^{(2)}$};
		\node[sym] (bottom) at (0,-1.25) {$\PP_a^{(2)}$};
		\draw[->] (bottom.60) -- node[label=right:{$\alpha$}, near end, dot] {} (top.300);
		\draw[->] (bottom.120) -- node[label=left:{$\beta$}, near start, dot] {} (top.240);
	\end{tikzpicture}.
	\]
	Thus $i$ parallel strands are naturally labeled by elements of $\Sym^i\Hom_{\basecat}(a,b)$ (using the Koszul sign convention as always).
\end{Remark}

In the remainder of this subsection we set up the maps occurring in
Theorem~\ref{thm:cat_heisenberg_relations-dg} and
prove the relations in
Theorem~\ref{thm:cat_heisenberg_relations-dg}\ref{item-heisenberg-relations-dg}
which hold on the \dg level. In the next subsection we prove the relation
in Theorem~\ref{thm:cat_heisenberg_relations-dg}\ref{item-heisenberg-relations-homotopy}
which holds on the homotopy level. 

We begin with several remarks detailing some \dg $2$-morphisms between 
$\PP^{(n)}$s, $\QQ^{(n)}$s and $\RR^{(n)}$s which can be induced 
from those between $\PP^{n}$s, $\QQ^{n}$s and $\RR^{n}$s:

\begin{Remark}
	\label{remark-canonical-morphisms-in-out-(n)s}
	
	We have the canonical $2$-morphisms defined by $e$ on degree $0$ terms:
	\[ \PP_a^{n} \xrightarrow{e} \PP_a^{(n)} \xrightarrow{e} \PP_a^{n}, \]
	Any $2$-morphism in or out of $\PP_a^{n}$ induces via
	pre- or postcomposition a $2$-morphism in or out of $\PP_a^{(n)}$. 
	
	In the homotopy category, where as in any triangulated category 
	all idempotents are split, $\PP_a^{(n)}$ is a direct summand of 
	$\PP_a^{n}$. The canonical $2$-morphisms above become 
	the morphisms of inclusion of and projection onto this direct summand.  
	Thus, pre- or postcompositions with them are DG
	equivalents of taking the component corresponding to this direct
	summand. 
	
	The same holds for $\QQ$s, $\RR$s, and any $1$-composition of these.  
\end{Remark}

\begin{Remark}
	\label{remark-between-(n)s-symmetric-group-images}
	
	Let $\alpha\colon \PP_a^{n} \rightarrow \PP_a^{n}$ in 
	$\hcat*\basecat$. Recall that when illustrating maps of
twisted complexes\index{twisted complex} we only draw their non-zero components.  In the homotopy category $\PP_a^{n}$ splits as $\PP_a^{(n)} \oplus 
		\overline{\PP_a^{(n)}}$, where 
		$\overline{\PP_a^{(n)}}$ is the complement summand. By Remark~\ref{remark-canonical-morphisms-in-out-(n)s}, 
		the $2$-morphism 
		\begin{equation*}
			e\alpha{e} \coloneqq
			\begin{tikzcd}
				\dots
				\ar{r}{\;e\;}
				&
				\PP_a^{n} 
				\ar{r}{1 - e}
				&
				\PP_a^{n} 
				\ar{r}{\;e\;}
				&
				\PP_a^{n} 
				\ar{r}{1 - e}
				&
				\PP_a^{n} 
				\ar{d}{e{\alpha}e}
				\\
				\dots
				\ar{r}{\;e\;}
				&
				\PP_a^{n} 
				\ar{r}{1 - e}
				&
				\PP_a^{n} 
				\ar{r}{\;e\;}
				&
				\PP_a^{n} 
				\ar{r}{1- e}
				&
				\PP_a^{n}.
			\end{tikzcd}
		\end{equation*}
		gives in the homotopy category the $\PP_a^{(n)} \rightarrow \PP_a^{(n)}$ 
		component of $\alpha$. Note, that so does 
		\begin{equation*}
			\alpha{e} \coloneqq
			\begin{tikzcd}
				\dots
				\ar{r}{\;e\;}
				&
				\PP_a^{n} 
				\ar{r}{1 - e}
				&
				\PP_a^{n} 
				\ar{r}{\;e\;}
				&
				\PP_a^{n} 
				\ar{r}{1 - e}
				&
				\PP_a^{n} 
				\ar{d}{{\alpha}e}
				\\
				\dots
				\ar{r}{\;e\;}
				&
				\PP_a^{n} 
				\ar{r}{1 - e}
				&
				\PP_a^{n} 
				\ar{r}{\;e\;}
				&
				\PP_a^{n} 
				\ar{r}{1- e}
				&
				\PP_a^{n}.
			\end{tikzcd}
		\end{equation*}
		
		However on the \dg level $\alpha{e}$ contains extra information. Indeed, both morphisms
		are defined by a map of twisted
complexes\index{twisted complex} with a single component $\PP_a^{n} \rightarrow \PP_a^{n}$. These two maps $\PP_a^{n} \rightarrow \PP_a^{n}$, 
		$\alpha{e}\alpha$ and $\alpha{e}$, 
		are different even in the homotopy category: $e\alpha{e}$ has a single component $\PP_a^{(n)} \rightarrow \PP_a^{(n)}$, 
		while $\alpha{e}$ also has
		a $\PP_a^{(n)} \rightarrow \overline{\PP_a^{(n)}}$ component. 
	
	If $\alpha$ supercommutes with $e$ we have another $2$-morphism 
	$\PP_a^{(n)} \rightarrow \PP_a^{(n)}$ given by
	\begin{equation*}
		\tilde{\alpha} \coloneqq
		\begin{tikzcd}
			\dots
			\ar{r}{\;e\;}
			&
			\PP_a^{n} 
			\ar{r}{1 - e}
			\ar{d}{\alpha}
			&
			\PP_a^{n} 
			\ar{r}{\;e\;}
			\ar{d}{\alpha}
			&
			\PP_a^{n} 
			\ar{r}{1 - e}
			\ar{d}{\alpha}
			&
			\PP_a^{n} 
			\ar{d}{\alpha}
			\\
			\dots
			\ar{r}{\;e\;}
			&
			\PP_a^{n} 
			\ar{r}{1 - e}
			&
			\PP_a^{n} 
			\ar{r}{\;e\;}
			&
			\PP_a^{n} 
			\ar{r}{1- e}
			&
			\PP_a^{n}.
		\end{tikzcd}
	\end{equation*}
	It is homotopic to $\alpha{e}$ and thus to $e\alpha{e}$: consider the degree
	$-1$ twisted complex map comprising degree $i \rightarrow (i-1)$
	components given by $\alpha$. As operations on $\alpha$, 
	both preserve the degree and commute with the differential. 
	In particular, if $\alpha$
	is closed of degree $0$, so are $\alpha{e}$ and $\tilde{\alpha}$.

	When $\alpha$ is an image of some $\sigma \in S_n$ under $\Xi^{\PP}$, 
	it commutes with $e$ because in $\kk[S_n]$
	\[ \sigma e = e = e \sigma. \]
	Hence $\tilde{\sigma}$ is well defined and homotopic to 
	$e\sigma{e} = e$. Thus, all $\tilde{\sigma}$ are homotopic to 
	$\id = \tilde{\id}$. 
	
	Similar considerations apply to a $1$-composition of
	several powers of $\PP$s. Let $\sigma \in S_m$ and
	let $\alpha$ be a $2$-morphism 
	\[ \alpha \colon \PP_{a_1}^{n_1} \PP_{a_2}^{n_2} \dots \PP_{a_m}^{n_m} 
	\rightarrow \PP_{a_{\sigma(1)}}^{n_{\sigma(1)}} 
	\PP_{a_{\sigma(2)}}^{n_{\sigma(2)}}
	\dots \PP_{a_{\sigma(m)}}^{n_{\sigma(m)}}.\]
	If $\alpha$ supercommutes with the symmetriser $e$ of each 
	$\PP_{a_i}^{n_i}$, then we have a map 
	\[ \tilde{\alpha} \colon \PP_{a_1}^{(n_1)} \PP_{a_2}^{(n_2)} \dots
	\PP_{a_m}^{(n_m)} 
	\rightarrow \PP_{a_{\sigma(1)}}^{(n_{\sigma(1)})} 
	\PP_{a_{\sigma(2)}}^{(n_{\sigma(2)})}
	\dots \PP_{a_{\sigma(m)}}^{(n_{\sigma(m)})}\]
	defined by the twisted complex map comprising degree $i \rightarrow i$ 
	components $\sum \alpha$.  Note that $\PP_{a_1}^{(n_1)}
	\PP_{a_2}^{(n_2)} \dots \PP_{a_m}^{(n_m)}$ is the product of the 
	twisted complexes defining each individual $\PP_{a_j}^{(n_j)}$ and thus a
	twisted complex whose degree $i$ element is the direct sum 
	$\bigoplus_{i_1 + \dots + i_n = i} 
	\PP_{a_1}^{n_1} \PP_{a_2}^{n_2} \dots \PP_{a_m}^{n_m}$
	where the multi-index $(i_1, \dots, i_n)$ gives the degrees in each 
	twisted complex of the product where each $\PP_{a_j}^{(n_j)}$ comes
	from. By $\sum \alpha$ above we mean the map sending each  
	$(i_1, \dots, i_n)$-indexed summand of the source to the 
	$(i_1, \dots, i_n)$-indexed summand of the target via $\alpha$. 
	In the simple case when $\sigma = \id$ and 
	$\alpha = \alpha_1 \alpha_2 \dots \alpha_n$ with 
	$\alpha_i\colon \PP_{a_i}^{n_i} \rightarrow \PP_{a_i}^{n_i}$ 
	we get $\tilde{\alpha} = \tilde{\alpha}_1 \dots \tilde{\alpha}_n$.  
	
	Now, let $n = \sum_{i = 0}^{m} n_i$, let $\phi: S_m \hookrightarrow S_n$ be
	the embedding of $S_m$ as the permutation group\index{permutation group} of $n_i$-tuples of
	elements, and let $S_{n_1} \times \dots \times S_{n_m} \leq S_n$
	be the subgroup of permutations which respect the partition $(n_1,
	\dots, n_m)$. If $\rho \in S_n$ is such that $\Xi^{\PP}(\rho)$ 
	is a morphism 
	\[ \PP_{a_1}^{n_1} \PP_{a_2}^{n_2} \dots \PP_{a_m}^{n_m} 
	\rightarrow \PP_{a_{\sigma(1)}}^{n_{\sigma(1)}} 
	\PP_{a_{\sigma(2)}}^{n_{\sigma(2)}}
	\dots \PP_{a_{\sigma(m)}}^{n_{\sigma(m)}}, \]
	then $\rho = \phi(\sigma) \tau$ for some 
	\[ \tau = (\tau_1, \dots, \tau_m) \in S_{n_1} \times \dots \times S_{n_m}. \]
	Indeed, by its definition the $2$-morphism $\Xi^{\PP}(\rho)$ has only
	unmarked strings which can only go from $\PP_{a_i}$ to $\PP_{a_i}$ and 
	not some other $\PP_{a_j}$. Thus $\rho$ must send each $n_i$-tuple in 
	the partition $(n_1, \dots, n_m)$ of $n$, in some order,
	to the corresponding $n_i = n_{\sigma(\sigma^{-1}(i))}$-tuple 
	in the permuted partition $(n_{\sigma(1)}, \dots, n_{\sigma(m)})$. 
	Thus doing $\rho$ is the same as individually permuting the elements
	of each $n_i$-tuple by some $\tau \in S_{n_1} \times \dots \times
	S_{n_m}$ and then doing $\phi(\sigma)$ to permute the $n_i$-tuples.
	
	Now, $\tau$ commutes with the symmetrisers $e_{n_i} \in \kk[\SymGrp{n_i}] \subseteq \kk[\SymGrp n]$ since
	\[  \tau e_{n_i} = (\tau_1, \dots, \tau_{i-1}, e_{n_i},
	\tau_{i+1}, \dots, \tau_m) = e_{n_i} \tau. \]
	The corresponding map 
	\[ \tilde{\tau}\colon  
	\PP_{a_1}^{(n_1)} \PP_{a_2}^{(n_2)} \dots \PP_{a_m}^{(n_m)} \rightarrow 
	\PP_{a_1}^{(n_1)} \PP_{a_2}^{(n_2)} \dots \PP_{a_m}^{(n_m)}, \]
	is the $1$-composition of the maps 
	$\tilde{\tau_i}\colon \PP_{a_i}^{n_i} \rightarrow \PP_{a_i}^{n_i}$
	described above, each of which is homotopic to $\id$. Thus
	$\tilde{\tau}$ itself is homotopic to $\id$. 
	
	On the other hand, $\phi(\sigma)$ commutes with
	the symmetrisers $e_{n_i}$ since their action is contained
	within each $n_i$-tuple. The corresponding map 
	\begin{equation}
		\label{eqn-tilde-sigma-isomorphism} 
		\widetilde{\phi(\sigma)}\colon \PP_{a_1}^{(n_1)} \PP_{a_2}^{(n_2)} \dots
		\PP_{a_m}^{(n_m)} 
		\rightarrow \PP_{a_{\sigma(1)}}^{(n_{\sigma(1)})} 
		\PP_{a_{\sigma(2)}}^{(n_{\sigma(2)})}
		\dots \PP_{a_{\sigma(m)}}^{(n_{\sigma(m)})}, 
	\end{equation}
	is then a 
	$2$-isomorphism, whose inverse is 
	$\widetilde{\phi(\sigma^{-1})}$. In particular, 
	in the simplest possible case $m = 2$ and $\sigma = (12)$,
	we get a $2$-isomorphism  
	\begin{equation}
		\label{eqn-tilde-12-isomorphism} 
		\widetilde{\phi(12)}\colon \PP_{a_1}^{(n_1)} \PP_{a_2}^{(n_2)}
		\xrightarrow{\sim} \PP_{a_2}^{(n_2)} \PP_{a_1}^{(n_1)}. 
	\end{equation}
	Similar considerations apply to $1$-compositions of powers of $\QQ$s
	and $\RR$s. 
\end{Remark}

\begin{Remark}
	\label{remark-symmetric-power-labels-on-thick-strands}
	Let $\alpha\colon \PP^n_a \rightarrow \PP_b^n$. Arguing as in
	Remark~\ref{remark-between-(n)s-symmetric-group-images} we see that if
	$\alpha$ commutes with the symmetrisers of $\PP^n_a$ and $\PP^n_b$, it
	defines a $2$-morphism 
	\[ \tilde{\alpha}\colon \PP^{(n)}_a \rightarrow \PP^{(n)}_b. \]
	
	Suppose such $\alpha$ lies in the image of the functor $\Xi^{\PP}$. 
	Then, as per Remark~\ref{rem:Xiprime}, we have
	\[ \alpha = \Xi^{\PP}(\beta)\;\Xi^{\PP}(\sigma), \quad \quad\text{for } 
	\beta \in 
	\Hom_{\basecat^{\otimes n}}(a^n,b^n),\; \sigma \in S_n. 
	\]
	We saw in Remark~\ref{remark-between-(n)s-symmetric-group-images}
	that $\Xi^{\PP}(\sigma)\colon \PP^{n}_a \rightarrow \PP^{n}_a$
	commutes with $e$ and the corresponding map
	$\tilde{\sigma}\colon \PP^{(n)}_a \rightarrow \PP^{(n)}_a$ is
	homotopic to the identity. 
	
	For any $\tau \in S_n$ we have in $\sym^n\basecat$
	\[ \tau \circ \beta = \tau(\beta) \circ \tau, \]
	and hence $\beta$ commutes with $e$ if and only if
	$e(\beta) = \beta$. In other words, if and only if $\beta$ lies 
	in the image of the canonical embedding 
	\[ \psi\colon \Sym^n \Hom_\basecat(a,b) 
	\rightarrow 
	\Hom_{\basecat^{\otimes n}}(a^n,b^n).
	\]
	In particular, for any $\gamma \in \Sym^n \Hom_\basecat(a,b)$, 
	$\psi(\gamma)$ commutes with $e$ and hence its image under
	$\Xi^{\PP}$ gives a well-defined map 
	\[ \widetilde{\psi(\gamma)}\colon \PP^{(n)}_a \rightarrow \PP^{(n)}_b, 
	\quad \quad \gamma \in \Sym^n \Hom_\basecat(a,b). \]
	By the above, up to homotopy, all the maps $\PP^{(n)}_a \rightarrow
	\PP^{(n)}_b$ induced from those in the image of $\Xi^{\PP}$ 
	are of this form. 
\end{Remark}

Throughout this section, we draw diagrams to define morphisms 
between $1$-compositions of $\PP^{(n)}$s and $\QQ^{(n)}$s. 

Any diagram defining a morphism $\alpha$ between the 
corresponding $1$-compositions of $\PP^{n}$s and $\QQ^{n}$s
defines a morphism $e{\alpha}e$ between 
those of $\PP^{(n)}$s and $\QQ^{(n)}$s as detailed in 
Remark~\ref{remark-canonical-morphisms-in-out-(n)s}. 

If $\alpha$ commutes with the symmetriser differentials of source and
target $1$-compositions of $\PP^{n}$s and $\QQ^{n}$s, it furthermore
defines a morphism $\tilde{\alpha}$ between 
$1$-compositions of $\PP^{(n)}$s and $\QQ^{(n)}$s as detailed in 
Remark~\ref{remark-between-(n)s-symmetric-group-images}. 

The two morphisms $e{\alpha}e$ and $\tilde{\alpha}$ thus produced
are homotopic. In this subsection, working on \dg level, we only want to 
work with the $(\widetilde{-})$ construction of
Remark~\ref{remark-between-(n)s-symmetric-group-images}, as it 
can produce termwise \dg isomorphisms\index{DG isomorphism} of 
twisted complexes. 
Thus we only consider the diagrams which commute with the symmetriser 
differentials. In Section~\ref{subsec:cat_heisenberg_relations_homotopy}, 
working in the homotopy category, we employ arbitrary diagrams and use 
the symmetrising $e(-)e$ construction of
Remark~\ref{remark-canonical-morphisms-in-out-(n)s}. 
It produces twisted complex maps concentrated in degree $0$, which
can only be homotopy equivalences. We stress again, that in the
homotopy category there is no difference between the two
constructions. 

It is crucial for our proofs that the construction of a morphism 
between $1$-compositions of $\PP^{(n)}$s and $\QQ^{(n)}$s from a
diagram defining the morphism between the corresponding
$1$-compositions of $\PP^{(n)}$s and $\QQ^{(n)}$s is compatible
with $2$-composition, that is –– with vertical concatenation of
diagrams. For the $(\widetilde{-})$ construction this is automatic. 
For the $e(-)e$ construction this means that any two diagrams
$\alpha$ and $\beta$ we compose must satisfy  
\begin{align}
	\label{eqn-the-condition-for-composition-of-diagrams-between-Pn-and-Qns}
	e{\alpha}e{\beta}e = e{\alpha}{\beta}e 
\end{align}

In this subsection, we use diagrams which commute with the
symmetrisers and use the $(\widetilde{-})$ construction, so this is 
not an issue. In \S\ref{subsec:cat_heisenberg_relations_homotopy}
we use arbitrary diagrams and use the $e(-)e$ construction, 
so we check the condition 
\eqref{eqn-the-condition-for-composition-of-diagrams-between-Pn-and-Qns}
by hand. In Section~\ref{subsec:gfid}, this is a simple idempotent 
absorption argument: the symmetriser idempotent 
of a subgroup can be absorbed into the symmetriser idempotent of the
group. In Section~\ref{subsec:fgid} a more elaborate argument is
necessary and we show that
\eqref{eqn-the-condition-for-composition-of-diagrams-between-Pn-and-Qns}
only holds up to a desired numerical coefficient.

We use the following conventions to simplify the diagrams in 
the context of this section: 

\begin{enumerate}
	\item A box containing $a^n$  at the top or the bottom of the
	diagram denotes both $1$-morphisms $\PP_a^{(n)}$ and $\QQ_a^{(n)}$:
	\begin{center}
		\begin{tikzpicture}
			\node[sym] (tr) at (1,1) {$a^{n}$};
		\end{tikzpicture}
	\end{center}
	We never use type $\RR$ $1$-morphisms, so the orientation of 
	the attached strands makes clear what is meant. 
	
	When such box occurs inside the diagram, it is 
	the symmetriser idempotent $e_{S_n}$. Note that in the context of
	$e(\widetilde{-})e$ construction, the boxes at the top and the bottom
	can also be viewed as occurences of symmetriser idempotents. We mainly
	use this notation to differentiate between the LHS and the RHS of
	the condition
	\eqref{eqn-the-condition-for-composition-of-diagrams-between-Pn-and-Qns}. 
	For example, if we start with diagrams
	\[
	\alpha = 
	\begin{tikzpicture}[baseline=-0.5ex]
		\node (tl) at (-1,1.25)  {$\QQ_a^m$};
		\node (tr) at (1,1.25)   {$\PP_a^n$};
		\node (bl) at (-1,-1.25) {$\PP_a^{n-i}$};
		\node (br) at (1,-1.25)  {$\QQ_a^{m-i}$};
		\draw[down, <-, many] (br.60) -- node[near start, desc] {$m-i$} (tl.240);
		\draw[cc, ->, many] (tl.300) to[out=330, in=210] (tr.240);
		\draw[up, ->, many] (bl.120) -- node[near start, desc] {$n-i$} (tr.300);
	\end{tikzpicture}
	\quad \quad \quad
	\beta =
	\begin{tikzpicture}[baseline=-0.5ex]
		\node (tl) at (-1,1.25)  {$\PP_a^{n-i}$};
		\node (tr) at (1,1.25)   {$\QQ_a^{m-i}$};
		\node (bl) at (-1,-1.25) {$\QQ_a^{m}$};
		\node (br) at (1,-1.25)  {$\PP_a^{n}$};
		\draw[up, ->, many] (br.60) -- node[near end, desc] {$m-i$} (tl.240);
		\draw[down, <-, many] (bl.120) -- node[near end, desc] {$n-i$} (tr.300);
		\draw[cc, <-, many] (bl.60) to[out=30, in=150] (br.120);
	\end{tikzpicture}
	\]
	then the induced morphisms between 
	$\QQ_a^{(m)}\PP_a^{(n)}$ and 
	$\PP_a^{(n-i)}\QQ_a^{(m-i)}$ are
	\[
	e{\alpha}e = 
	\begin{tikzpicture}[baseline=-0.5ex]
		\node[sym] (tl) at (-1,1.25)  {$a^m$};
		\node[sym] (tr) at (1,1.25)   {$a^n$};
		\node[sym] (bl) at (-1,-1.25) {$a^{n-i}$};
		\node[sym] (br) at (1,-1.25)  {$a^{m-i}$};
		\draw[down, <-, many] (br.60) -- node[near start, desc] {$m-i$} (tl.240);
		\draw[cc, ->, many] (tl.300) to[out=330, in=210] (tr.240);
		\draw[up, ->, many] (bl.120) -- node[near start, desc] {$n-i$} (tr.300);
	\end{tikzpicture}
	\quad \quad \quad 
	e{\beta}e =
	\begin{tikzpicture}[baseline=-0.5ex]
		\node[sym] (tl) at (-1,1.25)  {$a^{n-i}$};
		\node[sym] (tr) at (1,1.25)   {$a^{m-i}$};
		\node[sym] (bl) at (-1,-1.25) {$a^{m}$};
		\node[sym] (br) at (1,-1.25)  {$a^{n}$};
		\draw[up, ->, many] (br.60) -- node[near end, desc] {$m-i$} (tl.240);
		\draw[down, <-, many] (bl.120) -- node[near end, desc] {$n-i$} (tr.300);
		\draw[cc, <-, many] (bl.60) to[out=30, in=150] (br.120);
	\end{tikzpicture}
	\]
	and the LHS and the RHS of 
	\eqref{eqn-the-condition-for-composition-of-diagrams-between-Pn-and-Qns}
	are
	\[
	e{\alpha}e{\beta}e = 
	\begin{tikzpicture}[baseline=-0.5ex]
		\node[sym] (tl) at (-1,1.5) {$a^{m}$};
		\node[sym] (tr) at (1,1.5) {$a^{n}$};
		\node[sym] (ml) at (-1,0) {$a^{n-i}$};
		\node[sym] (mr) at (1,0) {$a^{m-i}$};
		\node[sym] (bl) at (-1,-1.5) {$a^{m}$};
		\node[sym] (br) at (1,-1.5) {$a^{n}$};
		\draw[cc, ->, many] (tl.300) to[out=330, in=210] (tr.240);
		\draw[cc, <-, many] (bl.60) to[out=30, in=150] (br.120);
		\draw[down, <-, many] (bl.120) to[out=90, in=270] (mr.270);
		\draw[down, <-, many] (mr.90) to[out=90,in=270] (tl.240);
		\draw[up, ->, many] (br.60) to[out=90, in=270] (ml.270);
		\draw[up, ->, many] (ml.90) to[out=90, in=270] (tr.300);
	\end{tikzpicture}
	\quad \quad \quad 
	e{\alpha}{\beta}e = 
	\begin{tikzpicture}[baseline=-0.5ex]
		\node[sym] (tl) at (-1,1.5) {$a^{m}$};
		\node[sym] (tr) at (1,1.5) {$a^{n}$};
		\node[sym] (bl) at (-1,-1.5) {$a^{m}$};
		\node[sym] (br) at (1,-1.5) {$a^{n}$};
		\draw[cc, ->, many] (tl.300) to[out=330, in=210] node[desc] {$i$} (tr.240);
		\draw[cc, <-, many] (bl.60) to[out=30, in=150] node[desc] {$i$} (br.120);
		\draw[down, <-, many] (bl.120) to[out=30, in=270] (0.75,0) node[desc] {$m-i$} to[out=90, in=330] (tl.240);
		\draw[up, ->, many] (br.60) to[out=150, in=270] (-0.75,0) node[desc] {$n-i$} to[out=90, in=210] (tr.300);
	\end{tikzpicture}.
	\]

	\item Upwards strands are coloured blue, downwards strands red 
	and counterclockwise turns green (clockwise turns will not appear 
	in the argument):
	\begin{center}
		\begin{tikzpicture}
			\node[sym] (t) at (0,1) {$a$};
			\node[sym] (b) at (0,0) {$a$};
			\draw[up, ->] (b) --  (t);
		\end{tikzpicture}\quad\quad
		\begin{tikzpicture}
			\node[sym] (t) at (0,1) {$a$};
			\node[sym] (b) at (0,0) {$a$};
			\draw[down, ->] (t) --  (b);
		\end{tikzpicture}\quad\quad
		\begin{tikzpicture}[baseline=-0.7cm]
			\node[sym] (l) at (0,0) {$a$};
			\node[sym] (r) at (1.5,0) {$a$};
			\draw[cc, ->] (l.300) to[out=330, in=210]  (r.240);
			
		\end{tikzpicture}
	\end{center}
	This colouring is solely for the convenience of the reader 
	and does not have any additional meaning.
	
	\item We denote multiple unadorned parallel strands all starting at
	one box and ending at another box by a single thick strand
	labelled with the strand multiplicity.  
	
	Thus, an upward braid of thick strands of multiplicities
	$n_1, \dots, n_m$ permuting $m$ boxes $a_1^{n_1}$, \dots, $a_m^{n_m}$ 
	defines a $2$-isomorphism between the corresponding $1$-compositions 
	of $\PP_{a_i}^{n_i}$. It depends only on the  
	permutation type $\sigma \in S_m$ of the braid. Moreover, as seen in 
	Remark~\ref{remark-between-(n)s-symmetric-group-images}, 
	it commutes with the symmetrisers and thus defines the
	$2$-isomorphism $\widetilde{\phi(\sigma)}$ of 
	\eqref{eqn-tilde-sigma-isomorphism} 
	between the corresponding
	$1$-compositions of $\PP_{a_i}^{(n_i)}$. 
	
	For example, the $2$-isomorphism $\PP_{a}^{(m)} \PP_{b}^{(n)} \rightarrow 
	\PP_{b}^{(n)} \PP_{a}^{(m)}$ of \eqref{eqn-tilde-12-isomorphism} is
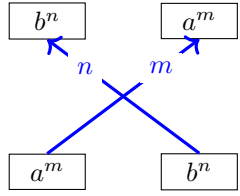
	\begin{equation*}
		\begin{tikzpicture}[baseline=-0.5ex]
			\node[sym] (tl) at (-1,1) {$b^{n}$};
			\node[sym] (tr) at (1,1) {$a^{m}$};
			\node[sym] (bl) at (-1,-1) {$a^{m}$};
			\node[sym] (br) at (1,-1) {$b^{n}$};
			\draw[up, ->, many] (br.90) -- node[near end, desc] {$n$} (tl.270);
			\draw[up, ->, many] (bl.90) -- node[near end, desc] {$m$} (tr.270);
		\end{tikzpicture}
	\end{equation*}

	\item A thick strand from box $a^n$ to box $b^n$ labelled with 
	an element 
	\[\alpha \in \Sym^n \Hom_\basecat(a,b)\]
	denotes the $2$-morphism 
	\[\Xi^\PP(\psi(\alpha))\] of Remark 
	\ref{remark-symmetric-power-labels-on-thick-strands}. 
	As it commutes with the symmetrisers, it defines 
	{ a} $2$-morphism $\widetilde{\psi(\alpha)}$ between the corresponding 
	symmetric powers of $\PP$s or $\QQ$s.

	For example, suppose that $\alpha$ is an elementary symmetric tensor
	\[ \alpha = \alpha_1 \stimes \dotsb \stimes \alpha_n \coloneqq \frac{1}{n!}\sum_{\sigma \in S_n} \alpha_{\sigma(1)} \dotsb  \alpha_{\sigma(n)} .\]
	A thick strand labelled $\alpha$ is the sum of all permutations of
	$n$ parallel strands adorned with the $\alpha_i$s.
	In particular, for even degree $\alpha$ and $\beta$ we have
	\[
	\begin{tikzpicture}[baseline=-0.5ex]
		\node[sym] (top) at (0,0.75) {$a^2$};
		\node[sym] (bottom) at (0,-0.75) {$a^2$};
		\draw[->, up, many] (bottom.90) -- node[dot, label=right:{$\alpha \stimes \beta$}] {} (top.270);
	\end{tikzpicture}
	= 
	\frac{1}{2}
	\left(
	\begin{tikzpicture}[baseline=-0.5ex]
		\node (top) at (0,0.75) {$\PP_b^2$};
		\node (bottom) at (0,-0.75) {$\PP_a^2$};
		\draw[->] (bottom.120) -- node[label=left:{$\alpha$}, near start, dot] {} (top.240);
		\draw[->] (bottom.60) -- node[label=right:{$\beta$}, near end, dot] {} (top.300);
	\end{tikzpicture}
	+ 
	\begin{tikzpicture}[baseline=-0.5ex]
		\node (top) at (0,0.75) {$\PP_b^2$};
		\node (bottom) at (0,-0.75) {$\PP_a^2$};
		\draw[->] (bottom.120) -- node[label=left:{$\beta$}, near start, dot] {} (top.240);
		\draw[->] (bottom.60) -- node[label=right:{$\alpha$}, near end, dot] {} (top.300);
	\end{tikzpicture}
	\right)
	\]
\end{enumerate}

With the above notation in mind, we have immediately:

\begin{proof}[Proof of
	Theorem~\ref{thm:cat_heisenberg_relations-dg}\ref{item-heisenberg-relations-dg}]

	We claim that the $2$-morphisms
	\[
	\begin{tikzpicture}[baseline=-0.5ex]
		\node[sym] (tl) at (-1,1) {$b^{n}$};
		\node[sym] (tr) at (1,1) {$a^{m}$};
		\node[sym] (bl) at (-1,-1) {$a^{m}$};
		\node[sym] (br) at (1,-1) {$b^{n}$};
		\draw[up, ->, many] (br.90) -- node[near end, desc] {$n$} (tl.270);
		\draw[up, ->, many] (bl.90) -- node[near end, desc] {$m$} (tr.270);
	\end{tikzpicture}
	\text{ and }
	\begin{tikzpicture}[baseline=-0.5ex]
		\node[sym] (tl) at (-1,1) {$a^{m}$};
		\node[sym] (tr) at (1,1) {$b^{n}$};
		\node[sym] (bl) at (-1,-1) {$b^{n}$};
		\node[sym] (br) at (1,-1) {$a^{m}$};
		\draw[up, ->, many] (br.90) -- node[near start, desc] {$m$} (tl.270);
		\draw[up, ->, many] (bl.90) -- node[near start, desc] {$n$} (tr.270);
	\end{tikzpicture}. 
	\]
	are inverse to each other. Indeed, as $\widetilde{(-)}$ is
	compatible with the compositions, we can vertically concatenate 
	the diagrams and then apply
	Lemma~\ref{lem:basic-relations-dg}~\ref{it:basic-relations-dg3}
	multiple times to get the claim.
	Thus we have 
	the relation $\PP_a^{(m)}\PP_{b}^{(n)} \cong \PP_{b}^{(n)} \PP_a^{(m)}$. The second relation $\QQ_a^{(m)}\QQ_{b}^{(n)} \cong \QQ_{b}^{(n)} \QQ_a^{(m)}$
	is implied by a similar pair of diagrams but involving
downward strands\index{downward strand}.
\end{proof}

\section{The Heisenberg commutation relations: homotopy
	level}\label{subsec:cat_heisenberg_relations_homotopy}

Next, let us construct the $2$-morphism giving~\eqref{eq:Heisenberg_map}.
Defining a map 
\[
g_i\colon \Sym^i\Hom_{\basecat}(a,b) \otimes_\kk \PP_{b}^{(n-i)}\QQ_a^{(m-i)} \to \QQ_a^{(m)}\PP_{b}^{(n)}.
\]
is equivalent to defining a map
\[
\tilde g_i\colon \Sym^i\Hom_{\basecat}(a,b) \to \Hom_{\hcat\basecat}\bigl(\PP_{b}^{(n-i)}\QQ_a^{(m-i)},\, \QQ_a^{(m)}\PP_{b}^{(n)}\bigr).
\]
For any $\alpha \in \Sym^i\Hom(a,b)$, define
\[
\tilde g_i(\alpha) = 
\begin{tikzpicture}[baseline=-0.5ex]
	\node[sym] (tl) at (-1,1.25)  {$a^m$};
	\node[sym] (tr) at (1,1.25)   {$b^n$};
	\node[sym] (bl) at (-1,-1.25) {$b^{n-i}$};
	\node[sym] (br) at (1,-1.25)  {$a^{m-i}$};
	\draw[down, <-, many] (br.60) -- node[near start, desc] {$m-i$} (tl.240);
	\draw[cc, ->, many] (tl.300) to[out=330, in=210] node[near start, dot] {} node[above] {$\alpha$} (tr.240);
	\draw[up, ->, many] (bl.120) -- node[near start, desc] {$n-i$} (tr.300);
\end{tikzpicture}
\]
Now, define $g \coloneqq \sum_i g_i$.
The map $g$ does in general not have an inverse on the \dg level.
However, we can define an inverse map $f$ in $H^*(\hcat\basecat)$.
To define a map 
\[
f_i\colon \QQ_a^{(m)}\PP_{b}^{(n)} \to \Sym^i\Hom_{H^*(\basecat)}(a,b) \otimes_\kk \PP_{b}^{(n-i)}\QQ_a^{(m-i)},
\]
we define a map
\[
\tilde f_i\colon \bigr(\Sym^i\Hom_{H^*(\basecat)}(a,b)\bigl)^* \to \Hom_{\hcat\basecat}\bigl(\QQ_a^{(m)}\PP_{b}^{(n)},\, \PP_{b}^{(n-i)}\QQ_a^{(m-i)}\bigr),
\]
or equivalently a map
\[
\tilde f_i' \colon \Sym^i \Hom_{H^*(\basecat)}(b,Sa) \to
\Hom_{H^*(\hcat\basecat)}\bigl(\QQ_a^{(m)}\PP_{b}^{(n)},\, \PP_{b}^{(n-i)}\QQ_a^{(m-i)}\bigr),
\]
where we use the identification $\Hom(a,\, b)^* = \Hom(b,\, Sa)$ in $H^*(\basecat)$.
Set
\[
\tilde f_i'(\alpha) = 
\begin{tikzpicture}[baseline=-0.5ex]
	\node[sym] (tl) at (-1,1.25)  {$b^{n-i}$};
	\node[sym] (tr) at (1,1.25)   {$a^{m-i}$};
	\node[sym] (bl) at (-1,-1.25) {$a^{m}$};
	\node[sym] (br) at (1,-1.25)  {$b^{n}$};
	\draw[up, ->, many] (br.60) -- node[near end, desc] {$m-i$} (tl.240);
	\draw[down, <-, many] (bl.120) -- node[near end, desc] {$n-i$} (tr.300);
	\draw[cc, <-, many] (bl.60) to[out=30, in=150] node[near end, dot] {} node[above right] {$\alpha$} (br.120);
\end{tikzpicture}
\]
Finally, set 
\[
f = \sum_i i! \binom{m}{i} \binom{n}{i} f_i.
\]

We now show that $f$ and $g$ are inverse isomorphisms in $H^*(\hcat\basecat)$.
The proof is entirely combinatorial: one composition follows from
repeated application of the second relation
in~\eqref{eq:up_down_braids-add}, which holds in $H^*(\hcat\basecat)$
by Lemma~\ref{lem:up_down_braids-homotopy}. The other follows
from relations \eqref{eq:circle_and_curl-dg}.
The reader uninterested in combinatorics may want to skip ahead to Section~\ref{subsec:transposed_generators}.

\subsection{The composition $g \circ f$ is the identity.}\label{subsec:gfid}
For simplicity, in this section we denote the image of any closed degree zero $2$-morphism of $\hcat\basecat$ in $H^*(\hcat\basecat)$ by the same symbol as the original $2$-morphism.

\begin{Remark}\label{rem:f_in_basis}
	Choose a basis  $\{\beta_\ell\}$ for $H^*(\Hom_\basecat(a,\, b))$ with dual basis $\{\beta_\ell^\dual\}$ of $H^*(\Hom_\basecat(b,\, Sa))$.
	Let $I=(\ell_1,\dots,\ell_i)$ be a multi-index.
	Then the dual to $\beta_{\ell_1} \stimes \dotsb \stimes \beta_{\ell_i} \in \Sym^i H^*(\Hom_\basecat(a,\, b))$ is $\frac1{m(I)}\beta_{\ell_1}^\dual \stimes \dotsb \stimes \beta_{\ell_i}^\dual \in \Sym^i H^*(\Hom_\basecat(b,\, Sa))$, where $m(I) = \prod m_j(I)!$ with $m_j(I)$ the number of times the index $j$ appears in $I$.
	
	Let $\dcross00{m-i}{n-i}i$ denote the $2$-morphism $g_i \circ f_i$. We have
	\[
	\dcross00{m-i}{n-i}i = \,
	\sum_{\ell_1,\dotsc,\ell_i}\ 
	\begin{tikzpicture}[baseline={($(0,1)-(0,.5ex)$)}]
		\node[sym] (tl) at (-1,3.5)  {$a^m$};
		\node[sym] (tr) at (1,3.5)   {$b^n$};
		\node[sym] (ml) at (-1,1)  {$b^{n-i}$};
		\node[sym] (mr) at (1,1)   {$a^{m-i}$};
		\node[sym] (bl) at (-1,-1.5) {$a^{m}$};
		\node[sym] (br) at (1,-1.5)  {$b^{n}$};
		\draw[up, ->, many] (br.60) -- node[near end, desc] {$m-i$} (ml.240);
		\draw[down, <-, many] (bl.120) -- node[near end, desc] {$n-i$} (mr.300);
		\draw[cc, <-, many] (bl.60) to[out=30, in=150] node[near end, dot] {} node[above right] {$\beta_{\ell_1}^\dual \stimes \dotsb \stimes \beta_{\ell_i}^\dual$} (br.120);
		\draw[down, <-, many] (mr.60) -- node[near start, desc] {$m-i$} (tl.240);
		\draw[cc, ->, many] (tl.300) to[out=330, in=210] node[near start, dot] {} node[below] {$\beta_{\ell_1} \stimes \dotsb \stimes \beta_{\ell_i}$} (tr.240);
		\draw[up, ->, many] (ml.120) -- node[near start, desc] {$n-i$} (tr.300);
	\end{tikzpicture}
	\]
\end{Remark}

We first verify that the composition condition
\eqref{eqn-the-condition-for-composition-of-diagrams-between-Pn-and-Qns}
holds:
\[
\begin{tikzpicture}[baseline=-0.5ex]
	\node[sym] (tl) at (-1,1.5) {$a^{m}$};
	\node[sym] (tr) at (1,1.5) {$b^{n}$};
	\node[sym] (ml) at (-1,0) {$b^{n-i}$};
	\node[sym] (mr) at (1,0) {$a^{m-i}$};
	\node[sym] (bl) at (-1,-1.5) {$a^{m}$};
	\node[sym] (br) at (1,-1.5) {$b^{n}$};
	\draw[cc, ->, many] (tl.300) to[out=330, in=210] node[dot, near start] {} (tr.240);
	\draw[cc, <-, many] (bl.60) to[out=30, in=150] node[dot, near end] {} (br.120);
	\draw[down, <-, many] (bl.120) to[out=90, in=270] (mr.270);
	\draw[down, <-, many] (mr.90) to[out=90,in=270] (tl.240);
	\draw[up, ->, many] (br.60) to[out=90, in=270] (ml.270);
	\draw[up, ->, many] (ml.90) to[out=90, in=270] (tr.300);
\end{tikzpicture}
\,\,=
\begin{tikzpicture}[baseline=-0.5ex]
	\node[sym] (tl) at (-1,1.5) {$a^{m}$};
	\node[sym] (tr) at (1,1.5) {$b^{n}$};
	\node[sym] (bl) at (-1,-1.5) {$a^{m}$};
	\node[sym] (br) at (1,-1.5) {$b^{n}$};
	\draw[cc, ->, many] (tl.300) to[out=330, in=210] node[desc] {$i$} node[dot, near start] {} (tr.240);
	\draw[cc, <-, many] (bl.60) to[out=30, in=150] node[desc] {$i$} node[dot, near end] {} (br.120);
	\draw[down, <-, many] (bl.120) to[out=30, in=270] (0.75,0) node[desc] {$m-i$} to[out=90, in=330] (tl.240);
	\draw[up, ->, many] (br.60) to[out=150, in=270] (-0.75,0) node[desc] {$n-i$} to[out=90, in=210] (tr.300);
\end{tikzpicture}.
\]
It does because we can absorb the middle idempotents into the top or bottom 
ones. We can move elements of (or their sums) of $S_{n-i}$ (resp. $S_{m-i}$) 
in the middle idempotents all the way up or down their strands 
where they can be viewed as elements of the corresponding 
subgroup $S_{n_i} < S_n$ (resp. $S_{m-i} < S_m$). 
Pre- or postcomposing with these does not change the 
symmetriser idempotent of $S_{n}$ (resp. $S_{m}$).

Adding $s$ downward strands on the left and $t$ upward strands on the right, we denote the resulting $2$-endomorphism of $\QQ_a^{(m+s)}\PP_{b}^{(n+t)}$ by $\dcross st{m-i}{n-i}i$.
Relabeling slightly, with a choice of basis as in Remark \ref{rem:f_in_basis} this gives
\[
\dcross stmni = 
\sum_{\ell_1,\dotsc,\ell_i}
\begin{tikzpicture}[baseline=-0.5ex]
	\node[sym] (tl) at (-1,1.5) {$a^{s+m+i}$};
	\node[sym] (tr) at (1,1.5) {$b^{t+n+i}$};
	\node[sym] (bl) at (-1,-1.5) {$a^{s+m+i}$};
	\node[sym] (br) at (1,-1.5) {$b^{t+n+i}$};
	\draw[cc, ->, many] (tl.300) to[out=330, in=210] node[desc] {$i$} (tr.240);
	\draw[cc, <-, many] (bl.60) to[out=30, in=150] node[desc] {$i$} (br.120);
	\draw[down, <-, many] (bl.140) -- node[desc] {$s$} (tl.220);
	\draw[down, <-, many] (bl.120) to[out=30, in=270] (0.75,0) node[desc] {$m$} to[out=90, in=330] (tl.240);
	\draw[up, ->, many] (br.60) to[out=150, in=270] (-0.75,0) node[desc] {$n$} to[out=90, in=210] (tr.300);
	\draw[up, ->, many] (br.40) -- node[desc] {$t$}(tr.320);
\end{tikzpicture},
\]
where the arcs are labeled by $\beta_{\ell_1}\stimes \dotsb \stimes \beta_{\ell_i}$ and $\beta_{\ell_1}^\dual\stimes \dotsb \stimes \beta_{\ell_i}^\dual$ respectively.
One notes that
\[
\dcross s t m0i = \dcross{s+m}t00i
\quad\text{and}\quad
\dcross s t 0ni = \dcross s{t+n}00i.
\]
To simplify notation, write $\sdcross smi \coloneqq \dcross ssmmi$ for the symmetric situation.

\begin{Lemma}
	\[
	\sdcross 0 n 0 =
	- n \sdcross 0 {n-1} 1 +
	\sum_{i=0}^{n-1} (-1)^i \frac{(n-1)!}{(n-1-i)!} \sdcross 1 {n-1-i} i.
	\]
\end{Lemma}

\begin{proof}
	First move the left-most downwards strand of $\sdcross0n0$ all the way to the left.
	To do so, one has to untwist $n$ down-up double crossings, introducing $n$ terms of the from $-\sdcross0{n-1}1$ via relation~\eqref{eq:up_down_braids-add}:
	\[
	\sdcross0n0 = \dcross10{n-1}n0 - n \sdcross0{n-1}1,
	\]
	or graphically,
	\[
	\begin{tikzpicture}[baseline=-0.5ex]
		\node[sym] (tl) at (-1,1.5) {$a^{n}$};
		\node[sym] (tr) at (1,1.5) {$b^{n}$};
		\node[sym] (bl) at (-1,-1.5) {$a^{n}$};
		\node[sym] (br) at (1,-1.5) {$b^{n}$};
		\draw[down, <-, many] (bl.120) to[out=30, in=270] (0.75,0) node[desc] {$n$} to[out=90, in=330] (tl.240);
		\draw[up, ->, many] (br.60) to[out=150, in=270] (-0.75,0) node[desc] {$n$} to[out=90, in=210] (tr.300);
	\end{tikzpicture}
	=
	\begin{tikzpicture}[baseline=-0.5ex]
		\node[sym] (tl) at (-1,1.5) {$a^{n}$};
		\node[sym] (tr) at (1,1.5) {$b^{n}$};
		\node[sym] (bl) at (-1,-1.5) {$a^{n}$};
		\node[sym] (br) at (1,-1.5) {$b^{n}$};
		\draw[down, <-] (bl.140) -- (tl.220);
		\draw[down, <-, many] (bl.120) to[out=30, in=270] (0.75,0) node[desc] {$n-1$} to[out=90, in=330] (tl.240);
		\draw[up, ->, many] (br.60) to[out=150, in=270] (-0.75,0) node[desc] {$n$} to[out=90, in=210] (tr.300);
	\end{tikzpicture}
	-n\sum_{\ell}\!\!
	\begin{tikzpicture}[baseline=-0.5ex]
		\node[sym] (tl) at (-1,1.5) {$a^{n}$};
		\node[sym] (tr) at (1,1.5) {$b^{n}$};
		\node[sym] (bl) at (-1,-1.5) {$a^{n}$};
		\node[sym] (br) at (1,-1.5) {$b^{n}$};
		\draw[cc, ->] (tl.300) to[out=330, in=210] node[dot, near start, label={above right:{$\beta_\ell$}}] {} (tr.240);
		\draw[cc, <-] (bl.60) to[out=30, in=150] node[dot, near end, label={below left:{$\beta_\ell^\vee$}}] {} (br.120);
		\draw[down, <-, many] (bl.120) to[out=30, in=270] (0.75,0) node[desc] {$n-1$} to[out=90, in=330] (tl.240);
		\draw[up, ->, many] (br.60) to[out=150, in=270] (-0.75,0) node[desc] {$n-1$} to[out=90, in=210] (tr.300);
	\end{tikzpicture}.
	\]
	Now move the rightmost upward strand of $\dcross10{n-1}n0$ all the way to the right.
	To do so, one has to untwist with $n-1$ downwards strands, introducing $n-1$ terms of the form $-\dcross10{n-2}{n-1}1$:
	\[
	\dcross10{n-1}n0 = \sdcross 1{n-1}0 -(n-1) \dcross10{n-2}{n-1}1.
	\]
	\[
	\begin{tikzpicture}[baseline=-0.5ex]
		\node[sym] (tl) at (-1,1.5) {$a^{n}$};
		\node[sym] (tr) at (1,1.5) {$b^{n}$};
		\node[sym] (bl) at (-1,-1.5) {$a^{n}$};
		\node[sym] (br) at (1,-1.5) {$b^{n}$};
		\draw[down, <-] (bl.140) -- (tl.220);
		\draw[down, <-, many] (bl.120) to[out=30, in=270] (0.75,0) node[desc] {$n-1$} to[out=90, in=330] (tl.240);
		\draw[up, ->, many] (br.60) to[out=150, in=270] (-0.75,0) node[desc] {$n$} to[out=90, in=210] (tr.300);
	\end{tikzpicture}
	=
	\begin{tikzpicture}[baseline=-0.5ex]
		\node[sym] (tl) at (-1,1.5) {$a^{n}$};
		\node[sym] (tr) at (1,1.5) {$b^{n}$};
		\node[sym] (bl) at (-1,-1.5) {$a^{n}$};
		\node[sym] (br) at (1,-1.5) {$b^{n}$};
		\draw[down, <-] (bl.140) -- (tl.220);
		\draw[down, <-, many] (bl.120) to[out=30, in=270] (0.75,0) node[desc] {$n-1$} to[out=90, in=330] (tl.240);
		\draw[up, ->, many] (br.60) to[out=150, in=270] (-0.75,0) node[desc] {$n-1$} to[out=90, in=210] (tr.300);
		\draw[up, ->] (br.40) -- (tr.320);
	\end{tikzpicture}
	- (n-1)\sum_{\ell}
	\begin{tikzpicture}[baseline=-0.5ex]
		\node[sym] (tl) at (-1,1.5) {$a^{n}$};
		\node[sym] (tr) at (1,1.5) {$b^{n}$};
		\node[sym] (bl) at (-1,-1.5) {$a^{n}$};
		\node[sym] (br) at (1,-1.5) {$b^{n}$};
		\draw[cc, ->] (tl.300) to[out=330, in=210] node[dot, near start, label={above right:{$\beta_\ell$}}] {} (tr.240);
		\draw[cc, <-] (bl.60) to[out=30, in=150] node[dot, near end, label={below left:{$\beta_\ell^\vee$}}] {} (br.120);
		\draw[down, <-] (bl.140) -- (tl.220);
		\draw[down, <-, many] (bl.120) to[out=30, in=270] (0.75,0) node[desc] {$n-2$} to[out=90, in=330] (tl.240);
		\draw[up, ->, many] (br.60) to[out=150, in=270] (-0.75,0) node[desc] {$n-1$} to[out=90, in=210] (tr.300);
	\end{tikzpicture}
	\]
	Repeat the last step for $-(n-1) \dcross10{n-2}{n-1}1$, obtaining $-(n-1)\sdcross1{n-2}1$ and $(n-1)(n-2)$ terms of the form $\dcross10{n-3}{n-2}2$:
	\[
	\begin{tikzpicture}[baseline=-0.5ex]
		\node[sym] (tl) at (-1,1.5) {$a^{n}$};
		\node[sym] (tr) at (1,1.5) {$b^{n}$};
		\node[sym] (bl) at (-1,-1.5) {$a^{n}$};
		\node[sym] (br) at (1,-1.5) {$b^{n}$};
		\draw[cc, ->] (tl.300) to[out=330, in=210] node[dot, near start, label={above right:{$\beta_\ell$}}] {} (tr.240);
		\draw[cc, <-] (bl.60) to[out=30, in=150] node[dot, near end, label={below left:{$\beta_\ell^\vee$}}] {} (br.120);
		\draw[down, <-] (bl.140) -- (tl.220);
		\draw[down, <-, many] (bl.120) to[out=30, in=270] (0.75,0) node[desc] {$n-2$} to[out=90, in=330] (tl.240);
		\draw[up, ->, many] (br.60) to[out=150, in=270] (-0.75,0) node[desc] {$n-1$} to[out=90, in=210] (tr.300);
	\end{tikzpicture}
	=
	\begin{tikzpicture}[baseline=-0.5ex]
		\node[sym] (tl) at (-1,1.5) {$a^{n}$};
		\node[sym] (tr) at (1,1.5) {$b^{n}$};
		\node[sym] (bl) at (-1,-1.5) {$a^{n}$};
		\node[sym] (br) at (1,-1.5) {$b^{n}$};
		\draw[cc, ->] (tl.300) to[out=330, in=210] node[dot, near start, label={above right:{$\beta_\ell$}}] {} (tr.240);
		\draw[cc, <-] (bl.60) to[out=30, in=150] node[dot, near end, label={below left:{$\beta_\ell^\vee$}}] {} (br.120);
		\draw[down, <-] (bl.140) -- (tl.220);
		\draw[down, <-, many] (bl.120) to[out=30, in=270] (0.75,0) node[desc] {$n-2$} to[out=90, in=330] (tl.240);
		\draw[up, ->, many] (br.60) to[out=150, in=270] (-0.75,0) node[desc] {$n-2$} to[out=90, in=210] (tr.300);
		\draw[up, ->] (br.40) -- (tr.320);
	\end{tikzpicture}
	- (n-2)\sum_{\ell'}
	\begin{tikzpicture}[baseline=-0.5ex]
		\node[sym] (tl) at (-1,1.5) {$a^{n}$};
		\node[sym] (tr) at (1,1.5) {$b^{n}$};
		\node[sym] (bl) at (-1,-1.5) {$a^{n}$};
		\node[sym] (br) at (1,-1.5) {$b^{n}$};
		\draw[cc, ->, many] (tl.300) to[out=330, in=210] node[dot, near start] {} (tr.240);
		\draw[cc, <-, many] (bl.60) to[out=30, in=150] node[dot, near end] {} (br.120);
		\draw[down, <-] (bl.140) -- (tl.220);
		\draw[down, <-, many] (bl.120) to[out=30, in=270] (0.75,0) node[desc] {$n-3$} to[out=90, in=330] (tl.240);
		\draw[up, ->, many] (br.60) to[out=150, in=270] (-0.75,0) node[desc] {$n-2$} to[out=90, in=210] (tr.300);
	\end{tikzpicture},
	\]
	where the dots are marked with $\beta_\ell \stimes \beta_{\ell'}$ and $\beta_\ell^\vee \stimes \beta_{\ell'}^\vee$ respectively.
	Recursive application of this procedure yields the desired formula.
\end{proof}

Rearranging and changing indices by $1$, we obtain
\begin{equation}\label{eq:link_1}
	\sdcross 1 n 0 =
	\sdcross 0 {n+1} 0 +
	(n+1) \sdcross 0 n 1 +
	\sum_{i=1}^{n} (-1)^{i+1} \frac{n!}{(n-i)!} \sdcross 1 {n-i} i.
\end{equation}
The remainder of the argument is now just repeated application of this formula.

\begin{Lemma}\label{lem:sdcross1n0}
	\[
	\sdcross 1 n 0 = 
	\sdcross 0{n+1}0 + (2n+1)\sdcross0n1 + n^2\sdcross0{n-1}2.
	\]
\end{Lemma}

\begin{proof}
	For $n=0$, this is just equation \eqref{eq:link_1}.
	Using induction and \eqref{eq:link_1} we get
	\begin{align*}
		\sdcross 1{n+1}0 & =
		\sdcross 0{n+1}0 + (n+2) \sdcross 0 {n+1} 1 + \sum_{i=1}^{n+1} (-1)^{i+1} \frac{(n+1)!}{(n+1-i)!} \sdcross 1 {n+1-i} i \\ &=
		\sdcross 0{n+1}0 + (n+2) \sdcross 0 {n+1} 1 + \sum_{i=1}^{n+1} (-1)^{i+1} \frac{(n+1)!}{(n+1-i)!} \biggl( \sdcross 0 {n+2-i} i +\\& \qquad + (2(n+1-i)+1)\sdcross0{n+1-i}{i+1} + (n+1-i)^2\sdcross0{n-i}{i+2}\biggr).
	\end{align*}
	Carefully rearranging terms, one obtains
	\begin{multline*}
		\sdcross 0{n+1}0 + (n+2) \sdcross 0 {n+1} 1 + \frac{(n+1)!}{n!}\sdcross0{n+1}1 \\ + \biggl(-\frac{(n+1)!}{(n-1)!} + (2n+1)\frac{(n+1)!}{n!}\biggr)\sdcross0n2 + \\
		+ \sum_{\ell=3}^{n+2} \biggl(
		(-1)^{\ell+1} \frac{(n+1)!}{(n+1-\ell)!} + 
		(-1)^{\ell} \frac{(n+1)!}{(n+2-\ell)!}(2(n+2-\ell)+1) + \\
		(-1)^{\ell-1} \frac{(n+1)!}{(n+3-\ell)!}(n+3-\ell)^2
		\biggr)\sdcross0{n+2-\ell}{\ell} \\
		+\biggl(\frac{(n+1)!}{0!}(-1)^{n+2} + \frac{(n+1)!}{1!}(-1)^{n+1}1^2\biggr)\sdcross00{n+2},
	\end{multline*}
	which one easily checks to be equal to
	\[
	\sdcross 0{n+2}0 + (2n+3)\sdcross0{n+1}1 + (n+1)^2\sdcross0n2.
	\qedhere
	\]
\end{proof}

\begin{Lemma}
	\[
	\sdcross k00 = 
	\sum_{i=0}^k i! \binom{k}{i}^2 \sdcross0{k-i}{i}.
	\]
\end{Lemma}

\begin{proof}
	For $k = 1$ this is immediate from \eqref{eq:link_1}.
	Assume that the identity holds for some integer $k$.
	Then
	\[
	\sdcross{k+1}00 =
	\sum_{i=0}^k i! \binom{k}{i}^2 \sdcross1{k-i}{i}.
	\]
	We can now substitute in the identity of Lemma~\ref{lem:sdcross1n0}.
	\[
	\sdcross{k+1}00 =
	\sum_{i=0}^k i! \binom{k}{i}^2 \biggl(
	\sdcross0{k-i+1}{i} +
	(2(k-i)+1)\sdcross0{k-i}{i+1} +
	(k-i)^2\sdcross0{k-i-1}{i+2}
	\biggr) \\
	\]
	Rearranging gives
	\begin{multline*}
		\sdcross0{k+1}0 +
		\bigl(k^2 + (2k+1)\bigr)\sdcross0k1\\ + 
		\sum_{\ell = 2}^k \biggl(
		\ell!\binom{k}{\ell}^2 +
		(\ell-1)!\binom{k}{\ell-1}^2 (2(k-\ell+1)+1) + \\
		(\ell-2)!\binom{k}{\ell-2}^2 (k-\ell+2)^2
		\biggr)\sdcross0{k+1-\ell}{\ell}\\ +
		\biggl(k! \binom{k}{k}^2 + (k-1)!\binom{k}{k-1}^2 1^2\biggr) \sdcross00{k+1}.
	\end{multline*}
	Again, one easily checks this to be equal to
	\[
	\sum_{i=0}^{k+1} i! \binom{k+1}{i}^2 \sdcross0{k+1-i}{i}.
	\qedhere
	\]
\end{proof}

\begin{Corollary}
	\[
	\dcross mn000 = 
	\sum_{i=0}^{\min(m,n)} i! \binom{m}{i}\binom{n}{i} \dcross00{m-i}{n-i}{i}.
	\]
	In other words $g \circ f = 1$.
\end{Corollary}

\begin{proof}
	Without loss of generality we can assume that $m \ge n$, say $m = n + j$.
	We will induct on $j$.
	We already considered the case that $j=0$.
	
	We ignore the left-most string and use the induction hypothesis to obtain
	\[
	\dcross{n+j+1}{n}000 =
	\sum_{i=0}^{n} i! \binom{n+j}{i}\binom{n}{i} \dcross10{n+j-i}{n-i}{i}.
	\]
	As in the first step of the proof of Lemma~\ref{lem:sdcross1n0}, we have
	\[
	\dcross10{n+j-i}{n-i}{i} = 
	\dcross00{n+j-i+1}{n-i}{i} + 
	(n-i)\dcross00{n+j-i}{n-i-1}{i+1}. 
	\]
	Thus,
	\[
	\dcross{n+j+1}{n}000 =
	\sum_{i=0}^{n} i! \binom{n+j}{i}\binom{n}{i} (\dcross00{n+j-i+1}{n-i}{i} + (n-i)\dcross00{n+j-i}{n-i-1}{i+1}).
	\]
	Grouping terms, this is
	\begin{multline*}
		0!\binom{n+j}0\binom{n}{0} \dcross00{n+j+1}n0 + \\
		\sum_{\ell=1}^{n} \biggl( \ell! \binom{n+j}{\ell}\binom{n}{\ell} + (\ell-1)!\binom{n+j}{\ell-1}\binom{n}{\ell-1}(n-\ell+1)\biggr)\dcross00{n+j+1-\ell}{n-\ell}{\ell}.
	\end{multline*}
	This is easily shown to be equal to the desired expression
	\[
	\sum_{\ell=0}^{n} \ell! \binom{n+j+1}{\ell}\binom{n}{\ell} \dcross00{n+j+1-\ell}{n-\ell}{\ell}.
	\qedhere
	\]
\end{proof}

\subsection{The composition $f \circ g$ is the identity.}\label{subsec:fgid}
We have
\[
f_j \circ g_i \;=\; 
\begin{tikzpicture}[baseline=-0.5ex]
	\node[sym] (tl) at (-1,2) {$b^{n-j}$};
	\node[sym] (tr) at (1,2) {$a^{m-j}$};
	\node[sym] (ml) at (-1,0) {$a^m$};
	\node[sym] (mr) at (1,0) {$b^n$};
	\node[sym] (bl) at (-1,-2) {$b^{n-i}$};
	\node[sym] (br) at (1,-2) {$a^{m-i}$};
	\draw[cc, ->, many] (ml.300) to[out=300, in=240] node[dot, near start] {} node[desc] {$i$} (mr.240);
	\draw[cc, <-, many] (ml.60) to[out=60, in=120] node[dot, near end] {} node[desc] {$j$} (mr.120);
	\draw[up, ->, many] (bl.90) to[out=90, in=270] (mr.300);
	\draw[up, ->, many] (mr.60) to[out=90,in=270] (tl.270);
	\draw[down, <-, many] (br.90) to[out=90, in=270] (ml.240);
	\draw[down, <-, many] (ml.120) to[out=90, in=270] (tr.270);
\end{tikzpicture}.
\]
When $i \ne j$, every combination of summands of the middle symmetriser 
idempotents produces a diagram which contains a left curl and hence vanishes.
Thus $f_j \circ g_i = 0$ if $i \neq j$.

When $i = j$, we claim that the composition condition 
\eqref{eqn-the-condition-for-composition-of-diagrams-between-Pn-and-Qns}
holds up to a coefficient:
\begin{equation}
	\label{eqn-composition-condition-for-f_i-g_i}
	\begin{tikzpicture}[baseline=-0.5ex]
		\node[sym] (tl) at (-1,2) {$b^{n-i}$};
		\node[sym] (tr) at (1,2) {$a^{m-i}$};
		\node[sym] (ml) at (-1,0) {$a^m$};
		\node[sym] (mr) at (1,0) {$b^n$};
		\node[sym] (bl) at (-1,-2) {$b^{n-i}$};
		\node[sym] (br) at (1,-2) {$a^{m-i}$};
		\draw[cc, ->, many] (ml.300) to[out=300, in=240] node[dot, near start] {} node[desc] {$i$} (mr.240);
		\draw[cc, <-, many] (ml.60) to[out=60, in=120] node[dot, near end] {} node[desc] {$j$} (mr.120);
		\draw[up, ->, many] (bl.90) to[out=90, in=270] (mr.300);
		\draw[up, ->, many] (mr.60) to[out=90,in=270] (tl.270);
		\draw[down, <-, many] (br.90) to[out=90, in=270] (ml.240);
		\draw[down, <-, many] (ml.120) to[out=90, in=270] (tr.270);
	\end{tikzpicture}
	\quad = \quad
	\frac{1}{i!\binom{m}{i}\binom{n}{i}} 
	\cdot
	\begin{tikzpicture}[baseline=-0.5ex]
		\node[sym] (tl) at (-1,2) {$b^{n-i}$};
		\node[sym] (tr) at (1,2) {$a^{m-i}$};
		\node[sym] (bl) at (-1,-2) {$b^{n-i}$};
		\node[sym] (br) at (1,-2) {$a^{m-i}$};
		\draw[up, ->, many] (bl.90) to[out=30, in=270] (1,0) node[desc] {$n-i$} to[out=90, in=330] (tl.270);
		\draw[down, <-, many] (br.90) to[out=150, in=270] (-1,0) node[desc] {$n-i$} to[out=90, in=210] (tr.270);
	\end{tikzpicture}.
\end{equation}
Indeed, the pair of the middle idempotents in the LHS of 
\eqref{eqn-composition-condition-for-f_i-g_i} are a $2$-morphism
\begin{equation}
	\label{eqn-middle-idempotents-in-f_i-g_i}
	\frac{1}{m!}\frac{1}{n!} \sum_{\sigma \in S_m, \tau \in S_n} \sigma \circ_1\tau
\end{equation}
where $\circ_1$ denotes $1$-composition. We first observe that if 
$\sigma \notin S_{m-i} \times S_i < S_m$ or $\tau \notin S_{i} \times
S_{n-i}$ the resulting diagram contains a left curl and hence vanishes. 
Let 
$$ \sigma = (\sigma_{m-i}, \sigma_{i}) \in S_{m-i} \times S_i, $$
$$ \tau = (\tau_{i}, \tau_{n-i}) \in S_{i} \times S_{n-i}. $$
On the diagram coresponding to this summand, we can slide $\sigma_i$
along the central bubble\index{bubble relation} and compose it with $\tau_i$. We obtain a
counterclockwise bubble of $i$ parallel strands with a single 
element $\tau_i \sigma_i \in S_i$ inserted into it. Unless this
element is $\id_{S_i}$, the resulting diagram contains a left curl. 
When it is $\id_{S_i}$, we get an unmarked $i$-stranded counterclockwise 
bubble which is the identity endomorphism of $\hunit$ and hence can be
erased. On the remaining diagram, we can absorb $\sigma_{m-i}$ and
$\tau_{n-i}$ into the top or bottom idempotents and thus obtain
the diagram on the RHS of
\eqref{eqn-composition-condition-for-f_i-g_i}. 

Thus when expanding the middle idempotents in the LHS of 
\eqref{eqn-composition-condition-for-f_i-g_i} the non-vanishing
diagrams are given by the summands  
$$ \sigma_{m-i} \circ_1 \upsilon_i \circ_1 \upsilon^{-1}_i \circ_1, \tau_{n-i}
\quad \quad \quad \sigma_{m-i} \in S_{m-i}, \upsilon_i \in S_i,
\tau_{n-i} \in S_{n-i} $$
of \eqref{eqn-middle-idempotents-in-f_i-g_i}. 
There are $(m-i)! i! (n-i)!$ of them and each produces  
the diagram on the RHS of
\eqref{eqn-composition-condition-for-f_i-g_i}, whence 
the equality in \eqref{eqn-composition-condition-for-f_i-g_i} holds. 

By the left relation in \eqref{eq:up_down_braids-add} the  
RHS of \eqref{eqn-composition-condition-for-f_i-g_i} is
\[
\frac{1}{i!\binom{m}{i}\binom{n}{i}} \id_{\PP_{b}^{(n-i)}\QQ_a^{(m-i)}}.
\]
Since $g = \sum g_i$ and $f = \sum_i i! \binom{m}{i} \binom{n}{i} f_i$, 
it follows that $f \circ g = \id$. 
This finishes the proof of Theorem~\ref{thm:cat_heisenberg_relations-dg}.

\section{The transposed generators}\label{subsec:transposed_generators}

Given any partition $\lambda$ of $n$ write $e_\lambda \in \kk[\SymGrp n]$ for the corresponding Young symmetriser.
It is a minimal idempotent of $\kk[\SymGrp n]$.
Thus, similar to the definition of the $1$-morphisms $\PP_a^{(n)}$ and
$\QQ_a^{(n)}$, it induces $1$-morphisms $\PP_a^{\lambda}$ and
$\QQ_a^{\lambda}$ in $\hcat\basecat$. 

Recall the transposed generators $p_a^{(1^n)}$ and $q_a^{(1^n)}$, $n \in \ZZ_{>0}$ from Section~\ref{subsubsec:transposedgenalg}.
We have the antisymmetriser idempotent corresponding to the partition $(1^n)$
\[
e_\sign = \frac{1}{n!} \sum_{\sigma \in \SymGrp n} \sgn(\sigma)\sigma \in \kk[\SymGrp n]
\]
on which $S_n$ acts by the sign character. 
Let $\PP_a^{(1^n)}$ and $\QQ_a^{(1^n)}$ be the corresponding
$1$-morphisms  defined analogously to
Definition~\ref{defn-symmetric-powers-of-PP-QQ-and-RR}. 

Arguing as in
Remark~\ref{remark-symmetric-power-labels-on-thick-strands}, we see  
that elements of $\Sym^n \Hom(a, b)$ define morphisms 
from $\PP_a^{(1^n)}$ to $\PP_{b}^{(1^n)}$, 
while those of $\bigwedge^n \Hom(a,b)$ define morphisms 
from $\PP_a^{(n)}$ to $\PP_{b}^{(1^n)}$.  

The category $\hcat*\basecat$ has a covariant autoequivalence $F$
which 
\begin{itemize}
	\item is identity on objects and $1$-morphisms, 
	\item on $2$-morphisms it multiplies the crossings by $-1$, while 
	preserving all other generating diagrams.  
\end{itemize}
The induced autoequivalence $F$ of $\hcat\basecat$ swaps the
$1$-morphisms above with those of Section~\ref{subsec:cat_heisenberg_relations}:
\[
F(\PP_a^{(n)}) = \PP_a^{(1^n)}, \quad
F(\PP_a^{(1^n)}) = \PP_a^{(n)}, \quad
F(\QQ_a^{(n)}) = \QQ_a^{(1^n)}, \quad
F(\QQ_a^{(1^n)}) = \QQ_a^{(n)}.
\]
Thus the relations of Theorem~\ref{thm:cat_heisenberg_relations-dg}
also hold for the transposed $1$-morphisms. 

\begin{Lemma}\label{lemma:shifts_and_symmetrisers}
	If $\basecat$ is pretriangulated, then for any $a \in \basecat$
	we have in $\hcat\basecat$ isomorphisms 
	\[
	\PP_{a[1]} \cong \PP_{a}[1] \text{ and } 
	\QQ_{a[1]} \cong \QQ_{a}[-1], 
	\] 
	and isomorphisms
	\[
	\PP_{a[1]}^{(n)} \cong \PP_a^{(1^n)}[n]
	\quad\text{and}\quad
	\QQ_{a[1]}^{(n)} \cong \QQ_a^{(1^n)}[-n].
	\]
\end{Lemma}

\begin{proof}
	We prove the statements about $\PP$s.  Those about $\QQ$s are proved similarly with a twist in the sign; see the end of the proof below.
	
	Let $i\colon a[1] \rightarrow a$ be the degree $-1$ morphism in
	$\basecat$ defined by $\id_a$. Let $\iota'\colon \PP_{a[1]} \rightarrow
	\PP_{a}$ be the corresponding morphism
	$\begin{tikzpicture}[scale=0.5, baseline={(0,0.125)}]
		\draw[->] (1,0) -- node[label=right:{$i$}, dot] {} (1,1);
	\end{tikzpicture}$
	in $\hcat*\basecat$. Finally, let 
	$\iota\colon \PP_{a[1]} \rightarrow \PP_{a}[1]$ be the degree zero
	morphism in $\hcat*\basecat$ defined by $\iota'$. It is an isomorphism 
	as it has an inverse $\iota^{-1}: \PP_{a}[1] \rightarrow  \PP_{a[1]}$
	which is similarly defined by $\id_a$. 
	
	By definition, $\PP_{a[1]}^{(n)}$ is the convolution of the twisted
	complex 
	\begin{equation*}
		\begin{tikzcd}
			\dots
			\ar{r}{e_\triv}
			&
			\PP_{a[1]}^{n} 
			\ar{r}{1 - e_\triv}
			&
			\PP_{a[1]}^{n} 
			\ar{r}{e_\triv}
			&
			\PP_{a[1]}^{n} 
			\ar{r}{1 - e_\triv}
			&
			\underset{\degzero}{\PP_{a[1]}^{n}},
		\end{tikzcd}
	\end{equation*}
	while is $\PP_a^{(1^n)}[n]$ the convolution of
	\begin{equation*}
		\begin{tikzcd}
			\dots
			\ar{r}{e_\sign}
			&
			\PP_{a}^{n}[n]
			\ar{r}{1 - e_\sign}
			&
			\PP_{a}^{n}[n]
			\ar{r}{e_\sign}
			&
			\PP_{a}^{n}[n]
			\ar{r}{1 - e_\sign}
			&
			\underset{\degzero}{\PP_{a}^{n}[n]}.
		\end{tikzcd}
	\end{equation*}
	Consider the following map of twisted complexes
	\begin{equation*}
		\tilde{\iota}^n \coloneqq
		\begin{tikzcd}
			\dots
			\ar{r}{e_\triv}
			&
			\PP_{a[1]}^{n} 
			\ar{r}{1 - e_\triv}
			\ar{d}{\iota^n}
			&
			\PP_{a[1]}^{n} 
			\ar{r}{e_\triv}
			\ar{d}{\iota^n}
			&
			\PP_{a[1]}^{n} 
			\ar{r}{1 - e_\triv}
			\ar{d}{\iota^n}
			&
			\PP_{a[1]}^{n} 
			\ar{d}{\iota^n}
			\\
			\dots
			\ar{r}{e_\sign}
			&
			\PP_a^{n}[n]
			\ar{r}{1 - e_\sign}
			&
			\PP_a^{n}[n] 
			\ar{r}{e_\sign}
			&
			\PP_a^{n}[n] 
			\ar{r}{1- e_\sign}
			&
			\PP_a^{n}[n].
		\end{tikzcd}
	\end{equation*}
	We claim that $\iota^n\colon \PP_{a[1]}^{n} \rightarrow \PP_a^{n}[n]$ 
	intertwines the idempotents $e_\triv$ and $e_\sign$:
	
	\[ \iota^ne_\triv =e_\sign\iota^n. \]
	It follows that $\tilde{\iota}^n$ is closed of degree $0$. We conclude that
	it is an isomorphism, as $\iota^n$ is one. 
	
	To prove the claim, it suffices to show that degree $-n$ map  
	$\iota^n\colon \PP_{a[1]}^{n} \rightarrow \PP_a^{n}$ intertwines 
	$e_\triv$ and $e_\sign$. 
	This is a straightforward verification in 
	$\hcat*\basecat$. We give the details for $n = 2$; the general case follows in the same manner. 
	
	When $n = 2$, we have
	\begin{align*}
		e_\triv &=
		\frac12 \biggl(
		\begin{tikzpicture}[scale=0.5, baseline={(0,0.125)}]
			\draw[->] (0,0) -- (0,1);
			\draw[->] (1,0) -- (1,1);
		\end{tikzpicture}
		+ 
		\begin{tikzpicture}[scale=0.5, baseline={(0,0.125)}]
			\draw[->] (0,0) --  (1,1);
			\draw[->] (1,0) --  (0,1);
		\end{tikzpicture}
		\biggr),\\
		e_\sign & = 
		\frac12\biggl(
		\begin{tikzpicture}[scale=0.5, baseline={(0,0.125)}]
			\draw[->] (0,0) -- (0,1);
			\draw[->] (1,0) -- (1,1);
		\end{tikzpicture}
		-
		\begin{tikzpicture}[scale=0.5, baseline={(0,0.125)}]
			\draw[->] (0,0) --  (1,1);
			\draw[->] (1,0) --  (0,1);
		\end{tikzpicture}
		\biggr).
	\end{align*}
	The $2$-morphism $\begin{tikzpicture}[scale=0.5, baseline={(0,0.125)}]
		\draw[->] (0,0) -- (0,1);
		\draw[->] (1,0) -- (1,1);
	\end{tikzpicture}$ is the identity map, and clearly $\iota^2$
	intertwines the identity maps. It remains to show that 
	it intertwines $\begin{tikzpicture}[scale=0.5, baseline={(0,0.125)}] \draw[->] (0,0) --  (1,1);
		\draw[->] (1,0) --  (0,1);
	\end{tikzpicture}$ and 
	$-\begin{tikzpicture}[scale=0.5, baseline={(0,0.125)}]
		\draw[->] (0,0) --  (1,1);
		\draw[->] (1,0) --  (0,1);
	\end{tikzpicture}$, that is:
	\begin{equation*}
		\begin{tikzpicture}[baseline={(0,0.875)}]
			\draw[->] (0,0) -- (1,1) -- node[label=right:{$i$}, dot] {} (1,2);
			\draw[->] (1,0) -- (0,1) -- node[label=right:{$i$}, dot] {} (0,2);
		\end{tikzpicture}
		\quad
		= 
		\quad
		-
		\;\;
		\begin{tikzpicture}[baseline={(0,0.875)}] 
			\draw[->] (0,0) -- node[label=right:{$i$}, dot] {}  (0,1) -- (1,2);
			\draw[->] (1,0) -- node[label=right:{$i$}, dot] {}  (1,1) -- (0,2);
		\end{tikzpicture}.
	\end{equation*}
	To see this, recall that according to our convention explained in
	Remark \ref{rem:$2$-morphism-interchange-additive}, the diagram 
	$\begin{tikzpicture}[scale = 0.5, baseline={(0,0.125)}] 
		\draw[->] (0,0) -- node[label=right:{$i$}, dot] {}  (0,1);
		\draw[->] (1,0) -- node[label=right:{$i$}, dot] {}  (1,1);
	\end{tikzpicture}$
	should be read as 
	$\begin{tikzpicture}[scale = 0.5, baseline={(0,0.125)}] 
		\draw[->] (0,0) -- node[label=right:{$i$}, dot, pos=0.55] {} (0,1);
		\draw[->] (1,0) -- node[label=right:{$i$}, dot, pos=0.25] {} (1,1);
	\end{tikzpicture}$. Since $i$ has degree $-1$,  
	the graded interchange law states
	\[
	\begin{tikzpicture}[scale = 0.5, baseline={(0,0.125)}] 
		\draw[->] (0,0) -- node[label=right:{$i$}, dot, pos=0.25] {} (0,1);
		\draw[->] (1,0) -- node[label=right:{$i$}, dot, pos=0.55] {} (1,1);
	\end{tikzpicture} 
	= (-1)^{(-1)(-1)}
	\begin{tikzpicture}[scale = 0.5, baseline={(0,0.125)}] 
		\draw[->] (0,0) -- node[label=right:{$i$}, dot, pos=0.55] {} (0,1);
		\draw[->] (1,0) -- node[label=right:{$i$}, dot, pos=0.25] {} (1,1);
	\end{tikzpicture} 
	= 
	-
	\begin{tikzpicture}[scale = 0.5, baseline={(0,0.125)}] 
		\draw[->] (0,0) -- node[label=right:{$i$}, dot, pos=0.55] {} (0,1);
		\draw[->] (1,0) -- node[label=right:{$i$}, dot, pos=0.25] {} (1,1);
	\end{tikzpicture}.
	\]
	Consequently:
	\begin{equation*}
		\begin{tikzpicture}[baseline={(0,0.875)}]
			\draw[->] (0,0) -- (1,1) -- node[label=right:{$i$}, dot] {} (1,2);
			\draw[->] (1,0) -- (0,1) -- node[label=right:{$i$}, dot] {} (0,2);
		\end{tikzpicture}
		= \ 
		\begin{tikzpicture}[baseline={(0,0.875)}]
			\draw[->] (0,0) -- (1,1) -- node[label=right:{$i$}, dot, pos=0.25] {} (1,2);
			\draw[->] (1,0) -- (0,1) -- node[label=right:{$i$}, dot] {} (0,2);
		\end{tikzpicture}
		= \ 
		\begin{tikzpicture}[baseline={(0,0.875)}] 
			\draw[->] (0,0) -- node[label=right:{$i$}, dot, pos=0.25] {} (0,1) -- (1,2);
			\draw[->] (1,0) -- node[label=right:{$i$}, dot] {}  (1,1) -- (0,2);
		\end{tikzpicture}
		= \ 
		-
		\;\;
		\begin{tikzpicture}[baseline={(0,0.875)}] 
			\draw[->] (0,0) -- node[label=right:{$i$}, dot] {}  (0,1) -- (1,2);
			\draw[->] (1,0) -- node[label=right:{$i$}, dot] {}  (1,1) -- (0,2);
		\end{tikzpicture}.
	\end{equation*}
	
	For $Q$, let $i\colon a \rightarrow a[1]$ be the degree $1$ morphism in
	$\basecat$ defined by $\id_a$. Let $\iota'\colon \QQ_{a[1]} \rightarrow
	\QQ_{a}$ be the corresponding morphism
	$\begin{tikzpicture}[scale=0.5, baseline={(0,0.125)}]
		\draw[<-] (1,0) -- node[label=right:{$i$}, dot] {} (1,1);
	\end{tikzpicture}$
	in $\hcat*\basecat$. Moreover, let 
	$\iota\colon \QQ_{a[1]} \rightarrow \QQ_{a}[-1]$ be the degree zero
	morphism in $\hcat*\basecat$ defined by $\iota'$. Again, it is an isomorphism 
	with inverse $\iota^{-1}: \QQ_{a}[-1] \rightarrow  \QQ_{a[1]}$
	defined similarly by $\id_a$. The rest of the proof is similar. 
\end{proof}

This result affords us the following further relations:

\begin{Proposition}\label{prop:transposed_cat_heisenberg_relations}\leavevmode
	\begin{enumerate}
		\item\label{it:transposed_cat_heisenberg_relations1}  For any $a, b \in \basecat$ and $n, m \in \NN$ the following holds
		in $\hcat\basecat$: 
		\[
		\PP_a^{(1^m)}\PP_{b}^{(n)} \cong \PP_{b}^{(n)} \PP_a^{(1^m)}, \quad
		\QQ_a^{(1^m)}\QQ_{b}^{(n)} \cong \QQ_{b}^{(n)} \QQ_a^{(1^m)},
		\]
		\item\label{it:transposed_cat_heisenberg_relations2} For any $a, b \in \basecat$ and $n, m \in \NN$ we have a
		homotopy equivalence in $\hcat\basecat$:
		\begin{equation*}
			\bigoplus_{i=0}^{\min(m,n)} \bigwedge^i \Hom_{\basecat}(a, b) \otimes_\kk \PP_{b}^{(n-i)} \QQ_a^{(1^{m-i})} 
			\to
			\QQ_a^{(1^m)}\PP_{b}^{(n)}.
		\end{equation*}
		and thus the following holds in $H^*(\hcat\basecat)$ 
		\[
		\QQ_a^{(1^m)}\PP_{b}^{(n)} \cong \bigoplus_{i=0}^{\min(m,n)}\bigwedge^i \Hom_{H^*(\basecat)}(a,b) \otimes_\kk \PP_{b}^{(n-i)}\QQ_a^{(1^{m-i})}
		\]
	\end{enumerate}
	The above also holds with the roles of $(1^m)$ and $(n)$ interchanged.
\end{Proposition}

\begin{proof}
	Replace $a$ with $b[-1]$, resp. with $b[1]$ in Lemma~\ref{lemma:shifts_and_symmetrisers} to get (up to a shift) claim~\ref{it:transposed_cat_heisenberg_relations1} from  Theorem~\ref{thm:cat_heisenberg_relations-dg}~\ref{it:cat_heisenberg_relations-dg1}.
	Claim~\ref{it:transposed_cat_heisenberg_relations2} follows similarly from Theorem~\ref{thm:cat_heisenberg_relations-dg}~\ref{it:cat_heisenberg_relations-dg2} using the identification 
	\[ \mathrm{Sym}^i \mathrm{Hom}(a[1],b)\simeq \mathrm{Sym}^i (\mathrm{Hom}(a,b)[-1])\simeq \bigwedge^i (\mathrm{Hom}(a,b))[-1] \]
	of graded symmetric powers. For the final statement, apply the automorphism $F$.
\end{proof}

\begin{Example}\label{ex:CautisLicata_part3}
	Let $\Gamma \subset \mathrm{SL}(2,\CC)$ be a finite subgroup.
	In Example~\ref{ex:CautisLicata_part1} we defined the $1$-morphisms $P_i =
	\PP_{\sheaf E_i}$ and $Q_i = \QQ_{\sheaf E_i}[1]$ for each $i \in I_\Gamma$.
	Thus the $1$-morphism $Q_i^{(n)}$ of $\mathcal H^\Gamma$ in \cite{cautis2012heisenberg} corresponds to the $1$-morphism $\QQ_{\sheaf E_i[-1]}^{(1^n)}$.
	From \eqref{eq:cl_homs} one obtains 
	\[
	\Hom^*(\sheaf E_i[-1],\, \sheaf E_j) = 
	\begin{cases} 
		\CC[1] \oplus \CC[-1], & i = j \\
		\CC,                   & \langle i, j \rangle = -1 \\
		0,                     & \text{otherwise.}
	\end{cases}
	\]
	The $k$-th exterior power of $\CC[1] \oplus \CC[-1]$ is
	$\bigoplus_{j=0}^k \CC[k-2j]$. Identifying it with
	$H^{\ast}(\mathbb{P}^k)[k]$, we see that Proposition~\ref{prop:transposed_cat_heisenberg_relations} agrees with \cite[Proposition 2]{cautis2012heisenberg}:
	\begin{equation*}
		\begin{gathered} 
			P_i^{(m)}P_j^{(n)} \cong P_j^{(n)} P_i^{(m)}, \quad
			Q_i^{(m)}\QQ_j^{(n)} \cong Q_j^{(n)} Q_i^{(m)},
			\\[1ex]
			Q_i^{(m)}P_j^{(n)} \cong 
			\begin{cases}
				\bigoplus_{k=0}^{\min(m,n)} H^{\ast}(\mathbb{P}^k) [k] \otimes_\kk P_j^{(n-k)} Q_i^{(m-k)} & \textrm{ if } i = j \in I_{\Gamma} ,\\
				P_j^{(n)} Q_i^{(m)} \oplus P_j^{(n-1)}Q_i^{(m-1)} & \textrm{ if } \langle i, j\rangle=-1,\\
				P_j^{(n)} Q_i^{(m)} & \textrm{ if } \langle i, j\rangle = 0.
			\end{cases}
		\end{gathered}
	\end{equation*}
\end{Example}

\section{Grothendieck groups}\label{subsec:grothendieck_groups}

Recall the definition of the numerical Grothendieck group\index{numerical Grothendieck group} $\numGgp{}$ 
of a \dg category given in
Section~\ref{subsec:prelim_grothendieck_group}. It 
is the quotient of the usual Grothendieck group by the kernel of the Euler pairing.
Recall from Section~\ref{subsubsec:idempotent_modification} that we write $\halg\basecat$ for the idempotent modified
Heisenberg algebra of the lattice $(\numGgp\basecat,\, \chi)$.
We note again that we use the numerical Grothendieck group to ensure that this algebra has trivial centre.

In this section we compare $\halg\basecat$ to the Grothendieck group of 
the Heisenberg category $\hcat\basecat$.
Let $\mathrm{K}_0(\hcat\basecat,\,\kk)$ be the $\kk$-linear category with the same objects as $\hcat\basecat$ and morphism spaces
\[
\Hom_{\mathrm{K}_0(\hcat\basecat,\,\kk)}(\ho,\ho*) = \mathrm{K}_0\bigl(\Hom_{\hcat\basecat}(\ho,\ho*),\, \kk\bigr),
\]
where for any \dg category $\A$ we set $\mathrm{K}_0(\A,\,\kk) = \mathrm{K}_0(\A) \otimes_\ZZ \kk$.
As forming Grothendieck groups is functorial, 
the $1$-composition of $\hcat\basecat$ induces the composition on
$\mathrm{K}_0(\hcat\basecat,\,\kk)$. 

A closed string diagram defines an endomorphism of $\hunit$. 
Some of these endomorphisms are non-trivial and are not subject to any
relations. For example, those defined by clockwise bubbles, the
compositions of clockwise cups followed by clockwise caps. Thus 
the categories $\Hom_{\hcat\basecat}(\ho,\ho*)$ are not $\homm$-finite. Thus we cannot use the Euler pairing to obtain the corresponding numerical Grothendieck groups.

\begin{Remark}
	For $\basecat = \catdgfVect$ the Hom-spaces  of $\hcat\basecat$, while infinite-dimensional, are controlled by $\End(\hunit)$ and the degenerate affine Hecke algebra\index{degenerate affine Hecke algebra} \cite[Proposition 4]{khovanov2014heisenberg}.
	Some version of this observation is expected to hold more generally, see for example \cite[Conjecture~2]{cautis2012heisenberg}.
	It is not however clear how to define the degenerate affine Hecke
	algebra in our generality. We intend to return to this question in
	future work. Instead, we use an ad-hoc definition of the numerical
		Grothendieck group given in Definition~\ref{def:knumheis} below.
	
\end{Remark}

To kill the centre, we need to at least quotient each $\mathrm{K}_0\bigl(\Hom_{\hcat\basecat}(\ho,\ho*),\, \kk\bigr)$ by the classes $[\PP_a]$ and $[\QQ_a]$ for $[a]$ in the kernel of the Euler pairing on $\mathrm{K}_0(\basecat)$, as well as by any direct summands of these coming from the symmetric group action on parallel strands.

To formulate this, recall the functors of Remark~\ref{rem:Xiprime}:
\[
\Xi^{\prime\PP}_{\ho,\ho+n} \colon \symbc n \to
\Hom_{\hcat\basecat}(\ho,\,\ho+n).
\]
Taking h-perfect hulls we obtain functors
\[
\Xi^{\PP}_{\ho,\ho+n} \colon \hperf(\symbc n) \to \Hom_{\hcat\basecat}(\ho,\,\ho+n),
\]
and similarly contravariant functors $\Xi^{\QQ}_{\ho,\ho+n}$.
These further package up into $2$-functors
$$\Xi^{\PP}, \Xi^{\QQ}\colon \bihperf(\bicat{Sym}_\basecat)
\rightarrow \hcat\basecat. $$ 
As these are integral parts of the structure of $\hcat\basecat$, we expect 
them to descend to the numerical Grothendieck groups. We thus make 
the following definition:

\begin{Definition}
	\label{def:knumheis}
	Let $I$ be the two-sided ideal of $\mathrm{K}_0(\hcat\basecat,\, \kk)$ generated by the images under $\Xi^\PP$ and $\Xi^\QQ$ of the kernels of the Euler pairings on $\mathrm{K}_0(\bihperf(\bicat{Sym}_\basecat),\, \kk)$.
	The $1$-category $\numGgp{\hcat\basecat,\,\kk}$ is the quotient of $\mathrm{K}_0(\hcat\basecat,\,\kk)$ by $I$.
\end{Definition}

\begin{Remark} Recall the $1$-morphisms $\PP_a^\lambda$ and $\QQ_a^\lambda$ defined in Section~\ref{subsec:transposed_generators}.
	The ideal $I$ contains the classes $[\PP_a^\lambda]$ and $[\QQ_a^\lambda]$ for all $a \in \basecat$ with $[a]$ in the kernel of the Euler pairing
	and all Young diagrams $\lambda$.  
	
	If $I$ is generated by these classes, then using the Giambelli identity, $I$ is in this case equivalently generated by classes of the form $[\PP_a^{(n)}]$ and $[\QQ_a^{(n)}]$, see for example \cite[Remark~6]{cautis2012heisenberg}.
	This is exactly the minimal ideal one needs to quotient out in order for the Heisenberg algebra to have no centre.
	
	In general, however, there may exist images of additional homotopy 
	idempotents in the kernel of the Euler pairing on $\mathrm{K}_0(\symbc n,\,\kk)$.
	In order to catch these and to obtain the expected natural morphisms
	$\numGgp{\bicat{Sym}_\basecat,\,\kk} \to \numGgp{\hcat\basecat,\,\kk}$ one needs to use the less intuitive definition of $I$ given above.
\end{Remark}

{ At the outset, we completed $\basecat$ to $\hperf\basecat$, see 
	the introduction to Chapter~\ref{sec:dg-Heisenberg-2-cat}. 
	We can therefore choose 
	a basis of $\numGgp{\basecat}$ consisting of the classes of objects 
	of $\basecat$. The elements $p_a^{(n)}$ and $q_a^{(n)}$ indexed by
	the objects $a$ in this basis generate the Heisenberg algebra
	$\halg\basecat$.} Theorem~\ref{thm:cat_heisenberg_relations-dg} implies 
that there is a canonical morphism of $\kk$-algebras
\[
\pi\colon \halg\basecat \to \numGgp{\hcat\basecat,\,\kk}
\]
sending the generators $p_a^{(n)}$ to the class of $\PP_a^{(n)}$ and
$q_a^{(n)}$ to the class of $\QQ_a^{(n)}$.

\begin{Theorem}\label{thm:injective}
	The map $\pi\colon \halg\basecat \to \numGgp{\hcat\basecat,\,\kk}$ is an injective map of  $\kk$-algebras.
\end{Theorem}

\begin{proof}
	In Chapter~\ref{sec:cat_fock} we construct a categorical analogue of the Fock space together with a 2-representation of $\hcat\basecat$ on it.
	By Corollary~\ref{cor:KgpFock}, this 2-representation induces on the level of K-groups a homomorphism of algebras
	\begin{equation}\label{eq:prop:injective:representation}
		\halg\basecat \xrightarrow{\pi} \numGgp{\hcat\basecat,\,\kk} \to
		\End\left(\bigoplus_{\ho \geq 0}  \numGgp{\symbc\ho,\, \kk}\right). 
	\end{equation}
	As $1 \in \numGgp{\sym^0 \basecat,\, \kk} \cong \kk$ is annihilated by
	$\halg\basecat^- \setminus \{1_0\}$ and is fixed by $1_0$, 
	Lemma~\ref{lem:fock_embeds} produces an embedding
	\begin{equation} 
		\label{eq:falgembedstonumGgp}
		\falg{\basecat} \to \bigoplus_{\ho \geq 0} \numGgp{\symbc \ho,\, \kk}\end{equation}
	of the classical Fock space.
	Hence the representation \eqref{eq:prop:injective:representation} of $\halg\basecat$ on 
	\[\bigoplus_{\ho \geq 0} \numGgp{\symbc\ho,\, \kk}\] is faithful.
	Therefore $\pi$ is necessarily injective.
\end{proof}

Surjectivity of $\pi$ is a considerably subtler question, due to the possible 
appearance of additional homotopy idempotents when taking the perfect hull 
$\hcat*\basecat$.
This is closely related to the question of whether 
$\numGgp{\symbc \ho,\, \kk}$ and 
\[\falg{\basecat}^N = \bigoplus_{k_1+2k_2+\dots=N}\bigotimes_{i}\Sym^{k_i}  (\numGgp{\basecat,\, \kk}),\]
the degree $\ho$ part of the Fock space
are  isomorphic.
To the authors' knowledge, there exists no general criterion for this, cf.~the remarks in Section~\ref{subsec:grothfock}.

\begin{Conjecture}\label{conj:pi_iso}
	If the canonical morphism  $\falg{\basecat}^N \to \numGgp{\symbc \ho}$ is an isomorphism, then so is $\pi$.
\end{Conjecture}

The main content of the conjecture is that on the level of Grothendieck groups the operation of taking perfect hulls only adds the classes $[\PP_a^{(n)}]$ and $[\QQ_a^{(n)}]$ as additional generators.
On the homotopy categories, taking the perfect hull corresponds to
taking the triangulated hull and Karoubi completion\index{Karoubi
completion}.
Thus, alternatively, the statement is that the only relevant
idempotents in the homotopy category are those arising from the action
of the symmetric groups on upward\index{upward strand} or downward
strands\index{downward strand}.
We prove a converse to Conjecture~\ref{conj:pi_iso} in Section~\ref{subsec:fock_cat_2}.

We want to stress that a 2-representation of $\hcat\basecat$ is completely determined by the images of $\PP_a$, $\QQ_a$, $\RR_a$, and the generating $2$-morphisms.
Thus the possible appearance of additional idempotents in
$\hcat\basecat$ (i.e., $\pi$ being possibly non-surjective) does not
complicate the construction of categorical Heisenberg actions.

\begin{Example}
	\label{ex:khovanovnumgrp}
	Taking  $\basecat = \kk$, the $1$-morphisms in $\hcat\basecat$ are 
	homotopy direct summands of one-sided twisted complexes\index{twisted complex} of direct sums of $\PP_\kk$ and $\QQ_\kk$.
	As $\Hom_\basecat(\kk,\kk) = \kk$, such one-sided complexes are actual
	complexes and their morphisms are morphisms of complexes.
	Idempotents of such complexes must be idempotent in each degree.
	It follows that $\mathrm{K}_0(\hcat\basecat,\,\kk) = \numGgp{\hcat\basecat,\,\kk}$
	coincides with the Grothendieck group of Khovanov's category
	\cite{khovanov2014heisenberg}. By the main result of \cite{brundan2018degenerate} this further coincides with the infinite Heisenberg algebra.
\end{Example}

In general, $1$-morphisms in $\hcat\basecat$ may be one-sided twisted complexes
with non-trivial higher differentials. One cannot then simply 
take idempotents in each degree.  The conjecture says that the situation is however no worse than in $\symbc n$.

\section{Quantum enhancement}
Several previous works on Heisenberg categorification, like \cite{cautis2012heisenberg}, use a quantum deformation of the Heisenberg algebra. This \emph{quantum Heisenberg algebra}\index{quantum Heisenbergalgebra} $\halg\basecat^t$ has coefficients taken from $\kk[t,t^{-1}]$, where $t$ is a formal variable. 

For a graded vector space $V$ define
\[ [V] \coloneqq \sum_{n \in \ZZ} \dim\,V_n t^n. \]
Using this, the unital algebra $\chalg\basecat^t$ is defined by the same generators and relations as $\chalg\basecat$ except that relation \eqref{eq:heisrel3} is replaced by
\begin{equation*}\label{eq:heisrel3q}
	q_{a}^{(n)}p_{b}^{(m)} = 
	\sum_{k = 0}^{\mathclap{\min(m,n)}} [\Sym^k H^{\ast} \Hom(a, b)]\, p_{b}^{(m-k)}q_{a}^{(n-k)}.
\end{equation*}
Its idempotent modification\index{idempotent modification} $\halg\basecat^q$ is then obtained exactly as in Section~\ref{subsubsec:idempotent_modification}. 

\begin{Example}
	For $n \in \NN$, let $[n]$ denote the quantum integer
	\[ [n] \coloneqq \frac{t^{-n}-t^{n}}{t^{-1}-t}=t^{-n+1}+t^{-n+3}+\dots+t^{n-3}+t^{n-1}.\]
	Note that when setting $t=1$ in the last expression, one gets $[n]=n$.
	Define moreover $[n] := [-n]$ for $n \in \ZZ_{<0}$. 
	Suppose that there is a set of generating objects of $\basecat$ such that Hom-spaces between these objects satisfy
	\[ [H^{\ast} \Hom(a, b)] = [\langle a, b \rangle_{\chi}].\]
	If moreover the form $\chi$ is symmetric, then our definition coincides with \cite[Definition~5.1]{suarez2017integral} (see also \cite[Equation (6)]{cautis2012heisenberg}). These conditions hold e.g. in Example \ref{ex:CautisLicata_part3}.
\end{Example}	




Since $\hcat\basecat$ is graded, $\numGgp{\hcat\basecat,\,\kk}$ is an algebra over $\kk[t,t^{-1}]$, where $t$ acts via the shift. Similarly, $\numGgp{\symbc\ho,\, \kk}$ is naturally a $\kk[t,t^{-1}]$-module, such that $\numGgp{\hcat\basecat,\,\kk}$ acts $\kk[t,t^{-1}]$-linearly on it. Hence, there is a $\kk[t,t^{-1}]$-algebra homomorphism  
\[\numGgp{\hcat\basecat,\,\kk} \to \End_{\kk[t,t^{-1}]}\left(\bigoplus_{\ho \geq 0}  \numGgp{\symbc\ho,\, \kk}\right).\]
\begin{Proposition}
	The morphism $\pi$ extends to an injective map of $\kk[t,t^{-1}]$-algebras
	\[ \pi: \halg\basecat^t \to \numGgp{\hcat\basecat,\,\kk}.\]
\end{Proposition}
\begin{proof} We need only show that $\pi$ is a map of $\kk[t,t^{-1}]$-algebras, that is, $\pi$ is compatible with the $t$-action on the source and the target. This is straigthforward from the definitions.
\end{proof}

Letting $\halg\basecat^{t-} \subset \halg\basecat^{t}$ denote again the subalgebra generated by the set 
\[
\bigl\{ q_a^{(n)}1_k : a \in M,\, k \leq 0,\, n \geq 0 \bigr\},
\] 
the quantum Fock space is obtained as the induced representation
\[ 
\falg\basecat^t = \Ind_{\halg\basecat^{t-}}^{\halg\basecat^{t}}(\triv_0) \cong \halg\basecat^{t} \otimes_{\halg\basecat^{t-}} \kk[t,t^{-1}].
\]
The embedding \eqref{eq:falgembedstonumGgp} is also compatible with the shift, so it can be enhanced to
\[
\falg\basecat^t \to \bigoplus_{\ho \geq 0}  \numGgp{\symbc\ho,\, \kk}.
\]
%

\chapter{The Categorical Fock Space}\label{sec:cat_fock}

As in the additive case, we construct a category called the
categorical Fock space from the symmetric powers of the \dg category
$\basecat$. We show that the Heisenberg category $\hcat\basecat$ 
acts on this categorical Fock space. The relation between
this representation and the classical Fock space representation is
explored in the next section.

\section{Symmetric powers of \texorpdfstring{\dg}{DG} categories}
\label{sec:symmetric-powers-of-dg-categories}

Recall from Definition~\ref{def:symmetric-power} that the $\ho$th symmetric power of $\basecat$ is defined as $\symbc\ho = \basecat^{\otimes \ho} \rtimes \SymGrp\ho$.

\begin{Example}
	\label{ex:symcomp}
	If $X$ is a scheme, then $\sym^{\ho}\catDGCoh{X} \cong \catDGCoh{X^{\ho}}^{\SymGrp \ho}$ 
	is Morita equivalent to the standard \dg enhancement 
	$\catDGCoh{[X^{\ho}/\SymGrp \ho]}$ of the $\ho$-th symmetric quotient stack\index{symmetric quotient stack} of $X$.
	We thus have $\catDc(\sym^{\ho}\catDGCoh{X}) \simeq \catDbCoh{[X^{\ho}/\SymGrp \ho]}$, 
	the derived category of $\SymGrp \ho$-equivariant perfect complexes on $X^{\ho}$ 
	\cite[Example 2.2.8(a)]{SymCat}.
\end{Example}

\begin{Definition}
	For any $1 \leq k \leq \ho$ define the group monomorphism 
	\[
	\iota_k\colon \SymGrp{\ho-1} \hookrightarrow \SymGrp\ho,
	\]
	by identifying $\SymGrp {\ho-1}$ with the subgroup of $\SymGrp \ho$
	consisting of permutations which keep $k$ fixed. 
\end{Definition}

\begin{Lemma}
	\label{lem-the-decomposition-of-Sn+1-into-Sn}
	The group $\SymGrp \ho$ admits the following decomposition into $\SymGrp{\ho-1}$-cosets:
	\begin{equation*}
		\SymGrp{\ho} =
		\sum_{i = 1}^{\ho} (1i) \iota_{1}(\SymGrp{\ho-1}) = 
		\sum_{i = 1}^{\ho} \iota_{1}(\SymGrp{\ho-1}) (1i)
	\end{equation*}
\end{Lemma}

This observation
can be used to rearrange the complete decomposition 
\eqref{equation-decomposition-of-the-equivariant-diagonal-bimodule}
of the diagonal bimodule of $\symbc\ho$ as follows.

\begin{Corollary}
	\label{cor-decomposition-of-the-diagonal-bimodule-S_l-into-S_l-1}
	There is the following direct sum decompositions of the diagonal bimodule: 
	\begin{equation*}
		\symbc\ho \simeq 
		\bigoplus_{i = 1}^{\ho} 
		\leftidx{_{i}}{\basecat}_1 \otimes 
		\leftidx{_{\hat{1} \circ (1i)}}{\left(\symbc{\ho-1}\right)}{_{\hat{1}}}
		\simeq
		\bigoplus_{i = 1}^{\ho}
		\leftidx{_{1}}{\basecat}_i 
		\otimes 
		\leftidx{_{\hat{1}}}{\left(\symbc{\ho-1}\right)}{_{\hat{1} \circ (1i)}}
	\end{equation*}
	where the left and right indices denote taking the left and right
	arguments of the bimodule $\symbc\ho$ 
	and applying the following:
	\begin{itemize}
		\item for $i$ in $\left\{1,\dots,n\right\}$, 
		the map $i\colon  \mathrm{Ob}(\symbc\ho) \rightarrow  \mathrm{Ob}(\basecat)$ 
		projects to the $i$-th factor, 
		\item for $i$ in $\left\{1,\dots,n\right\}$, the map 
		$\hat{i}\colon  \mathrm{Ob}(\symbc\ho) \rightarrow \mathrm{Ob}(\basecat^{\otimes(\ho-1)})$
		projects to all factors but $i$-th,
		\item for $i$, $j$ in $\left\{1,\dots,n\right\}$, the map $(ij)\colon 
		\mathrm{Ob}(\symbc\ho) \rightarrow \mathrm{Ob}(\symbc\ho)$ transposes 
		$i$-th and $j$-th factors. 
	\end{itemize}
\end{Corollary}

We illustrate this notation. Let $\underline{a} = a_1 \otimes
\cdots \otimes a_{\ho}, \; \underline{b} = b_1 \otimes \cdots \otimes b_{\ho} \in \symbc\ho$. Then 
\[
\leftidx
{_{\underline{b}}}
{\left(\symbc\ho\right)}
{_{\underline{a}}}
= 
\homm_{\symbc\ho}\left(a_1 \otimes \cdots \otimes a_{\ho},\, b_1 \otimes \cdots \otimes b_{\ho}\right). 
\]
Our notation gives 
\begin{equation*}
	\leftidx
	{_{\underline{b}}}
	{\left(\leftidx{_{i}}{\basecat}{_1}\right)}
	{_{\underline{a}}}
	= \homm_{\basecat}(a_1,b_i),
	\end{equation*}
and
\begin{multline*}
\leftidx
{_{\underline{b}}}
{\left(\leftidx{_{\hat{1} \circ (1i)}}{\left(\symbc{\ho-1}\right)}{_{\hat{1}}}
	\right)}
{_{\underline{a}}}
=\\ 
\homm_{\symbc{\ho-1}}\left(a_2 \otimes \cdots \otimes a_{\ho},\,
b_2 \otimes \cdots \otimes b_{i-1} \otimes b_1 \otimes b_{i+1} \otimes
\cdots \otimes b_{\ho}\right).
\end{multline*}
It is clear that there is natural inclusion of \dg $\kk$-modules
\begin{equation*}
	\begin{tikzcd}[row sep = 0.5cm, font=\small]
		\homm_{\basecat}(a_1,b_i) \otimes
		\homm_{\symbc{\ho-1}}\left(a_2 \otimes \cdots \otimes a_{\ho},\,
		b_2 \otimes \cdots \otimes b_{i-1} \otimes b_1 \otimes b_{i+1} \otimes
		\cdots \otimes b_{\ho}\right)
		\ar[hook]{d}
		\\
		\homm_{\symbc\ho}\left(a_1 \otimes
		\cdots \otimes a_{\ho},\, b_1 \otimes \cdots \otimes b_{\ho}\right), 
	\end{tikzcd}
\end{equation*}
and the proof below demonstrates that summing this over all $i \in \{1,
\dots, \ho\}$ gives a complete decomposition of the diagonal bimodule. 

Let us stress that the index maps $i$, $\hat{i}$ and $(ij)$ are
maps of sets and are not functorial. Thus the expressions 
like $\leftidx{_{1}}{\basecat}_1$ in
Corollary~\ref{cor-decomposition-of-the-diagonal-bimodule-S_l-into-S_l-1} are
not $\symbc\ho$-bimodules by themselves:
while 
\[
{_{\underline{b}}} {\left(\leftidx{_{i}}{\basecat}{_1}\right)}
{_{\underline{a}}} = \homm_{\basecat}(a_1,b_i)
\]
is perfectly well-defined, one cannot uniquely pick out the first factor in 
some \[\alpha \in \Hom_{\symbc\ho}(\underline{b},\, \underline{b}')\]
to act with it on $\homm_{\basecat}(a_1,b_i)$. Nonetheless,
if we use Lemma \ref{lem-the-decomposition-of-Sn+1-into-Sn}
to decompose $\alpha$ with respect to the permutation type  
into $\sum \alpha_i$, then each $\alpha_i$ does act naturally 
on the summand
$\leftidx{_{i}}{\basecat}_1 \otimes \leftidx{_{\hat{1} \circ
		(1i)}}{\left(\symbc{\ho-1}\right)}{_{\hat{1}}}$.  
Thus we can view
Corollary~\ref{cor-decomposition-of-the-diagonal-bimodule-S_l-into-S_l-1} 
as an isomorphism of $\symbc\ho$-bimodules, with 
the index maps indicating the left and right actions
of $\symbc\ho$ on the decompositions. 

\begin{proof}[Proof of Corollary~\ref{cor-decomposition-of-the-diagonal-bimodule-S_l-into-S_l-1}]
	First, by the decomposition 
	\eqref{equation-decomposition-of-the-equivariant-diagonal-bimodule}
	we have:
	\begin{equation*}
		\symbc{\ho}
		\simeq
		\bigoplus_{\sigma \in \SymGrp \ho}
		\left(\basecat^{\otimes \ho}\right)_\sigma. 
	\end{equation*}
	We then use the decomposition $\SymGrp \ho = \sum_{i = 1}^{\ho} (1i) \iota_{1}(\SymGrp{\ho-1})$ from  Lemma~\ref{lem-the-decomposition-of-Sn+1-into-Sn}
	to obtain 
	\begin{equation*}
		\bigoplus_{\sigma \in \SymGrp \ho}
		\left(\basecat^{\otimes \ho}\right)_{\sigma} 
		\simeq 
		\bigoplus_{i = 1}^{\ho}
		\bigoplus_{\sigma \in \SymGrp{\ho-1}}
		\left(\basecat^{\otimes \ho}\right)_{(1i)\iota_1(\sigma)}. 
	\end{equation*}
	The  $\basecat^{\otimes \ho}$-bimodule isomorphism 
	$ (\basecat^{\otimes \ho})_{(1i)\iota_1(\sigma)} 
	\simeq 
	\leftidx{_{(1i)}}{\left(\basecat^{\otimes \ho}\right)}{_{\iota_1(\sigma)}}$
	given by $\alpha \mapsto (1i)\cdot\alpha$ implies that
	\begin{equation*}
		\bigoplus_{i = 1}^{\ho}
		\bigoplus_{\sigma \in \SymGrp{\ho-1}}
		\left(\basecat^{\otimes \ho}\right)_{(1i)\iota_1(\sigma)}
		\simeq 
		\bigoplus_{i = 1}^{\ho}
		\bigoplus_{\sigma \in \SymGrp{\ho-1}}
		\leftidx{_{(1i)}}{\left(\basecat^{\otimes \ho}\right)}{_{\iota_1(\sigma)}}. 
	\end{equation*}
	Now we can decompose $\basecat^{\otimes \ho}$ into 
	$\leftidx{_1}\basecat_1 \otimes
	\leftidx{_{\hat{1}}}{\left(\basecat^{\otimes(\ho-1)}\right)}{_{\hat{1}}}$, 
	which further gives us
	\begin{equation*}
		\bigoplus_{i = 1}^{\ho}
		\bigoplus_{\sigma \in \SymGrp{\ho-1}}
		\leftidx{_{(1i)}}{\left(\basecat^{\otimes \ho}\right)}{_{\iota_1(\sigma)}}
		= 
		\bigoplus_{i = 1}^{\ho}
		\bigoplus_{\sigma \in \SymGrp{\ho-1}}
		\leftidx{_i}{\basecat}{_1} \otimes
		\leftidx{_{\hat{1}\circ(1i)}}{\left(\basecat^{\otimes(\ho-1)}\right)}{_{\sigma
				\circ \hat{1}}}. 
	\end{equation*}
	Finally, by 
	\eqref{equation-decomposition-of-the-equivariant-diagonal-bimodule}
	we have $\symbc{\ho-1} = 
	\bigoplus_{\sigma \in \SymGrp{\ho-1}} \left(\basecat^{\otimes(\ho-1)}\right)_{\sigma}$
	and therefore 
	\begin{equation*}
		\bigoplus_{i = 1}^{\ho}
		\bigoplus_{\sigma \in \SymGrp{\ho-1}}
		\leftidx{_i}\basecat_1 \otimes
		\leftidx{_{\hat{1}\circ(1i)}}{\left(\basecat^{\otimes(\ho-1)}\right)}{_{\sigma  \circ {\hat{1}}}}
		\simeq 
		\bigoplus_{i = 1}^{\ho}
		\leftidx{_i}{\basecat}{_1} \otimes
		\leftidx{_{\hat{1}\circ(1i)}}{\left(\symbc{\ho-1}\right)}{_{\hat{1}}}. 
	\end{equation*}
	This establishes the first decomposition. The second decomposition 
	is proved similarly.  
\end{proof}

Recall from Section~\ref{subsec:equivariant_cats} that $\symbc{\ho}$ and $\basecat^{\otimes \ho}$ have 
the same objects, while  
the morphisms of $\symbc{\ho}$ are 
generated under composition by those of $\basecat^{\otimes \ho}$ plus 
the formal isomorphisms corresponding to the elements of $\SymGrp{\ho}$. 
Thus the data of a \dg functor from $\symbc{\ho}$
to some \dg category $\B$ is the data of a functor 
$\basecat^{\otimes \ho} \rightarrow \B$ plus the data of where
the formal isomorphisms go.

\begin{Definition}
	\label{def:phia}
	Let  $a \in \basecat$.  Define the functor 
	\begin{equation*}
		\phi_a \colon  \symbc{\ho - 1} \rightarrow \symbc\ho
	\end{equation*}
	to be the extension of the functor 
	\begin{equation*}
		\basecat^{\otimes(\ho-1)} \xrightarrow{a \otimes \id} 
		\basecat^{\otimes \ho} 
	\end{equation*}
	which sends the formal isomorphisms of $\SymGrp{\ho-1}$ to those
	of $\SymGrp \ho$ via \[\iota_1\colon \SymGrp{\ho-1} \hookrightarrow
	\SymGrp{\ho},\] the embedding as the subgroup of permutations which are trivial on the first element. 
\end{Definition}

As explained in Section~\ref{section-restriction-and-extension-of-scalars}, 
we have three induced functors
\begin{align*}
	\phi^*_a\colon &
	\rightmod{\symbc{\ho-1}} \to \rightmod{\symbc\ho},
	\\
	\phi_{a *}\colon &
	\rightmod{\symbc\ho} \to \rightmod{\symbc{\ho-1}},
	\\
	\phi^!_a\colon &
	\rightmod{\symbc{\ho-1}} \to \rightmod{\symbc\ho},
\end{align*}
which form two adjoint pairs 
$(\phi_a^*,\, \phi_{a *})$ and $(\phi_{a *},\, \phi_a^!)$. 
The action of the first two functors on representable object\index{representable object}s can be described as follows. 
\begin{Lemma}\label{lem:onrepresentables}\leavevmode
	Let $h^r$ denote right representable modules, 
	as per Section~\ref{section-dg-categories-and-dg-modules}. Then: 
	
	\begin{enumerate}
		\item\label{it-phi_a^*-on-representables}
		For any $a_1 \otimes \dots \otimes a_{\ho-1} \in \symbc{\ho-1}$ we have
		\begin{equation*}
			\phi_a^*\bigl(h^r(a_1 \otimes \dots \otimes a_{\ho-1})\bigr)
			\simeq 
			h^r(a \otimes a_1\otimes \dots \otimes a_{\ho-1}).
		\end{equation*}
		\item\label{it-phi_a_*-on-representables}
		For any $a_1 \otimes \dots \otimes a_{\ho} \in \symbc\ho$ we have
		\begin{equation*}
			\phi_{a*}\bigl(h^r(a_1 \otimes \cdots \otimes a_{\ho})\bigr)
			\simeq 
			\bigoplus_{i = 1}^{\ho}
			\homm_{\basecat}(a,a_i) \otimes
			h^r(a_1 \otimes \cdots \widehat{a_{i}} \cdots \otimes a_{\ho})
		\end{equation*}
	\end{enumerate}
\end{Lemma}

\begin{proof}
	For Part~\ref{it-phi_a^*-on-representables}, we have:
	\begin{equation*}
		\begin{multlined}
		\phi_a^*\bigl(h^r(a_1 \otimes \dots \otimes a_{\ho-1})\bigr) 
		:=
		h^r(a_1 \otimes \dots \otimes a_{\ho-1})
		\otimes_{\symbc{\ho-1}}
		\leftidx{_{\phi_a}}{\symbc\ho}
		\\ \simeq 
		h^r(a \otimes a_1\otimes \dots \otimes a_{\ho-1}).
		\end{multlined}
	\end{equation*}
	
	For Part~\ref{it-phi_a_*-on-representables}, we have 
	\begin{equation*}
		\phi_{a*}\bigl(h^r(a_1 \otimes \dots \otimes a_{\ho})\bigr) 
		:= 
		h^r(a_1 \otimes \dots \otimes a_{\ho})
		\otimes_{\symbc\ho}
		\symbc\ho_{\phi_a}
	\end{equation*}
	By Corollary~\ref{cor-decomposition-of-the-diagonal-bimodule-S_l-into-S_l-1}
	we have 
	\begin{equation*}
		{\symbc \ho}_{\phi_a}
		\simeq 
		\bigoplus_{i = 1}^{\ho}
		\leftidx{_{i}}{\basecat}{_a} \otimes 
		\leftidx{_{\hat{1} \circ (1i)}}{\symbc{\ho-1}}
	\end{equation*}
	and hence 
	\begin{equation*}
		\begin{multlined}
		h^r(a_1 \otimes \dots \otimes a_{\ho})
		\otimes_{\symbc\ho}
		\symbc\ho_{\phi_a}
		\\ \simeq 
		\bigoplus_{i = 1}^{\ho}
		\leftidx{_{a_{i}}}{\basecat}{_a} \otimes 
		\leftidx{_{a_2 \otimes \dots \otimes a_{i-1} \otimes a_1 \otimes
				a_{i+1} \otimes \dots \otimes a_{\ho}}}
		{(\basecat^{\otimes(\ho-1)} \rtimes \SymGrp{\ho-1})}. 
\end{multlined}	
\end{equation*}
	Since in $\symbc{\ho-1}$ we have 
	\[a_2 \otimes \cdots \otimes a_{i-1} \otimes a_1 \otimes a_{i+1} \otimes \dots a_{\ho}
	\simeq a_1 \otimes \cdots \widehat{a_{i}} \cdots \otimes a_{\ho},\] we
	have 
	\begin{equation*}
		\phi_{a*}\bigl(h^r(a_1 \otimes \cdots \otimes a_{\ho})\bigr)  
		\simeq  
		\bigoplus_{i = 1}^{\ho}
		\homm_{\basecat}(a,a_i) \otimes
		h^r(a_1 \otimes \cdots \widehat{a_{i}} \cdots \otimes a_{\ho}).
		\qedhere
	\end{equation*}
\end{proof}

Lemma~\ref{lem:onrepresentables}~\ref{it-phi_a^*-on-representables} shows that 
the bimodule $\leftidx{_{\phi_a}}{\symbc\ho}$ 
defining $\phi^*_{a}$ is always right-representable. 
Thus it is always right-perfect and right-h-projective. 
On the other hand, by Lemma~\ref{lem:onrepresentables}~\ref{it-phi_a_*-on-representables} the bimodule 
${\symbc\ho}_{\phi_a}$ defining $\phi_{a *}$ is always right-h-flat,
but is right-perfect and right h-projective if and only if 
$\basecat$ is proper. Similarly, 
${\symbc\ho}_{\phi_a}$ is always left representable,  
while $\leftidx{_{\phi_a}}{\symbc\ho}$ 
is always left-h-flat, but is left-perfect and left-h-projective
if and only if $\basecat$ is proper.  
We conclude that when $\basecat$ is proper 
both $\leftidx{_{\phi_a}}{\symbc\ho}$
and ${\symbc\ho}_{\phi_a}$ are left- and right-perfect and left- and
right-h-projective. In particular, they define $1$-morphisms in
$\EnhCatKCdg$ and, by abuse of notation, we denote these again 
by $\phi^*_{a}$ and $\phi_{a *}$, respectively. 

The twisted inverse image functor $\phi^!_a$ is not a priori 
a functor of tensoring with a bimodule. However, in presence of a
homotopy Serre functor, it is quasi-isomorphic to one:

\begin{Proposition}
	\label{prop-phi-a-shriek-isomorphic-to-phi-Sa-upper-star}
	Let $\basecat$ be proper and assume that $\basecat$ admits a homotopy Serre
	functor $S$. Then there is a quasi-isomorphism of \dg functors
	\[ \starmap{a}\colon \phi_{Sa}^* \to \phi_a^!. \]
\end{Proposition}

\begin{proof}
	Let $E \in \rightmod{\symbc{\ho-1}}$. 
	By Corollary~\ref{cor-decomposition-of-the-diagonal-bimodule-S_l-into-S_l-1} we have
	\begin{equation*}
		\phi_{Sa}^* E
		= 
		E \otimes_{\symbc{\ho-1}}\;\leftidx{_{\phi_{Sa}}}{\symbc\ho}
		\simeq 
		E \otimes_{\symbc{\ho-1}}
		\left(
		\bigoplus_{i = 1}^{\ho}
		\leftidx{_{Sa}}{\basecat}_i \otimes_\kk 
		\symbc{\ho-1}_{\hat{1} \circ (1i)}
		\right). 
	\end{equation*}
	
	The homotopy Serre functor $S$ on $\basecat$ comes with a
	quasi-isomorphism \[\eta\colon \basecat \rightarrow
	\left(\leftidx{_S}{\basecat}{}\right)^*.\] Since $\basecat$ is proper, 
	$\eta^*\colon \leftidx{_S}{\basecat}{} \rightarrow \basecat^*$
	is also a quasi-isomorphism. Hence so is
	\begin{equation}\label{eq:Snatmorph}
	\begin{multlined}
		E \otimes_{\symbc{\ho-1}}
		\left(
		\bigoplus_{i = 1}^{\ho}
		\leftidx{_{Sa}}{\basecat}_i 
		\otimes_\kk
		\symbc{\ho-1}_{\hat{1} \circ (1i)}
		\right)
		\\ \to 
		E \otimes_{\symbc{\ho-1}}
		\left(
		\bigoplus_{i = 1}^{\ho}
		(\leftidx{_{i}}{\basecat}_a)^*
		\otimes_\kk 
		\symbc{\ho-1}_{\hat{1} \circ (1i)}
		\right).
	\end{multlined}
	\end{equation}
	Since $\leftidx{_{\hat{1} \circ (1i)}}{\symbc{\ho-1}}$ are representables, we have
	\begin{equation*}
		\homm_{\symbc{\ho-1}}
		(
		\leftidx{_{\hat{1} \circ (1i)}}
		{\symbc{\ho-1}},\, E
		)
		\simeq 
		E_{\hat{1} \circ (1i)}
		\simeq 
		E \otimes_{\symbc{\ho-1}}
		\symbc{\ho-1}_{\hat{1} \circ (1i)}
	\end{equation*}
	and therefore 
	\begin{equation*}
	\begin{multlined}
		E \otimes_{\symbc{\ho-1}}
		\left(
		\bigoplus_{i = 1}^{\ho}
		(\leftidx{_{i}}{\basecat}_a)^*
		\otimes 
		\symbc{\ho-1}_{\hat{1} \circ (1i)}
		\right)
		\\ 
		\simeq\,
		\bigoplus_{i = 1}^{\ho}
		(\leftidx{_{i}}{\basecat}_a)^* \otimes 
		\homm_{\symbc{\ho-1}}
		\left(
		\leftidx{_{\hat{1} \circ (1i)}}{\symbc{\ho-1}},\, E
		\right). 
	\end{multlined}
	\end{equation*}
	Since $\basecat$ is proper, $\leftidx{_{i}}{\basecat}{_a}$ are perfect 
	as $\kk$-modules. Thus the natural map 
	\begin{equation*}
		\begin{multlined}
		\bigoplus_{i = 1}^{\ho}
		\left(
		\leftidx{_{i}}{\basecat}_a
		\right)^* 
		\otimes 
		\homm_{\symbc{\ho-1}}
		\left(
		\leftidx{_{\hat{1} \circ (1i)}}{\symbc{\ho-1}},\, E
		\right)
		\\ \longrightarrow 
		\bigoplus_{i = 1}^{\ho}
		\homm_{\symbc{\ho-1}}
		\left(
		\leftidx{_{i}}{\basecat}{_a} 
		\otimes 
		\leftidx{_{\hat{1} \circ (1i)}}{\symbc{\ho-1}},\, E
		\right),
		\end{multlined}
	\end{equation*}
	is a quasi-isomorphism. Finally, by
	Corollary~\ref{cor-decomposition-of-the-diagonal-bimodule-S_l-into-S_l-1}
	again, we have  
	\begin{equation*}
		\bigoplus_{i = 1}^{\ho}
		\homm_{\symbc{\ho-1}}
		\left(
		\leftidx{_{i}}{\basecat}{_a} 
		\otimes 
		\leftidx{_{\hat{1} \circ (1i)}}{\symbc{\ho-1}},\, E
		\right)
		\simeq
		\homm_{\symbc{\ho-1}}
		\left(\symbc{\ho}_{\phi_a},\, E\right) 
		=
		\phi_a^! E.
	\end{equation*}
\end{proof}

\begin{Corollary}
	Let $\basecat$ be proper and assume it admits a homotopy Serre
	functor $S$. The bimodule approximation $\bimodapx(\phi_a^!)$ is a
	right- and left-perfect and left-h-projective
	$\symbc{\ho-1}$-$\symbc{\ho}$-bimodule.
\end{Corollary}

\begin{proof}
	By the definition of the bimodule approximation functor in
	Section~\ref{section-bimodule-approximation}, for any $b \in \symbc{\ho-1}$,
	the fibre $\leftidx{_b}{\bimodapx(\phi_a^!)}$ is the $\symbc{\ho}$-module
	$\phi_a^!(h^r(b))$. By Proposition
	\ref{prop-phi-a-shriek-isomorphic-to-phi-Sa-upper-star},
	$\phi_a^!(h^r(b))$
	is quasi-isomorphic to $\phi_{Sa}^*(h^r(b))$. Since the latter is the
	representable object $h^r(\phi_a(b))$, we conclude that the former is
	perfect. 
	
	Now let $c \in \symbc{\ho}$. We have 
	\[
	\begin{multlined}
	{\bimodapx(\phi_a^!)}_c = 
	{\phi_a^! ({\symbc{\ho-1}})}_c
	\simeq  
	\homm_{\symbc{\ho-1}}
	\left(\leftidx{_c}{(\symbc{\ho})}_{\phi_a},\, \symbc{\ho-1}\right) 
	\\ = 
	\homm_{\symbc{\ho-1}}
	\left(\phi_{a *}(c),\, \symbc{\ho-1}\right). 
	\end{multlined}
	\]
	It is well known that the dualisation functor\index{dualisation functor} sends h-projective and 
	perfect modules to h-projective and perfect modules
	\cite[Section~2.2]{AnnoLogvinenko-SphericalDGFunctors}. 
	By Lemma~\ref{lem:onrepresentables}~\ref{it-phi_a_*-on-representables}
	and properness of $\basecat$, the $\symbc{\ho-1}$-module $\phi_{a *}(c)$ is 
	h-projective and perfect, hence so is its dual  
	$\homm_{\symbc{\ho-1}} \left(\phi_{a *}(c),\, \symbc{\ho-1}\right)$. 
\end{proof}

\section{The categorical Fock space \texorpdfstring{$\fcat\basecat$}{F}}
\label{subsec:catfockconst}

In Chapter~\ref{sec:dg-Heisenberg-2-cat} we fixed a smooth and proper 
enhanced triangulated category $\basecat \in \EnhCatKCdg$.
That is, $\basecat$ is a smooth and proper \dg category 
considered as a Morita enhancement\index{Morita
enhancement} of the triangulated category 
$\catDc(\basecat) = \Hzero(\hperf \basecat)$. As 
$\basecat$ is smooth and proper, it admits 
an enhanced Serre functor $S$ given by the bimodule
$\basecat^*$ 
\cite{Shklyarov-OnSerreDualityForCompactHomologicallySmoothDGAlgebras}, 
and in Section~\ref{subsec:dg-homotopy-serre} we proved that it lifts to a
homotopy Serre functor $S$ on $\hperf \basecat$.
Replacing $\basecat$ by $\hperf\basecat$ if necessary, we can assume that $\basecat$ itself admits a homotopy Serre functor $S$.

We then defined the Heisenberg $2$-category $\hcat\basecat$ of $\basecat$. 
It was constructed in two steps:
\begin{enumerate}
	\item First, we defined  in Sections~\ref{subsec:heisencatdef-dg} and \ref{subsec:heisencatdef2-dg} a strict \dg $2$-category
	$\hcat*{\basecat}$. 
	Its object set is $\mathbb{Z}$, its $1$-morphisms are freely generated 
	by formal symbols $\PP_a, \RR_a\colon \ho \rightarrow \ho + 1$ and
	$\QQ_a\colon \ho \rightarrow \ho - 1$ for $a \in \basecat$, and its 
	$2$-morphisms are certain string diagrams connecting up the endpoints
	which correspond to $\PP$s, $\QQ$s, and $\RR$s of the source and target 
	$1$-morphisms. 
	
	\item Next,  in Section~\ref{subsec:idempotent-hcat} we took
	the perfect hull (see Section~\ref{subsec:hperf}) of
	$\hcat*{\basecat}$ and then a monoidal Drinfeld quotient\index{monoidal Drinfeld quotient} (see Section~\ref{subsec:monoidal-drinfeld-quotients})of
	$\bihperf(\hcat*{\basecat})$ by a certain $2$-sided ideal
	$I_{\basecat}$ of $1$-morphisms. This was to make each $\RR_a$
	homotopy equivalent to $\PP_{Sa}$ and impose a certain homotopy
	relation on $\PP$s and $\QQ$s.  The resulting $\HoDGCat$-enriched
	bicategory is the Heisenberg $2$-category $\hcat\basecat$. 
\end{enumerate}

Our next aim is to construct a $2$-representation $\fcat{\basecat}$ of 
$\hcat\basecat$ analogous to the Fock space representation of
a Heisenberg algebra. 

\begin{Lemma} 
	\label{lemma-enhcatkcdg-to-hperf-enhcatkcdg-is-quasi-equivalence}
	The Yoneda embedding of $\EnhCatKCdg$ into
	$\bihperf(\EnhCatKCdg)$ is a quasi-equivalence. In particular, both of
	these are \dg enhancements of the strict $2$-category $\EnhCatKC$ of
	enhanced triangulated categories.
\end{Lemma}

\begin{proof}
	The procedure of taking the perfect hull does not change the Morita equivalence class of a
	\dg category and, if the \dg category is pre-triangulated and its homotopy
	category is Karoubi-complete, it does not change
	the homotopy category either. The $1$-morphism categories 
	$\homm_{\EnhCatKCdg}(\A,\B)$ of $\EnhCatKCdg$ are defined so that
	their homotopy categories are $D_{\Bperf}(\AbimB)$
	In particular, they are triangulated and Karoubi-complete. We conclude
	that taking the perfect hull of $\EnhCatKCdg$ does not change its homotopy 
	$2$-category.
\end{proof}

We therefore make the following definition.

\begin{Definition}
	\label{def:catfockspacedg}
	\leavevmode
	\begin{enumerate}
		\item 
		The strict \dg $2$-category $\fcat*{\basecat}$ is the 
		$1$-full subcategory of $\DGModCat$ (see 
		Section~\ref{section-bimodule-approximation}) whose objects are 
		symmetric powers $\symbcn$ with $N \in \mathbb{Z}$.  
		By convention, $\symbcn$ is the zero category       
		if $N < 0$ and is the unit object $\kk$ of $\DGModCat$ if $N=0$.
		
		\item The \emph{categorical Fock space}\index{categorical Fock space} $\fcat{\basecat}$ of $\basecat$ 
		is the $\HoDGCat$-enriched bicategory which is the perfect
		hull of the $1$-full subcategory of $\EnhCatKCdg$ whose
		objects are the symmetric powers $\symbcn$ with $N \in \mathbb{Z}$.  
	\end{enumerate}
\end{Definition}

\section{The representation \texorpdfstring{$\Phi'_{\basecat}$}{Phi'}: the generators}
\label{subsec:fock2gen}

In this and the next section we carry out the first step of the  
construction outlined in Section~\ref{subsec:catfockconst} and 
define a strict \dg $2$-functor
\begin{equation*}
	\Phi'_{\basecat} \colon \hcat*{\basecat} \rightarrow \fcat*{\basecat}.
\end{equation*}

\underline{\emph{Objects:}} For any object $\ho \in \mathbb{Z}$ of 
$\hcat*{\basecat}$ we define
\begin{equation*}
	\Phi'_{\basecat}(\ho) = \symbcn. 
\end{equation*}

\underline{\emph{$1$-morphisms:}}\index{$1$-morphism} 
The $1$-morphisms of $\hcat*{\basecat}$ are freely generated by
$\PP_a, \RR_a\colon  \ho \rightarrow \ho+1$ and
$\QQ_a\colon \ho \rightarrow \ho-1$ for all $a \in \basecat$ and $\ho \in \mathbb{Z}$. We define
$\Phi'_{\basecat}$ on morphisms by setting
\begin{align*}
	\Phi'_{\basecat}(\PP_a) &= \phi^*_{a},
	\\
	\Phi'_{\basecat}(\QQ_a) &= \phi_{a *},
	\\
	\Phi'_{\basecat}(\RR_a) &= \phi^!_{a} 
\end{align*}
where $\phi^*_a$, $\phi_{a*}$, and $\phi_{a}^!$ are the \dg functors 
we constructed in 
Section~\ref{sec:symmetric-powers-of-dg-categories}
for any $a \in \basecat$ and $\ho \in \mathbb{Z}$. 

For clarity, we write $P_a$ (respectively, $Q_a$, $R_a$) for $\phi^*_{a}$, (respectively, $\phi_{a *}$ , $\phi^!_{a}$),
when considered as the image of $\PP_a$ (respectively, $\QQ_a$, $\RR_a$) under
$\Phi'_{\basecat}$. 

\begin{Example}\label{ex:Krug_part1}
	Let $X$ be a smooth projective variety and 
	$\basecat = \catDGCoh{X}$ 
	be the standard enhancement of $D^b_{\text{coh}}(X)$. 
	As per Example~\ref{ex:symcomp}, the symmetric powers
	$\symbcn$ of $\basecat$ are Morita enhancements\index{Morita
enhancement} of 
	the derived categories $\catDGCoh{ [ X^{\ho} / \SymGrp{\ho}]  }$ of the 
	symmetric quotient stacks\index{symmetric quotient stack} of $X$.
	Functors $P_a$ and $Q_a$ are the \dg
	enhancements of functors $P_a^{(1)}$ and $Q_a^{(1)}$ defined by 
	Krug in \cite[Section~2.4]{krug2018symmetric}, 
	while $Q_{S^{-1}a}$ correspond to the left adjoints considered in \cite[Section~3.2]{krug2018symmetric}.
	The higher powers $P_a^{(n)}$ and $Q_a^{(n)}$ will arise automatically
	from our calculus, cf.~Example~\ref{ex:Krug_part2}.
\end{Example}

\begin{Example}\label{ex:QAPBandPBQA}
	Let $a,b \in \basecat$ and let $a_1 \otimes \cdots \otimes a_{\ho} \in
	\symbcn$. Let $h^r(a_1 \otimes \cdots \otimes a_{\ho})$ be the
	corresponding representable module in $\hperf(\symbcn)$.
	\begin{enumerate}
		\item\label{it:QAPB} We have
		\begin{align*}
			Q_b P_a h^r(a_1 \otimes \dots \otimes a_{\ho}) & \simeq  
			\phi_{b*}^{} \phi^*_a h^r(a_1 \otimes \cdots \otimes a_{\ho})
			\\
			&\simeq 
			\phi_{b*}^{}  h^r(a \otimes a_1 \otimes \cdots \otimes a_{\ho})
			\\
			&\simeq
			\homm_{\basecat}(b,a) \otimes h^r(a_1 \otimes \cdots \otimes a_{\ho}) \oplus {}
			\\
			&\quad \oplus
			\left( \bigoplus_{i=1}^{\ho}  \homm_{\basecat}(b,a_i) \otimes
			h^r(a \otimes a_1 \otimes \cdots \widehat{a_i} \cdots \otimes a_{\ho}) \right).
		\end{align*}
		\item\label{it:PBQA} On the other hand,
		\begin{align*}
			P_a Q_b h^r(a_1 \otimes \dots \otimes a_{\ho}) & \simeq 
			\phi^*_a \phi_{b*}^{} h^r(a_1 \otimes \cdots \otimes a_{\ho})
			\\
			& \simeq 
			\phi^*_a \left(
			\bigoplus_{i = 1}^{\ho} \homm_{\basecat}(b,a_i) \otimes
			h^r(a_1 \otimes \cdots \widehat{a_i}\cdots \otimes a_{\ho})
			\right)
			\\
			& \simeq
			\bigoplus_{i = 1}^{\ho} \homm_{\basecat}(b,a_i) \otimes
			\phi^*_a h^r(a_1 \otimes \cdots \widehat{a_i}\cdots \otimes a_{\ho})
			\\
			& \simeq
			\bigoplus_{i = 1}^{\ho} \homm_{\basecat}(b,a_i) \otimes
			h^r(a \otimes a_1 \otimes \cdots \widehat{a_i}\cdots \otimes a_{\ho}).
		\end{align*}
	\end{enumerate}
\end{Example}

\underline{\emph{$2$-morphisms:}}\index{$2$-morphism}
The $2$-morphisms of $\hcat*{\basecat}$ are generated, subject to
relations, by four sets of generating $2$-morphisms, cf.~Section~\ref{subsec:heisencatdef-dg}:
\begin{enumerate}[itemsep=0.5em,topsep=0.5em, partopsep=0.5em]
	\item The marked arrows
	\begin{tikzpicture}[baseline=-1ex, scale=0.67]
		\draw[->] (0,-0.5) node[below] {$\PP_{a}$} -- node[label=right:{$\alpha$}, dot, pos=0.5] {} (0,0.5) node[above] {$\PP_{b}$};
	\end{tikzpicture},
	\begin{tikzpicture}[baseline=-1ex, scale=0.67]
		\draw[<-] (0,-0.5) node[below] {$\QQ_{b}$} -- node[label=right:{$\alpha$}, dot, pos=0.5] {} (0,0.5) node[above] {$\QQ_{a}$};
	\end{tikzpicture} and 
	\begin{tikzpicture}[baseline=-1ex, scale=0.67]
		\draw[->] (0,-0.5) node[below] {$\RR_{a}$} -- node[label=right:{$\alpha$}, dot, pos=0.5] {} (0,0.5) node[above] {$\RR_{b}$};
	\end{tikzpicture}.
	\item The Serre relation 
	\begin{tikzpicture}[baseline=-1ex, scale=0.67]
		\draw[->] (0,-0.5) node[below] {$\PP_{Sa}$} -- node[serre, pos=0.5] {} (0,0.5) node[above] {$\RR_{a}$};
	\end{tikzpicture}.
	\item The cups and caps
	\begin{tikzpicture}[baseline=0ex, scale=0.67]
		\draw[->] (0,0) node[below] {$\PP_{a}$} arc[start angle=180, end angle=0, radius=.5] node[below] {$\QQ_{a}$};
	\end{tikzpicture},
	\begin{tikzpicture}[baseline=0ex, scale=0.67]
		\draw[->] (0,0) node[below] {$\RR_{a}$} arc[start angle=0, end angle=180, radius=.5] node[below] {$\QQ_{a}$};
	\end{tikzpicture},
	\begin{tikzpicture}[baseline=0ex, scale=0.67]
		\draw[->] (0,0.5) node[above] {$\QQ_{a}$}arc[start angle=-180, end angle=0, radius=.5] node[above] {$\PP_{a}$};
	\end{tikzpicture} and
	\begin{tikzpicture}[baseline=0ex, scale=0.67]
		\draw[->] (0,0.5) node[above] {$\QQ_{a}$} arc[start angle=0, end angle=-180, radius=.5] node[above] {$\RR_{a}$};
	\end{tikzpicture}.
	\item The crossing
	\begin{tikzpicture}[baseline=-1ex, scale=0.67]
		\draw[<-] (0,-0.5) node[below] {$\QQ_{a}$} -- (1,0.5) node[above] {$\QQ_{a}$};
		\draw[<-] (1,-0.5) node[below] {$\QQ_{b}$} -- (0,0.5) node[above] {$\QQ_{b}$};
	\end{tikzpicture}.
\end{enumerate}

We define $\Phi'_{\basecat}$ on these generating $2$-morphisms as
follows:
\begin{enumerate}
	\item Given $\alpha \in \Hom_\basecat(a,b)$, we have a natural
	transformation of functors $\symbcnmone \rightarrow \symbcn$:
	\[
	\alpha \otimes \id\colon
	\quad
	\phi_a = a \otimes \id 
	\; \longrightarrow \;
	\phi_b = b \otimes \id .
	\]
	We set 
	\begin{equation*}
		\Phi'_{\basecat}\left(
		\;
		\begin{tikzpicture}[baseline=1.5ex, scale=0.67]
			\draw[->] (0,0) node[below] {$\PP_a$} -- node[label=right:{$\alpha$}, dot, pos=0.5] {} (0,1) node[above] {$\PP_b$};
		\end{tikzpicture} 
		\right)
		= (\alpha \otimes \id)^* 
		\text{, }\quad
		\Phi'_{\basecat}\left(
		\;
		\begin{tikzpicture}[baseline=1.5ex, scale=0.67]
			\draw[->] (0,1) node[above] {$\QQ_a$} -- node[label=right:{$\alpha$}, dot, pos=0.5] {} (0,0) node[below] {$\QQ_b$};
		\end{tikzpicture}
		\right)
		= (\alpha \otimes \id)_*
	\end{equation*}
and
\begin{equation*}
		\Phi'_{\basecat}\left(
		\;
		\begin{tikzpicture}[baseline=1.5ex, scale=0.67]
			\draw[->] (0,0) node[below] {$\RR_a$} -- node[label=right:{$\alpha$}, dot, pos=0.5] {} (0,1) node[above] {$\RR_b$};
		\end{tikzpicture} 
		\right)
		= (\alpha \otimes \id)^!.
	\end{equation*}
	We denote these natural transformations by $P_\alpha$, $Q_\alpha$ and $R_\alpha$ respectively.
	\item 
	With $\starmap{a}\colon \phi^*_{Sa} \to \phi^!_a$ as in
	Proposition~\ref{prop-phi-a-shriek-isomorphic-to-phi-Sa-upper-star}, we set
	\begin{equation*}
		\Phi'_{\basecat}\left(
		\;
		\begin{tikzpicture}[baseline=1.5ex, scale=0.67]
			\draw[->] (0,0) node[below] {$\PP_{Sa}$} -- node[serre, pos=0.5] {} (0,1) node[above] {$\RR_a$};
		\end{tikzpicture} 
		\right)
		= \starmap{a},
	\end{equation*}
	\item
	
	As seen in Section~\ref{section-restriction-and-extension-of-scalars},
	we have adjunctions $(\phi^*_a \dashv \phi_{a *})$ and $(\phi_{a *} \dashv \phi^!_a)$. We set:
	\begin{equation*}
		\Phi'_{\basecat}\left(
		\begin{tikzpicture}[baseline={(0,0.1)}, scale=0.67]
			\draw[->] (0,0) node[below] {$\PP_a$} arc[start angle=180, end angle=0, radius=.5] node[label=above:{$\hunit$},pos=0.5] {} node[below] {$\QQ_a$};
		\end{tikzpicture}
		\right)
		= \left[\phi_a^* \phi_{a *} \xrightarrow{\counit} \Id\right], 
		\quad
		\Phi'_{\basecat}\left(
		\begin{tikzpicture}[baseline={(0,0.1)}, scale=0.67]
			\draw[->] (0,0) node[below] {$\RR_{a}$} arc[start angle=0, end angle=180, radius=.5] node[label=above:{$\hunit$},pos=0.5]{} node[below] {$\QQ_{a}$};
		\end{tikzpicture}
		\right)
		= \left[\phi_{a *} \phi^!_{a} \xrightarrow{\counit} \Id\right], 
	\end{equation*}
	\begin{equation*}
		\Phi'_{\basecat}\left(
		\begin{tikzpicture}[baseline={(0,-0.25)}, scale=0.67]
			\draw[->] (0,0) node[above] {$\QQ_a$} arc[start angle=0, end angle=-180, radius=.5] node[label=below:{$\hunit$},pos=0.5]{} node[above] {$\RR_{a}$};
		\end{tikzpicture}
		\right)
		= \left[\Id \xrightarrow{\unit} \phi^!_{a} \phi_{a *}\right],
		\quad
		\Phi'_{\basecat}\left(
		\begin{tikzpicture}[baseline={(0,-0.25)}, scale=0.67]
			\draw[->] (0,0) node[above] {$\QQ_{a}$} arc[start angle=-180, end angle=0, radius=.5] node[label=below:{$\hunit$},pos=0.5]{} node[above] {$\PP_{a}$};
		\end{tikzpicture}
		\right)
		= \left[\Id \xrightarrow{\unit} \phi_{a *} \phi^*_{a}\right].
	\end{equation*}
	
	\item We have an isomorphism of functors $\symbcnmtwo \to \symbcn$
	\[ (12): \phi_a \circ \phi_b \simeq \phi_b \circ \phi_a \]
	given objectwise by the transposition $(12) \in S_n$. We set 
	\[
	\Phi'_{\basecat}\left(
	\begin{tikzpicture}[baseline={(0,0.2)}, scale=0.67]
		\draw[<-] (0,0) node[below] {$\QQ_{a}$}-- (1,1) node[above] {$\QQ_{a}$};
		\draw[<-] (1,0) node[below] {$\QQ_{b}$}-- (0,1) node[above] {$\QQ_{b}$};
	\end{tikzpicture}
	\right)
	= (12)_{\ast}.
	\]
\end{enumerate}

\begin{Remark}
	The differentials on natural transformations in $\fcat*\basecat$ match those in $\hcat*\basecat$.
	For the dots this follows from $d(\alpha \otimes \id) = d(\alpha) \otimes \id$, while all the other defining transformations (the Serre map $\eta$, adjunctions and the transposition) are closed.
\end{Remark}

\begin{Example}\label{ex:PQ-adjunction}
	In the notation of Example~\ref{ex:QAPBandPBQA}, the adjunction
	unit $\id \to Q_aP_a$ is given on representables by embedding
	$h^r(a_1 \otimes \dots \otimes a_{\ho})$ as $\id_a \otimes h^r(a_1
	\otimes \dots \otimes a_{\ho})$ into the first summand. The 
	adjunction counit $P_aQ_a \to \id$ is induced by the evaluation 
	maps $\Hom_\basecat(a, a_i) \otimes a \to a_i$, followed by the transposition $(1i)$ and the universal morphism out of the direct sum.
\end{Example}

\begin{Example}\label{ex:QR-adjunction}
	Using the decomposition of Corollary~\ref{cor-decomposition-of-the-diagonal-bimodule-S_l-into-S_l-1}, we have for any $\symbcn$-module $E$
	\begin{multline*}
		Q_bR_a(E) =
		\phi_{b*}^{}\phi_a^!(E)  \simeq
		\Hom_{\symbcn}\bigl( {}_{\phi_b}(\symbcnpone)_{\phi_a},\, E\bigr) \\  \simeq
		\Hom_{\symbcn}( {}_b\basecat_a \otimes \symbcn,\, E) \oplus
		\bigoplus_{i=1}^{\ho} \Hom_{\symbcn}( {}_i\basecat_a \otimes {}_{\hat 1 \circ (1i)}\symbcn,\, E).
	\end{multline*}
	The adjunction counit $Q_aR_a \to \id$ is given by projecting onto
	\[
	\Hom_{\symbcn}( {}_a\basecat_a \otimes \symbcn,\, E)
	\]
	followed by the morphism induced by the map 
	$\kk \to {}_a\basecat_a$ sending $1 \mapsto \id_{a}$:
	\[
	\Hom_{\symbcn}( {}_a\basecat_a \otimes \symbcn,\, E) \to 
	\Hom_{\symbcn}( \symbcn,\, E) \cong E.
	\]
	To see this, note that by the description of adjunction
	units and counits for Tensor-Hom adjunction
	\cite[Section~2.1]{AnnoLogvinenko-BarCategoryOfModulesAndHomotopyAdjunctionForTensorFunctors}
	our adjunction counit comes from the natural evaluation map
	\[
	\Hom_{\symbcn}( {\symbcnpone} _{\phi_a}, E)
	\otimes_{\symbcnpone} \symbcnpone_{\phi_a} \to E,
	\quad \sum f \otimes g \mapsto \sum f(g)
	\]
	via the identification of the left-hand side with 
	$\Hom_{\symbcn}({}_{\phi_a} {\symbcnpone}_{\phi_a}, E)$
	via the isomorphism $f \mapsto f \otimes 1$. Thus our counit is 
	the map 
	\[\Hom_{\symbcn}(\leftidx{_{\phi_a}}{\symbcnpone}{_{\phi_a}}, E) \to E\]
	given by $f \mapsto f(1)$. Since  
	$1 \in \leftidx{_{\phi_a}}{\symbcnpone}{_{\phi_a}}$
	lies in the component
	\[\Hom(a,a)\otimes {\symbcn},\]
	we can project to that. Then evaluating at $1_a \otimes 1_{\symbcn}$
	is first mapping $\Hom_{\symbcn}(\Hom(a,a)\otimes {\symbcn}, E)$ to
	$\Hom_{\symbcn}({\symbcn}, E)$ and then identifying this with $E$. 
	This gives the claim.
\end{Example}

\begin{Example}\label{ex:CautisLicata_part2}
	Let $\Gamma \leq \mathrm{SL}(2,\mathbb{C})$ be finite and $\basecat$ as 
	in Examples~\ref{ex:CautisLicata_part1} and \ref{ex:CautisLicata_part3}. 
	Let $A_1^{\Gamma}$ denote $\mathbb{C}[x,y] \rtimes \Gamma$, the skew
	group algebra. Its abelian category of modules
	$\text{Mod-}A_1^{\Gamma}$ is equivalent to
	$\catCoh{[\mathbb{C}^2/\Gamma]}$, 
	the abelian category of coherent sheaves on the quotient stack. 
	We can therefore view the algebra $A_1^{\Gamma}$ as a Morita \dg
	enhancement of $\catDbCoh{[\mathbb{C}^2/\Gamma]}$
	and view $\sym^{\ho} A_1^{\Gamma}$ as a Morita \dg enhancement of 
	$\catDbCoh{\Sym^{\ho} [\mathbb{C}^2/\Gamma]}$. In $\hperf \sym^{\ho} A_1^{\Gamma}$ take the full subcategory 
	corresponding to the sheaves supported at the origin $(0,\dots,0)  
	\in \Sym^{\ho} [\mathbb{C}^2/\Gamma]$ where $0$ is the origin 
	of $\mathbb{C}^2$. Its homotopy category is the target of the
	2-representation considered in \cite[Section~4]{cautis2012heisenberg}.
	The functors $P_{i}$ and $Q_{i}$ representing $\PP_{i}$ and $\QQ_{i}$ from
	Example~\ref{ex:CautisLicata_part1} as well as the natural
	transformations defined above are the same as those constructed in
	\cite[Section~4.3]{cautis2012heisenberg}.  Again, the higher powers
	$P_i^{(n)}$ and $Q_i^{(n)}$ arise automatically from our calculus (see
	Example~\ref{ex:CautisLicata_part3} and
	Section~\ref{subsec:Fock_symmetrisers}).  
\end{Example}

\section{The representation \texorpdfstring{$\Phi'_{\basecat}$}{Phi'}: the Heisenberg 2-relations} 
\label{subsec:fock2rel}

We now prove the following:
\begin{Theorem}\label{thm:fock-2-functor-step1}
	The images assigned in Section~\ref{subsec:fock2gen} to 
	the generating $2$-morphisms  
	of $\hcat*\basecat$ satisfy the Heisenberg $2$-relation\index{$2$-relation}s of
	Section~\ref{subsec:heisencatdef2-dg}. We thus have a 
	strict \dg $2$-functor 
	\[\Phi'_\basecat \colon \hcat*\basecat \to \fcat*\basecat.\]
\end{Theorem}

We verify the Heisenberg $2$-relations of
Section~\ref{subsec:heisencatdef2-dg}  in a series of lemmas.

\begin{Lemma}\label{lem:fcat-check-on-representables}
	Let $\alpha$ be a $2$-morphism in $\hcat*{\basecat}$ 
	between $1$-morphisms $\ho \rightarrow \ho*$ which only involve $\PP$s and $\QQ$s. 
	The natural transformation $\Phi'_{\basecat}(\alpha)$ of 
	\dg functors $\rightmod{\sym^{\ho} \basecat} \rightarrow \rightmod{\sym^{\ho*} \basecat}$
	is completely determined by its action on representable modules
	$h^r(a_1 \otimes \dots \otimes a_{\ho})$.
\end{Lemma}

This Lemma means that any relation in $\hcat*{\basecat}$ whose source 
and target only involve $\PP$s and $\QQ$s can be verified 
in $\fcat*{\basecat}$ by checking it on the representable modules.  

\begin{proof}
	By definition, $\Phi'_{\basecat}$ maps $\PP$s and $\QQ$s to the functors
	of extension and restriction of scalars. These are tensor functors\index{tensor functor} -- 
	they are given by tensoring with a bimodule. In other words, they lie
	in the image of the fully faithful functor
	\[ \tensorfn\colon \AmodB \rightarrow \DGFun(\modA,\,\modB), \]
	described in the section Section~\ref{section-bimodule-approximation}. 
	Its right adjoint is the bimodule approximation functor $\bimodapx$
	and the fully faithfullness of $\tensorfn$ implies that
	a natural transformations of tensor functors is completely determined
	by its image under $\bimodapx$. The claim now follows, since 
	$\bimodapx$ is the restriction to the diagonal bimodule, 
	i.e.~to the representables. 
\end{proof}

\begin{Lemma}\label{lem:fcat_adj_isotopy}
	The straightening relation \eqref{eq:straighten-dg} is satisfied in $\fcat*\basecat$:
	\[
	\Phi'_{\basecat}\left(\;
	\begin{tikzpicture}[scale=0.5, baseline={(0,0.4)}]
		\draw (0,0) -- (0,1)  arc[start angle=180, end angle=0, radius=.5] arc[start angle=-180, end angle=0, radius=.5] -- (2,2);
	\end{tikzpicture}
	\;\right)
	=
	\Phi'_{\basecat}\left(\;
	\begin{tikzpicture}[scale=0.5, baseline={(0,0.4)}]
		\draw (3,0) -- (3,2);
	\end{tikzpicture}
	\;\right)
	=
	\Phi'_{\basecat}\left(\;
	\begin{tikzpicture}[scale=0.5, baseline={(0,0.4)}]
		\draw (6,0) -- (6,1)  arc[start angle=0, end angle=180, radius=.5] arc[start angle=0, end angle=-180, radius=.5] -- (4,2);
	\end{tikzpicture}
	\;\right)
	\]
	for any allowed orientation and labeling of the strands.
\end{Lemma}

\begin{proof}
	Caps and cups are sent to the unit and counit morphisms of adjoint
	pairs of functors. The claim now follows from the standard relations
	\[ 
	\bigl(F \xrightarrow{F\eta} FGF \xrightarrow{\varepsilon F} F\bigr) =
	\id_F
	\quad \quad\text{and}\quad \quad
	\bigl(G  \xrightarrow{\varepsilon G} GFG \xrightarrow{G\eta} G\bigr) =
	\id_G
	\]
	satisfied by any adjunction $(F \dashv G)$ with unit $\eta$ and counit 
	$\epsilon$. 
\end{proof}

\begin{Lemma}
	Relation~\eqref{eq:cupsslide-dg} is satisfied is $\fcat*\basecat$: 
	dots may slide through cups and caps.
\end{Lemma}

\begin{proof}
	We need to show that the following pairs of maps are equal for any $\alpha \in \Hom_\basecat(a,b)$:
	\begin{enumerate}
		\item 
		$P_aQ_{b} \xrightarrow{P_{\alpha}Q_{\id_{b}}} P_{b}Q_{b} \xrightarrow{\counit} \id$ and
		$P_aQ_{b} \xrightarrow{P_{\id_a}Q_{\alpha}} P_aQ_{a} \xrightarrow{\counit} \id$;
		\item
		$\id \xrightarrow{\unit} Q_{a}P_a \xrightarrow{Q_{\id_a}P_{\alpha}} Q_{a}P_{b}$ and
		$\id \xrightarrow{\unit} Q_{b}P_{b} \xrightarrow{Q_{\alpha}P_{\id_{b}}} Q_{a}P_{b}$.
		\item 
		$Q_{b}R_{a} \xrightarrow{Q_{\id_{b}}R_{\alpha}}   Q_{b}R_{b} \xrightarrow{\counit} \id$ and
		$Q_{b}R_{a} \xrightarrow{Q_{\alpha}R_{\id_{a}}} Q_{a}R_{a} \xrightarrow{\counit} \id$;
		\item
		$\id \xrightarrow{\unit} R_{a}Q_{a} \xrightarrow{R_{\alpha}Q_{\id_a}}   R_{b}Q_{a}$ and  
		$\id \xrightarrow{\unit} R_{b}Q_{b} \xrightarrow{R_{\id_{b}}Q_{\alpha}} R_{b}Q_{a} $;
	\end{enumerate}
	By adjunction, (1) and (2) are equivalent, as are (3) and (4).
	We will show (1). The proof of (3) is similar, using the description of Example~\ref{ex:QR-adjunction}.
	
	From Example~\ref{ex:QAPBandPBQA}~\ref{it:PBQA} it follows that
	\begin{equation}\label{eq:PAQBel}
		P_aQ_{b}h^r(a_1\otimes \dots \otimes a_{\ho}) =
		\bigoplus_{i=1}^{\ho} \Hom_\basecat(b,a_i) \otimes h^r(a \otimes a_1 \otimes \cdots  \widehat{a_{i}} \cdots \otimes a_{\ho}).
	\end{equation}
	The map $P_{\alpha}Q_{\id_{b}}$ is given on each summand 
	by applying $\alpha$ to the second factor. It lands in 
	\[
	\bigoplus_{i=1}^{\ho} \Hom_\basecat(b,a_i) \otimes h^r(b \otimes a_1 \otimes \cdots  \widehat{a_{i}} \cdots \otimes a_{\ho}).
	\]
	The counit map takes each summand and evaluates the first factor on
	the second factor: 
	\[
	\bigoplus_{i=1}^{\ho} \Hom_\basecat(b,a_i) \otimes h^r(b \otimes a_1 \otimes \cdots  \widehat{a_{i}} \cdots \otimes a_{\ho}) \to h^r(a_1 \otimes \cdots \otimes a_{\ho}).
	\]
	
	Computing the second composition in a similar way, we see that 
	the equality of these compositions is equivalent to 
	the commutativity of the following diagram:
	\[
	\begin{tikzcd}[sep=1.4em, font=\small]
		\Hom(b,a_i) \otimes h^r(a \otimes a_1 \otimes \cdots \widehat{a_{i}} \cdots \otimes a_{\ho})
		\arrow[r, "\id \otimes \alpha"] \arrow[d,"\alpha \otimes \id"] &
		\Hom(b,a_i) \otimes h^r(b \otimes a_1 \otimes \cdots  \widehat{a_{i}} \cdots \otimes a_{\ho}) \arrow[d] \\
		\Hom(a,a_i) \otimes h^r(a \otimes a_1 \otimes \cdots  \widehat{a_{i}} \cdots \otimes a_{\ho}) \arrow[r] &
		h^r(a_1 \otimes \cdots \otimes a_{\ho})
	\end{tikzcd}. 
	\]
	This diagram commutes by the functoriality of tensor product. 
\end{proof}

The next observation is immediate from the construction.

\begin{Lemma}
	Relation~\eqref{eq:crossingslide-dg} is satisfied in $\fcat*\basecat$.
	That is, dots move freely through crossings:
	\[
	\Phi'_{\basecat}\left(
	\begin{tikzpicture}[scale=0.67, baseline={(0,0.25)}]
		\draw[<-] (0,0) -- (1,1);
		\draw[<-] (1,0) -- (0,1);
		\node[label=left:{$\alpha$}, dot] at (0.25,0.25) {};
	\end{tikzpicture}
	\;\right)
	=
	\Phi'_{\basecat}\left(\;
	\begin{tikzpicture}[scale=0.67, baseline={(0,0.25)}]
		\draw[<-] (2,0) -- (3,1);
		\draw[<-] (3,0) -- (2,1);
		\node[label=right:{$\alpha$}, dot] at (2.75,0.75) {};
	\end{tikzpicture}
	\right).
	\]
\end{Lemma}

\begin{Lemma}
	The symmetric group relations \eqref{eq:symmetric_group_relations-dg} hold in $\fcat*\basecat$.
\end{Lemma}

\begin{proof}
	For the double crossing, the identity $((12)_\ast)^2=\id$ follows from the fact that $(12)^2=\id$ in $\symbcn$.
	The triple move similarly follows by splitting the steps as
	\[ (12) \circ (23) \circ (12)  = (23) \circ (12) \circ (23). \qedhere \]
\end{proof}

\begin{Lemma}\leavevmode
	\begin{enumerate}
		\item\label{it:colliding-dots-fock} The composition relation \eqref{eq:colliding_dots_down-dg} holds in $\fcat*\basecat$. Namely,
		$(Q_{b} \xRightarrow{Q_\alpha} Q_a) \circ (Q_{c} \xRightarrow{Q_\beta} Q_{b})$ is equal to $(-1)^{|\alpha||\beta|}\cdot Q_{c} \xRightarrow{Q_{\beta \circ \alpha}} Q_a$.
		\item\label{it:dot-sliding-past-star-fock} Relation~\eqref{eq:dot_sliding_past_star} holds in $\fcat*\basecat$. Namely,
		$(P_{Sb} \xRightarrow{\starmap{b}} R_{b}) \circ (P_{Sa} \xRightarrow{P_{S\alpha}} P_{Sb})$ is equal to
		$(R_{a} \xRightarrow{R_{\alpha}} R_{b}) \circ (P_{Sa} \xRightarrow{\starmap{A}} R_a)$.
	\end{enumerate}
\end{Lemma}

\begin{proof}
	Part~\ref{it:colliding-dots-fock} is clear from $(\beta \otimes \id) \circ (\alpha \otimes \id) = (\beta \circ \alpha) \otimes \id$, taking the sign rules for contravariant \dg functors into account.
	Part~\ref{it:dot-sliding-past-star-fock} is a consequence of naturality of the Serre morphism $\eta^*$.
\end{proof}

\begin{Lemma}
	\label{lem:PQQP}	
	For every $a, b \in \mathrm{Ob}({\basecat})$ and $a_1 \otimes \dots \otimes a_{\ho} \in \symbcn$ there exists a natural isomorphism on representable objects
	\[ 
	Q_bP_a(h^r(a_1 \otimes \dots \otimes a_{\ho})) \cong
	\bigl(\Hom_\basecat(b,a) \otimes h^r(a_1 \otimes \dots \otimes a_{\ho})\bigr) \oplus P_a Q_b(h^r(a_1 \otimes \dots \otimes a_{\ho})).
	\]
	The image of 
	\[
	\begin{tikzpicture}[baseline=-1ex, scale=0.67]
		\draw[->] (0,-0.5) node[below] {$\PP_{a}$} -- (1,0.5) node[above] {$\PP_{a}$};
		\draw[<-] (1,-0.5) node[below] {$\QQ_{b}$} -- (0,0.5) node[above] {$\QQ_{b}$};
	\end{tikzpicture}
	\]
	under $\Phi'_{\basecat}$ embeds $P_a Q_b(h^r(a_1 \otimes \dots \otimes a_{\ho}))$ as the second summand.
\end{Lemma}

\begin{proof}
	The first assertion follows from Example~\ref{ex:QAPBandPBQA}. The
	image of the crossing under $\Phi'_\basecat$ is:
	\[
	\phi_a^*\phi_{b*}^{}
	\xrightarrow{\unit} 
	\phi_a^*\phi_{b*}^{}\phi_{a*}^{}\phi_a^*
	\cong
	\phi_a^*\phi_{a*}^{}\phi_{b*}^{}\phi_a^*
	\xrightarrow{\counit} 
	\phi_{b*}^{}\phi_a^*.
	\]
	Here we used that the commutativity of the the tensor product implies that
	\[ \phi_{b*}^{}\phi_{a*}^{} \simeq\phi_{b \otimes a*}\simeq\phi_{a \otimes b*} \simeq\phi_{a*}^{}\phi_{b*}^{}.\]
	The second assertion follows from the
	description of unit and counit maps in Example~\ref{ex:PQ-adjunction}. 
\end{proof}

The following gives a description of the image of the \enquote{starred cup.}
\begin{Lemma}\label{lem:starred-cap-fock}
	The natural transformation
	\[
	\zeta = 
	\Phi'_\basecat\left(
	\begin{tikzpicture}[baseline={(0,0)}, scale=0.67]
		\draw[->] (1,0) node[below] {$\PP_{Sa}$} -- node[serre, pos=0.5] {} (1,0.5) arc[start angle=0, end angle=180, radius=0.5] -- (0,0) node[below] {$\QQ_a$};
	\end{tikzpicture}
	\right)
	\]
	is given by the bimodule map
	\begin{align*}
		\zeta\colon 
		\leftidx{_{\phi_{Sa}}}{\symbcn}
		\otimes_{\symbcn}
		\symbcn_{\phi_a}
		&\rightarrow 
		\symbcnmone
		\\
		f \otimes h
		& \mapsto 
		\Tr\left(\leftidx{_{Sa}}{(f \circ h)}{_{a}}\right) \;
		\leftidx{_{\widehat{Sa}}}{(f \circ h)}{_{\widehat{a}}},
	\end{align*}
	where the notation indicates that we take the first summand in terms of the decomposition of 
	$\leftidx{_{\phi_{Sa}}}{\symbcn}{_{\phi_a}}$
	provided by
	Corollary~\ref{cor-decomposition-of-the-diagonal-bimodule-S_l-into-S_l-1}. In terms of Example~\ref{ex:QAPBandPBQA}, $\zeta$ maps
	\begin{multline*}
		Q_aP_{Sa}(h^r(a_1\otimes\dots\otimes a_{\ho-1})) \cong 
		\homm_{\basecat}(a,Sa) \otimes h^r(a_1 \otimes \cdots \otimes a_{\ho-1}) \oplus {}\\
		\oplus \left( \bigoplus_{i=1}^{\ho-1}  \homm_{\basecat}(a,a_i) \otimes h^r(a \otimes a_1 \otimes \cdots \widehat{a_i} \cdots \otimes a_{\ho-1}) \right)
	\end{multline*}
	onto the first summand, followed by applying the Serre trace map $\Tr$.
\end{Lemma} 
\begin{proof}
	Proposition~\ref{prop-phi-a-shriek-isomorphic-to-phi-Sa-upper-star} gives the star quasi-isomorphism on $a$. Then, similarly as in Example~\ref{ex:unitcounit_add}, the counit is a projection onto the first summand from Corollary~\ref{cor-decomposition-of-the-diagonal-bimodule-S_l-into-S_l-1} followed by the Serre trace applied to $\Hom(a, Sa)$.
\end{proof}


\begin{Lemma}
	The Serre trace relation on the right hand side of~\eqref{eq:circle_and_curl-dg} holds in $\fcat*\basecat$:
	\[
	\Phi'_{\basecat}\left(\;
	\begin{tikzpicture}[scale=0.67,baseline={(0,0.25)}]
		\draw[decoration={markings, mark=at position 0.36 with {\arrow{>}}, mark=at position 0.85 with {\arrow{>}}}, postaction={decorate}] (1,0)
		--
		node[label=right:{$\alpha$}, dot, pos=0.25] {}
		node[serre, pos=0.75] {}
		(1,1)
		arc[start angle=0, end angle=180, radius=.5]
		--
		(0,0)
		arc[start angle=180, end angle=360, radius=.5];
	\end{tikzpicture}
	\right)
	=
	\Tr(\alpha), \quad \quad 
	\]
\end{Lemma}

\begin{proof}
	Assume first that $\ho=0$. Then we need to  compute the image of 
	$h^r(1)$ for $1 \in \sym^0\basecat = \kk$.
	By Example~\ref{ex:PQ-adjunction}, the unit corresponding to the cup at the bottom sends this to 
	\[
	\id_a \in Q_{a}P_{a}(h^r(1))\simeq \phi_{a\ast}\phi_a^\ast h^r(1)\simeq \Hom_{\basecat}(a,a).
	\]
	Composing with $Q_{\id_a} P_\alpha$ sends this to $\alpha \in \Hom(a, Sa)$.
	Finally, the starred cup \[\zeta = \counit \circ (\phi_{A,*}\starmap{A})\] sends $\alpha$ to $\Tr(\alpha)$ by Lemma~\ref{lem:starred-cap-fock}.
	For general $\ho$, we need to compute the image of
	$h^r(a_1 \otimes \dots \otimes a_{\ho})$ for $a_1 \otimes \dots
	\otimes a_{\ho} \in \sym^{\ho}\basecat$. We get the same computation as
	above but tensored over $\kk$ with the identity morphism of 
	$a_1 \otimes \dots \otimes a_{\ho}$.
\end{proof}

\begin{Lemma}\label{lem:leftcurlvanish}
	The left curl on the left side of \eqref{eq:circle_and_curl-dg} vanishes in $\fcat*\basecat$:
	\[
	\Phi'_{\basecat}\left(
	\begin{tikzpicture}[baseline={(0,-0.05)}, xscale=-.67]
		\draw[<-]
		(1,-0.9) node[below] {$\QQ_{a}$} -- 
		(1,-0.5) to[out=90, in=0] (0.3,0.5) to[out=180,in=90]
		(-0.1,0) node[serre] {} to[out=270,in=180] (0.3,-.5) to[out=0,in=270]
		(1,0.5) -- (1,0.9) node[above] {$\QQ_{Sa}$};
	\end{tikzpicture}
	\right)
	= 0.
	\]
\end{Lemma}

\begin{proof}
	This diagram decomposes as 
	\[ 
	Q_{a} \xrightarrow{Q_{\id_{a}}\mathrm{unit} }
	Q_a Q_{Sa} P_{Sa} \xrightarrow{(12)_{\ast} } 
	Q_{Sa} Q_{a} R_a \xrightarrow{Q_{\id_{Sa}} \zeta}Q_{Sa},
	\]
	where $\zeta$ is as in Lemma~\ref{lem:starred-cap-fock}.
	Using the notation of Example~\ref{ex:QAPBandPBQA}, the first step embeds
	\[
	Q_a(h^r(a_1 \otimes \dots a_{\ho})) \cong
	\bigoplus_{i=1}^{\ho} \Hom(a,a_i) \otimes h^r(a_1\otimes \cdots \widehat{a_i} \cdots \otimes a_{\ho})
	\]
	into the first factor of
	\begin{multline*}
		Q_aQ_{Sa}P_{Sa}(h^r(a_1 \otimes \dots a_{\ho})) \\ \cong
		\left( \bigoplus_{i=1}^{\ho} \Hom(a,a_i) \otimes \Hom(Sa,Sa) \otimes h^r(a_1\otimes \cdots \widehat{a_i} \cdots \otimes a_{\ho}) \right) \oplus {} \\
		\oplus \bigoplus_{j=1}^{\ho} \Biggl(\Hom(a,Sa) \otimes \Hom(Sa ,a_j)  \otimes h^r(a_1\otimes \cdots \widehat{a_j} \cdots \otimes a_{\ho})  \\
		\oplus \bigoplus_{\substack{i=1 \\ i \neq j}}^{\ho} \Hom(a,a_i) \otimes \Hom(Sa,a_j)  \otimes h^r(Sa \otimes a_1\otimes \cdots \widehat{a_i} \cdots \widehat{a_j} \cdots \otimes a_{\ho}) \Biggr)
	\end{multline*}
	by tensoring with $\id_{Sa} \in \Hom(Sa, Sa)$.
	The crossing changes the order of the summands, and the starred cap $\zeta$ projects onto the summand
	\[
	\bigoplus_{i=1}^{\ho} \Hom(a,Sa) \otimes \Hom(Sa ,a_i)  \otimes h^r(a_1\otimes \cdots \widehat{a_i} \cdots \otimes a_{\ho})
	\]
	followed by the Serre trace applied to $\Hom(a, Sa)$.
	As the component corresponding to this summand is zero after the first step, the whole composition vanishes.
\end{proof}

\begin{Remark}
	The proof of Lemma \ref{lem:leftcurlvanish} also explains why the right curls in $\hcat\basecat$ are not required to vanish. Therein the unit at the first step and the counit at the last step are both given by diagonal maps, and hence they do not automatically compose to zero.
\end{Remark}

\begin{Lemma}\label{lem:up_down_braid-fock}
	The relation in \eqref{eq:up_down_braid-dg} holds in $\fcat*\basecat$, i.e.
	\[
	\Phi'_\basecat\left(
	\begin{tikzpicture}[baseline={(0,0.6)}, scale=0.67]
		\draw[->] (0,0) node[below] {$\PP_{Sa}$}  -- (0.5,0.5) 
		to[out=45,in=-45] 
		node[serre, pos=0.5] {} 
		(0.5,1.5) -- (0,2) node[above] {$\RR_{a}$};
		\draw[->] (1,2) node[above] {$\QQ_{b}$} -- (0.5,1.5)
		to[out=225, in = 135] (0.5,0.5) -- (1,0) node[below] {$\QQ_{b}$};
	\end{tikzpicture}
	\right)
	=
	\Phi'_\basecat\left(
	\begin{tikzpicture}[baseline={(0,0.6)}, scale=0.67]
		\draw[->] (2,0) node[below] {$\PP_{Sa}$} -- node[serre, pos=0.5]{} (2,2) node[above] {$\RR_{a}$} ;
		\draw[<-] (3,0) node[below] {$\QQ_{b}$} -- (3,2) node[above] {$\QQ_{b}$} ;
	\end{tikzpicture}
	\right).
	\]
\end{Lemma}

\begin{proof}
	To use Lemma~\ref{lem:fcat-check-on-representables},
	we prove the statement which is equivalent by adjunction: 
	\[
	\Phi'_\basecat \left( 
	\begin{tikzpicture}[baseline={(0,0.6)}, scale=0.67]
		\draw[->]
		(0,0)     node[below] {$\PP_{Sa}$}  --
		(0.5,0.5) to[out=45,in=-45] node[serre, pos=0.5] {} 
		(0.5,1.5) --
		(0,2)     arc[start angle=45, end angle=180, radius={sqrt(2)/(1+sqrt(2))}]
		--
		(-1,0)    node[below] {$\QQ_{a}$};
		\draw[->]
		(1,2)     node[above] {$\QQ_{b}$} --
		(0.5,1.5) to[out=225, in = 135]
		(0.5,0.5) --
		(1,0)     node[below] {$\QQ_{b}$};
	\end{tikzpicture}
	=
	\begin{tikzpicture}[baseline={(0,0.6)}, scale=0.67]
		\draw[->]
		(0,0)   node[below] {$\PP_{Sa}$} to[out=45,in=270]
		(1,1.5) -- node[serre, pos=0.5] {} (1,1.6)
		arc[start angle=0, end angle=180, radius=0.5]
		to[out=270, in=90] (-1,0.2)
		--
		(-1,0)  node[below] {$\QQ_{a}$};
		\draw[->]
		(-1,2) node[above] {$\QQ_{b}$}
		to[out=270, in = 135]
		(1,0)  node[below] {$\QQ_{b}$};
	\end{tikzpicture}
	\right)
	=
	\Phi'_\basecat \left(
	\begin{tikzpicture}[baseline={(0,0.6)}, scale=0.67]
		\draw[->] (2,0) node[below] {$\PP_{Sa}$} -- node[serre, pos=0.6]{} (2,1.5) arc[start angle=0, end angle=180, radius=0.5] -- (1,0) node[below] {$\QQ_{a}$} ;
		\draw[<-] (3,0) node[below] {$\QQ_{b}$} -- (3,2) node[above] {$\QQ_{b}$} ;
	\end{tikzpicture}
	\right). 
	\]
	In other words, 
	$\Phi'_\basecat$ preserves the commutativity of the diagram
	\[
	\begin{tikzcd}
		\QQ_a\PP_{Sa}\QQ_b 
		\arrow[rr, "{\begin{tikzpicture}[scale=0.3] \draw[->] (0,2) -- (2,0); \draw[->] (1,2) -- (0,0); \draw[->] (1,0) -- (2,2); \end{tikzpicture}}"]
		\arrow[dr, "{\begin{tikzpicture}[scale=0.3] \draw[->] (1,0) -- (1,0.5) arc[start angle=0, end angle=180, radius=0.5] -- (0,0); \node[serre] at (1,0.4) {}; \draw[->] (1.8, 1) -- (1.8,0); \end{tikzpicture}}"']
		& &
		\QQ_b\QQ_a\PP_{Sa}
		\arrow[dl, "{\begin{tikzpicture}[scale=0.3] \draw[->] (1,0) -- (1,0.5) arc[start angle=0, end angle=180, radius=0.5] -- (0,0); \node[serre] at (1,0.4) {}; \draw[->] (-0.8, 1) -- (-0.8,0); \end{tikzpicture}}"]
		\\
		& \QQ_b &
	\end{tikzcd}
	\]
	By Lemma~\ref{lem:PQQP} the first crossing in $\Phi_\basecat'\bigl({\begin{tikzpicture}[scale=0.2, baseline={(0,0.1)}] \draw[->] (0,2) -- (2,0); \draw[->] (1,2) -- (0,0); \draw[->] (1,0) -- (2,2); \end{tikzpicture}}\bigr)$ embeds
	$ Q_aP_{Sa}Q_b(h^r(a_1 \otimes \dots \otimes a_{\ho}))$, that is
	\begin{multline}\label{eq:lem:up_down_braid-fock:QPQ}
		\bigoplus_{i=1}^{\ho} \biggl( \Hom(a,Sa) \otimes \Hom(b,a_i) \otimes h^r (a_1 \otimes \dotsb \widehat{a_i} \dotsb \otimes a_{\ho}) \oplus {} \\
		\oplus \bigoplus_{\substack{j=1\\j \ne i}}^{\ho} \Hom(a,a_j) \otimes \Hom(b,a_i) \otimes h^r(Sa \otimes a_1 \otimes \dotsb \widehat{a_i} \dotsb \widehat{a_j} \dotsb \otimes a_{\ho}) \biggr),
	\end{multline}
	into $Q_aQ_bP_{Sa}(h^r(a_1 \otimes \dots \otimes a_{\ho}))$, that is
	\begin{multline*}
		Q_a\bigl(\Hom(b,Sa) \otimes h^r(a_1 \otimes \dots \otimes a_{\ho})\bigr) \oplus Q_aP_{Sa}Q_b\bigl(h^r(a_1 \otimes \dots \otimes a_{\ho})\bigr) = \\
		 \bigoplus_{j=1}^{\ho} \Hom(a,a_i) \otimes \Hom(b,Sa) \otimes h^r(a_1 \otimes \dotsb \widehat{a_i} \dotsb \otimes a_{\ho}) \oplus Q_aP_{Sa}Q_b\bigl(h^r(a_1 \otimes \dots \otimes a_{\ho})\bigr).
	\end{multline*}
	The second crossing changes the summand order. By Lemma~\ref{lem:starred-cap-fock} $\Psi_\basecat'\bigl({\begin{tikzpicture}[scale=0.3, baseline={(0,0.05)}] \draw[->] (1,0) -- (1,0.5) arc[start angle=0, end angle=180, radius=0.5] -- (0,0); \node[serre] at (1,0.4) {}; \draw[->] (1.8, 1) -- (1.8,0); \end{tikzpicture}}\bigr)$ projects onto
	\begin{equation}\label{eq:lem:up_down_braid-fock:end}
		\bigoplus_{i=1}^{\ho} \Hom(a,Sa) \otimes \Hom(b,a_i) \otimes h^r (a_1 \otimes \dotsb \widehat{a_i} \dotsb \otimes a_{\ho}),
	\end{equation}
	followed by $\Tr\colon \Hom(a,Sa) \to \kk$.
	On the other hand, $\Psi_\basecat'\bigl({\begin{tikzpicture}[scale=0.3, baseline={(0,0.05)}] \draw[->] (1,0) -- (1,0.5) arc[start angle=0, end angle=180, radius=0.5] -- (0,0); \node[serre] at (1,0.4) {}; \draw[->] (-0.8, 1) -- (-0.8,0); \end{tikzpicture}}\bigr)$ projects \eqref{eq:lem:up_down_braid-fock:QPQ} directly onto \eqref{eq:lem:up_down_braid-fock:end}, followed by the Serre trace.
	Thus the two sides are the same natural transformation.
\end{proof}

\section{From \texorpdfstring{$\Phi'_{\basecat}$ to $\Phi_{\basecat}$}{Phi' to Phi}}
\label{sec:magic-wand}

In the previous two sections, we constructed a strict $2$-functor 
\[ \Phi'_\basecat\colon \hcat*\basecat \rightarrow \fcat*\basecat \]
of strict \dg $2$-categories. Recall that $\fcat*\basecat$ 
is a $1$-full subcategory of $\DGModCat$, the strict \dg $2$-category 
whose objects are small \dg categories, and whose
$1$-morphisms are \dg functors between their module categories. 
We next apply the lax $2$-functor of bimodule approximation defined in 
Section~\ref{section-bimodule-approximation}:
\[ \bimodapx\colon \DGModCat \rightarrow \DGBiMod. \]
Its target is the \dg bicategory $\DGBiMod$ whose objects are
small \dg categories and whose $1$-morphisms are their \dg bimodule
categories. On objects, $\bimodapx$ is the identity map. 
On $1$-morphisms, for any small \dg categories $\A$ and $\B$ it
is the \dg functor
\[ \bimodapx\colon \DGFun(\modA,\, \modB) \rightarrow \AmodB, \]
defined by $F \mapsto F(\A)$.

The $1$-morphisms of $\hcat*\basecat$ are generated by
$\PP_a$, $\QQ_a$, and $\RR_a$ for $a \in \basecat$. 
$2$-functor $\Phi'_\basecat$ sends these to \dg functors 
$\phi_a^*$, $\phi_{a *}$, and $\phi^!_a$. 
In Section~\ref{sec:symmetric-powers-of-dg-categories} we proved that
the images of these under $\bimodapx$ are left 
h-projective and right-perfect bimodules. We thus obtain
a composition 
\[  \hcat*\basecat \xrightarrow{\Phi'_\basecat} \fcat*\basecat
\xrightarrow{\bimodapx} \DGBiMod, \]
whose image is contained in the $2$-full subcategory $\DGBiMod_{\lfrp}$ of $\DGBiMod$
consisting of the left-h-flat and right-perfect bimodules.

We remark that the $2$-functor
$\bimodapx$ does not send all $1$-morphisms of $\fcat*\basecat$ to 
$\DGBiMod_{\lfrp}$. Indeed, by definition $\homm_{\fcat*\basecat}(0,1)$
consists of all \dg functors $\modk \rightarrow
\modd\text{-}\basecat$. For any $E \in \modd\text{-}\basecat$ we have
the functor $(-) \otimes E$ which $\bimodapx$ sends to $E$
considered as $\kk$-$\basecat$-bimodule. Thus for any non-perfect $E$
the corresponding tensor functor $(-) \otimes E$ is a $1$-morphism of
$\fcat*\basecat$ whose image under $\bimodapx$ isn't right-perfect.

Recall the $\HoDGCat$-enriched bicategory $\EnhCatKCdg$ 
of enhanced triangulated categories defined 
in Section~\ref{section-enhanced-categories}.
We next apply a strict $2$-functor 
\[ L \colon \DGBiMod_{\lfrp} \rightarrow \EnhCatKCdg. \]
On objects, $L$ is the identity map. 
On $1$-morphisms, depending on which of the two definitions 
of $\EnhCatKCdg$ one uses,
$L$ is either the natural embedding
\[ \AmodB_{\lfrp} \hookrightarrow \AmodbarB_{\lfrp}, \]
into the bar category\index{bar category} of bimodules, or the natural embedding
\[ \AmodB_{\lfrp} \hookrightarrow \AmodB_{\lfrp}/\acyc, \]
into the Drinfeld quotient by acyclics. On the level of homotopy
categories, both are just the standard localisation of \dg bimodules 
by quasi-isomorphisms.  

We thus obtain a composition 
\begin{equation}
	\label{eqn-hlax-functor-H'-to-EnhCatKCdg}
	\hcat*\basecat \xrightarrow{\Phi'_\basecat} \fcat*\basecat
	\xrightarrow{\bimodapx} \DGBiMod_{\lfrp} \xrightarrow{L}
	\EnhCatKCdg. 
\end{equation}
The $2$-functors $\Phi'_\basecat$ and $L$ are strict. In general, 
the $2$-functor $\bimodapx$ is lax, but it follows from Proposition 
\ref{prop-bimod-approximation-is-quasi-iso-lax-on-tensor-and-hom} 
that on the \dg functors $\phi_a^*$, $\phi_{a *}$, and $\phi^!_a$ 
its coherence morphisms are quasi-isomorphisms. Since $L$
sends quasi-isomorphisms to homotopy equivalences, it follows
that the composition \eqref{eqn-hlax-functor-H'-to-EnhCatKCdg} is
a homotopy strong $2$-functor. 

Next, we take perfect hulls as per 
Section~\ref{defn-the-perfect-hull-of-a-bicategory}. By definition, 
$\fcat\basecat$ is the perfect hull 
of the $1$-full subcategory of $\EnhCatKCdg$ comprising the symmetric powers $\symbcn$.
Thus it contains the perfect hull of the image of 
\eqref{eqn-hlax-functor-H'-to-EnhCatKCdg}.
We thereby
obtain a homotopy strong $2$-functor
\begin{equation}
	\label{eqn-hlax-functor-hperf-H'-to-F}
	\bihperf(\hcat*\basecat) \xrightarrow{\hperf(L \circ \bimodapx \circ
		\Phi'_\basecat)} \fcat\basecat.
\end{equation}

The Heisenberg $2$-category $\hcat\basecat$ is 
the monoidal Drinfeld quotient\index{monoidal Drinfeld
quotient} of 
$\bihperf(\hcat*\basecat)$ by the two-sided $1$-morphism ideal 
$\mathcal{I}_\basecat$ 
generated by the following two classes of $1$-morphisms:
\begin{enumerate}
	\item For each $a \in \basecat$, the cone of the Serre relation
	$2$-morphism
	\begin{equation}
		\label{eqn-generators-of-the-ideal-to-quotient-by-1}
		\PP_{Sa} \xrightarrow{\tikz[scale=0.5]
			\draw[->](0,0) -- node[serre, pos=0.5] {} (0,1);} \RR_a,
	\end{equation}
	\item For each $a,b \in \basecat$, the cone of the $2$-morphism
	\begin{equation}
		\label{eqn-generators-of-the-ideal-to-quotient-by-2}
		\PP_{b}\QQ_a \oplus (\hunit \otimes \Hom(a,b)) 
		\xrightarrow{
			\left[
			\begin{tikzpicture}[baseline={(0,0.15)},scale=0.5]
				\draw[->](0.5,0) -- (1.5,1);
				\draw[->](0.5,1) -- (1.5,0);
			\end{tikzpicture}
			\,,\,
			\,\psi_2
			\right]
		}
		\QQ_a \PP_{b}.
	\end{equation}
\end{enumerate}
We claim that \eqref{eqn-hlax-functor-hperf-H'-to-F}
sends these to null-homotopic $1$-morphisms in $\fcat\basecat$. 
It suffices to check that \eqref{eqn-hlax-functor-hperf-H'-to-F}
sends the $2$-morphisms
\eqref{eqn-generators-of-the-ideal-to-quotient-by-1}
and
\eqref{eqn-generators-of-the-ideal-to-quotient-by-2}
to homotopy equivalences.
Recall that in both definitions of $\EnhCatKCdg$
in Section~\ref{section-enhanced-categories} its $1$-morphisms
are \dg bimodules and its $2$-morphisms are defined in terms 
of morphisms of \dg bimodules. In first definition we take bar
morphisms and in the second we take the Drinfeld quotient of the usual
bimodule category by acyclics. In both cases, all usual morphisms 
of \dg bimodules are valid $2$-morphisms. We say that a $2$-morphism
is a quasi-isomorphism if it is an usual morphism of \dg bimodules which 
is a quasi-isomorphism. All such $2$-morphisms are homotopy equivalences: for bar morphisms this is shown in \cite[Cor.~3.8]{AnnoLogvinenko-BarCategoryOfModulesAndHomotopyAdjunctionForTensorFunctors}, 
while in the Drinfeld quotient by acyclics the cone of a
quasi-isomorphism is null-homotopic because it is acyclic. 
It thus suffices to check that \eqref{eqn-hlax-functor-hperf-H'-to-F} 
sends 
\eqref{eqn-generators-of-the-ideal-to-quotient-by-1}
and
\eqref{eqn-generators-of-the-ideal-to-quotient-by-2}
to quasi-isomorphisms. For the former this follows by Lemma  
\ref{prop-phi-a-shriek-isomorphic-to-phi-Sa-upper-star}, 
and for the latter by Example \ref{ex:QAPBandPBQA}. 

We conclude that \eqref{eqn-hlax-functor-hperf-H'-to-F}
sends all the $1$-morphisms in $\mathcal{I}_\basecat$ to
null-homotopic ones. By the universal property of the Drinfeld quotient, 
\eqref{eqn-hlax-functor-hperf-H'-to-F} lifts to a homotopy-lax $2$-functor
\begin{equation*}
	\Phi_\basecat\colon \hcat\basecat =
	\bihperf(\hcat*\basecat)/\mathcal{I}_{\basecat} \rightarrow 
	\fcat\basecat.
\end{equation*}
This homotopy strong $2$-functor gives our categorical Fock space
$\fcat\basecat$ the structure of a representation of the Heisenberg
$2$-category $\hcat\basecat$:

\begin{Theorem}\label{thm:cat_Fock_representation}
	The constructions above give a homotopy strong $2$-functor
	\[
	\Phi_\basecat \colon \hcat\basecat \to \fcat\basecat,
	\]
	that is, a $2$-categorical representation of $\hcat\basecat$ on $\fcat\basecat$.
\end{Theorem}

\begin{Corollary}\label{cor:homotopy_Fock_representation}
	There exists a $2$-categorical representation of $\hcat{H^*(\basecat)}$ on the categories $H^*(\hperf\symbcn)$.
\end{Corollary}

\begin{proof}
	This follows immediately by combining Theorem~\ref{thm:cat_Fock_representation} with Corollary~\ref{cor:graded-to-graded}.
\end{proof}

If one is only interested in the action of homotopy categories, the functor $L \circ \bimodapx$ above can safely be ignored.
More precisely, on homotopy categories one has a canonical isomorphism $\phi_{Sa} \cong \phi_a^!$ and hence one only needs to understand the functors $\phi_a^*$ and $\phi_{a,*}$.
As these functors are already given by bimodules, the functor $L \circ \bimodapx$ simply restricts them to $\hperf\symbcn$.

\chapter{Structure of the Categorical Fock Space}
\label{sec:strfock}

\section{The symmetrised operators}\label{subsec:Fock_symmetrisers}

As described in Section~\ref{subsec:cat_heisenberg_relations}, the $1$-morphisms $\PP_{a}$, $\QQ_{a}$  and $\RR_{a}$ induce $1$-morphisms $\PP_{a}^{(n)}$, $\QQ_{a}^{(n)}$ and $\RR_{a}^{(n)}$ of $\hcat{\basecat}$ for $n \ge 0$ via symmetrisers.
These are represented by operators $P_{a}^{(n)}$, $Q_{a}^{(n)}$ and $R_{a}^{(n)}$ on $\fcat\basecat$.
In order to explicitly describe the effect of these operators on $\fcat\basecat$, we consider the functor
\[
\phi_{a^n}\colon \symbc{\ho} \to \symbc{\ho+n}, \qquad a_1 \otimes
\dots \otimes a_{\ho} \mapsto a \otimes \dots \otimes a \otimes a_1 \otimes \dots a_{\ho}.
\] The $1$-morphisms 
$P_{a}^{n}$, $Q_{a}^{n}$ and $R_{a}^{n}$ are the images of the functors
\begin{align*}
	\phi^*_{a^n}\colon& \rightmod{\symbc{\ho}} \to
	\rightmod{\symbc{\ho+n}}, \\
	\phi_{a^n,*}\colon& \rightmod{\symbc{\ho+n}} \to \rightmod{\symbc{\ho}}, 
	\intertext{and}
	\phi^!_{a^n}\colon& \rightmod{\symbc{\ho}} \to
	\rightmod{\symbc{\ho_n}},
\end{align*}
under the functor $L \circ \smash[b]{\bimodapx}$ of
Section~\ref{sec:magic-wand}, with $P_{a}^{(0)} = Q_{a}^{(0)} =
R_a^{(0)} = \id$.  Recall that in Definition \ref{def:catfockspacedg}
we defined the Fock space $\fcat\basecat$ as a $1$-full subcategory 
of the $2$-category $\bihperf(\EnhCatKCdg)$ which by Lemma 
\ref{lemma-enhcatkcdg-to-hperf-enhcatkcdg-is-quasi-equivalence}
is a \dg enhancement of the strict $2$-category 
$\EnhCatKC$ of enhanced triangulated categories. Thus $1$-morphisms in 
the Fock space are enhanced functors between enhanced triangulated 
categories. The underlying exact functors have the following explicit description: 
\begin{Lemma}
	\label{lemma-explicit-descriptions-of-p^n_a-q^n_a}
	Let $a$ be an object of $\basecat$. Then:
	\begin{enumerate}
		\item 
		\label{item-explicit-description-of-the-exact-functor-p^n_a}
		The exact functor 
		\[ p_{a}^{n}\colon \catDc(\symbc{\ho}) \rightarrow \catDc(\symbc{\ho+n}) \] 
		underlying the enhanced functor $P_{a}^{n}$ is isomorphic to the composition 
		\[
		\catDc(\symbc{\ho}) \xrightarrow{h^r(a^n) \otimes (-)}
		\catDc(\symbc{n} \otimes \symbc{\ho})
		\xrightarrow{\Ind_{\SymGrp n \times \SymGrp{\ho}}^{{\SymGrp{\ho+n}}}}
		\catDc(\symbc{\ho+n}), 
		\]
		where $h^r(a^n) \otimes (-)$ is the evaluation of 
		\eqref{eqn-modA-otimes-modB-to-modAotimesB} at $h^r(a^n)$. This \dg
		functor sends 
		any $E \in \modd\text{-}\symbc{\ho}$ to the module over $\symbc{\ho}
		\otimes \symbc{n}$ whose fibers are given by the tensor product over $\kk$
		of the fibers of $h^r(a^n)$ and the fibers of $E$. As it sends
		acyclics to acyclics, it descends to the derived categories as-is.  
		
		\item 
		\label{item-explicit-description-of-the-exact-functor-q^n_a}
		The exact functor 
		\[ q_{a}^{n}\colon \catDc(\symbc{\ho}) \rightarrow \catDc(\symbc{\ho+n}) \]
		underlying the enhanced functor $Q_{a}^{n}$ is isomorphic to the composition 
		\[
		\catDc(\symbc{\ho+n})
		\xrightarrow{\Res_{{\SymGrp{\ho+n}}}^{\SymGrp n \times \SymGrp{{\ho}}}}
		\catDc(\symbc{\ho} \otimes \symbc{n})  
		\xrightarrow{\homm_{\symbc{n}}({h^r(a^n),-})}
		\catDc(\symbc{\ho}), 
		\]
		where $\homm_{\symbc{n}}({h^r(a^n),-})$ is the right adjoint of 
		$h^r(a^n) \otimes (-)$. It is the \dg functor of taking $\homm$-spaces
		as $\symbc{\ho}$ modules. With $h^r(a^n)$ in 
		the first argument, it is isomorphic to the functor
		$(-)_{a^n}$ of taking fibers over $a^n \in \symbc{n}$. As it sends
		acyclics to acyclics, it descends to the derived categories as-is.  
	\end{enumerate}
\end{Lemma}

\begin{proof}
	Since $\QQ^n_a$ is the $2$-categorical right adjoint of $\PP^n_a$ and
	$\Phi_\basecat$ is homotopy monoidal, $Q^n_a$ is a homotopy right
	adjoint of $\PP^n_a$. Therefore $q^n_a$ is the right adjoint of
	$p^n_a$. We thus only prove 
	\ref{item-explicit-description-of-the-exact-functor-p^n_a}, 
	as \ref{item-explicit-description-of-the-exact-functor-q^n_a}
	follows by adjunction. 
	
	By definition, $P^n_a$ is the image of $\phi_{a^n}^*$ under the
	functor $L \circ \smash[b]{\bimodapx}$ of taking bimodule
	approximation, and then projecting to the derived category of
	bimodules. Since $\phi_{a^n}^*$ is already a tensor functor,
	it restricts to
	$$ \phi_{a^n}^*\colon \hperf (\symbc{\ho}) \rightarrow 
	\hperf (\symbc{\ho+n}), $$
	and the corresponding exact functor $p^n_a$ is the $\Hzero$-truncation 
	of this restriction. 
	
	We can view $\phi_{a^n}$ as the image of $a^n \in \symbc{n}$
	under the \dg functor
	\begin{align*}
		\phi\colon \quad & \symbc{n} \xrightarrow{\text{unit}} \DGFun(\symbc{\ho},\, 
		\symbc{n} \otimes \symbc{\ho}) \xrightarrow{i_{n,\ho}^{n+\ho} \circ
			(-)} \\ 
		\rightarrow  &\DGFun(\symbc{\ho},\, \symbc{\ho+n})
		\xrightarrow{(-)^*} 
		\DGFun\left(\hperf(\symbc{\ho}),\, \hperf(\symbc{\ho+n})\right), 
	\end{align*}
	where $i_{n,\ho}^{n+\ho}\colon \symbc{n} \otimes \symbc{\ho} \rightarrow 
	\symbc{\ho+n}$ is the natural inclusion. 
	It can now be readily verified that $\phi$ is isomorphic to the composition
	of the Yoneda embedding $\symbc{n} \hookrightarrow \hperf(\symbc{n})$
	with 
	\begin{small}
		\begin{align}
			\label{eqn-filtering-phi-through-yoneda-embedding}
			\begin{tikzcd}[row sep=0.5cm]
				\hperf(\symbc{n})
				\ar{d}{\text{unit}}
				\\
				\DGFun\left(\hperf(\symbc{\ho}),\, \hperf(\symbc{n}) \otimes
				\hperf(\symbc{\ho})\right)
				\ar{d}{\eqref{eqn-modA-otimes-modB-to-modAotimesB} \circ (-)}
				\\
				\DGFun\left(\hperf(\symbc{\ho}),\, \hperf(\symbc{n} \otimes \symbc{\ho})\right)
				\ar{d}{(i_{n,\ho}^{n+\ho})^* \circ (-)}
				\\
				\DGFun\left(\hperf(\symbc{\ho}),\, \hperf(\symbc{\ho+n})\right). 
			\end{tikzcd}
		\end{align}
	\end{small}
	Since the \dg category isomorphism 
	$$ \symbc{n} \otimes \symbc{\ho} \simeq (\SymGrp n \times \SymGrp{\ho})
	\rtimes \basecat^{\ho+n}, $$
	and the equivalence 
	\eqref{eqn-equivalence-modules-over-twalg-with-equiv-modules}
	identify 
	$$ (i_{n,\ho}^{n+\ho})^* \colon 
	\hperf(\symbc{n} \otimes \symbc{\ho}) 
	\rightarrow \hperf(\symbc{\ho + n}) $$
	with the induction functor
	$$
	\Ind_{S_n \times S_{\ho}}^{{S_{\ho+n}}}\colon 
	\hperf^{S_n \times S_{\ho}}(\basecat^{\ho+n})
	\rightarrow 
	\hperf^{S_{\ho + n}}(\basecat^{\ho+n}), $$
	the desired claim follows. 
\end{proof}

Let $\phi^*_{e_\triv}$, 
$\phi_{e_\triv,*}$, and $\phi^!_{e_\triv}$ be the images of the idempotent 
$\phi_{e_\triv}\colon \phi_{a^n} \rightarrow \phi_{a^n}$ under the functors 
$(-)^*$, $(-)_*$, and $(-)^!$, respectively. While the
idempotents $e_\triv$ and $\phi_{e_\triv}$ are not apriori split, 
the idempotents $\phi^*_{e_\triv}$, $\phi_{e_\triv,*}$, and $\phi^!_{e_\triv}$  
always are. The splitting is obtained by taking all  elements invariant
under the action of $S_n$ on $a^n$. For example, given any module
$E \in \modd\text{-}\symbc{\ho}$ we consider the  elements of 
$\phi^*_{a^n}(E)$ which are
invariant under the endomorphisms induced by $\sigma\colon a^n
\rightarrow a^n$ for all $\sigma \in S_n$. These form a submodule
which splits the idempotent $\phi^*_{e_\triv}$ on $E$. 
In fact, $\phi^*_{a^n}$ is the functor of tensoring with the bimodule
$$
\leftidx{_{\phi_{a^n}}}{\symbc{\ho}} :=
\homm_{\symbc{\ho}}(-, a^n \otimes -), $$
and $\phi^*_{e_\triv}$ is split by 
its submodule of  elements invariant under the action of $S_n$ on $a^n$. 

By construction, the $1$-morphisms $P_{a}^{(n)}$, $Q_{a}^{(n)}$ and
$R_{a}^{(n)}$ are the images under projection 
$L \circ \smash[b]{\bimodapx}$ to $\fcat\basecat$ of the homotopical
splittings of idempotents 
$\phi^*_{e_\triv}$, $\phi_{e_\triv,*}$, and $\phi^!_{e_\triv}$
given by the construction in Section~\ref{subsec:cat_heisenberg_relations}. Thus 
$P_{a}^{(n)}$, $Q_{a}^{(n)}$ and $R_{a}^{(n)}$ are homotopy equivalent
to the images under $L \circ \smash[b]{\bimodapx}$ of the genuine
splittings of these idempotents. We thus have:

\begin{Corollary}
	Let $a$ be an object of $\basecat$. Let $h^r(a^n)^{S_n} \in \hperf
	\symbc{n}$ be the submodule of $h^r(a^n)$ consisting of
	$S_n$-invariant  elements. 
	Then:
	\begin{enumerate}
		\item 
		\label{item-explicit-description-of-the-exact-functor-p^(n)_a}
		The exact functor 
		$$ p_{a}^{(n)}\colon \catDc(\symbc{\ho}) \rightarrow \catDc(\symbc{\ho+n}) $$ 
		underlying the enhanced functor $P_{a}^{(n)}$ is isomorphic to the composition 
		$$ \catDc(\symbc{\ho}) \xrightarrow{h^r(a^n)^{S_n} \otimes (-)}
		\catDc(\symbc{n} \otimes \symbc{\ho}) \xrightarrow{\Ind_{S_n
				\times S_{\ho}}^{{S_{\ho+n}}}} \catDc(\symbc{\ho+n}). $$
		
		\item 
		\label{item-explicit-description-of-the-exact-functor-q^(n)_a}
		The exact functor 
		$$ q_{a}^{(n)}\colon \catDc(\symbc{\ho}) \rightarrow \catDc(\symbc{\ho+n}) $$
		underlying enhanced functor $Q_{a}^{(n)}$ is isomorphic to the composition 
		$$ 
		\catDc(\symbc{\ho+n})
		\xrightarrow{\Res_{{S_{\ho+n}}}^{S_n \times S_{{\ho}}}}
		\catDc(\symbc{\ho} \otimes \symbc{n})  
		\xrightarrow{\homm_{\symbc{n}}({h^r(a^n)^{S_n},-})}
		\catDc(\symbc{\ho}).
		$$
	\end{enumerate}
	\begin{proof}
		As before, \ref{item-explicit-description-of-the-exact-functor-q^(n)_a}
		follows by adjunction from 
		\ref{item-explicit-description-of-the-exact-functor-p^(n)_a}. 
		
		To prove the latter, recall that in the proof of Lemma ~\ref{lemma-explicit-descriptions-of-p^n_a-q^n_a} we have established 
		that $p^n_a$ is isomorphic to the
		$\Hzero$-truncation of $\phi(a^n)$ where $\phi$
		is isomorphic to the composition of the Yoneda embedding
		and \eqref{eqn-filtering-phi-through-yoneda-embedding}. 
		The idempotent $e_\triv: a^n \rightarrow a^n$ becomes split once
		we apply the Yoneda embedding $h^r(-)$ and $h^r(a^n)^{S_n}$
		is the corresponding direct summand. Therefore 
		the idempotent $\phi(e_\triv)$ is split and 
		the corresponding direct summand of $\phi(a_n)$ is given 
		by the image of $h^r(a^n)^{S_n}$ under 
		\eqref{eqn-filtering-phi-through-yoneda-embedding}. Since $p^{(n)}_a$ 
		is isomorphic to this direct summand of $p^n_a$, the claim folllows. 
	\end{proof}
	
\end{Corollary}

The $1$-morphisms $\PP_{a}^{(n)}$, $\QQ_{a}^{(n)}$ and $\RR_{a}^{(n)}$
satisfy a number of relations arising from the relations between  in $\hcat{\basecat}$. For example:
\begin{enumerate}
	\item There are adjunctions  $P_{a}^{(n)} \dashv Q_{a}^{(n)}$ and $Q_{a}^{(n)} \dashv R_{a}^{(n)}$.
	\item By Remark~\ref{rem:why_sym}, for every $\alpha \in \mathrm{Sym}^n(\mathrm{Hom}(a,b))$ there are natural transformations
	$P_{a}^{(n)} \xRightarrow{\alpha} P_{b}^{(n)}$ and $Q_{b}^{(n)} \xRightarrow{\alpha} Q_{a}^{(n)}$.
	\item By Theorem~\ref{thm:cat_heisenberg_relations-dg}, for any $a,\, b \in \basecat$ and $n,\, m \in \NN$ we have natural isomorphisms
	\begin{equation}\label{eq:fock_cat_heisenberg_relations_1} 
		P_a^{(m)}P_{b}^{(n)} \cong P_{b}^{(n)} P_a^{(m)}, \quad
		Q_a^{(m)}Q_{b}^{(n)} \cong Q_{b}^{(n)} Q_a^{(m)},
	\end{equation}
	and a homotopy isomorphism
	\begin{equation}\label{eq:fock_cat_heisenberg_relations_2} 
		\bigoplus_{i=0}^{\mathclap{\min(m,n)}}\ \Sym^i(\Hom_\basecat(a,b)) \otimes_\kk P_{b}^{(n-i)}Q_a^{(m-i)} \to Q_a^{(m)}P_{b}^{(n)}.
	\end{equation}
\end{enumerate}

\begin{Example}\label{ex:Krug_part2}
	Let $X$ be a smooth and projective variety.
	Continuing Example~\ref{ex:Krug_part1}, we obtain the symmetrised operators $P_{a}^{(n)}$ and $Q_{a}^{(n)}$.
	These reproduce the remaining functors from \cite[Section~2.4]{krug2018symmetric} for the \dg derived categories.
	Thus Corollary~\ref{cor:homotopy_Fock_representation} enhances the representation defined by Krug to a $2$-categorical action of $H^*(\hcat{\catDGCoh{X}})$.
	On the homotopy categories, \eqref{eq:fock_cat_heisenberg_relations_1} and \eqref{eq:fock_cat_heisenberg_relations_2} become
	\begin{equation*}
		\begin{gathered} 
			P_a^{(m)}P_{b}^{(n)} \cong P_{b}^{(n)} P_a^{(m)}, \quad
			Q_a^{(m)}Q_{b}^{(n)} \cong Q_{b}^{(n)} Q_a^{(m)}, \\
			Q_a^{(m)}P_{b}^{(n)} \cong \ \bigoplus_{i=0}^{\mathclap{\min(m,n)}}\ \Sym^i \Hom^*(a,b) \otimes_\kk P_{b}^{(n-i)}Q_a^{(m-i)}.
		\end{gathered}
	\end{equation*}
	This provides a new proof of \cite[Theorem 1.4]{krug2018symmetric}.
\end{Example}

\section{Grothendieck groups and the classical Fock space}\label{subsec:grothfock}

\subsection{Constructing a representation of the Heisenberg algebra}

In Section~\ref{subsec:prelim_grothendieck_group} we defined  
the numerical Grothendieck group\index{numerical Grothendieck group} $\numGgp{\basecat,\, \kk}$ of
a smooth and proper \dg category $\basecat$. 
As finite tensor products of \dg categories preserve both 
of these properties, $\basecat^{{ \otimes} N}$ is smooth and proper. It is then 
evident from the decomposition
\eqref{equation-decomposition-of-the-equivariant-diagonal-bimodule} of
the diagonal bimodule, that $\symbc\ho$ is smooth and proper as well. 
Thus its numerical Grothendieck group is well-defined.

With this in mind, define the $\kk$-linear $1$-category 
\[
\End\biggl(\bigoplus_{\ho}  \numGgp{\symbc\ho,\, \kk}\biggr)
\]
to have as objects $\numGgp{\symbc{\ho},\, \kk} \coloneqq \kk
\otimes_\ZZ \numGgp{\symbc{\ho}}$, and as morphisms the $\kk$-linear 
maps between these vector spaces.
Thus a $\kk$-linear functor into this category is an idempotent-modified version of a representation on the vector space $\bigoplus_{\ho}  \numGgp{\symbc\ho,\, \kk}$.

We next use the $2$-functor
$\Phi_\basecat\colon \hcat\basecat \to \fcat\basecat$ 
to define a $1$-functor 
from $\numGgp{\hcat\basecat,\,\kk}$ to 
$\bigoplus_{\ho}  \numGgp{\symbc\ho,\, \kk}$. 
Recall our definition of 
$\numGgp{\hcat\basecat,\,\kk}$: it is the quotient of 
$\mathrm{K}_0(\hcat\basecat,\,\kk)$ by the two-sided ideal generated by the
images under $\Xi^P,\, \Xi^Q: \bihperf(\bicat{Sym}_\basecat) \rightarrow 
\hcat\basecat$ of the kernel of the Euler pairing, see Section~\ref{subsec:grothendieck_groups}.
To show that this two-sided ideal gets sent by $\Phi_\basecat$ 
to the kernel of the Euler pairing on $\bigoplus_{\ho}  \mathrm{K}_0(\symbc\ho,\, \kk)$, 
we need the following lemma:

\begin{Lemma}
	\label{lemma-the-description-of-composition-Phi-Xi^P}
	The composition   
	$$ \Phi_\basecat \circ \Xi^P: 
	\bihperf(\bicat{Sym}_\basecat) \rightarrow \fcat\basecat, $$
	is the following $2$-functor. On the object sets, it is 
	$\id \colon \mathbb{Z} \rightarrow \mathbb{Z}$. 
	On the $1$-morphism categories $\ho \rightarrow \ho + n$, it is
	homotopy equivalent to the composition of $L \circ \smash[b]{\bimodapx}$ 
	with the \dg functor
	$$ \hperf\left(\symbc{n}\right) 
	\xrightarrow{\eqref{eqn-filtering-phi-through-yoneda-embedding}}
	\DGFun\left(\hperf(\symbc{\ho}),\, \hperf\left(\symbc{\ho + n}\right)\right). $$
\end{Lemma}

\begin{proof}
	By definition, $\Phi_\basecat \circ \Xi^P$ is the perfect hull of
	the $2$-functor
	$$ \bicat{Sym}_\basecat \xrightarrow{L \circ \smash[b]{\bimodapx}
		\circ \Phi'_\basecat \circ \Xi^{P'}_\basecat} \EnhCatKCdg. $$
	On the $1$-morphism categories $\ho \rightarrow \ho + n$, 
	the composition $\Phi'_\basecat \circ \Xi^{P'}_\basecat$
	is the functor 
	$$ \phi\colon \symbc{n} \rightarrow  
	\DGFun\left(\hperf(\symbc{\ho}),\, \hperf(\symbc{\ho+n})\right), $$
	defined in the proof of Lemma
	\ref{lemma-explicit-descriptions-of-p^n_a-q^n_a}. 
	Therefore, on the $1$-morphism categories $\ho \rightarrow \ho + n$, 
	the composition $\Phi_\basecat \circ \Xi^P$ is the functor
	$(L \circ \smash[b]{\bimodapx} \circ \phi)^*$. 
	
	Since $\phi$ is isomorphic to 
	$$ \symbc{n} \xrightarrow{\text{Yoneda}} \hperf\left(\symbc{n}\right) 
	\xrightarrow{\eqref{eqn-filtering-phi-through-yoneda-embedding}}
	\DGFun\left(\hperf(\symbc{\ho}),\, \hperf\left(\symbc{\ho +
		n}\right)\right) $$
	the desired assertion now follows from the following
	fundamental fact. Let $\A$ and $\B$ be any \dg categories and 
	$F\colon \A \rightarrow \B$ any \dg functor. If $F$ decomposes as 
	$$ \A \xrightarrow{\text{Yoneda}} \hperfA \xrightarrow{G} \B, $$
	for some \dg functor $G$, then $F^*$ is homotopy equivalent to
	$$ \hperfA \xrightarrow{G} \B \xrightarrow{\text{Yoneda}} \hperfB. $$
	To see this, consider the commutative square
	\begin{equation*}
		\begin{tikzcd}
			\hperfA 
			\ar{r}{G} 
			\ar{d}{\text{Yoneda}} 
			&
			\B
			\ar{d}{\text{Yoneda}}
			\\
			\hperf(\hperfA)
			\ar{r}{G^*}
			&
			\hperfB,
		\end{tikzcd}
	\end{equation*}
	and observe that the \dg functor
	$$\text{Yoneda}\colon \hperfA \rightarrow \hperf(\hperfA)$$
	is homotopy equivalent to the \dg functor
	\[
	\text{Yoneda}^* \colon \hperfA \rightarrow \hperf(\hperfA).
	\qedhere
	\]
\end{proof}

We can now construct the desired $1$-functor:

\begin{Corollary}\label{cor:KgpFock}
	The $2$-functor $\Phi_\basecat\colon \hcat\basecat \to \fcat\basecat$ from Theorem~\ref{thm:cat_Fock_representation} induces a $1$-functor
	\[
	\numGgp{\hcat\basecat,\,\kk} \to \End\biggl(\bigoplus_{\ho}  \numGgp{\symbc\ho,\, \kk}\biggr).
	\]
	In other words one obtains a representation of $\numGgp{\hcat\basecat,\,\kk}$ on $\bigoplus \numGgp{\symbc\ho,\, \kk}$.
\end{Corollary}

\begin{proof}
	Functoriality of Grothendieck groups gives a $1$-functor 
	$$ \Phi_{\basecat}\colon \mathrm{K}_0(\hcat\basecat,\,\kk) \to \End \left(
	\bigoplus_{\ho} \mathrm{K}_0(\symbc\ho,\, \kk)\right).$$
	We claim that any morphism $\mathrm{K}_0(\symbc\ho,\, \kk) \rightarrow 
	K^0(\symbc{M}, \, \kk)$ in its image takes 
	the kernel of the Euler pairing $\chi$ on $\mathrm{K}_0(\symbc\ho,\, \kk)$ 
	to its kernel on $K^0(\symbc{M}, \, \kk)$. 
	As per Section~\ref{sec:symmetric-powers-of-dg-categories}, 
	the $1$-morphisms to which $\Phi_\basecat$ maps generating $1$-morphisms
	$\PP$s, $\QQ$s, and $\RR$s of $\hcat*\basecat$ are left- and right-perfect
	bimodules. Hence the same is true of all $1$-morphisms
	in $\Phi_\basecat(\hcat*\basecat)$.
	By construction, $\Phi_\basecat(\hcat\basecat)$ lies in the
	$\hperf$-hull of $\Phi_\basecat(\hcat*\basecat)$, and thus the
	$1$-morphisms in $\Phi_\basecat(\hcat\basecat)$ are also left- and
	right-perfect bimodules. By \cite[Theorem
	4.1]{AnnoLogvinenko-BarCategoryOfModulesAndHomotopyAdjunctionForTensorFunctors}
	the corresponding exact functors $\catDc(\symbc\ho) \rightarrow
	\catDc(\symbc{M})$ have left adjoints. 
	Arguing as in Lemma \ref{lem:indmapnum}, we see that the induced maps
	$\mathrm{K}_0(\symbc\ho,\, \kk) \rightarrow K^0(\symbc{M}, \, \kk)$  
	take $\ker \chi$ to $\ker \chi$. 
	
	We thus have a $1$-functor 
	$$ \Phi_{\basecat}\colon \mathrm{K}_0(\hcat\basecat,\,\kk) \to 
	\End \left( \bigoplus_{\ho} \numGgp{\symbc\ho,\, \kk} \right).$$
	It remains to show that this functor descends to
	$\numGgp{\hcat\basecat,\,\kk}$. For the definition of the latter, 
	see Section~\ref{subsec:grothendieck_groups}. 
	
	Let $E \in \hperf(\symbc n)$ be in the kernel of the Euler pairing. 
	Let us view $E$ as a $1$-morphism $\ho \rightarrow \ho + n$ in 
	$\bihperf(\bicat{Sym}_\basecat)$. 
	By Lemma \ref{lemma-the-description-of-composition-Phi-Xi^P}, 
	$\Phi_{\basecat} \circ \Xi^P(E)$ is an enhanced
	functor whose underlying exact functor is
	$$ \catDc(\symbc{\ho}) \xrightarrow{E \otimes (-)}
	\catDc(\symbc{n} \otimes \symbc{\ho}) 
	\xrightarrow{\Ind^{S_{\ho+n}}_{S_n \times S_{\ho}}} 
	\catDc(\symbc{\ho+n}). $$ 
	We have to show that its image lies in $\ker \chi$, 
	and thus the induced map of $K^{\text{num}}_0$ is zero. 
	Let $F \in \hperf(\symbc{\ho})$ and observe that 
	$$ \Phi_{\basecat} \circ \Xi^P(E)(F) 
	\simeq \Ind^{S_{\ho+n}}_{S_n \times S_{\ho}}(E \otimes F) \simeq 
	\Ind^{S_{\ho+n}}_{S_{\ho} \times S_{n}}(F \otimes E) 
	\simeq \Phi_{\basecat} \circ \Xi^P(F)(E). $$
	Above we already established that the underlying exact functor
	of any $1$-morphism in the image of $\Phi_\basecat$ takes 
	$\ker \chi$ to $\ker \chi$. Thus $\Phi_{\basecat} \circ \Xi^P(F)(E)$
	lies in $\ker \chi$, and hence so does 
	$\Phi_{\basecat} \circ \Xi^P(E)(F)$. 
	By adjunction, $\Phi_{\basecat} \circ \Xi^Q(E)(F)$ lies in $\ker \chi$
	as well. 
	
	We have now established that on the level of Grothendieck groups
	$\Phi_{\basecat}$ kills the image under $\Xi^P$ and $\Xi^Q$ of
	the kernel of the Euler pairing on $\mathrm{K}_0(\bihperf(\bicat{Sym}_\basecat), \kk)$. 
	Since $\numGgp{\hcat\basecat,\,\kk}$ is the quotient of 
	$\mathrm{K}_0(\hcat\basecat,\,\kk)$ by the two-sided ideal generated by this
	image, we conclude that $\Phi_{\basecat}$ descends to a functor
	$\numGgp{\hcat\basecat,\,\kk} \rightarrow \End \left( \bigoplus_{\ho}
	\numGgp{\symbc\ho,\, \kk} \right)$, as desired. 
\end{proof}

\subsection{Genuine categorification}
\label{subsec:genuine_categorification}

Consider $\Phi_{\basecat}$ as homomorphism of algebras and compose it with the algebra homomorphism $\pi\colon \halg\basecat
\to \numGgp{\hcat\basecat,\, \kk}$ of Section~\ref{subsec:grothendieck_groups} to obtain a homomorphism
\[
\halg\basecat \to \End\biggl(\bigoplus_{\ho}  \numGgp{\symbc\ho,\, \kk}\biggr).
\]
The vector $1 \in \numGgp{\symbc 0,\, \kk} \cong \kk$ is annihilated by all elements of $\halg\basecat^- \setminus \{1_0\}$ and is kept invariant by $1_0$.
Lemma~\ref{lem:fock_embeds} then implies that there is  a graded
$\halg\basecat$-module embedding
\begin{equation}\label{eq:fock_embedding}
	\phi\colon \falg{\basecat} = 
	\bigoplus_{\ho} \falg{\basecat}^{\ho} \hookrightarrow
	\bigoplus_{\ho} \numGgp{\symbc\ho,\, \kk}
\end{equation}
of the appropriate classical Fock space.

The following is a generalisation of \cite[Section~3.1]{krug2018symmetric}. 
For a partition $\lambda$, write $r(\lambda)_i$ for the number of parts of $\lambda$ of size $i$.
\begin{Corollary}\label{cor:fockcateg} 
	Suppose that the following dimension formula holds:
	\[
	\dim \numGgp{\symbc\ho,\, \kk} =
	\sum_{\lambda \dashv \ho } \prod_{i} \dim \Sym^{r(\lambda)_i} \numGgp{\basecat,\, \kk}   \]
	where the sum runs over all partitions $\lambda$
	of $\ho$ and the product over all sizes $i$ of parts of $\lambda$.
	Then \eqref{eq:fock_embedding} is an $\halg{\basecat}$-module
	isomorphism. That is, $\fcat\basecat$ categorifies $\falg\basecat$.
\end{Corollary}

\begin{proof}
	The assumption and \eqref{eq:fockdegk} implies that the dimensions of 
	the graded vector spaces $\falg{\basecat}$ and $\bigoplus_{\ho} \numGgp{\symbc\ho, \kk}$ agree in each degree. Hence, these graded spaces must be isomorphic.
\end{proof}

\begin{Example}
	\label{ex:fockcatdim}
	The assumption of Corollary~\ref{cor:fockcateg}  is satisfied in the following cases:
	\begin{enumerate}
		\item 
		\label{it:fockcatdim1}
		Let $X$ be a smooth projective variety and $\basecat = \catDGCoh{X}$ as in Examples~\ref{ex:Krug_part1} and \ref{ex:Krug_part2}.
		Assume moreover that the numerical Grothendieck group\index{numerical Grothendieck group} satisfies a Künneth formula:
		\[
		\numGgp{ \basecat^{\otimes \ho}} \cong (\numGgp{\basecat})^{\otimes \ho}.
		\]
		This is the case, by Example~\ref{ex:mukai}, if the Chow
		groups of $X$ tensored with $\mathbb{Q}$ satisfy the
K{\"u}nneth formula\index{K{\"u}nneth formula}.
		A sufficient condition for this is that the Chow motive of $X$ is a
		summand of a direct sum of Tate motives \cite{totaro2016motive}. 
		Note that this is a very strong assumption which is closely related to
		$\catDbCoh{X}$ having a full exceptional collection\index{full exceptional collection}. It is already
		false for elliptic curves as we see in the counterexample in 
		\S\ref{section:counterexample}. 
		
		As the Chern character is additive on disjoint varieties, 
		by Hirzebruch–-Riemann-–Roch\index{Hirzebruch--Riemann--Roch theorem} we can replace $\mathrm{K}_0$ with
		$\mathrm{K}^{\text{num}}_0$ in  \cite[Theorem 1]{vistoli1991higher}
		to get a direct sum decomposition:
		\[
		\numGgp{\symbc\ho,\, \kk} \cong \bigoplus_{\lambda \dashv \ho } \bigotimes_i \Sym^{ r(\lambda)_i} \numGgp{\basecat,\, \kk} 
		\]
		where the sum runs over all partitions $\lambda$ of $\ho$ and the product 
		over all sizes $i$ of parts of $\lambda$.
		
		Hence, $\basecat$ satisfies the assumption of Corollary~\ref{cor:fockcateg}.
		%
		\item 
		\label{it:fockcatdim2}
		Let $\Gamma \subset \mathrm{SL}(2,\mathbb{C})$ be a finite subgroup and $\basecat$ as in Examples~\ref{ex:CautisLicata_part1}, \ref{ex:CautisLicata_part3} and \ref{ex:CautisLicata_part2}. Then the dimension assumption for the usual K-groups follows from the combination of \cite[Proposition 5]{wang2000equivariant},  Göttsche's formula for the Betti numbers of Hilbert schemes\index{Hilbert schemes} and the fact the topological and algebraic K-theories agree on the minimal resolution of the quotient variety $\mathbb{C}^2/\Gamma$ (as both are described by the representation theory of $G$  \cite[Chapter 4]{nakajima1999lectures}). The Euler form equals the intersection form on the resolution, which is given by the appropriate finite type Cartan matrix. This is known to be non-degenerate. Hence, the kernel of $\chi$ is trivial in each case, and the dimension assumption descends to the numerical K-groups.
	\end{enumerate}
\end{Example}

\begin{Remark}
	An alternative way to obtain Example~\ref{ex:fockcatdim}~\ref{it:fockcatdim1} in many cases is to combine the main result of \cite{brundan2018degenerate} proving that the map
	\[ \pi: \halg\basecat \to \numGgp{\hcat\basecat,\kk} \] 
	is an isomorphism when $X=\Spec(\kk)$ is a point (and hence also when $\catDbCoh{X}$ has a full exceptional collection) with our Theorem~\ref{thm:piisothenphi} below.
\end{Remark}



\subsection{A counterexample}
\label{section:counterexample}
We now give an example of $\pi$ not being an isomorphism.
Let $X$ be a smooth projective curve and $n \in \ZZ_{>0}$. Denote by 
\[X^{(\ho)}=X^{\ho}/S_{\ho}\]
the $\ho$-th symmetric power of $X$. This is a smooth projective variety of dimension $\ho$. 

{ Let $\lambda$ be any partition of $\ho$. Write $r_i$ for the number 
	of parts of size $i$ in $\lambda$. Define the closed subvariety 
	\[X[\lambda] \subset X^{\ho}\] 
	to be the fixed point locus of some $\sigma \in S_{\ho}$ of cycle type
	$\lambda$. Different choices of $\sigma$ produce canonically
	isomorphic $X[\lambda]$. Explicitly, $X[\lambda]$ consists of $(x_1,
	\dots, x_{\ho})$ where $x_i = x_{\sigma(i)}$ for all $i \in 1, \dots ,
	\ho$. Thus $X[\lambda] \simeq X^k$ where $k$ is the total number of parts 
	in $\lambda$. The action of the centraliser $C(\sigma) \subset
	S_{\ho}$ on $X^{\ho}$ restricts to $X[\lambda]$ as 
	the action of $S_{\lambda}=\prod_i S_{r_i}$ which permutes 
	the factors of $X^k$ which correspond to the parts of the same size in
	$\lambda$.} The quotient variety is 
\[ X[\lambda]/S_{\lambda}= \prod_i X^{(r_i)}.\] 
For any ordering $\lambda^1,\dots,\lambda^p$ of partitions of $\ho$ 
refining the dominance order, there is a semiorthogonal decomposition
\[
\begin{aligned}
	\catDbCoh{[X^{\ho}/S_{\ho}]} & =\left\langle \catDbCoh{X[\lambda^1]/S_{\lambda^1}},\dots, \catDbCoh{X[\lambda^p]/S_{\lambda^p}}  \right\rangle \\
	& =\left\langle \mathrm{D}^{\mathrm{b}}_{\mathrm{coh}}\left(\prod_i X^{(r(\lambda^1)_i)}\right),\dots, \mathrm{D}^{\mathrm{b}}_{\mathrm{coh}}\left(\prod_i X^{(r(\lambda^p)_i)}\right)  \right\rangle
\end{aligned}
\]
by \cite{polishchuk2019semiorthogonal}. As $\ho=2$ has two partitions ($\lambda^1=(2)$ and $\lambda^2=(1,1)$), we have for the second symmetric quotient stack the semiorthogonal decomposition
\
\begin{equation}
	\label{eq:dbcohcurvesym2}
	\catDbCoh{[X^2/S_2]}= \left\langle \catDbCoh{X},\catDbCoh{X^{(2)}} \right\rangle.\end{equation}

Let now $X$ be an elliptic curve. It is known that for each $\ho>1$ the Abel-Jacobi map realizes $X^{(\ho)}$ as a $\ps{\ho-1}$-bundle over $X$, see \cite[Section~1.1]{catanese1993symmetric}. Hence, by \cite{orlov1992projective} there is a semiorthogonal decomposition
\[ \catDbCoh{X^{(\ho)}}=\langle \underbrace{ \catDbCoh{X}, \dots,\catDbCoh{X} }_{\ho \textrm{ times } } \rangle.\]
Combining this for $\ho=2$ with \eqref{eq:dbcohcurvesym2}, we obtain a semiorthogonal decomposition
\begin{equation}
	\label{eq:dbcohx2dec}	
	\catDbCoh{[X^2/S_2]} = \left\langle \catDbCoh{X}, \catDbCoh{X}, \catDbCoh{X} \right\rangle.
\end{equation}


Recall that 
\[
\begin{aligned}
	\mathrm{K}_0(X) & \xrightarrow{\sim} & \ZZ\oplus  \mathrm{Pic}(X)   \\
	[F] & \mapsto & ( \mathrm{rk} F, \det F)
\end{aligned}
\]
is an isomorphism \cite[{ Exercise~II.6.11}]{hartshorne1977algebraic}. From this we get that 
\[
\begin{aligned}
	\numGgp{X} & \xrightarrow{\sim} &  \ZZ\oplus  \mathrm{Pic}(X)/\mathrm{Pic}^0(X) \cong \ZZ  \oplus \ZZ  \\
	[F] & \mapsto & (\mathrm{rk} F, \deg \det F)
\end{aligned}
\]
is also an isomorphism. Here $\mathrm{Pic}(X)/\mathrm{Pic}^0(X)=NS(X)$ is the Neron-Severi group\index{Neron-Severi group} of X. Hence, by \eqref{eq:dbcohx2dec}
\[ \dim \numGgp{[X^2/S_2],\kk}=6. \]
On the other hand, the dimension of the degree 2 part of the classical Fock space of $X$ by \eqref{eq:fockdegk} is
\[ \dim \Sym^{1}\numGgp{X,\kk} + \dim \Sym^{2}\numGgp{X,\kk}= \binom{2+1-1}{1}+\binom{2+2-1}{2}=5.\]
Therefore, $\phi$ cannot be an isomorphism when $X$ is an elliptic curve. By Theorem~\ref{thm:piisothenphi} below the same holds for $\pi$. 

\section{The Fock space as a quotient}\label{subsec:fock_cat_2}

The following constructions are easier to express in a monoidal
setting, rather than in the $2$-categorical setting we worked in so far.
Thus, let $\monhcat\basecat$ be the $\HoDGCatone$-monoidal \dg 
$1$-category obtained from $\hcat\basecat$ 
by identifying all objects and all $1$-morphism categories 
$\Hom(\ho,\, \ho+n)$ for fixed $n \in \ZZ$. Concretely, set
\[
\monhcat\basecat = \bigoplus_{n\in \ZZ} \Hom_{\hcat\basecat}(0,n)
\]
with the monoidal structure given by the horizontal composition in $\hcat\basecat$ via the identification
\[
\begin{multlined}
\Hom_{\hcat\basecat}(0,n_1) \otimes \Hom_{\hcat\basecat}(0,n_2) \cong
\Hom_{\hcat\basecat}(0,n_1) \otimes \Hom_{\hcat\basecat}(n_1,n_1+n_2) \\ \to
\Hom_{\hcat\basecat}(0,n_1+n_2).
\end{multlined}
\]

Applying the same flattening\index{flattening} procedure to $\fcat\basecat$, we obtain 
a \dg category 
\begin{equation}
	\label{eqn-monfcat}
	\monfcat\basecat = \bigoplus_{n \geq 0} \hperf \perfbar (\symbc n).
\end{equation}
In $\fcat\basecat$ we do not have a
$\homm$-category isomorphism \[\Hom_{\fcat\basecat}(0,n_2) \cong
\Hom_{\fcat\basecat}(n_1,n_1+n_2).\] However, there is a natural
functor between the two:
\begin{small}
	\begin{multline*}
		\Hom_{\fcat\basecat}(0,n_2) = \hperf \perfbar (\symbc{n_2}) \\
		 \xrightarrow{(\symbc{n_1} \otimes_k (-))^{*}}
		\hperf \bigl(\bimodbar{\symbc{n_1}}{(\symbc{n_1}\otimes_k \symbc{n_2})}\bigr)_{\lfrp} \\
		 \xrightarrow{\Ind^{S_{n_1 + n_2}}_{S_{n_1},S_{n_2}}}
		\hperf \bigl(\bimodbar{\symbc{n_1}}{\symbc{n_1+n_2}}\bigr)_{\lfrp}
		= \Hom_{\fcat\basecat}(n_1,n_1 + n_2). 
	\end{multline*}
\end{small}
Using it, we obtain from the $1$-composition of $\fcat\basecat$ 
a monoidal structure on $\monfcat\basecat$ given by 
$$ \hperf \perfbar (\symbc{n_1}) \otimes \hperf \perfbar (\symbc{n_2})
\rightarrow \hperf \perfbar (\symbc{n_1 + n_2}), $$ 
induced by the natural inclusion $\symbc{n_1} \otimes \symbc{n_2}
\hookrightarrow \symbc{n_1 + n_2}$.

Applying the flattening to the $2$-functor 
$\Phi_\basecat\colon \hcat\basecat \rightarrow \fcat\basecat$
constructed in Chapter~\ref{sec:cat_fock} we obtain a \dg functor
$\Phi_\basecat\colon \monhcat\basecat \rightarrow \monfcat\basecat$.
One can readily check that it is homotopy monoidal with 
respect to the monoidal structures on $\monhcat\basecat$ and 
$\monfcat\basecat$ described above. 

Next, take the functor $L\colon \DGBiMod_{\lfrp} \rightarrow \EnhCatKCdg$
defined in Section~\ref{sec:magic-wand}, restrict it to the
categories $\symbc{n}$ and apply the flattening. 
We obtain a \dg functor 
$$ \bigoplus_{n} \perf \symbc{n} \xrightarrow{L}
\bigoplus_{n} \perfbar  \symbc{n}. $$
Precomposing it with the inclusions $\hperf \symbc{n} \hookrightarrow
\perf \symbc{n}$ we obtain the quasi-equivalence  
$$ \bigoplus_{n} \hperf \symbc{n} \xrightarrow{L}
\bigoplus_{n} \perfbar  \symbc{n}. $$
Since for any \dg category $\A$ the Yoneda embedding $\perfbar \A
\rightarrow \hperf \perfbar \A$ is a quasi-equivalence, we further
obtain a quasi-equivalence
\begin{equation}
	\label{eqn-quasi-equiv-oplus-hperf-S^nV-to-monfcat}
	\bigoplus_{n} \hperf \symbc{n}
	\xrightarrow{L}
	\bigoplus_{n} \perfbar \symbc{n}
	\xrightarrow{\text{Yoneda}} 
	\bigoplus_{n} \hperf \perfbar \symbc{n}
	= \monfcat\basecat. 
\end{equation}

\begin{Corollary}
	\label{cor-oplus-hperf-S^nV-is-a-retract-of-H_V}
	The \dg functor  
	$$ \bigoplus_{n} \hperf \symbc{n}
	\xrightarrow{\Xi^P}
	\monhcat\basecat
	\xrightarrow{\Phi_{\basecat}}
	\monfcat\basecat, $$
	filters through the quasi-quivalence 
	$\eqref{eqn-quasi-equiv-oplus-hperf-S^nV-to-monfcat}$ 
	as a functor homotopic to 
	\[\Id_{\bigoplus_{n} \hperf \symbc{n}}.\] 
\end{Corollary}
\begin{proof}
	This follows from the proof of Lemma
	\ref{lemma-the-description-of-composition-Phi-Xi^P} and the fact
	that the composition 
	$$ 
	\hperf\left(\symbc{n}\right) 
	\xrightarrow{\eqref{eqn-filtering-phi-through-yoneda-embedding}}
	\DGFun\left(\hperf(\kk),\, \hperf\left(\symbc{n}\right)\right)
	\xrightarrow{\bimodapx}
	\hperf\left(\symbc{n}\right) 
	$$
	is the identity functor. 
\end{proof}
\begin{Remark}
	Corollary \ref{cor-oplus-hperf-S^nV-is-a-retract-of-H_V} implies, 
	in particular, that the DG-category $\bigoplus \hperf \symbc{n}$ is 
	a homotopy retract of $\monhcat\basecat$, that is --- a retract in the category 
	$\HoDGCatone$. 
	
	In particular, $\hperf
	\basecat$ itself is a homotopy retract of $\monhcat\basecat$. 
	Thus, on the level of underlying triangulated categories, we have
	a faithful embedding 
	$\catDc(\basecat) \hookrightarrow \catDc(\monhcat\basecat)$. 
\end{Remark}

As explained in Section~\ref{subsec:heisenberg_algebra},
the classical Fock space $\cfalg\basecat$ is isomorphic to $\chalg\basecat/I$ 
where $I$ is the left ideal generated by the $q_{[a]}^{(n)}$ for $[a] \in \numGgp{\basecat,\,\kk}$ and $n > 0$.
Moreover, as seen in Section~\ref{subsec:grothfock}, we have
an embedding $\phi\colon \cfalg\basecat \hookrightarrow \bigoplus_{\ho} \numGgp{\symbc\ho,\, \kk}$.

Motivated by this, we define 
\[
\widetilde{\fcat\basecat} \coloneqq \monhcat\basecat/\mathcal I,
\]
where $\mathcal I$ is the left ideal generated by objects $\QQ_a$ for $a
\in \basecat$. 

\begin{Lemma}
	\label{lemma-oplus-hperf-S^n-V-homotopy-retract-of-tilde-F_V}
	The \dg category $\bigoplus \hperf \symbc{n}$ is a homotopy retract 
	of $\widetilde{\fcat\basecat}$. Specifically, the following
	composition is a homotopy retraction:
	\begin{equation}
		\label{eqn-homotopy-retraction-of-tildeF_V-onto-oplus-hperfS^nV}
		\bigoplus \hperf \symbc{n}
		\xrightarrow{\Xi^P}
		\monhcat\basecat
		\xrightarrow{\text{\rm Drinfeld}}  
		\widetilde{\fcat\basecat}. 
	\end{equation}
	Moreover, this composition is quasi-essentially surjective on
	objects.
\end{Lemma}
\begin{proof}
	In view of Corollary \ref{cor-oplus-hperf-S^nV-is-a-retract-of-H_V}
	it suffices to prove that the homotopy retraction 
	$$ \Phi_\basecat\colon \monhcat\basecat \rightarrow
	\monfcat\basecat $$
	filters in $\HoDGCatone$ through the Drinfeld quotient functor 
	$$ \monhcat\basecat \rightarrow \monhcat\basecat/\I =
	\widetilde{\fcat\basecat}. $$
	By the universal property of Drinfeld quotient (Theorem
	\ref{theorem-main-properties-of-drinfeld-quotients}), it suffices to
	prove that $\Phi_\basecat$ sends all objects of $\I$ to null-homotopic
	ones. By the definition of the monoidal structure on
	$\monhcat\basecat$, it suffices to check this on objects 
	$\QQ_a$ for $a \in \basecat$ which generate $\I$ as a left ideal. 
	
	In fact, $\Phi_\basecat$ sends all of these to zero.  Indeed, 
	we compute $\Phi_\basecat(\QQ_a)$ by
	evaluating the corresponding $2$-functor on $1$-morphisms
	$\QQ_a \in \homm_{\hcat\basecat}(0,-1)$. By construction, the $2$-functor 
	$\Phi_\basecat$ sends all objects $n \in \mathbb{Z}_{< 0}$ to zero, and 
	hence for any $n < 0$ it sends the whole $1$-morphism category 
	$\homm_{\hcat\basecat}(0,n)$ to zero. 
	
	For the final claim, recall that $1$-morphism categories of
	$\hcat\basecat$ are Drinfeld quotients of the perfect hulls of those
	of $\hcat*\basecat$. We then take a further Drinfeld quotient to
	obtain $\widetilde{\fcat\basecat}$. As taking Drinfeld quotient 
	doesn't change the objects of a category, 
	the objects of $\widetilde{\fcat\basecat}$ are perfect modules over 
	the $1$-morphism categories $\homm_{\hcat*\basecat}(0,n)$. We 
	can therefore view them as homotopy idempotents of twisted complexes\index{twisted complex} 
	over $\homm_{\hcat*\basecat}(0,n)$. 
	
	The objects of $\homm_{\hcat*\basecat}(0,n)$ are words on $\PP$, 
	$\QQ$, and $\RR$s. In $\hcat\basecat$, $\PP$ and $\RR$ become
	homotopy equivalent. Furthermore, the homotopy equivalence 
	\eqref{eq:Heisenberg_map} in $\hcat\basecat$ allows 
	us to turn any subword $\QQ\PP$ into a direct sum of $\PP\QQ$s and
	$\hunit$s. Since any word ending in $\QQ$ is null-homotopic 
	in $\widetilde{\fcat\basecat}$, we conclude that all objects of
	$\homm_{\hcat*\basecat}(0,n)$ are homotopy equivalent 
	in $\widetilde{\fcat\basecat}$ to direct sums of words on just $\PP$s. 
	
	It remains to show that any morphism between words on $\PP$s in 
	$\homm_{\hcat*\basecat}(0,n)$ becomes homotopic in 
	$\widetilde{\fcat\basecat}$ to something that lies in the image of $\Xi^P$. 
	In other words, homotopic to a diagram containing just the
	crossings. Since there are no $\QQ$s involved, we only need to show 
	that we can get rid of curls and of bubbles. The relations
	\eqref{eq:circle_and_curl-dg} imply that counterlclockwise curls
	are homotopic to zero, while counterclockwise bubbles are homotopic to 
	scalar multiples of identity maps. 
	
	Suppose we have a clockwise bubble. 
	If there is no vertical string to the right of it, the diagram can be
	written as a $2$-composition filtering through a word ending
	in $\QQ$ and hence vanishes. If there is a vertical string to the
	right of the bubble, we use the homotopy relations in Lemma 
	\ref{lem:up_down_braids-homotopy} to replace a downward string in the
	bubble and the (upward) vertical string by a cup and a cap plus a
	double crossing. The replacement by a cup and a cap absorbs the bubble into 
	the vertical string and gets rid of it. The replacement by a crossing 
	makes the vertical string cross the bubble. We then use the
	symmetric group relations on upward strand\index{upward strand}s to move the bubble
	completely to right of the vertical string. If there are any more
	vertical strings to the right of the bubble, we repeat this
	procedure. 
	
	A similar argument works for clockwise curls. If there are 
	no vertical strings to the right of it, the diagram passes through 
	a word ending in $\QQ$ and hence vanishes. If there are, we can
	similarly move the curl to the right of string: the replacement by
	a cup and a cap turns the curl into a crossing and gets rid of it, 
	while the replacement by a double crossing makes the vertical string cross
	the curl, and we can then use a triple move to finish moving the 
	curl completely to the right of the vertical string.
\end{proof}

We have shown above that $\Phi_\basecat$ filters through the
Drinfeld quotient $\monhcat\basecat \rightarrow
\widetilde{\fcat\basecat}$. Let 
$$ \tilde{\Phi}_\basecat\colon 
\widetilde{\fcat\basecat} \rightarrow \monfcat\basecat $$
be the corresponding quasi-functor. On the other hand, let
the quasi-functor
$$ \tilde{\Xi}^P\colon 
\monfcat\basecat \rightarrow \widetilde{\fcat\basecat} $$ 
be the composition of 
$\eqref{eqn-homotopy-retraction-of-tildeF_V-onto-oplus-hperfS^nV}$
with the formal inverse of the quasi-equivalence
$\eqref{eqn-quasi-equiv-oplus-hperf-S^nV-to-monfcat}$. 
By Corollary \ref{cor-oplus-hperf-S^nV-is-a-retract-of-H_V}, 
$\tilde{\Phi}_\basecat$ is a homotopy left inverse of $\tilde{\Xi}^P$. 

Define the numerical Grothendieck groups\index{numerical Grothendieck group} $\numGgp{\monhcat\basecat,\, \kk}$ 
and $\numGgp{\widetilde{\fcat\basecat},\, \kk}$ similarly to the definition of 
$\numGgp{\hcat\basecat,\,\kk}$ in Section~\ref{subsec:grothendieck_groups}.
Namely, they are the quotients of $\mathrm{K}_0(\monhcat\basecat,\,\kk)$
and $\mathrm{K}_0(\widetilde{\fcat\basecat}, \kk)$ under the images of the
kernel of the Euler form on $\bigoplus_{n} \mathrm{K}_0(\symbc n, \kk)$
under $\Xi_P$ and 
$\eqref{eqn-homotopy-retraction-of-tildeF_V-onto-oplus-hperfS^nV}$, 
respectively. Then $\numGgp{\hcat\basecat,\,\kk}$ is the idempotent 
modification of $\numGgp{\monhcat\basecat,\,\kk}$.

By Lemma \ref{lemma-oplus-hperf-S^n-V-homotopy-retract-of-tilde-F_V}, 
the map of $K$-groups induced by $\tilde{\Xi}^P$ 
is injective. By our definitions, it descends 
to an injective map of numerical $K$-groups, and so does 
any left inverse of it. We thus obtain:
\begin{Corollary}
	\label{cor-numgp-oplus-S^nV-into-tildeF_V}
	The following composition is the identity map:
	$$ \numGgp{\monfcat\basecat,\, \kk}
	\xrightarrow{\tilde{\Xi}^P}
	\numGgp{\widetilde{\fcat\basecat},\, \kk}
	\xrightarrow{\tilde{\Phi}_\basecat}
	\numGgp{\monfcat\basecat,\, \kk}.$$ 
\end{Corollary}

Corollary~\ref{cor-numgp-oplus-S^nV-into-tildeF_V} together with 
the morphism $\phi$ of~\eqref{eq:fock_embedding} gives an embedding 
of the classical Fock space into the numerical Grothendieck group 
of the category $\widetilde{\fcat\basecat}$:
\[ 
\cfalg{\basecat} =
\bigoplus_{n}\cfalg{\basecat}^n
\xrightarrow{\phi}
\bigoplus_{n} \numGgp{\symbc n,\, \kk} 
\simeq 
\numGgp{\monfcat\basecat,\, \kk} 
\xrightarrow{\tilde{\Xi}^P}
\numGgp{\widetilde{\fcat\basecat},\, \kk}
\]
where
\[ \bigoplus_{n}\cfalg{\basecat}^n \cong \bigoplus_{n \geq 0} \bigoplus_{\lambda \dashv n} \bigotimes_i\Sym^{r(\lambda)_i} \numGgp{ \basecat,\, \kk} \]
and $r(\lambda)_i$ is the number of parts of size $i$ in $\lambda$.

We now prove a converse to Conjecture~\ref{conj:pi_iso}.

\begin{Theorem} 
	\label{thm:piisothenphi}
	If $\pi\colon \halg\basecat \to \numGgp{\hcat\basecat,\,\kk}$ is an
	isomorphism, then so are $\phi$ and $\tilde{\Xi}^P$:
	\[
	\bigoplus_{n \geq 0} \bigoplus_{\lambda \dashv n} \bigotimes_i\Sym^{r(\lambda)_i} \numGgp{ \basecat,\, \kk} \simeq 
	\bigoplus_{n \geq 0} \numGgp{\symbc n,\, \kk}
	\simeq  \numGgp{\widetilde{\fcat\basecat},\, \kk}.
	\]
\end{Theorem}

\begin{proof}
	Let $I$ be the left ideal of $\chalg\basecat$ generated by $q_{[a]}^{(n)}$ with $n > 0$ and $a \in \basecat$. We have
	\begin{equation}
		\label{eqn-diagram-of-various-Knum-groups}
		\begin{tikzcd}
			\chalg\basecat
			\ar[two heads]{r}{\text{quotient by } I}
			\ar[hookrightarrow]{d}[']{\underline{\pi}}
			&
			\cfalg{\basecat}
			\ar[hookrightarrow]{d}{\phi}
			\\
			\numGgp{\monhcat\basecat,\, \kk} 
			\ar{r}{\Phi_\basecat}
			\ar[two heads]{dr}[']{\text{Drinfeld}}
			&
			\numGgp{\monfcat\basecat,\, \kk} 
			\\
			&
			\numGgp{\widetilde{\fcat\basecat},\, \kk}. 
			\ar[two heads]{u}[']{\tilde{\Phi}_\basecat}
		\end{tikzcd}
	\end{equation}
	The Drinfeld quotient induces a surjective map of the $K$-groups by
	\cite[Proposition~VIII.3.1]{grothendieck1977cohomologie}. By our
	definitions, this descends to the surjective map 
	of the numerical $K$-groups in \eqref{eqn-diagram-of-various-Knum-groups}. 
	
	By assumption of the Theorem, the map $\underline{\pi}$ is an
	isomorphism. By \eqref{eqn-diagram-of-various-Knum-groups}, 
	the injective map $\phi$ is then surjective, 
	and thus an isomorphism. Now observe that the Drinfeld quotient 
	map kills $\underline{\pi}(I)$. The surjective map $\tilde{\Phi}$ is
	therefore injective and thus an isomorphism. Its right inverse 
	$\tilde{\Xi}^P$ is then also an isomorphism. 
\end{proof}

\section{Reconstruction of the base category}

It is natural to ask to what extent we can recover 
the base category $\basecat$ from its Heisenberg category $\hcat\basecat$. 
Given the nature of our construction, 
the best we can hope for is to recover $\basecat$ up to
Morita equivalence. This recovers $\hperf\basecat$, that is --
the compact derived category $\catDc(\basecat)$. 

We are not allowed to use our categorical Fock space $\fcat\basecat$
in this reconstruction as it is not built from $\hcat\basecat$,
but directly from $\basecat$. In particular, $\fcat\basecat$
contains $\hperf\basecat$ as 
the $1$-morphism category $\homm_{\fcat\basecat}(0,1)$. 
However, this gives us our strategy: we obtain our categorical
Fock space quotient $\widetilde{\fcat\basecat}$ intrinsically from  
$\hcat\basecat$ together with $\mathbb{Z}$-grading which remembers
the flattening 
\[
\monhcat\basecat = \bigoplus_{n\in \ZZ} \Hom_{\hcat\basecat}(0,n). 
\]
If we could show that the natural functor of 
Lemma \ref{lemma-oplus-hperf-S^n-V-homotopy-retract-of-tilde-F_V}
\begin{equation*}
	\bigoplus \hperf \symbc{n}
	\xrightarrow{\eqref{eqn-homotopy-retraction-of-tildeF_V-onto-oplus-hperfS^nV}}
	\widetilde{\fcat\basecat}. 
\end{equation*}
is a quasi-equivalence, we could recover $\hperf \basecat$ as $1$-graded
part $\widetilde{\fcat\basecat^1}$ of $\widetilde{\fcat\basecat}$. 
In Lemma \ref{lemma-oplus-hperf-S^n-V-homotopy-retract-of-tilde-F_V} we come
tantalisingly close: we show
\eqref{eqn-homotopy-retraction-of-tildeF_V-onto-oplus-hperfS^nV} to be 
quasi-faithful and quasi-essentially surjective on objects. In fact, 
in the proof of Lemma
\ref{lemma-oplus-hperf-S^n-V-homotopy-retract-of-tilde-F_V} we show
that it is also quasi-full on those morphisms in $\widetilde{\fcat\basecat}$
which come from the perfect hull of $\hcat*\basecat$. The only
morphisms we can't get so far are those added by taking the two
Drinfeld quotients -- the first one to get $\hcat\basecat$ and the second
one to get $\widetilde{\fcat\basecat}$. 

We conjecture that one can get even these and thus
\eqref{eqn-homotopy-retraction-of-tildeF_V-onto-oplus-hperfS^nV} 
is a quasi-equivalence. This would allow one to recover $\hperf \basecat$ as 
$\widetilde{\fcat\basecat^1}$. For the moment, however, we only have:

\begin{Lemma}
	\label{lem:HoDGCatequivinducessubcatequiv}
	Let $\basecat$ and $\mathcal{W}$ be smooth and proper \dg categories
	and assume that there is a quasi-equivalence of $\HoDGCat$-enriched 
	bicategories which is the identity on objects:
	\[ \hcat\basecat \simeq \hcat{\mathcal{W}}. \] 
	Then:
	\begin{enumerate}
		\item There is a $\mathbb{Z}$-graded quasi-equivalence 
		\[\widetilde{\fcat\basecat} \simeq \widetilde{\fcat{\mathcal{W}}}.\]
		\item\label{it:HoDGCatequivinducessubcatequiv1} There are
		quasi-faithful quasi-essentially surjective functors 
		\[ \hperf\basecat \rightarrow \widetilde{\fcat\basecat^1} 
		\simeq \widetilde{\fcat{\mathcal{W}}^1} \leftarrow \hperf\mathcal{W}. \] 
	\end{enumerate}
\end{Lemma}
\begin{proof}
	For the first claim, recall that we construct the categorical Fock 
	space quotient $\widetilde{\fcat\basecat}$ as the Drinfeld quotient 
	of $\hcat\basecat$ by the left ideal $I$ generated by objects $\QQ_a$
	for $a \in \basecat$. We can equivalently take $I$ to be the left
	ideal generated by all $1$-morphisms in $\homm_{\hcat\basecat}(n,n-k)$ 
	for $k > 0$. Since the quasi-equivalence 
	$\hcat\basecat \simeq \hcat{\mathcal{W}}$ is identity on the objects $n
	\in \mathbb{Z}$ it preserves this ideal and hence descends to a
	quasi-equivalence 
	$\widetilde{\fcat\basecat} \simeq \widetilde{\fcat{\mathcal{W}}}$.
	
	The second claim follows directly from Lemma
	\ref{lemma-oplus-hperf-S^n-V-homotopy-retract-of-tilde-F_V}. 
\end{proof}

This is enough to show that the Heisenberg categories
of $\catDGCoh{\ps 1}$ and $\catDGCoh{\mathrm{pt} \sqcup \mathrm{pt}}$
are distinct:  

\begin{Example} For the categories $\catDGCoh{\ps 1}$ and $\catDGCoh{\mathrm{pt} \sqcup \mathrm{pt}}$ of Example~\ref{Ex:ps1pt2} our Lemma~\ref{lem:HoDGCatequivinducessubcatequiv} is still enough to see that
	\[ \hcat{\catDGCoh{\ps 1}}  \not\simeq \hcat{\catDGCoh {\mathrm{pt} \sqcup \mathrm{pt}}}.\]
	The decomposition $\catDGCoh{\mathrm{pt} \sqcup \mathrm{pt}} = \catDGCoh{\mathrm{pt}} \oplus \catDGCoh{\mathrm{pt}}$ induces a decomposition
	\[ \Hom_{\hcat{\catDGCoh{\mathrm{pt} \sqcup \mathrm{pt}}}'}(0,1) =  \Hom_{\hcat{\catDGCoh{\mathrm{pt}}}'}(0,1) \oplus \Hom_{\hcat{\catDGCoh{\mathrm{pt}}}'}(0,1)  \]
	as follows. The objects of the Hom-space on the LHS consists of words with one more P than Q. There is no morphism between two such words if the difference of P's and Q's indexed by \emph{one} of the two generating objects is positive in one of the words but nonpositive in the other word. This decomposition then induces an orthogonal decomposition
	\[ \widetilde{\fcat{\catDGCoh{\mathrm{pt}}}^1} \oplus \widetilde{\fcat{\catDGCoh{\mathrm{pt}}}^1}.\]
	{ By Lemma~\ref{lem:HoDGCatequivinducessubcatequiv}~\ref{it:HoDGCatequivinducessubcatequiv1}, 
		there exists a faithful and essentially surjective functor \[\catDbCoh{\ps 1}= \Hzero(\catDGCoh{\ps 1}) \to \Hzero(\widetilde{\fcat{\catDGCoh{\ps 1}}}).\] 
		Therefore, if $\Hzero(\widetilde{\fcat{\catDGCoh{\ps 1}}})$ had an orthogonal decomposition, so would have $\catDbCoh{\ps 1}$. But this would imply that $\ps 1$ is disconnected.}
\end{Example}


\backmatter
\bibliographystyle{amsplain}
\bibliography{references}
\printindex

\end{document}